\documentclass[12pt]{amsart}
\usepackage{fullpage}
\usepackage{amsmath}
\usepackage{amssymb}
\usepackage{verbatim}
\usepackage[all]{xy}
\usepackage{color}
\usepackage[dvips]{epsfig}
\usepackage{graphicx}
\def\p{\partial}
\newtheorem{Theorem}{Theorem}[section]
\newtheorem{Definition}{Definition}[section]

\newtheorem{Lemma}[Theorem]{Lemma}
\newtheorem{Proposition}[Theorem]{Proposition}
\newtheorem{Example}[Theorem]{Example}

\newtheorem{Corollary}[Theorem]{Corollary}

\newtheorem{Remark}{Remark}[section]
\numberwithin{equation}{section}
\newcommand{\oD}{ \!{\buildrel \circ 
\over D}^{_{_{_{\mbox{{\small $_{q+1}$}}}}}}}

\newcommand{\oDnp}{ \!{\buildrel \circ 
\over D}^{_{_{_{\mbox{{\small $_{n-p}$}}}}}}}

\newcommand{\oDqm}{ \!{\buildrel \circ 
\over D}_{-}^{_{_{_{\mbox{{\small $_{q}$}}}}}}}

\def\N{{\mathcal N}}
\def\C{{\mathcal C}}
\def\F{{\mathcal F}}
\def\W{{\mathcal W}} 
\def\U{{\mathcal U}}
\def\V{{\mathcal V}}
\def\H{{\mathcal H}}

\def\M{{\mathcal M}} 

\def\G{{\mathcal G}}
\def\Riem{{\mathcal R}{\mathrm i}{\mathrm e}{\mathrm m}}

\begin{document}
\author{Mark Walsh} 
\title{Metrics of positive scalar curvature and
  generalised Morse functions, part 1}

\maketitle
\centerline{Department of Mathematics}
\centerline{University of Oregon} 
\centerline{Eugene, OR 97403}
\centerline{USA}

\vspace{1.0cm} 

\begin{abstract}
  It is well known that isotopic metrics of positive scalar curvature
  are concordant. Whether or not the converse holds is an open
  question, at least in dimensions greater than four. We show that for
  a particular type of concordance, constructed using the surgery
  techniques of Gromov and Lawson, this converse holds in the case of
  closed simply connected manifolds of dimension at least five.
\end{abstract}

\tableofcontents
\newpage

\section{Introduction}\label{intro}
\vspace{1cm}

\subsection{Background} Let $X$ be a smooth closed manifold. We denote
by $\Riem (X)$, the space of Riemannian metrics on $X$. Contained
inside $\Riem(X)$, as an open subspace, is the space $\Riem^{+}(X)$
which consists of metrics on $X$ which have positive scalar curvature
({\em psc-metrics}). The problem of whether or not this space is
non-empty (i.e. $X$ admits a psc-metric) has been extensively studied,
see \cite{GL1}, \cite{GL2}, \cite{RS} and \cite{S}. In particular,
when $X$ is simply connected and $\dim X\geq 5$, it is known that $X$
always admits a psc-metric when $X$ is not a spin manifold and, in the
case when $X$ is spin, it admits such a metric if and only if the
index $\alpha(X)\in KO_n$ of the Dirac operator vanishes, see
\cite{RS} and \cite{S}. Considerably less is known about the topology
of the space $\Riem^{+}(X)$, even for such manifolds as the sphere
$S^{n}$. For example, it is known that $\Riem^{+}(S^{2})$ is
contractible (as is $\Riem^{+}(\mathbb{R}P^{2})$), see \cite{RS}, but
little is known about the topology of $\Riem^{+}(S^{n})$ when $n\geq
3$. In this paper, we will focus on questions relating to the
path-connectivity of $\Riem^{+}(X)$.

Metrics which lie in the same path component of $\Riem^{+}(X)$ are
said to be {\em isotopic}. Two psc-metrics $g_0$ and $g_1$ on $X$ are
said to be {\em concordant} if there is a psc-metric $\bar{g}$ on the
cylinder $X\times I$ ($I=[0,1]$) which near $X\times \{0\}$ is the
product $g_0+dt^{2}$, and which near $X\times \{1\}$ is the product
$g_1+dt^{2}$. It is well known that isotopic metrics are concordant,
see Lemma \ref{isoconc} below. It is also known that concordant
metrics need not be isotopic when $\dim X =4$, where the difference
between isotopy and concordance is detected by the Seiberg-Witten
invariant, see \cite{Ru}. However, in the case when $\dim X\geq 5$,
the question of whether or not concordance implies isotopy is an open
problem and one we will attempt to shed some light on.

Before discussing this further, it is worth mentioning that the only
known method for showing that two psc-metrics on $X$ lie in distinct
path components of $\Riem^{+}(X)$, is to show that these metrics are
not concordant. For example, index obstruction methods may be used to
exhibit a countable inifinite collection of distinct concordance
classes for $X=S^{4k-1}$ with $k>1$, implying that the space
$\Riem^{+}(S^{4k-1})$ has at least as many path components, see
\cite{Carr} (or Example \ref{bottex} below for the case when
$k=2$). In \cite{BG}, the authors show that if $X$ is a connected spin
manifold with $\dim X=2k+1\geq 5$ and $\pi_1(X)$ is non-trivial and
finite, then $\Riem^{+}(X)$ has inifinitely many path components
provided $\Riem^{+}(X)$ is non-empty. Again, this is done by
exhibiting inifinitely many distinct concordance classes. For a
general smooth manifold $X$, understanding $\pi_0(\Riem^{+}(X))$ is
contingent on answering the following open questions?

(i) Are there more concordance classes undetected by the index theory?

(ii) When are concordant metrics isotopic?

\noindent 
For more on the first of these problems, the reader is referred to
\cite{S0} and \cite{RS}. We will focus our attention on the second
problem.

A fundamental difficulty when approaching this question is that an
arbitrary concordance may be extraordinarily complicated.  For
example, let $g_s, s\in I$ denote an isotopy in the space
$\Riem^{+}(S^{n})$. After an appropriate rescaling, see Lemma
\ref{isoconc}, we may assume that the warped product metric
$\bar{h}=g_t+dt^{2}$, on the cylinder $S^{n}\times I$, has positive
scalar curvature and a product structure near the boundary, i.e. is a
concordance of $g_0$ and $g_1$. Now let $g$ be any psc-metric on the
sphere $S^{n+1}$ (this metric may be very complicated indeed).  It is
possible to construct a psc-metric $\bar{g}$ on $S^{n}\times I$ by
taking a connected sum
$$
\bar{g}=\bar{h}\# g,
$$
see \cite{GL1}.  As this construction only alters the metric $\bar{h}$
on the interior of the cylinder, the resulting metric, $\bar{g}$, is
still a concordance of $g_0$ and $g_1$, see
Fig. \ref{difficultconc}. Unlike the concordance $\bar{h}$ however,
$\bar{g}$ could be arbitrarily complicated. In some sense, this makes
$\bar{g}$ ``unrecognisable" as an isotopy.

\begin{figure}[htbp]

\begin{picture}(0,0)%
\includegraphics{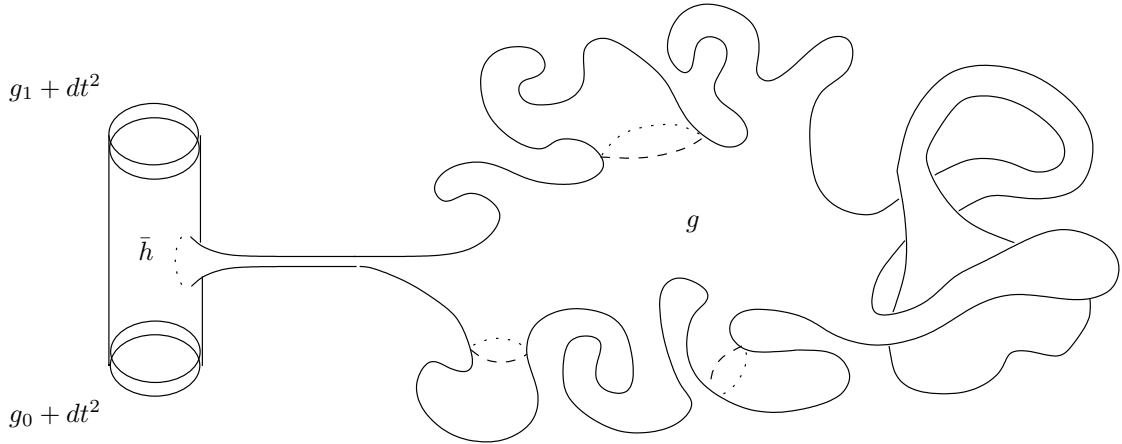}
\end{picture}%
\setlength{\unitlength}{3947sp}%
\begingroup\makeatletter\ifx\SetFigFont\undefined%
\gdef\SetFigFont#1#2#3#4#5{%
  \reset@font\fontsize{#1}{#2pt}%
  \fontfamily{#3}\fontseries{#4}\fontshape{#5}%
  \selectfont}%
\fi\endgroup%
\begin{picture}(6996,2775)(1249,-3013)
  \put(2076,-1861){\makebox(0,0)[lb]{\smash{{\SetFigFont{10}{8}{\rmdefault}{\mddefault}{\updefault}{\color[rgb]{0,0,0}$\bar{h}$}%
        }}}}
  \put(1264,-824){\makebox(0,0)[lb]{\smash{{\SetFigFont{10}{8}{\rmdefault}{\mddefault}{\updefault}{\color[rgb]{0,0,0}$g_1+dt^{2}$}%
        }}}}
  \put(1264,-2874){\makebox(0,0)[lb]{\smash{{\SetFigFont{10}{8}{\rmdefault}{\mddefault}{\updefault}{\color[rgb]{0,0,0}$g_0+dt^{2}$}%
        }}}}
  \put(5514,-1649){\makebox(0,0)[lb]{\smash{{\SetFigFont{10}{8}{\rmdefault}{\mddefault}{\updefault}{\color[rgb]{0,0,0}$g$}%
        }}}}
\end{picture}%
\caption{The concordance $\bar{g}$ on $S^{n}\times I$, formed by
  taking a conncected sum of metrics $\bar{h}$ and $g$.}
\label{difficultconc}
\end{figure}

\noindent Consequently, we will not approach this problem at the level of arbitrary concordance. Instead, we will restrict our attention to concordances which are constructed by a particular application of the surgery technique of Gromov and Lawson. Such concordances will be called {\em Gromov-Lawson concordances}. 

Before discussing the relationship between surgery and concordance, it is worth recalling how the surgery technique alters a psc-metric. The Surgery Theorem of Gromov-Lawson and Schoen-Yau states that any manifold 
$X'$ which is obtained from $X$ via a codimension$\geq 3$ surgery admits a psc-metric, see \cite{GL1}
and \cite{SY}. In their proof, Gromov and Lawson replace the metric $g$ with a psc-metric which is standard in a tubular neighbourhood of the embedded surgery sphere. More precisely, let $ds_{n}^{2}$ denote the standard round metric on the sphere $S^{n}$. We denote by $g_{tor}^{n}(\delta)$, the metric on the disk $D^{n}$ which, near $\p D^{n}$, is the Riemannian cylinder $\delta^{2}ds_{n-1}^{2}+dr^{2}$ and which near the centre of $D^{n}$ is the round metric $\delta^{2}ds_{n}^{2}$. The metric $g_{tor}^{n}(\delta)$ is known as a {\em torpedo metric}, see section \ref{prelim} for a detailed construction. For sufficiently small $\delta>0$ and provided $n\geq 3$, the scalar curvature of this metric can be bounded below by an arbitrarily large positive constant. Now, let $(X,g)$ be a smooth $n$-dimensional Riemannian 
manifold of positive scalar curvature and let $S^{p}$ denote an embedded $p$-sphere in $X$ with trivial normal bundle and with $p+q+1=n$ and $q\geq 2$. The metric $g$ can be replaced by a psc-metric on $X$ which, on a tubular neighbourhood of $S^{p}$, is the standard product $ds_{p}^{2}+g_{tor}^{q+1}(\delta)$ for some appropriately small $\delta$. In turn, surgery may be performed on this standard piece to obtain a psc-metric $g'$ on $X'$, which on the handle $D^{p+1}\times S^{q}$ is the standard product $g_{tor}^{p+1}+\delta^{2}ds_{q}^{2}$, see Fig. \ref{introsurgerypicture}.

\begin{figure}[htbp]
\begin{picture}(0,0)%
\includegraphics{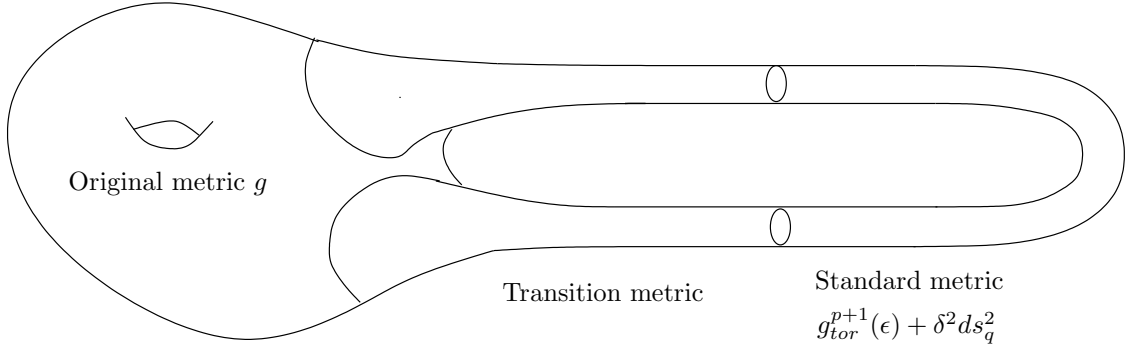}
\end{picture}%
\setlength{\unitlength}{3947sp}%
\begingroup\makeatletter\ifx\SetFigFont\undefined%
\gdef\SetFigFont#1#2#3#4#5{%
  \reset@font\fontsize{#1}{#2pt}%
  \fontfamily{#3}\fontseries{#4}\fontshape{#5}%
  \selectfont}%
\fi\endgroup%
\begin{picture}(7038,2154)(330,-2515)
\put(726,-1548){\makebox(0,0)[lb]{\smash{{\SetFigFont{10}{8}{\rmdefault}{\mddefault}{\updefault}{\color[rgb]{0,0,0}Original metric $g$}%
}}}}
\put(3451,-2248){\makebox(0,0)[lb]{\smash{{\SetFigFont{10}{8}{\rmdefault}{\mddefault}{\updefault}{\color[rgb]{0,0,0}Transition metric}%
}}}}
\put(5414,-2174){\makebox(0,0)[lb]{\smash{{\SetFigFont{10}{8}{\rmdefault}{\mddefault}{\updefault}{\color[rgb]{0,0,0}Standard metric}%
}}}}
\put(5426,-2449){\makebox(0,0)[lb]{\smash{{\SetFigFont{10}{8}{\rmdefault}{\mddefault}{\updefault}{\color[rgb]{0,0,0}$g_{tor}^{p+1}(\epsilon)+\delta^{2}ds_{q}^{2}$}%
}}}}
\end{picture}%
\caption{The psc-metric $g'$ obtained on $X'$ by the Gromov-Lawson construction}
\label{introsurgerypicture}
\end{figure}

There is an important strengthenning of this technique whereby the metric $g$ is extended over the trace of the surgery to obtain a psc-metric $\bar{g}$ which is a product metric near the boundary. This is sometimes referred to as the Improved Surgery Theorem, see \cite{Gajer}. Suppose $\{W; X_0, X_1\}$ is a smooth compact cobordism of closed $n$-manifolds $X_0$ and $X_1$, i.e. $\p W=X_0 \sqcup X_1$, and $f:W\rightarrow I$ is a Morse function with $f^{-1}(0)=X_0$ and $f^{-1}(1)=X_1$ and whose critical points lie in the interior of $W$. The Morse function $f$ gives rise to a decomposition of $W$ into elementary cobordisms. Let us assume that each elementary cobordism is the trace of a codimension$\geq 3$ surgery. This means that each critical point of $f$ has index $\leq n-2$. Roughly speaking, such Morse functions will be called ``admissible". It is now possible to extend a psc-metric $g_0$ on $X_0$ to a psc-metric $\bar{g}$ on $W$ which is a product near the boundary $\p W$, see Theorem \ref{GLcob} below. In particular, the restriction $g_1=\bar{g}|_{X_1}$ is a psc-metric on $X_1$. This method is a powerful tool for constructing new psc-metrics as the following example demonstrates.   

\begin{figure}[htbp]
\begin{picture}(0,0)%
\includegraphics{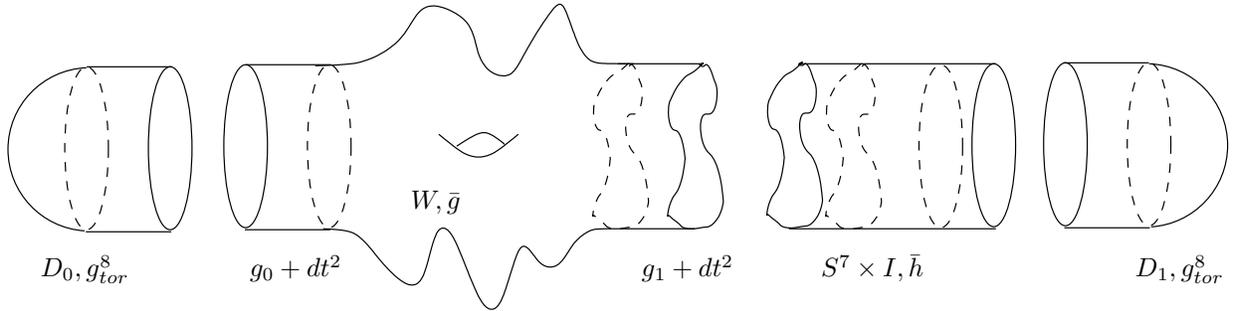}%
\end{picture}%
\setlength{\unitlength}{3947sp}%
\begingroup\makeatletter\ifx\SetFigFont\undefined%
\gdef\SetFigFont#1#2#3#4#5{%
  \reset@font\fontsize{#1}{#2pt}%
  \fontfamily{#3}\fontseries{#4}\fontshape{#5}%
  \selectfont}%
\fi\endgroup%
\begin{picture}(7678,1941)(862,-3979)
\put(1076,-3761){\makebox(0,0)[lb]{\smash{{\SetFigFont{10}{8}{\rmdefault}{\mddefault}{\updefault}{\color[rgb]{0,0,0}$D_0, g_{tor}^{8}$}%
}}}}
\put(3401,-3336){\makebox(0,0)[lb]{\smash{{\SetFigFont{10}{8}{\rmdefault}{\mddefault}{\updefault}{\color[rgb]{0,0,0}$W,\bar{g}$}%
}}}}
\put(5976,-3761){\makebox(0,0)[lb]{\smash{{\SetFigFont{10}{8}{\rmdefault}{\mddefault}{\updefault}{\color[rgb]{0,0,0}$S^{7}\times I, \bar{h}$}%
}}}}
\put(7951,-3761){\makebox(0,0)[lb]{\smash{{\SetFigFont{10}{8}{\rmdefault}{\mddefault}{\updefault}{\color[rgb]{0,0,0}$D_1, g_{tor}^{8}$}%
}}}}
\put(4851,-3761){\makebox(0,0)[lb]{\smash{{\SetFigFont{10}{8}{\rmdefault}{\mddefault}{\updefault}{\color[rgb]{0,0,0}$g_1+dt^{2}$}%
}}}}
\put(2389,-3761){\makebox(0,0)[lb]{\smash{{\SetFigFont{10}{8}{\rmdefault}{\mddefault}{\updefault}{\color[rgb]{0,0,0}$g_0+dt^{2}$}%
}}}}
\end{picture}%

\caption{The existence of a concordance $(S^{7}\times I, \bar{h})$ between $g_1$ and $g_0=ds_{7}^{2}$ would imply the existence of a psc-metric on $B$, which is impossible.}
\label{bott}
\end{figure} 

\begin{Example}\label{bottex} {\rm Let $B=B^8$ be a Bott manifold, i.e. an 8-dimensional
closed simply connected spin manifold with $\alpha(B)=1$, see
\cite{Jo} for a geometric construction of such a manifold. The fact that 
$\alpha(B)\neq 0$ means that $B$ does not admit a psc-metric. Let $W=B\setminus (D_0\sqcup D_1)$ denote the smooth manifold obtained by removing a disjoint pair of $8$-dimensional disks $D_0$ and $D_1$ from $B$. The boundary of $W$ is a pair of disjoint $7$-dimensional smooth spheres, which we denote $S_0^{7}$ and $S_1^{7}$ respectively. It is possible, although we do not include the details here, to equip $W$ with an admissible Morse function. This decomposes $W$ into a union
of elementary cobordisms, each the trace of a codimension$\geq 3$ surgery.
Thus, we can extend the standard round metric $g_0=ds_{7}^{2}$ 
from the boundary component $S_{0}^{7}$ to a psc metric $\bar{g}$ on $W$, which is a
product metric near both boundary components. In particular, the metric $\bar{g}$
restricts to a psc-metric $g_1$ on $S^7_1$. This metric however, is {\bf not} concordant (and hence not isotopic) to $g_0$. This is because the existence of a concordance $\bar{h}$ of $g_1$ and $g_0=ds_{7}^{2}$, would give rise to a psc-metric $g_B$ on $B$ (see Fig. \ref{bott}), defined by taking the union
\begin{equation*}
(B,g_{B})=(D_{0},g_{tor}^{8})\cup(W, \bar{g})\cup(S^{7}\times I, \bar{h})\cup (D_1, g_{tor}^{8}),
\end{equation*}
something we know to be impossible.}
\end{Example}

\subsection{Main Results}  
In Theorem \ref{GLcob}, we generalise the so-called Improved Surgery Theorem, as well as correcting an error from the proof in \cite{Gajer}, see Remark \ref{Gajercomment} in section \ref{surgerysection}.
 
\begin{Theorem} \label{GLcob}
Let $\{W^{n+1};X_0,X_1\}$ be a smooth compact cobordism. Suppose $g_0$ is a metric of positive scalar curvature on $X_0$ and $f:W\rightarrow I$ is an admissible Morse function. Then there is a psc-metric $\bar{g}=\bar{g}(g_0,f)$ on $W$ which extends $g_0$ and has a product structure near the boundary.
\end{Theorem}

We call the metric $\bar{g}=\bar{g}(g_0,f)$, a {\em Gromov-Lawson cobordism (GL-cobordism) with respect to $g_0$ and $f$}.
Essentially, the metric $\bar{g}$ restricts on a regular level set of $f$ to the metric obtained by repeated application of the surgery technique with respect to each of the critical points below that level set. In the case when $W$ is the cylinder $X\times I$, the metric $\bar{g}$ is a concordance of the metrics $g_0$ and $g_1=\bar{g}|_{X\times\{1\}}$. It will be referred to as a {\em Gromov-Lawson concordance (GL-concordance) with respect to $g_0$ and $f$}, see Fig. \ref{glconcf}.

\begin{figure}[htbp]
\begin{picture}(0,0)%
\includegraphics{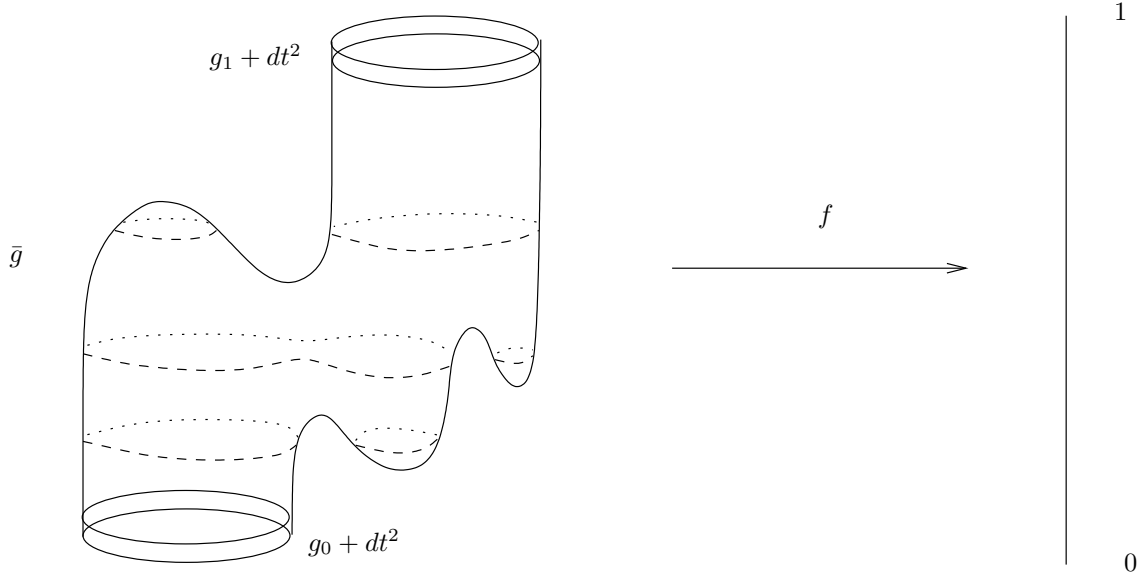}%
\end{picture}%
\setlength{\unitlength}{3947sp}%
\begingroup\makeatletter\ifx\SetFigFont\undefined%
\gdef\SetFigFont#1#2#3#4#5{%
  \reset@font\fontsize{#1}{#2pt}%
  \fontfamily{#3}\fontseries{#4}\fontshape{#5}%
  \selectfont}%
\fi\endgroup%
\begin{picture}(7230,3682)(1624,-3852)
\put(6914,-1599){\makebox(0,0)[lb]{\smash{{\SetFigFont{10}{8}{\rmdefault}{\mddefault}{\updefault}{\color[rgb]{0,0,0}$f$}%
}}}}
\put(8839,-3786){\makebox(0,0)[lb]{\smash{{\SetFigFont{10}{8}{\rmdefault}{\mddefault}{\updefault}{\color[rgb]{0,0,0}$0$}%
}}}}
\put(8776,-324){\makebox(0,0)[lb]{\smash{{\SetFigFont{10}{8}{\rmdefault}{\mddefault}{\updefault}{\color[rgb]{0,0,0}$1$}%
}}}}
\put(3096,-611){\makebox(0,0)[lb]{\smash{{\SetFigFont{10}{8}{\rmdefault}{\mddefault}{\updefault}{\color[rgb]{0,0,0}$g_1+dt^{2}$}%
}}}}
\put(3714,-3661){\makebox(0,0)[lb]{\smash{{\SetFigFont{10}{8}{\rmdefault}{\mddefault}{\updefault}{\color[rgb]{0,0,0}$g_0+dt^{2}$}%
}}}}
\put(1839,-1849){\makebox(0,0)[lb]{\smash{{\SetFigFont{10}{8}{\rmdefault}{\mddefault}{\updefault}{\color[rgb]{0,0,0}$\bar{g}$}%
}}}}
\end{picture}%

\caption{Obtaining a Gromov-Lawson concordance on the cylinder $X\times I$ with respect to a Morse function $f$ and a psc-metric $g_0$}
\label{glconcf}
\end{figure}
  
Any admissible Morse function $f:W\rightarrow I$ can be replaced by a Morse function denoted $1-f$, which has the gradient flow of $f$, but running in reverse. This function has the same critical points as $f$, however, each critical point of index $\lambda$ has been replaced with one of index $n+1-\lambda$. The following theorem can be obtained by ``reversing" the construction from Theorem \ref{GLcob}.

\begin{Theorem}\label{Reversemorse}
Let $\{W^{n+1};X_0,X_1\}$ be a smooth compact cobordism, $g_0$ a psc-metric on $X_0$ and $f:W\rightarrow I$, an admissible Morse function. Suppose that $1-f$ is also an admissible Morse function. Let $g_1=\bar{g}(g_0,f)|_{X_1}$ denote the restriction of the Gromov-Lawson cobordism $\bar{g}(g_0,f)$ to $X_1$. Let $\bar{g}(g_1,1-f)$ be a Gromov-Lawson cobordism with respect to $g_1$ and $1-f$ and let $g_0'=\bar{g}(g_1, 1-f)|_{X_0}$ denote the restriction of this metric to $X_0$. Then $g_0$ and $g_0'$ are canonically isotopic metrics in $\Riem^{+}(X_0)$.
\end{Theorem}

A careful analysis of the Gromov-Lawson construction shows that it can be applied continuously over a compact family of metrics as well as a compact family of embedded surgery spheres, see Theorem \ref{GLcompact} in section \ref{surgerysection}. It then follows that the construction of Theorem \ref{GLcob} can be applied continuously over a compact family of admissible Morse functions and so we obtain the following theorem. 

\begin{Theorem}\label{GLcobordismcompact}
Let $\{W, X_0, X_1\}$ be a smooth compact cobordim, $\mathcal{B}$, a compact continuous family of psc-metrics on $X_0$ and $\mathcal{C}$, a compact continuous family of admissible Morse functions on $W$. Then there is a continuous map 

\begin{equation*}
\begin{split}
\mathcal{B}\times \mathcal{C}&\longrightarrow \Riem^{+}(W)\\
(g_b,f_c)&\longmapsto \bar{g}_{b,c}=\bar{g}(g_b, f_c)
\end{split}
\end{equation*}

\noindent so that for each pair $(b,c)$, the metric $\bar{g}_{b,c}$ is the metric constructed in Theorem \ref{GLcob}.

\end{Theorem}

\begin{Remark}
In \cite{Che}, Chernysh also describes a parameterised version of the orginal Gromov-Lawson construction for compact families of psc-metrics.
\end{Remark}

In section \ref{GLconcordsection}
we construct an example of a GL-concordance on the cylinder $S^{n}\times I$. Here $g_0=ds_n^{2}$, the standard round metric and $f$ is an admissible Morse function with two critical points which have Morse indices $p+1$ and $p+2$ where $p+q+1=n$ and $q\geq 3$. The critical point of index $p+1$ corresponds to a $p$-surgery on $S^{n}$ resulting in a manifold diffeomorphic to $S^{p+1}\times S^{q}$. This is then followed by a $(p+1)$-surgery which restores the orginal manifold $S^{n}$. The restriction of the metric $\bar{g}(ds_{n}^{2}, f)$ to level sets of $f$ below, between and above these critical points is denoted $g_0$, $g_0'$ and $g_1$ respectively, see Fig. \ref{introconc}. The metric $g_1$ is also a psc-metric on $S^{n}$, but as Fig. \ref{introconc} suggests, looks radically different from the original metric $g_0$. Understanding why these metrics are in fact isotopic is crucial in proving our main result, stated below.

\begin{figure}[htbp]
\begin{picture}(0,0)%
\includegraphics{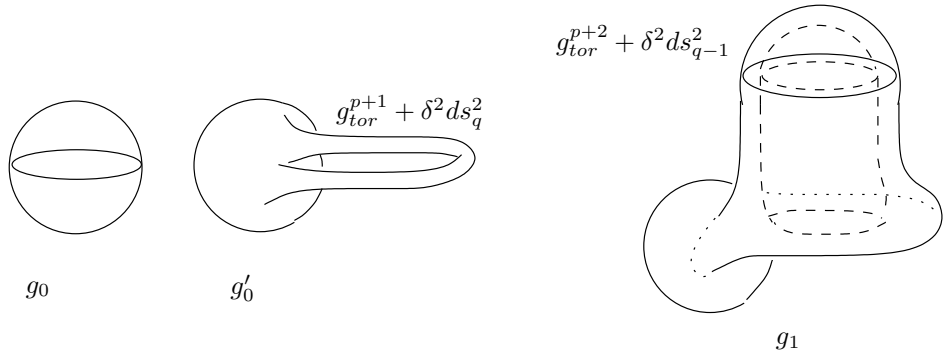}%
\end{picture}%
\setlength{\unitlength}{3947sp}%
\begingroup\makeatletter\ifx\SetFigFont\undefined%
\gdef\SetFigFont#1#2#3#4#5{%
  \reset@font\fontsize{#1}{#2pt}%
  \fontfamily{#3}\fontseries{#4}\fontshape{#5}%
  \selectfont}%
\fi\endgroup%
\begin{picture}(5878,2206)(978,-2540)
\put(1089,-2149){\makebox(0,0)[lb]{\smash{{\SetFigFont{10}{8}{\rmdefault}{\mddefault}{\updefault}{\color[rgb]{0,0,0}$g_0$}%
}}}}
\put(2376,-2161){\makebox(0,0)[lb]{\smash{{\SetFigFont{10}{8}{\rmdefault}{\mddefault}{\updefault}{\color[rgb]{0,0,0}$g_0'$}%
}}}}
\put(3039,-1061){\makebox(0,0)[lb]{\smash{{\SetFigFont{10}{8}{\rmdefault}{\mddefault}{\updefault}{\color[rgb]{0,0,0}$g_{tor}^{p+1}+\delta^{2}ds_{q}^{2}$}%
}}}}
\put(4421,-624){\makebox(0,0)[lb]{\smash{{\SetFigFont{10}{8}{\rmdefault}{\mddefault}{\updefault}{\color[rgb]{0,0,0}$g_{tor}^{p+2}+\delta^{2}ds_{q-1}^{2}$}%
}}}}
\put(5801,-2474){\makebox(0,0)[lb]{\smash{{\SetFigFont{10}{8}{\rmdefault}{\mddefault}{\updefault}{\color[rgb]{0,0,0}$g_1$}%
}}}}
\end{picture}%

\caption{Applying the Gromov-Lawson construction over a pair of cancelling surgeries of consecutive dimension}
\label{introconc}
\end{figure}

\begin{Theorem}\label{conciso} Let $X$ be a closed simply connected manifold with $dimX=n\geq5$ and let $g_0$ be a positive scalar curvature metric on $X$. Suppose $\bar{g}=\bar{g}(g_0,f)$ is a Gromov-Lawson concordance with respect to $g_0$ and an admissible Morse function $f:X\times I\rightarrow I$. Then the metrics $g_0$ and $g_1=\bar{g}|_{X\times\{1\}}$ are isotopic.
\end{Theorem}

\noindent The proof of Theorem \ref{conciso} takes place over sections \ref{doublesurgsection} and \ref{glconcisosection}. In section \ref{doublesurgsection} we prove the theorem in the case when $f$ has exactly two ``cancelling" critical points. This is the key geometric step and draws heavily from some important technical observations made in section \ref{prelim}. The general case then follows from Morse-Smale theory and the fact that the function $f$ can be deformed to one whose critical points are arranged in cancelling pairs.

\subsection{The connection with generalised Morse functions and Part II}\label{introthree}

This paper is the first part of a larger project to be completed in a subsequent paper. We will present here only a brief summary of the main results. The goal of this work is to better understand the space of Gromov-Lawson cobordisms on a smooth compact cobordism and in particular, the space of Gromov-Lawson concordances. More specifically, let $\{W;X_0,X_1\}$ be a smooth compact cobordism of closed $n$-manifolds $X_0$ and $X_1$. For the rest of this section we will assume that $W, X_0$ and $X_1$ are simply connected manifolds and $n\geq 5$. Let $g_0$ be a fixed psc-metric on $X_0$. We will consider the subspace of $\Riem^{+}(W)$ consisting of all Gromov-Lawson cobordisms which extend the metric $g_0$ from $X\times\{0\}$. This space is denoted $\G(g_0)$. Roughly speaking, Theorems \ref{GLcob} and \ref{GLcobordismcompact} allow us parameterise this space by admissible Morse functions. 

The space of admissible Morse functions on $W$ is denoted $\M^{adm}(W)$ and can be thought of as a subspace of the space of Morse functions $W\rightarrow I$, denoted $\M(W)$. A good deal is understood about the topology of the space $\M(W)$, in particular, see \cite{I3}. It is clear that this space is not path connected, as functions in the same path component must have the same number of critical points of the same index. The space $\G(g_0)$ therefore consists of separate components which correspond to the various path components of $\M^{adm}(W)$. This however, gives a rather misleading picture, as it is possible for appropriate pairs of Morse critical points to cancel, giving rise to a simpler handle decomposition of $W$. In the proof of Theorem \ref{conciso}, we discuss a corresponding ``geometric cancellation" which simplifies a psc-metric associated to this Morse function. In order to obtain a more complete picture of the space of GL-cobordisms, we need to incorporate this cancellation property into our description.

There is a natural setting in which to consider the cancellation of Morse critical points. Recall that near a critical point $w$, a Morse function $f\in\M(W)$ is locally equivalent to the map

\begin{equation*}
(x_1, \cdots, x_{n+1})\longmapsto -x_1^{2}\cdots-x_{p+1}^{2}+x_{p+2}^{2}+\cdots+x_{n+1}^{2}.
\end{equation*}

\noindent A critical point $w$ of a smooth function $f:W\rightarrow I$ is said to be of {\em birth-death} or {\em embryo} type if near $w$, $f$ is equivalent to the map

\begin{equation*}
(x_0, \cdots, x_{n})\longmapsto x_0^{3}-x_{1}^{2}\cdots-x_{p+1}^{2}+x_{p+2}^{2}+\cdots+x_{n}^{2}.
\end{equation*}

\noindent A {\em generalised Morse function} $f:W\rightarrow I$ is a smooth function with $f^{-1}(0)=X_0$, $f^{-1}(1)=X_1$ and whose singular set is contained in the interior of $W$ and consists of only Morse and birth-death critical points. There is a natural embedding of $\M(W)$ into the space of generalised Morse functions $\H(W)$. This is a path-connected space since birth-death singularities allow for the cancellation of Morse critical points of consecutive index, see \cite{Cerf}. Before going any further it is worth considering a couple of examples of this sort of cancellation.

\begin{Example}
{\rm The function $F(x,t)=x^{3}+tx$ can be thought of as a smooth family of functions $x\longmapsto F(x,t)$ parameterised by $t$. When $t<0$, the map $x\longmapsto F(x,t)$ is a Morse function with 2 critical points which cancel as a degenerate singularity of the function $x\longmapsto F(x,0)$. The function $x\longmapsto F(x,0)$ is an example of a generalised Morse function with a birth-death singularity at $x=0$.}
\end{Example}
 
\begin{Example}
{\rm More generally, any two Morse functions $f_0, f_1\in \M(W)$ may be connected by a path $f_t, t\in I$ in the space $\H(W)$ so that $f_t$ is Morse for all but finitely many $t\in I$. In Fig. \ref{genmorsetwo} we sketch using selected level sets, a path $f_t, t\in[-1,1]$, in the space $\H(S^{n}\times I)$ which connects a Morse function $f_{-1}$ with two critical points of consecutive Morse index to a Morse function $f_1$ which has no critical points. We will assume that the critical points of $f_{-1}$ lie on the level sets $f_{-1}=\frac{1}{4}$ and $f_{-1}=\frac{3}{4}$ and that $f_{0}$ has only a birth-death singularity on the level set $f_{0}=\frac{1}{2}$.

\begin{figure}[htbp]
\begin{picture}(0,0)%
\includegraphics{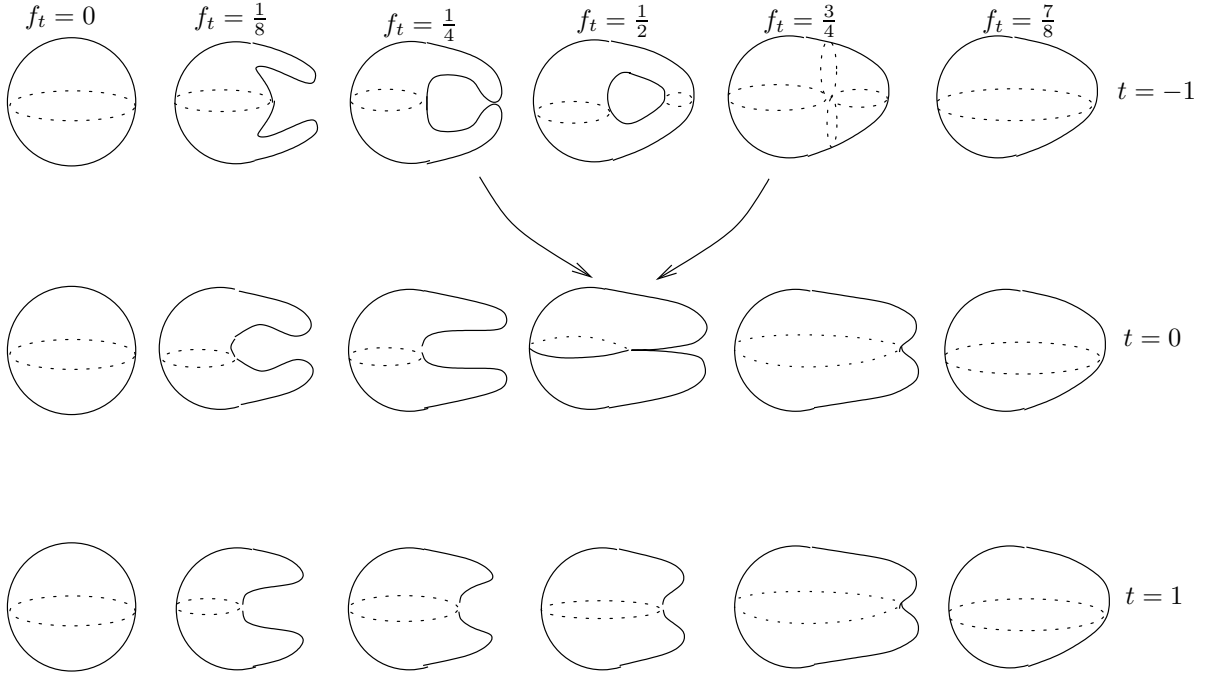}%
\end{picture}%
\setlength{\unitlength}{3947sp}%
\begingroup\makeatletter\ifx\SetFigFont\undefined%
\gdef\SetFigFont#1#2#3#4#5{%
  \reset@font\fontsize{#1}{#2pt}%
  \fontfamily{#3}\fontseries{#4}\fontshape{#5}%
  \selectfont}%
\fi\endgroup%
\begin{picture}(7044,4231)(1397,-5151)
\put(1514,-1074){\makebox(0,0)[lb]{\smash{{\SetFigFont{10}{8}{\rmdefault}{\mddefault}{\updefault}{\color[rgb]{0,0,0}$f_t=0$}%
}}}}
\put(2576,-1086){\makebox(0,0)[lb]{\smash{{\SetFigFont{10}{8}{\rmdefault}{\mddefault}{\updefault}{\color[rgb]{0,0,0}$f_t=\frac{1}{8}$}%
}}}}
\put(3764,-1124){\makebox(0,0)[lb]{\smash{{\SetFigFont{10}{8}{\rmdefault}{\mddefault}{\updefault}{\color[rgb]{0,0,0}$f_t=\frac{1}{4}$}%
}}}}
\put(4976,-1099){\makebox(0,0)[lb]{\smash{{\SetFigFont{10}{8}{\rmdefault}{\mddefault}{\updefault}{\color[rgb]{0,0,0}$f_t=\frac{1}{2}$}%
}}}}
\put(6151,-1099){\makebox(0,0)[lb]{\smash{{\SetFigFont{10}{8}{\rmdefault}{\mddefault}{\updefault}{\color[rgb]{0,0,0}$f_t=\frac{3}{4}$}%
}}}}
\put(7526,-1099){\makebox(0,0)[lb]{\smash{{\SetFigFont{10}{8}{\rmdefault}{\mddefault}{\updefault}{\color[rgb]{0,0,0}$f_{t}=\frac{7}{8}$}%
}}}}
\put(8376,-1549){\makebox(0,0)[lb]{\smash{{\SetFigFont{10}{8}{\rmdefault}{\mddefault}{\updefault}{\color[rgb]{0,0,0}$t=-1$}%
}}}}
\put(8414,-3099){\makebox(0,0)[lb]{\smash{{\SetFigFont{10}{8}{\rmdefault}{\mddefault}{\updefault}{\color[rgb]{0,0,0}$t=0$}%
}}}}
\put(8426,-4724){\makebox(0,0)[lb]{\smash{{\SetFigFont{10}{8}{\rmdefault}{\mddefault}{\updefault}{\color[rgb]{0,0,0}$t=1$}%
}}}}
\end{picture}%

\caption{Two Morse critical points cancelling at a birth-death singularity, from the point of view of selected level sets}
\label{genmorsetwo}

\end{figure}
}
\end{Example}

In order to capture this cancellation property in our description of the space of GL-cobordisms we specify a space of {\em admissible generalised Morse functions}, $\H^{adm}(W)$ on which a construction analogous to that of Theorem \ref{GLcob} can be applied, see Theorem A below. 
\\

\noindent {\bf Theorem A.} {\em Let $\{W^{n+1};X_0,X_1\}$ be a smooth compact cobordism. Let $g_0$ be a psc-metric on $X_0$ and $f:W\rightarrow I$, an admissible generalised Morse function. Then there is a psc-metric $\bar{g}_{gen}=\bar{g}_{gen}(g_0,f)$ on $W$, which extends $g_0$ and has a product structure near the boundary.}

The metric $\bar{g}(g_0,f)$ is called a {\em generalised Gromov-Lawson cobordism with respect to $g_0$ and $f$} and in the case when $W=X\times I$, this metric is called a {\em generalised Gromov-Lawson concordance}. In the case of admissible Morse functions, it was easy to extend the construction of Theorem \ref{GLcob}, as a continuous map on compact families, see Theorem \ref{GLcobordismcompact}. A far more difficult problem is proving an analogous theorem for admissible generalised Morse functions, in other words extending the construction of Theorem A as a continuous map on a compact family of generalised Morse functions. With such a theorem in hand however, it is possible to parameterise a space of {\em generalised Gromov-Lawson cobordisms}, which connects up the various path components of $\G(g_0)$ and presents a more realistic topological picture of what is going on. We denote this space by $\G^{gen}(g_0)$.

The main difficulty in proving this theorem is the geometric problem of continuously varying the construction of a GL-cobordism over a pair of cancelling critical points. The key ingredient which allows us to do this is Theorem \ref{conciso}. A convenient setting in which to prove an analogue of Theorem \ref{GLcobordismcompact} is described by Eliashberg and Mishachev in their work on ``wrinklings" of smooth maps, see \cite{El1} and \cite{El2}. Let $K$ be a smooth compact manifold of dimension $k$. Roughly speaking, a smooth map $f:W\times K\rightarrow K\times I$ is {\em wrinkled} if the restriction $f_{x}$ of $f$ to $W\times\{x\}$, for each $x\in K$, is a generalised Morse function and if the singular set of $f$ in $W\times K$ is a disjoint union of $k$ dimensional spheres, called wrinkles. The equator of each sphere consists of birth-death singularities, while the interiors of the northern and southern hemi-spheres consist of Morse critical points of consecutive index which cancel at the equator. After defining an appropriate notion of admissibility for a wrinkled map we can prove the following theorem. 
\\

\noindent{\bf Theorem B.} 
{\em Let $\{W;X_0,X_1\}$ be a smooth compact cobordism with $g_0$ a psc-metric on $X_0$. Let $K$ be a smooth compact manifold and $f:W\times K\rightarrow K\times I$ an admissible wrinkled map. Then there is a continuous map
\begin{equation*}
\begin{split}
K\longrightarrow&\Riem^{+}(W)\\
x\longmapsto &\bar{g}_{gen}^{x}=\bar{g}_{gen}(g_0, f_x)
\end{split}
\end{equation*}
where each $\bar{g}_{gen}^{x}$ is a generalised Gromov-Lawson cobordism.}
\\

\noindent {\bf Corollary C.} 
{\em Suppose $f_a, f_b\in\M^{adm}(W)$. Let ${g}_a=\bar{g}(g_0, f_a)|_{X_0}$ and ${g}_b=\bar{g}(g_0, f_b)|_{X_0}$ denote the restriction to $X_0$ of GL-cobordisms with respect to $f_a$ and $f_b$. Then $g_a$ and $g_b$ are isotopic metrics.}

\begin{Remark}
Corollary C is a generalisation of Theorem \ref{conciso}. However, the key geometric details which make Theorem B and Corollary C possible are contained in the proof of Theorem \ref{conciso}. 
\end{Remark}

\noindent Our main result is Theorem D below.
\\

\noindent {\bf Theorem D.} 
{\em The space of generalised GL-cobordisms $\G^{gen}(g_0)$ is weakly homotopy equivalent to the space of admissible generalised Morse functions $\H^{adm}(W)$.}
\\
 
In the case when $W$ is the cylinder $X\times I$, the space $\G^{gen}(g_0)$ consists of generalised GL-concordances. A considerable amount is already known about the space $\H(X\times I)$, of generalised Morse functions on the cylinder from the work of Cerf and Igusa, see \cite{Cerf}, \cite{I0}, \cite{I1}, \cite{I2} and \cite{I3}. Furthermore, from the work of Cohen in \cite{CO}, on the stable space of generalised Morse functions, it is possible to gain explicit information about some higher homotopy groups of $\H^{adm}(X\times I)$ and consequently about higher homotopy groups of $\G^{gen}(g_0)$. 
  
\subsection{Acknowledgements} This work forms part of the author's doctoral dissertation. I am deeply grateful to my advisor Boris Botvinnik for suggesting this problem and for his guidance over the years. Some of this work took place at the Royal Institute of Technology (KTH) in Stockholm, Sweden, as well as at SFB 478 - Geometrische Strukturen in der Mathematik in M{\"u}nster, Germany. My thanks to both institutions and in particular to M. Dahl, M. Joachim and W. L{\"u}ck for their hospitality. I am also grateful to W.M. Kantor at the University of Oregon for generous financial support from his NSF grant. Finally, it is a pleasure to thank David Wraith at NUI Maynooth, Ireland for many helpful conversations.  
  
\section{Definitions and preliminary results}\label{prelim}
\subsection{Isotopy and concordance in the space of 
metrics of positive scalar curvature}
\indent Throughout this paper, $X$ will denote a smooth closed
compact manifold of dimension $n$. In later sections we will also require
that $X$ be simply connected and that $n\geq 5$. We will denote by
$\Riem(X)$, the space of all Riemannian metrics on $X$. The topology
on this space is induced by the standard $C^k$-norm on Riemannian
metrics and defined $|g|_k=\max_{i\leq k}\sup_X|\nabla^{i}g|$. Here $\nabla$ 
is the Levi-Civita connection for some fixed reference metric and $|\nabla^{i} g|$ 
is the Euclidean tensor norm on $\nabla^{i}g$, see page 54 of \cite{P} for a definition.
Note that the topology on $\Riem(X)$ does not depend on the choice of reference metric. 
For our purposes it is sufficient (and convenient) to assume that $k=2$.

Contained inside $\Riem(X)$, as an open subspace, is the space
\begin{equation*}
\Riem^{+}(X)=\{g\in\Riem(X):R_g>0\}. 
\end{equation*}
Here $R_g:X\rightarrow \mathbb{R}$ denotes the scalar curvature of the
metric $g$, although context permitting we will sometimes denote the
scalar curvature of a metric as simply $R$. The space $\Riem^{+}(X)$
is the space of metrics on $X$ whose scalar curvature function is
strictly positive. Henceforth, such metrics will be referred to as
{\em positive scalar curvature metrics} and the term {\em positive
scalar curvature} will frequently be abbreviated by the initials {\em
psc}. As mentioned in the introduction, the problem of whether or not
$X$ admits any {\em psc}-metrics has been extensively studied and so
unless otherwise stated, we will assume we are working with $X$ so that
$\Riem^{+}(X)\neq \emptyset$.

It is a straightforward exercise in linear algebra to show that
$\Riem(X)$ is a convex space, i.e. for any pair $g_0,g_1\in\Riem(X)$,
the path $sg_0+(1-s)g_1$, where $s\in I$, lies entirely in
$\Riem(X)$. The topology of $\Riem^{+}(X)$ on the other hand is far
less understood, even at the level of $0$-connectedness. Before
discussing this any further it is necessary to define the following
equivalence relations on $\Riem^{+}(X)$.
\begin{Definition}
{\rm 
The metrics $g_0$ and $g_1$ are said to be {\em {isotopic}} if
they lie in the same path component of $\Riem^{+}(X)$. A path $g_s, s\in I$ in
$\Riem^{+}(X)$ connecting $g_0$ and $g_1$ is known as an {\em
{isotopy}}.}
\end{Definition}
\begin{Definition}
{\rm If there is a metric of positive scalar curvature $\bar{g}$ on the
cylinder $X\times I$ so that for some $\delta>0$,
$\bar{g}|_{X\times[0,\delta]}=g_0+ds^2$ and
$\bar{g}|_{X\times[1-\delta,1]}=g_1+ds^2$, then $g_0$ and $g_1$ are
said to be {\em {concordant}}. The metric $\bar{g}$ is known as a
{\em {concordance}}.}
\end{Definition}

The following lemma is well known and proofs of various versions of it
are found in \cite{GL1}, \cite{Gajer} and \cite{RS}.
\begin{Lemma}\label{isotopyimpliesconc}
Let $g_{r}, r\in I$ be a smooth path in $\Riem^{+}(X)$. Then there exists a constant $0<\Lambda\leq 1$ so that
for every smooth function $f:\mathbb{R}\rightarrow[0,1]$ with
$|\dot{f}|,|\ddot{f}|\leq\Lambda$, the metric $g_{f(s)}+ds^{2}$ on
$X\times\mathbb{R}$ has positive scalar curvature.
\end{Lemma}
\begin{Corollary}\label{isoconc} 
Metrics which are isotopic are also concordant.
\end{Corollary}
\begin{proof}
Let $g_0$ and $g_1$ be two psc-metrics, connected by the path $g_r$ in $\Riem^{+}(X)$, where $r\in I$. 
Any continuous path in $\Riem^{+}(X)$ may be approximated by a
smooth one and so we will assume that $g_r$ is a smooth
isotopy. Let $f$ be a smooth increasing function which is of the
form
\begin{equation*}
\begin{array}{cl}
f(s) =
\begin{cases}
1 & \text{if $s\geq k_2$}\\
0 & \text{if $s\leq k_1$}
\end{cases}
\end{array}
\end{equation*}
where $k_1<k_2$. The function $f$ can be chosen with $|\dot{f}|$ and
$|\ddot{f}|$ bounded by some arbitrarily small constant provided
$k_{2}-k_{1}$ is large enough. Now choose $A_{1},A_{2}$ so that
$A_{1}<k_{1}<k_{2}<A_{2}$. By the lemma above, the metric
$g_{f(s)}+ds^{2}$ on $X\times[A_{1},A_{2}]$ has positive scalar
curvature. This metric can easily be pulled back to obtain the desired
concordance on $X\times I$.
\end{proof}

\noindent Whether or not the converse of this corollary holds, i.e. concordant 
metrics are isotopic, is a much more complicated question and one we 
discussed in the introduction. In particular, when $\dim X\geq 5$, the 
problem of whether or not concordance implies isotopy is completely open.
Recall that a general concordance may be arbitrarily complicated. We will 
approach this problem restricting our attention to a particular 
type of concordance, which we construct using the surgery technique of Gromov and 
Lawson, and which we will call a Gromov-Lawson concordance. An important part of 
the surgery technique concerns modification of a
psc-metric on or near an embedded sphere. For the remainder of this
section we will consider a variety of psc-metrics both on the sphere
and the disk. These metrics will play an important technical role in
later sections.

\subsection{Warped product metrics on the sphere}
We denote by $S^{n}$, the standard $n$-dimensional sphere and assume that $n\geq 3$. 
We will study metrics on $S^{n}$ which take the form of {\it warped} and {\it doubly warped}
product metrics, see description below. All of the metrics we consider will have non-negative 
sectional and Ricci curvatures, positive scalar curvature and will be {\it isotopic} to the standard 
round metric on $S^{n}$. The latter fact will be important in the proof of the main theorem, Theorem \ref{conciso}.
     
The standard round metric of radius 1, can be induced 
on $S^{n}$ via the usual embedding into $\mathbb{R}^{n+1}$. 
We denote this metric $ds_{n}^{2}$. There are of course many 
different choices of coordinates with which to realise this 
metric. For example, the embedding
\begin{equation*}
\begin{split}
(0,\pi)\times{S^{n-1}}&\longrightarrow\mathbb{R}\times\mathbb{R}^{n}\\
(t,\theta)&\longmapsto(\cos{t},\sin{t}.\theta)
\end{split}
\end{equation*}
gives rise to the metric $dt^{2}+\sin^{2}(t)ds_{n-1}^{2}$ 
on $(0,\pi)\times S^{n-1}$. This extends uniquely to the 
round metric of radius $1$ on $S^{n}$. Similarly, 
the round metric of radius $\epsilon$ has the form 
$dt^{2}+\epsilon^{2}\sin^{2}(\frac{t}{\epsilon})ds_{n-1}^{2}$ on $(0,\epsilon\pi)\times S^{n-1}$. 
More generally, by replacing $\sin t$ with a suitable smooth 
function $f:(0,b)\rightarrow(0,\infty)$, we can construct other metrics on $S^{n}$. 
The following proposition specifies necessary and sufficent 
conditions on $f$ which guarantee smoothness of the metric 
$dt^{2}+f(t)^{2}ds_{n-1}^{2}$ on $S^{n}$.

\begin{Proposition} {\rm{({Chapter 1, section 3.4, \cite{P}})}} \label{smoothwarp}
Let $f:(0,b)\rightarrow(0,\infty)$ be a smooth function with $f(0)=0=f(b)$. Then the metric $g=dt^{2}+f(t)^{2}ds_{n-1}^{2}$ is a smooth metric on the sphere $S^{n}$ if and only if
$f^{(even)}(0)=0$, $\dot{f}(0)=1$, $f^{(even)}(b)=0$ and $\dot{f}(b)=-1$.
\end{Proposition}

Given the uniqueness of the extension, we will regard metrics of the form $dt^{2}+f(t)^{2}ds_{n-1}^2$ on $(0,b)\times S^{n-1}$ as simply metrics on $S^{n}$, provided $f$ satisfies the conditions above. For a general smooth function $f:(0,b)\rightarrow(0,\infty)$, a metric of the form $dt^{2}+f(t)^{2}ds_{n-1}^{2}$ on $(0,b)\times S^{n-1}$ is known as a {\it warped product metric}. From page 69 of \cite{P}, we obtain the following formulae for the Ricci and scalar curvatures of such a metric. Let $\p_t,e_1,\cdots,e_{n-1}$ be an orthonormal frame where $\p_t$ is tangent to the interval $(0,b)$ while each $e_i$ is tangent to the sphere $S^{n-1}$. Then
\\
\begin{equation*}\label{eqn;ricwarp}
\begin{split}
Ric(\p_t)&=-(n-1)\frac{\ddot{f}}{f},\\
Ric(e_i)&=(n-2)\frac{1-\dot{f}^2}{f^2}-\frac{\ddot{f}}{f} \hspace{0.2cm}\text{, when $i=1,\cdots ,n-1$}.
\end{split}
\end{equation*}

\noindent Thus, the scalar curvature is

\begin{equation}\label{eqn;scalwarp}
R=-2(n-1)\frac{\ddot{f}}{f}+(n-1)(n-2)\frac{1-\dot{f}^2}{f^2}.
\end{equation}

Let $\F(0,b)$ denote the space of all smooth functions 
$f:(0,b)\rightarrow(0,\infty)$ which satisfy the following conditions.

\begin{equation}\label{eqns;smoothwarp}
\begin{array}{rlc}
f(0)=0,&\qquad 
f(b)=0,\\
\dot{f}(0)=1,&\qquad 
\dot{f}(b)=-1,\\
f^{(even)}(0)=0,&\qquad 
f^{(even)}(b)=0,\\ 
\ddot{f}\leq 0,&\qquad 
\dddot{f}(0)<0,\qquad
\dddot{f}(b)>0,&\\
\ddot{f}(t)<0,&  
\text{when $t$ is near but not at $0$ and $b$}.
\end{array}
\end{equation}

\begin{figure}[!htbp]
\vspace{1cm}

\begin{picture}(0,0)%
\includegraphics{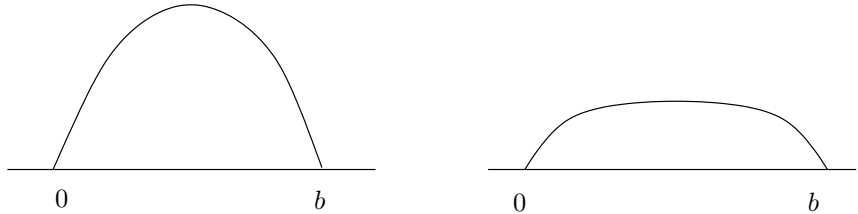}%
\end{picture}%
\setlength{\unitlength}{3947sp}%
\begingroup\makeatletter\ifx\SetFigFont\undefined%
\gdef\SetFigFont#1#2#3#4#5{%
  \reset@font\fontsize{#1}{#2pt}%
  \fontfamily{#3}\fontseries{#4}\fontshape{#5}%
  \selectfont}%
\fi\endgroup%
\begin{picture}(5355,1376)(627,-3209)
\put(939,-3118){\makebox(0,0)[lb]{\smash{{\SetFigFont{10}{8}{\rmdefault}{\mddefault}{\updefault}{\color[rgb]{0,0,0}$0$}%
}}}}
\put(2564,-3130){\makebox(0,0)[lb]{\smash{{\SetFigFont{10}{8}{\rmdefault}{\mddefault}{\updefault}{\color[rgb]{0,0,0}$b$}%
}}}}
\put(3814,-3143){\makebox(0,0)[lb]{\smash{{\SetFigFont{10}{8}{\rmdefault}{\mddefault}{\updefault}{\color[rgb]{0,0,0}$0$}%
}}}}
\put(5664,-3143){\makebox(0,0)[lb]{\smash{{\SetFigFont{10}{8}{\rmdefault}{\mddefault}{\updefault}{\color[rgb]{0,0,0}$b$}%
}}}}
\end{picture}%

\caption{Typical elements of $\F(0,b)$}
\label{fig:F(0,b)}
\end{figure}

For each function $f$ in $\F(0,b)$, there is an associated smooth metric 
$g=dt^{2}+f(t)^{2}ds_{n-1}^{2}$ on $S^{n}$. We will denote the space of 
all such metrics by $\W(0,b)$. Note that $\F(0,b)$ is assumed to have the 
standard $C^{k}$ function space topology with $k\geq 2$, see chapter 2 of \cite{Hirsch} for details.

\begin{Proposition}\label{firstpsc} 
The space $\W(0,b)=\{dt^{2}+f(t)^{2}ds_{n-1}^{2}:f\in\F(0,b)\}$ 
is a path-connected subspace of $\Riem^{+}(S^{n})$.
\end{Proposition}
\begin{proof}

The first three conditions of (\ref{eqns;smoothwarp}) guarantee smoothness of such metrics 
on $S^{n}$, by Proposition \ref{smoothwarp}.
We will now consider the scalar curvature when $0<t<b$.
Recall that $\ddot{f}(t)\leq 0$ and that near the endpoints this inequality is strict. 
This means that when $0<t<b$, $|\dot{f}(t)|<1$ and so while the first term in 
(\ref{eqn;scalwarp}) is at worst non-negative, the second term is strictly positive. At the end points, several applications of 
l'Hospital's rule give that
\begin{equation*}
\begin{array}{c}
\lim_{t\rightarrow0^{+}}\frac{-\ddot{f}}{f}=-\dddot{f}(0),\qquad \lim_{t\rightarrow0^{+}}\frac{1-\dot{f}^{2}}{f^{2}}=-\dddot{f}(0)>0,\\
\lim_{t\rightarrow b^{-}}\frac{-\ddot{f}}{f}=\dddot{f}(b),\qquad \lim_{t\rightarrow b^{-}}\frac{1-\dot{f}^{2}}{f^{2}}=\dddot{f}(b)>0.
\end{array}
\end{equation*}

\noindent Thus $\W(0,b)\subset \Riem^{+}S^{n}$. Path connectedness now follows from the convexity of $\F(0,b)$ which in turn follows 
from an elementary calculation. 
\end{proof}

\noindent It is convenient to allow $b$ to vary. Thus we will define 
$\F=\bigcup_{b\in(0,\infty)}\F(0,b)$ and $\W=\bigcup_{b\in(0,\infty)}\W(0,b)$. 
Each metric in $\W$ is defined on $(0,b)\times{S^{n-1}}$ for some $b>0$. 
In particular, the round metric of radius $\epsilon$, $\epsilon^{2}ds_{n}^{2}$, is an element of $\W(0,\epsilon\pi)$.

\begin{Proposition} \label{doubletorpedopsc}
The space $\W$ is a path-connected subspace 
of $\Riem^{+}(S^{n})$.
\end{Proposition}
\begin{proof}

Let $g$ be an element of $\W$. Then $g=dt^{2}+f(t)^{2}ds_{n-1}^{2}$ on 
$(0,b)\times S^{n-1}$ for some $f\in \F(0,b)$ and some $b>0$. As $\F(0,b)$ is convex,
there is a path connecting $g$ to the metric 
$dt^{2}+(\frac{b}{\pi})^{2}\sin^{2}(\frac{\pi t}{b})ds_{n-1}^{2}$, the round metric of 
radius $(\frac{b}{\pi})^{2}$ in $\W(0,b)$. As all round metrics on $S^{n}$ are isotopic by an obvious
rescaling, $g$ can be isotopied to any metric in the space.
\end{proof}

\subsection{Torpedo metrics on the disk}\label{torpedofunc}
\indent A $\delta $-{\it torpedo metric} on a disk $D^{n}$, 
denoted $g_{tor}^n(\delta)$, is an ${\rm O(n)}$ symmetric positive 
scalar curvature metric which is a product with the standard 
$n-1$-sphere of radius $\delta$ near the boundary of $D^{n}$ 
and is the standard metric on the $n$-sphere of radius $\delta$ 
near the centre of the disk. It is not hard to see how such 
metrics can be constructed. Let $f_\delta$ be a smooth 
function on $(0,\infty)$ which satisfies the following conditions.

(i) $f_\delta(t)=\delta\sin{(\frac{t}{\delta})}$ when $t$ is near $0$.

(ii) $f_\delta(t)=\delta$ when $t\geq\delta\frac{\pi}{2}$.

(iii) $\ddot{f_{\delta}}(t)\leq 0$.

\noindent From now on $f_{\delta}$ will be known as a {\it $\delta-$torpedo function}.
\\
\begin{figure}[!htbp]
\vspace{1cm}
\begin{picture}(0,0)%
\includegraphics{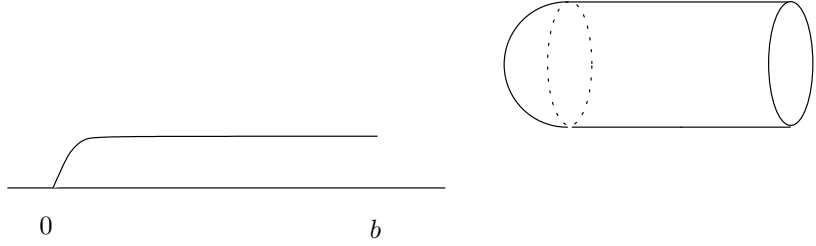}%
\end{picture}%
\setlength{\unitlength}{3947sp}%
\begingroup\makeatletter\ifx\SetFigFont\undefined%
\gdef\SetFigFont#1#2#3#4#5{%
  \reset@font\fontsize{#1}{#2pt}%
  \fontfamily{#3}\fontseries{#4}\fontshape{#5}%
  \selectfont}%
\fi\endgroup%
\begin{picture}(5079,1559)(1902,-7227)
\put(2114,-7136){\makebox(0,0)[lb]{\smash{{\SetFigFont{10}{8}{\rmdefault}{\mddefault}{\updefault}{\color[rgb]{0,0,0}$0$}%
}}}}
\put(4189,-7161){\makebox(0,0)[lb]{\smash{{\SetFigFont{10}{8}{\rmdefault}{\mddefault}{\updefault}{\color[rgb]{0,0,0}$b$}%
}}}}
\end{picture}%
\caption{A torpedo function and the resulting torpedo metric}
\label{torpedo}
\end{figure} 

Let $r$ be the standard radial distance function on $\mathbb{R}^n$. 
By Lemma \ref{smoothwarp}, the metric $dr^2 + f_\delta(r)^2ds_{n-1}^2$ on $(0,\infty)\times S^{n-1}$ 
extends smoothly as a metric on $\mathbb{R}^n$, as $\dot{f_{\delta}}(0)=1$ 
and $f_{\delta}^{even}(0)=0$. The resulting metric is a torpedo metric 
of radius $\delta$ on $\mathbb{R}^n$. By restricting to $(0,b)\times 
S^{n-1}$ for some $b>\delta\frac{\pi}{2}$ we obtain a torpedo metric on 
a disk $D^{n}$, see Fig. \ref{torpedo}. From formula (\ref{eqn;scalwarp}), it is clear that this metric has positive scalar curvature and moreover, the scalar curvature can be bounded below by an arbitrarily large constant by choosing $\delta$ sufficiently small. 

We will refer to the cylindrical part of this metric as the {\it tube}, and the remaining piece as the {\it cap} of $g_{tor}^{n}(\delta)$. Notice that we can always isotopy $g_{tor}^{n}(\delta)$ to make the tube arbitrarily long. Strictly speaking then, $g_{tor}^{n}(\delta)$ denotes a collection of isotopic metrics, each with isometric cap of radius $\delta$. It is convenient however, to think of $g_{tor}^{n}(\delta)$ as a fixed metric, the tube length of which may be adjusted if necessary.

The torpedo metric on a disk $D^{n}$ can be used to construct a collection of psc-metrics on $S^{n}$ which will be of use to us later on. The first of these is the double torpedo metric on $S^{n}$.
By considering the torpedo metric as a metric on a hemisphere, we can 
obtain a metric on $S^{n}$ by taking its double. More precisely let 
$\bar{f}_{\delta}(t)$ be the smooth function on $(0,b)$ which satisfies 
the following conditions.

(i) $\bar{f}_{\delta}(t)=f_\delta(t)$ on $[0,\frac{b}{2}]$

(ii) $\bar{f}_{\delta}(t)=f_\delta(b-t)$ on $[\frac{b}{2}, b]$,

\noindent where $\frac{b}{2}>\delta\frac{\pi}{2}$.

\begin{figure}[htbp]
\vspace{1cm}

\begin{picture}(0,0)%
\includegraphics{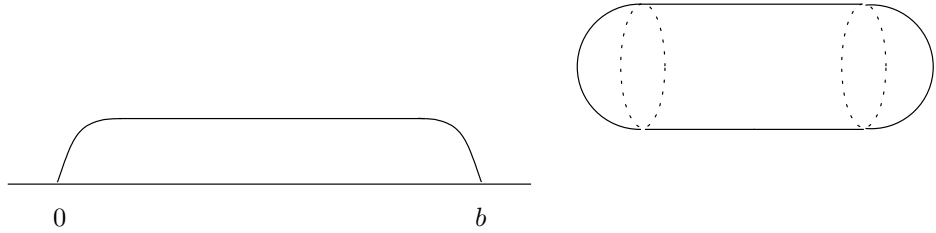}%
\end{picture}%
\setlength{\unitlength}{3947sp}%
\begingroup\makeatletter\ifx\SetFigFont\undefined%
\gdef\SetFigFont#1#2#3#4#5{%
  \reset@font\fontsize{#1}{#2pt}%
  \fontfamily{#3}\fontseries{#4}\fontshape{#5}%
  \selectfont}%
\fi\endgroup%
\begin{picture}(5833,1472)(995,-3093)
\put(3944,-3027){\makebox(0,0)[lb]{\smash{{\SetFigFont{10}{8}{\rmdefault}{\mddefault}{\updefault}{\color[rgb]{0,0,0}$b$}%
}}}}
\put(1294,-3027){\makebox(0,0)[lb]{\smash{{\SetFigFont{10}{8}{\rmdefault}{\mddefault}{\updefault}{\color[rgb]{0,0,0}$0$}%
}}}}
\end{picture}%

\caption{A double torpedo function and the resulting double torpedo metric}
\label{doubletorpedo}
\end{figure}

As $\bar{f}\in\F(0,b)$, the metric $dt^{2}+\bar{f}_{\delta}(t)^{2}ds_{n-1}^{2}$ 
on $(0,b)\times S^{n-1}$ gives rise to a smooth psc-metric on $S^{n}$. Such a 
metric will be called a {\it double torpedo metric of radius $\delta$} 
and denoted $g_{Dtor}^{n}(\delta)$, see Fig. \ref{doubletorpedo}. Then Proposition \ref{doubletorpedopsc} implies that 
$g_{Dtor}^{n}(\delta)$ is isotopic to $ds_{n}^{2}$.

\subsection{Doubly warped products and mixed torpedo metrics}
Henceforth $p$ and $q$ will denote a pair of non-negative integers satisfying $p+q+1=n$. 
The standard sphere $S^{n}$ decomposes as a union of sphere-disk products
as shown below.
\begin{equation*}
\begin{array}{cl}
S^{n} &= \p D^{n+1}\\
&=\p (D^{p+1}\times D^{q+1}),\\ 
&=S^{p}\times D^{q+1}\cup_{S^{p}\times S^{q}} D^{p+1}\times S^{q}.
\end{array}
\end{equation*}
\\
\begin{figure}[!htbp]

\includegraphics[height=40mm]{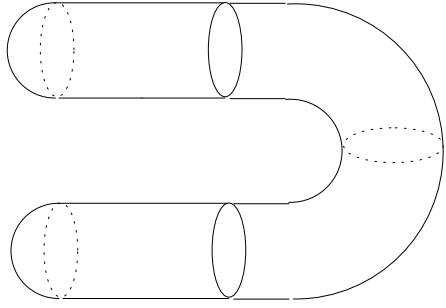}

\caption{$S^{n}$ decomposed as $S^{p}\times D^{q+1}\cup_{S^{p}\times S^{q}} D^{p+1}\times S^{q}$ and equipped with a mixed torpedo metric $g_{Mtor}^{p,q}$}
\label{fig:decomposing_S^n}
\end{figure}

We can utilise this decomposition to construct a new metric on $S^{n}$. 
Equip $S^{p}\times D^{q+1}$ with the product metric 
$\epsilon^{2}ds_{p}^{2}+g_{tor}^{q+1}(\delta)$.
Then equip $D^{p+1}\times S^{q}$ with $g_{tor}^{p+1}(\epsilon)+\delta^{2}ds_{q}^{2}$. 
These metrics glue together smoothly along the common boundary 
$S^{p}\times S^{q}$ to form a smooth metric on $S^{n}$. Such metrics will be known
as {\it mixed torpedo metrics} on $S^{n}$ and denoted $g_{Mtor}^{p,q}$. For the remainder of this section 
we will show how to realise these metrics in a more computationally useful form.

\indent Recall that a metric of the form $dt^{2}+f(t)^{2}ds_{n-1}^{2}$ 
on $(0,b)\times S^{n-1}$, where $f:(0,b)\rightarrow (0,\infty)$ is a 
smooth function, is known as a {\it warped product metric}. We have 
observed that the standard round sphere metric: $ds_{n}^{2}$, can be 
represented as the warped product metric $dt^{2}+\sin^{2}(t)ds_{n-1}^{2}$ 
on $(0,\pi)\times S^{n-1}$. The notion of a warped product metric on 
$(0,b)\times S^{n-1}$ generalises to something called a {\it doubly 
warped product metric} on $(0,b)\times S^{p}\times S^{q}$. Here the 
metric takes the form $dt^{2}+u(t)^{2}ds_{p}^{2}+v(t)^{2}ds_{q}^{2}$, 
where $u,v:(0,b)\rightarrow(0,\infty)$ are smooth functions. 

From page 72 of \cite{P}, we obtain the following curvature formulae. Let 
$\p_t,e_1,\cdots,e_{p}, e_{1}',\cdots,e_{q}'$ be an 
orthonormal frame where $e_1,\cdots,e_{p}$ are tangent to $S^{p}$ and 
$e_{1}',\cdots,e_{q}'$ are tangent to $S^{q}$. Then
\begin{equation*}
\begin{split}
Ric(\p_t)&=-(p)\frac{\ddot{u}}{u}-(q)\frac{\ddot{v}}{v},\\
Ric(e_i)&=(p-1)\frac{1-\dot{u}^2}{u^2}-\frac{\ddot{u}}{u}-q\frac{\dot{u}\dot{v}}{uv} \hspace{0.2cm}\text{, $i=1,\cdots ,p$},\\
Ric(e_i')&=(q-1)\frac{1-\dot{v}^2}{v^2}-\frac{\ddot{v}}{v}-p\frac{\dot{u}\dot{v}}{uv} \hspace{0.2cm}\text{, $i=1,\cdots ,q$}.
\end{split}
\end{equation*}

\noindent Thus, the scalar curvature is
\begin{equation}\label{eqn;scaldoublewarp}
\begin{split}
R=-2p\frac{\ddot{u}}{u}-2q\frac{\ddot{v}}{v}+p(p-1)\frac{1-\dot{u}^2}{u^2}+q(q-1)\frac{1-\dot{v}^2}{v^2}-2pq\frac{\dot{u}\dot{v}}{uv}.
\end{split}
\end{equation}

\indent We observe that the round metric $ds_{n}^{2}$ can be represented by a 
doubly warped product. Recalling that $p+q+1=n$, consider the map
\begin{equation}\label{map;doublewarpembedding} 
\begin{split}
(0,\frac{\pi}{2})\times{{S}^{p}}\times{{S}^{q}}&\longrightarrow\mathbb{R}^{p+1}\times\mathbb{R}^{q+1}\\
 (t,\phi,\theta)&\longmapsto(\cos{t}.\phi,\sin{t}.\theta)
\end{split}
\end{equation}

\noindent Here $S^{p}$ and $S^{q}$ denote the standard unit spheres in $\mathbb{R}^{p+1}$ and $\mathbb{R}^{q+1}$ respectively. 
The metric induced by this embedding is given by the formula
\begin{equation*} 
\begin{array}{c}
dt^{2}+\cos^{2}(t)ds_{p}^{2}+\sin^{2}(t)ds_{q}^{2},
\end{array}
\end{equation*} 
a doubly warped product representing the round metric on $S^{n}$. 
More generally the round metric of radius $\epsilon$ takes the form 
$dt^{2}+\epsilon^{2}\cos^{2}(\frac{t}{\epsilon})ds_{p}^{2}+\epsilon^{2}\sin^{2}(\frac{t}{\epsilon})ds_{q}^{2}$ 
on $(0,\epsilon\frac{\pi}{2})\times{S^{p}}\times{S^{q}}$.

As before, by imposing appropriate conditions on the functions 
$u,v:(0,b)\rightarrow(0,\infty)$, the metric 
$dt^{2}+u(t)^{2}ds_{p}^{2}+v(t)^{2}ds_{q}^{2}$ 
gives rise to a smooth metric on $S^{n}$. By combining propositions 1 and 2 on page 13 of \cite{P}, we obtain the following proposition which makes these conditions clear.
\vspace{2mm}

\begin{Proposition}
\label{smoothdoublewarp}
{\rm{(Page 13, \cite{P})}}
Let $u,v:(0,b)\rightarrow(0,\infty)$ be smooth functions with $u(b)=0$ and $v(0)=0$. Then the metric $dt^{2}+u(t)^{2}ds_{p}^{2}+v(t)^{2}ds_{q}^{2}$ on $(0,b)\times{S^{p}}\times{S^{q}}$ is a smooth metric on $S^{n}$ if and only if the following conditions hold.
\begin{eqnarray}
\label{eqns;usmooth}
&u(0)>0,\qquad 
u^{(odd)}(0)=0,\qquad 
\dot{u}(b)=-1,\qquad 
u^{(even)}(b)=0.\\
\label{eqns;vsmooth} 
&v(b)>0,\qquad
v^{(odd)}(b)=0,\qquad
\dot{v}(0)=1,\qquad
v^{(even)}(0)=0.
\end{eqnarray}
\end{Proposition}

Let $\U(0,b)$ denote the space of all functions $u:(0,b)\rightarrow(0,\infty)$ which satisfy (\ref{eqns;usmooth}) above and the condition that $\ddot{u}\leq 0$ with $\ddot{u}(t)<0$ when $t$ is near but not at $b$ and $\dddot{u}(b)>0$. 

Similarly $\V(0,b)$ will denote the space of all functions $v:(0,b)\rightarrow(0,\infty)$ which satisfy (\ref{eqns;vsmooth}) and for which $\ddot{v}\leq 0$ with $\ddot{v}(t)<0$ when $t$ is near but not at $0$ and $\dddot{v}(0)<0$. 

Each pair $u,v$ from the space $\U(0,b)\times\V(0,b)$ gives rise to a metric
$dt^{2}+u(t)^{2}ds_{p}^{2}+v(t)^{2}ds_{q}^{2}$ on $S^{n}$. We denote the space of such metrics
\begin{equation*}
\begin{array}{c}
\hat{\W}^{p,q}(0,b)=\{dt^{2}+u(t)^{2}ds_{p}^{2}+v(t)^{2}ds_{q}^{2}:(u,v)\in\U(0,b)\times\V(0,b)\}.
\end{array}
\end{equation*}
We now obtain the following lemma.

\begin{Lemma}\label{lem:markslemm1.5}
Let $n\geq 3$ and let $p$ and $q$ be any pair of non-negative integers satisfying $p+q+1=n$. Then the space $\hat{\W}^{p,q}(0,b)$ is a path-connected subspace of $\Riem^{+}{S^{n}}$.
\end{Lemma}
\begin{proof} Let $g=dt^{2}+u(t)^{2}ds_{p}^{2}+v(t)^{2}ds_{q}^{2}$ be an element of $\hat{\W}^{p,q}(0,b)$. 
Smoothness of this metric on $S^{n}$ follows from Proposition \ref{smoothdoublewarp}. 
We will first show that $g$ has positive scalar curvature when $0<t<b$. Recall that $u$ 
and $v$ are both concave downward, that is $\ddot{u}, \ddot{v}< 0$. 
This means that the first 2 terms in (\ref{eqn;scaldoublewarp}) are at worst non-negative.  
Downward concavity and the fact that $\dot{u}(0)=0$ and $\dot{u}(b)=-1$ 
imply that $-1<\dot{u}\leq 0$. A similar argument gives that 
$0\leq\dot{v}<1$. This means that the fifth term in (\ref{eqn;scaldoublewarp}) is also 
non-negative and at least one of the third and fourth terms in (\ref{eqn;scaldoublewarp}) 
is strictly positive (the other may be $0$ for dimensional reasons). 

When $t=0$ and $t=b$, some elementary limit computations using 
l'Hospital's rule show that the scalar curvature is positive.
Thus $\hat{\W}(0,b)^{p,q}\subset {\Riem}^{+}(S^{n})$.

Finally, path connectivity follows immediately from the convexity of the space $\U(0,b)\times\V(0,b)$. 
\end{proof}

As before, it is convenient to allow $b$ to vary. 
Thus, we define $\U\times\V=\bigcup_{b\in(0,\infty)}\U(0,b)\times\V(0,b)$ 
and $\hat{\W}^{p,q}=\bigcup_{b\in(0,\infty)}\hat{\W}^{p,q}(0,b)$. Finally we let $\hat{\W}=\bigcup_{p+q+1=n}\hat{\W}^{p,q}$ where $0\leq p,q\leq n+1$.

\begin{Proposition} 
Let $n\geq 3$. The space $\hat{\W}$ is a path-connected subspace of $\Riem^{+}{S^{n}}.$\label{mixedtorpedopsc}
\end{Proposition}
\begin{proof} The proof that $\hat{W}^{p,q}$ is path connected is almost identical to that of 
Proposition \ref{doubletorpedopsc}. The rest follows from the fact that each $\hat{\W}^{p,q}$ contains the round metric
$ds_{n}^{2}=dt^{2}+\cos^{2}{t}ds_{p}^{2}+\sin^{2}{t}ds_{q}^{2}$.
\end{proof}

At the beginning of this section we demonstrated that $S^{n}$ could be 
decomposed into a union of $S^{p}\times D^{q+1}$ and $D^{p+1}\times S^{q}$.
This can be seen explicitly by appropriate restriction of the embedding in(\ref{map;doublewarpembedding}). 
Thus, provided $t$ is near $0$, the metric $dt^{2}+u(t)^{2}ds_{p}^{2}+v(t)^{2}ds_{q}^{2}$,
with $u,v\in \U(0,b)\times\V(0,b)$, is a metric on $S^{p}\times D^{q+1}$. When $t$ is near $b$
we obtain a metric on $D^{p+1}\times S^{q}$. We can now construct a mixed torpedo metric on $S^{n}$, as follows.
Let $f_{\epsilon}(t)$ and $f_{\delta}(t)$ be the torpedo functions on $(0,b)$ defined in section \ref{torpedofunc} with
$b>\max\{\epsilon\pi,\delta\pi\}$.
Then the metric
\begin{equation}
\begin{array}{c}
 g_{Mtor}^{p,q}=dt^{2}+f_{\epsilon}(b-t)^{2}ds_{p}^{2}+f_{\delta}(t)^{2}ds_{q}^{2}
\end{array}
\end{equation}
is a mixed torpedo metric on $S^{n}$, see Fig. \ref{fig:mixed torpedo}. 
\\
\begin{figure}[htbp]
\vspace{1cm}

\includegraphics[height=30mm]{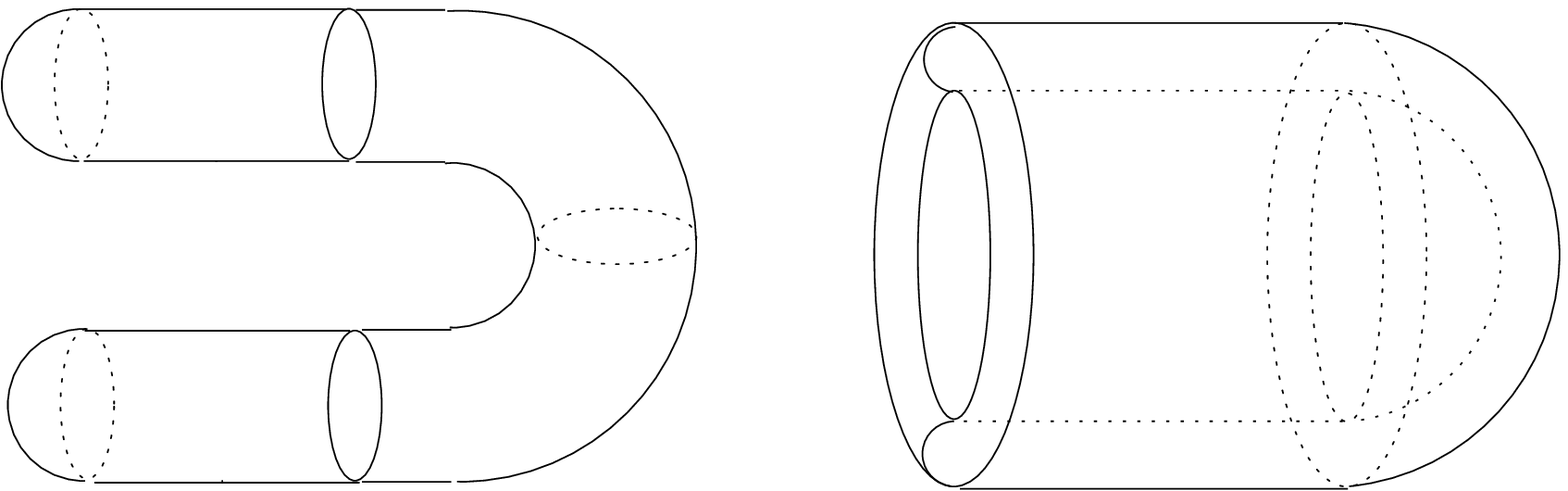}
\caption{The mixed torpedo metrics $g_{Mtor}^{p,q}$ and $g_{Mtor}^{p+1,q-1}$}
\label{fig:mixed torpedo}
\end{figure}

\begin{Lemma} \label{toriso}
Let $n\geq 3$. For any non-negative integers $p$ and $q$ with $p+q+1=n$, the metric $g_{Mtor}^{p,q}$ is isotopic to $ds_{n}^{2}$.
\end{Lemma}
\begin{proof} An elementary calculation shows that the functions $f_{\epsilon}(b-t)$ and 
$f_{\delta}(t)$ lie in $\U(0,b)$and $\V(0,b)$ respectively. Thus $g_{Mtor}^{p,q}\in\hat{\W}^{p,q}(0,b)$.
As the standard round metric lies in $\hat{\W}^{p,q}(0,b)$, the proof follows from Proposition 
\ref{mixedtorpedopsc}.
\end{proof}

\begin{figure}[htbp]
\begin{picture}(0,0)%
\includegraphics[height=60mm]{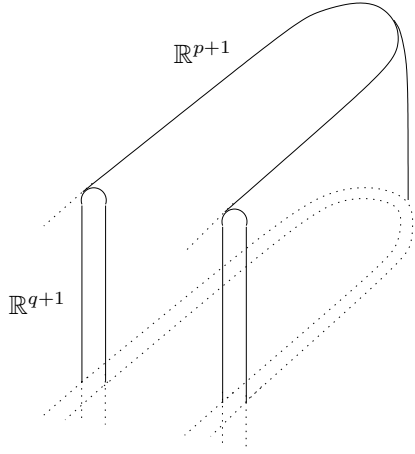}%
\end{picture}%
\setlength{\unitlength}{3947sp}%
\begingroup\makeatletter\ifx\SetFigFont\undefined%
\gdef\SetFigFont#1#2#3#4#5{%
  \reset@font\fontsize{#1}{#2pt}%
  \fontfamily{#3}\fontseries{#4}\fontshape{#5}%
  \selectfont}%
\fi\endgroup%
\begin{picture}(1968,2757)(2374,-5278)
\put(3226,-2836){\makebox(0,0)[lb]{\smash{{\SetFigFont{10}{8}{\rmdefault}{\mddefault}{\updefault}{\color[rgb]{0,0,0}$\mathbb{R}^{p+1}$}%
}}}}
\put(2189,-4411){\makebox(0,0)[lb]{\smash{{\SetFigFont{10}{8}{\rmdefault}{\mddefault}{\updefault}{\color[rgb]{0,0,0}$\mathbb{R}^{q+1}$}%
}}}}
\end{picture}%

\caption{The plane $\mathbb{R}^{n+1}$ equipped with the metric $h$}
\label{h}
\end{figure}

\subsection{Inducing a mixed torpedo metric with an embedding}\label{embedmixedtorp}
We close this section with a rather technical observation 
which will be of use later on. It is of course possible to realise mixed torpedo metrics on the sphere as 
the induced metrics of some embedding. Let $\mathbb{R}^{n+1}=\mathbb{R}^{p+1}\times\mathbb{R}^{q+1}$
where of course $p+q+1=n$. Let $(\rho, \phi)$ and $(r,\theta)$ denote
standard spherical coordinates on $\mathbb{R}^{p+1}$ and $\mathbb{R}^{q+1}$ 
where $\rho$ and $r$ are the respective Euclidean distance functions and $\phi\in S^{p}$ and
$\theta\in S^{q}$. Then equip 
$\mathbb{R}^{n+1}=\mathbb{R}^{p+1}\times\mathbb{R}^{q+1}$ with the metric $h=h^{p,q}$ defined
\begin{equation}
\begin{array}{c}
h^{p,q}=d\rho^{2}+f_\epsilon(\rho)^{2}ds_{p}^{2}+dr^{2}+f_\delta(r)^{2}ds_{q}^{2},
\end{array}
\end{equation}
shown in Fig. \ref{h}, where $f_\epsilon, f_\delta:(0,\infty)\rightarrow(0,\infty)$ are the torpedo functions defined in section \ref{torpedofunc}.

\begin{figure}
\begin{picture}(0,0)%
\includegraphics[height=40mm]{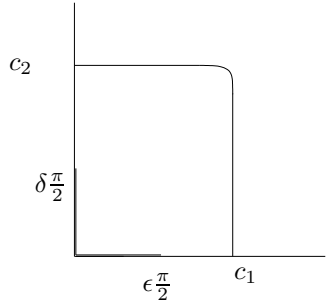}%
\end{picture}%
\setlength{\unitlength}{3947sp}%
\begingroup\makeatletter\ifx\SetFigFont\undefined%
\gdef\SetFigFont#1#2#3#4#5{%
  \reset@font\fontsize{#1}{#2pt}%
  \fontfamily{#3}\fontseries{#4}\fontshape{#5}%
  \selectfont}%
\fi\endgroup%
\begin{picture}(2677,2516)(2399,-4690)
\put(3826,-4549){\makebox(0,0)[lb]{\smash{{\SetFigFont{10}{8}{\rmdefault}{\mddefault}{\updefault}{\color[rgb]{0,0,0}$c_1$}%
}}}}
\put(2414,-3236){\makebox(0,0)[lb]{\smash{{\SetFigFont{10}{8}{\rmdefault}{\mddefault}{\updefault}{\color[rgb]{0,0,0}$c_2$}%
}}}}
\put(2576,-3999){\makebox(0,0)[lb]{\smash{{\SetFigFont{10}{8}{\rmdefault}{\mddefault}{\updefault}{\color[rgb]{0,0,0}$\delta\frac{\pi}{2}$}%
}}}}
\put(3251,-4624){\makebox(0,0)[lb]{\smash{{\SetFigFont{10}{8}{\rmdefault}{\mddefault}{\updefault}{\color[rgb]{0,0,0}$\epsilon\frac{\pi}{2}$}%
}}}}
\end{picture}%
\caption{The curve $\alpha$}
\label{fig:xandy}
\end{figure}

We will now parameterise an embedded sphere $S^{n}$ in $(\mathbb{R}^{n+1},h)$, the induced metric on which will be precisely the mixed torpedo metric described earlier. Let $c_1$ and $c_2$ be constants satisfying $c_1>\epsilon\frac{\pi}{2}$ and $c_2>\delta\frac{\pi}{2}$. Let $a=(a_1, a_2)$ denote a smooth unit speed curve in the first quadrant of $\mathbb{R}^{2}$ which begins at $(c_1,0)$ follows a vertical trajectory, bends by an angle of $\frac{\pi}{2}$ towards the vertical axis and continues as a horizontal line to end at $(0,c_2)$. We will assume that the bending takes place above the horizontal line through line $(0,\delta\frac{\pi}{2})$, see Fig. \ref{fig:xandy}. We also assume that $a_1\in\U(0,b)$ and $a_2\in\V(0,b)$ for sufficiently large $b>0$.

\begin{figure}[htbp]
\includegraphics[height=50mm]{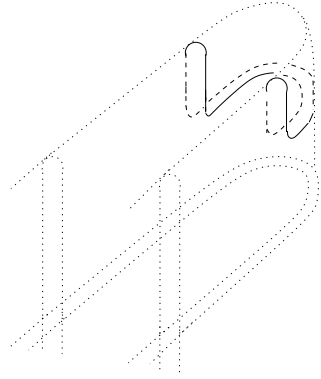}
\caption{The map $J$ gives a parameterisation for $S^{n}$}
\label{Fig.:J}
\end{figure}

We will now specify an embedding of the $n-$sphere into $(\mathbb{R}^{n+1}, h)$ which 
induces the mixed torpedo metric $g_{Mtor}^{p,q}$ described above. Let $J$ be the embedding defined as follows

\begin{equation*} 
\begin{split}
J:(0,b)\times{S^{p}}\times{S^{q}}&\longrightarrow\mathbb{R}^{p+1}\times\mathbb{R}^{q+1},\\
(t,\theta,\phi)&\longmapsto((a_1(t),\phi),(a_2(t),\theta)),
\end{split}
\end{equation*}
see Fig. \ref{Fig.:J}.
Provided that $\epsilon$ and $\delta$ are chosen sufficiently small, this embedding induces 
the mixed torpedo metric $g_{Mtor}^{p,q}$ on $S^{n}$. Indeed, we have

\begin{equation*}\label{mixedtorembcalc}
\begin{array}{cl}
J^{*}h&=J^{*}(d\rho^{2}+f_\epsilon(\rho)^{2}ds_{p}^{2}+dr^{2}+f_\delta(r)^{2}ds_{q}^{2})\\
&= dt^{2}+f_\epsilon(\alpha_1(t))^{2}ds_{p}^{2}+f_\delta(\alpha_2(t))^{2}ds_{q}^{2}\\
&= dt^{2}+f_\epsilon(b-t)^{2}ds_{p}^{2}+f_\delta(t)^{2}ds_{q}^{2}\\
&=g_{Mtor}^{p,q}.
\end{array}
\end{equation*}
The second equality follows from the fact that $\alpha$ is a unit speed curve and the third equality from the fact that $f_\epsilon(s)$ and $f_\delta(s)$ are both constant when $s>\max\{\epsilon\frac{\pi}{2},\delta\frac{\pi}{2}\}$.

\section{Revisiting the Surgery Theorem}\label{surgerysection}

Over the the next two sections we will provide a proof of Theorem \ref{GLcob}. The proof involves the construction of a psc-metric on a compact cobordism $\{W^{n+1};X_0,X_1\}$ which extends a psc-metric $g_0$ from $X_0$ and is a product near $\p{W}$. A specific case of this is Theorem \ref{ImprovedsurgeryTheorem} (stated below) which we prove in this section. It can be thought of as a building block for the more general case of the proof of Theorem \ref{GLcob} which will be completed in section \ref{GLcobordsection}. Before stating Theorem \ref{ImprovedsurgeryTheorem}, it is worth briefly reviewing some basic notions about surgery and cobordism.

\subsection{Surgery and cobordism}
A {\it surgery} on a smooth manifold $X$ of dimension $n$, is the construction of a new $n$-dimensional manfiold $X'$ by removing an embedded sphere of dimension $p$ from $X$ and replacing it with a sphere of dimension $q$ where $p+q+1=n$. More precisely, suppose $i:S^{p}\hookrightarrow X$ is an embedding. Suppose also that the normal bundle of this embedded sphere is trivial. Then we can extend $i$ to an embedding $\bar{i}:S^{p}\times D^{q+1}\hookrightarrow X$. The map $\bar{i}$ is known as a {\it framed embedding} of $S^{p}$. By removing an open neighbourhood of $S^{p}$, we obtain a manifold $X-\bar{i}(S^{p}\times {\oD})$ with boundary $S^{p}\times S^{q}$. Here ${\oD}$ denotes the interior of the disk $D^{q+1}$. As the handle $D^{p+1}\times S^{q}$ has the same boundary, we can use the map $\bar{i}|_{S^{p}\times S^{q}}$, to glue the manifolds $X-\bar{i}(S^{p}\times {\oD})$ and $D^{p+1}\times S^{q}$ along their common boundary and obtain the manifold 

\begin{equation*}
X'=(X-\bar{i}(S^{p}\times {\oD}))\cup_{\bar{i}}D^{p+1}\times S^{q}.
\end{equation*}

The manifold $X'$ can be taken as being smooth (although some minor smoothing of corners is necessary where the attachment took place). Topologically, $X'$ is quite different from the manifold $X$. It is well known that the topology of $X'$ depends on the embedding $i$ and the choice of framing $\bar{i}$, see \cite{R} for details. In the case when $i$ embeds a sphere of dimension $p$ we will describe a surgery on this sphere as either a {\it $p-$surgery} or a {\it surgery of codimenison $q+1$}.

The {\it trace} of a $p-$surgery is a smooth $n+1$-dimensional manifold $W$ with boundary $\p{W}=X\sqcup X'$, see Fig. \ref{fig:trace}. It is formed by attaching a solid handle $D^{p+1}\times D^{q+1}$ onto the cylinder $X\times I$, identifying the $S^{p}\times D^{q+1}$ part of the boundary of $D^{p+1}\times D^{q+1}$ with the embedded $S^{p}\times D^{q+1}$ in $X\times \{1\}$ via the framed embedding $\bar{i}$. The trace of a surgery is an example of a cobordism. In general, a {\it cobordism} between $n$-dimensional manifolds $X_0$ and $X_1$ is an $n+1$-dimensional manifold $W^{n+1}=\{W^{n+1};X_0,X_1\}$ with boundary $\p {W}=X_0\sqcup X_1$. Cobordisms which arise as the trace of a surgery are known as {\it elementary cobordisms}. By taking appropriate unions of elementary cobordisms it is clear that more general cobordisms can be constructed. An important consequence of Morse theory is that the converse is also true, that is any compact cobordism $\{W^{n+1};X_0,X_1\}$ may be decomposed as a finite union of elementary cobordisms. This is a subject we will return to later on.

\begin{figure}[htbp]
\begin{picture}(0,0)%
\includegraphics{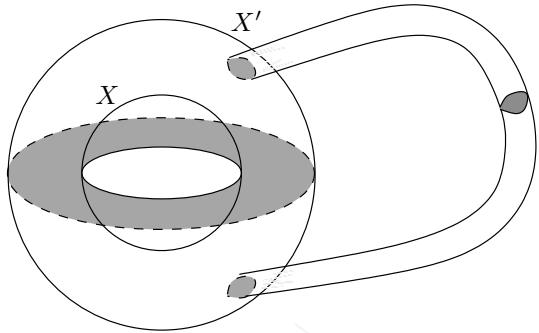}%
\end{picture}%
\setlength{\unitlength}{3947sp}%
\begingroup\makeatletter\ifx\SetFigFont\undefined%
\gdef\SetFigFont#1#2#3#4#5{%
  \reset@font\fontsize{#1}{#2pt}%
  \fontfamily{#3}\fontseries{#4}\fontshape{#5}%
  \selectfont}%
\fi\endgroup%
\begin{picture}(3339,2281)(1665,-3492)
\put(2226,-1849){\makebox(0,0)[lb]{\smash{{\SetFigFont{10}{8}{\rmdefault}{\mddefault}{\updefault}{\color[rgb]{0,0,0}$X$}%
}}}}
\put(3076,-1386){\makebox(0,0)[lb]{\smash{{\SetFigFont{10}{8}{\rmdefault}{\mddefault}{\updefault}{\color[rgb]{0,0,0}$X'$}%
}}}}
\end{picture}%

\caption{The trace of a $p$-surgery on $X$}
\label{fig:trace}
\end{figure}

\subsection{Surgery and positive scalar curvature}
The Surgery Theorem of Gromov-Lawson and Schoen-Yau can now be stated as follows. 
\\

\noindent{\emph{\bf Surgery Theorem.}} {\rm (\cite{GL1}, \cite{SY})}\label{SurgeryTheorem}
 {\sl Let $(X,g)$
    be a Riemannian manifold of positive scalar curvature. Let $X'$ be
    a manifold which has been obtained from $X$ by a surgery of
    codimension at least $3$. Then $X'$ admits a metric $g'$ which also has positive scalar curvature.}

\begin{Remark}
We will concentrate on the technique used by Gromov and Lawson, however, the proof of the Surgery Theorem by Schoen and Yau in \cite{SY} is rather different and involves conformal methods. There is in fact another approach to the problem of classifying manifolds of positive scalar curvature which involves conformal geometry, see for example the work of Akutagawa and Botvinnik in \cite{AB}. 
\end{Remark}

\begin{figure}[htbp]
\begin{picture}(0,0)%
\includegraphics{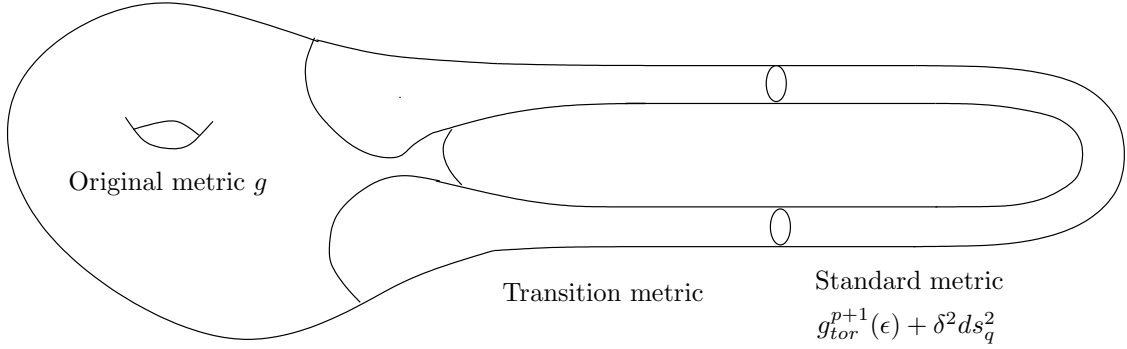}%
\end{picture}%
\setlength{\unitlength}{3947sp}%
\begingroup\makeatletter\ifx\SetFigFont\undefined%
\gdef\SetFigFont#1#2#3#4#5{%
  \reset@font\fontsize{#1}{#2pt}%
  \fontfamily{#3}\fontseries{#4}\fontshape{#5}%
  \selectfont}%
\fi\endgroup%
\begin{picture}(7038,2154)(330,-2515)
\put(726,-1548){\makebox(0,0)[lb]{\smash{{\SetFigFont{10}{8}{\rmdefault}{\mddefault}{\updefault}{\color[rgb]{0,0,0}Original metric $g$}%
}}}}
\put(3451,-2248){\makebox(0,0)[lb]{\smash{{\SetFigFont{10}{8}{\rmdefault}{\mddefault}{\updefault}{\color[rgb]{0,0,0}Transition metric}%
}}}}
\put(5414,-2174){\makebox(0,0)[lb]{\smash{{\SetFigFont{10}{8}{\rmdefault}{\mddefault}{\updefault}{\color[rgb]{0,0,0}Standard metric}%
}}}}
\put(5426,-2449){\makebox(0,0)[lb]{\smash{{\SetFigFont{10}{8}{\rmdefault}{\mddefault}{\updefault}{\color[rgb]{0,0,0}$g_{tor}^{p+1}(\epsilon)+\delta^{2}ds_{q}^{2}$}%
}}}}
\end{picture}%
\caption{The metric $g'$, obtained by the Surgery Theorem}
\label{fig:GLmetric}
\end{figure}

\noindent In their proof, Gromov and Lawson provide a technique for constructing the metric $g'$, see Fig. \ref{fig:GLmetric}. Their technique can be strengthened to yield the following theorem.

\begin{Theorem}\label{ImprovedsurgeryTheorem}
Let $(X,g)$ be a Riemannian manifold of positive scalar curvature. 
If $W^{n+1}$ is the trace of a surgery on $X$ in codimension at least $3$, then 
we can extend the metric $g$ to a metric $\bar{g}$ on $W$ which has positive scalar 
curvature and is a product near the boundary.
\end{Theorem}

\noindent In fact, the restriction of the metric $\bar{g}$ to $X'$, the boundary component of $W$ which is 
the result of the surgery, is the metric $g'$ of the Surgery Theorem. Theorem \ref{ImprovedsurgeryTheorem} is sometimes referred to as the Improved Surgery Theorem and was originally proved by Gajer in \cite{Gajer}.
We have two reasons for providing a proof of Theorem \ref{ImprovedsurgeryTheorem}. Firstly, there is an error 
in Gajer's original proof. Secondly, this construction will be used as a ``building block" for generating concordances. In turn, it will allow us to describe a space of concordances, see section \ref{introthree} for a discussion of this.

\begin{figure}[htbp]
\begin{picture}(0,0)%
\includegraphics{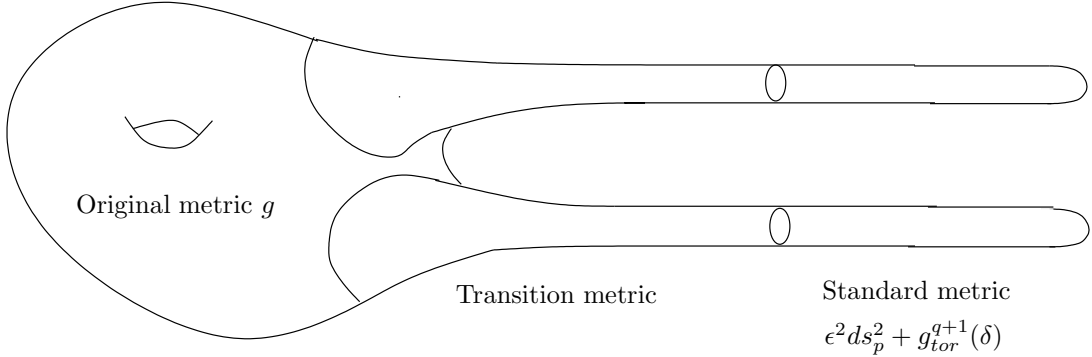}%
\end{picture}%
\setlength{\unitlength}{3947sp}%
\begingroup\makeatletter\ifx\SetFigFont\undefined%
\gdef\SetFigFont#1#2#3#4#5{%
  \reset@font\fontsize{#1}{#2pt}%
  \fontfamily{#3}\fontseries{#4}\fontshape{#5}%
  \selectfont}%
\fi\endgroup%
\begin{picture}(6820,2228)(618,-2077)
\put(3451,-1749){\makebox(0,0)[lb]{\smash{{\SetFigFont{10}{8}{\rmdefault}{\mddefault}{\updefault}{\color[rgb]{0,0,0}Transition metric}%
}}}}
\put(5751,-1724){\makebox(0,0)[lb]{\smash{{\SetFigFont{10}{8}{\rmdefault}{\mddefault}{\updefault}{\color[rgb]{0,0,0}Standard metric}%
}}}}
\put(5764,-2011){\makebox(0,0)[lb]{\smash{{\SetFigFont{10}{8}{\rmdefault}{\mddefault}{\updefault}{\color[rgb]{0,0,0}$\epsilon^{2}ds_{p}^{2}+g_{tor}^{q+1}(\delta)$}%
}}}}
\put(1064,-1186){\makebox(0,0)[lb]{\smash{{\SetFigFont{10}{8}{\rmdefault}{\mddefault}{\updefault}{\color[rgb]{0,0,0}Original metric $g$}%
}}}}
\end{picture}%
\caption{The ``surgery-ready" metric obtained by Theorem \ref{IsotopyTheorem}}
\label{fig:IsotopyLemmametric}
\end{figure}

The proof of Theorem \ref{ImprovedsurgeryTheorem} will dominate much of the rest of this section. 
We will first prove a theorem which strengthens the Surgery Theorem in a slightly different way. This is Theorem \ref{IsotopyTheorem} below, which will play a vital role throughout our work.
 
\begin{Theorem}\label{IsotopyTheorem}
Let $(X,g)$ be a Riemannian manifold of positive scalar curvature with $dim X=n$ and let $g_p$ be any metric on the sphere $S^{p}$. Suppose $i:S^{p}\hookrightarrow X$ is an embedding of $S^{p}$, with trivial normal bundle. Suppose also that $p+q+1=n$ and that $q\geq 2$. Then, for some $\delta>0$ there is an isotopy of $g$, to a psc-metric $g_{std}$ on $X$, which has the form $g_p+g_{tor}^{q+1}(\delta)$ on a tubular neighbourhood of the embedded $S^{p}$ and is the original metric $g$ away from this neighbourhood.
\end{Theorem}

\begin{Corollary}\label{GLconc}
There is a metric $\bar{g}$ on $X\times I$ satisfying\\
(i) $\bar{g}$ has positive scalar curvature.\\
(ii) $\bar{g}$ restricts to $g$ on $X\times{\{0\}}$, $g_{std}$ on $X\times{\{1\}}$ and is product near the boundary.\\
$\bar{g}$ is therefore a concordance of $g$ and $g_{std}$.
\end{Corollary}
\begin{proof} This follows immediately from Lemma \ref{isoconc}.
\end{proof}

\begin{Remark}
The proof of Theorem \ref{IsotopyTheorem} is not made any simpler by choosing a particular metric for $g_p$. Indeed, the embedded sphere $S^{p}$ can be replaced by any closed codimension$\geq 3$ submanifold with trivial normal bundle, and the result still holds with an essentially identical proof. That said, we are really only
interested in the case of an embedded sphere and moreover, the case when $g_p$ is the round metric $\epsilon^{2}ds_p^{2}.$ 
\end{Remark}

The proof of Theorem \ref{IsotopyTheorem} is long and technical. Contained in it is the proof 
of the original Surgery Theorem of Gromov and Lawson, see \cite{GL1}. Here the authors consider an embedded 
surgery sphere $S^{p}$ in a Riemannian manifold $(X^{n},g)$ where $n-p\geq 3$ and $g$ is a psc-metric. 
Their construction directly implies that the metric $g$ can be replaced 
by the psc-metric $g'$ described in the statement of Theorem \ref{IsotopyTheorem}, where in this case $g_p=\epsilon^{2}ds_{p}^{2}$. Thus, Gromov and 
Lawson prepare the metric for surgery by making it standard near the surgery sphere. By performing the surgery entirely on the standard region, it is then possible to 
attach a handle $D^{p+1}\times S^{n-p-1}$ with a correponding standard metric, $g_{tor}^{p+1}(\epsilon)+\delta^{2}ds_{n-p-1}^{2}$ onto $X-\bar{i}(S^{p}\times {\oDnp})$, as in Fig. \ref{fig:GLmetric}.

Rather than attaching a handle metric, Theorem \ref{IsotopyTheorem} states that the ``surgery-ready" 
metric $g_{std}$ on $X$, see Fig. \ref{fig:IsotopyLemmametric}, is actually isotopic to the original metric $g$. Thus the concordance $\bar{g}$ on $X\times I$, which is described in Corollary \ref{GLconc}, can be built. The proof of Theorem \ref{ImprovedsurgeryTheorem} then proceeds by attaching a solid handle $D^{p+1}\times D^{n-p}$ to $X\times I$, with an appropriate standard metric. After smoothing, this will result in a metric of positive scalar curvature on the trace of the surgery on $S^{p}$. The only remaining task in the proof of Theorem \ref{ImprovedsurgeryTheorem} is to show 
that this metric can be adjusted to also carry a product structure near the boundary.

\subsection{Outline of the proof of Theorem \ref{IsotopyTheorem}}
Although the result is known, Theorem \ref{IsotopyTheorem} is based on a number of technical 
lemmas from a variety of sources, in particular \cite{GL1}, \cite{RS}. 
For the most part, it is a reworking of Gromov and Lawson's proof of the Surgery Theorem. To aid the reader we relegate many of the more technical proofs to the appendix. We begin with a brief summary.
\\

\noindent{\bf Part 1:} Using the exponential map we can specify a tubular neighbourhood $N=S^{p}\times D^{q+1}$, of the embedded sphere $S^{p}$. Henceforth, all of our work will take place in this neighbourhood. We construct a hypersurface $M$ in $N\times \mathbb{R}$ where $N\times\mathbb{R}$ is equipped with the metric $g+dt^{2}$. Letting $r$ denote the radial distance from $S^{p}\times\{0\}$ in $N$, this hypersurface is obtained by pushing out bundles of geodesic spheres of radius $r$ in $N$ along the $t$-axis with respect to some smooth curve $\gamma$ of the type depicted in Fig. \ref{gamma}. In Lemmas \ref{principalM} and \ref{scalM}, we compute the scalar curvature of the metric $g_\gamma$ which is induced on the hypersurface $M$.

\noindent {\bf Part 2:} We recall the fact that $\gamma$ can be chosen so that the metric $g_\gamma$ has positive scalar curvature. This fact was originally proved in \cite{GL1} although later, in \cite{RS}, an error in the original proof was corrected. We will employ the method used by Rosenberg and Stolz in \cite{RS} to construct such a curve $\gamma$ and we will then demonstrate that $\gamma$ can be homotopied through appropriate curves back to the vertical axis inducing an isotopy from the psc-metric $g_\gamma$ back to the orginal psc-metric $g$. We will also comment on the error in the proof of the ``Improved Surgery Theorem", Theorem 4 in \cite{Gajer}, see Remark \ref{Gajercomment}. 

\noindent {\bf Part 3:} We will now make a further deformation to the metric $g_\gamma$ induced on $M$. Here we restrict our attention to the part of $M$ arising from the torpedo part of $\gamma$. Lemma \ref{GLlemma1} implies that $M$ can be chosen so that the metric induced on the fibre disks can be made arbitrarily close to the standard torpedo metric of radius $\delta$. It is therefore possible to isotopy the metric $g$, through psc-metrics, to one which near $S^{p}$ is a Riemannian submersion with base metric $g|_{S^{p}}$ and fibre metric $g_{tor}^{q+1}(\delta)$. Using the formulae of O'Neill, we will show that the positivity of the curvature on the disk factor allows us to isotopy through psc-submersion metrics near $S^{p}$ to obtain the desired metric $g_{std}=g_p+g_{tor}^{q+1}(\delta)$.
 
\begin{proof} Let $X^{n}$ be a manifold of dimension $n\geq 3$ and $g$ a metric of 
positive scalar curvature on $X$.
\\

\subsection{Part 1 of the proof: Curvature formulae for the first deformation}

Let $i:S^{p}\hookrightarrow X$ be an embedding with trivial normal bundle, denoted by $\N$, and with $q\geq 2$ where $p+q+1=n$. By choosing an orthonormal frame for $\N$ over $i(S^{p})$, we specify a bundle isomorphism $\tilde{i}:S^{p}\times \mathbb{R}^{q+1}\rightarrow\mathcal{N}$. Points in $S^{p}\times \mathbb{R}^{q+1}$ will be denoted $(y,x)$. Let $r$ denote the standard Euclidean distance function in $\mathbb{R}^{q+1}$ and let $D^{q+1}(\bar{r})=\{x\in\mathbb{R}^{q+1}:r(x)\leq\bar{r}\}$ denote the standard Euclidean disk of radius $\bar{r}$ in $\mathbb{R}^{q+1}$. Provided $\bar{r}$ is sufficiently small, the composition $\exp\circ\tilde{i}|_{S^{p}\times D^{q+1}(\bar{r})}$, where $\exp$ denotes the exponential map with respect to the metric $g$, is an embedding. We will denote by $N=N(\bar{r})$, the image of this embedding and the coordinates $(y,x)$ will be used to denote points on $N$. Note that curves of the form $\{y\}\times l$, where $l$ is a ray in $D^{{q}}(\bar{r})$ emanating from $0$, are geodesics in $N$.

Before proceeding any further we state a lemma concerning the metric induced on a geodesic sphere of a Riemannian manifold. Fix $z\in X$ and let $D$ be a normal coordinate ball of radius $\bar{r}$ around $z$. Recall, this means first choosing an orthonormal basis $\{e_1,...,e_n\}$ for $T_z X$. This determines an isomorphism $E:(x_1,...,x_n)\mapsto x_1 e_1+\cdots +x_n e_n$ from $\mathbb{R}^{n}$ to $T_z X$. The composition $E^{-1}\circ \exp^{-1}$ is a coordinate map provided we restrict it to an appropriate neighbourhood of $z$. Thus we identify $D=\{x\in \mathbb{R}^{n}:|x|\leq \bar{r}\}$. The quantity $r(x)=|x|$ is the radial distance from the point $z$, and  $S^{n-1}(\epsilon)=\{x\in\mathbb{R}^{n}:|x|=\epsilon\}$ will denote the geodesic sphere of radius $\epsilon$ around $z$.
\\
\begin{Lemma} \label{GLlemma1} {\rm{(Lemma 1, \cite{GL1})}}

\noindent (a) The principal curvatures of the hypersurfaces $S^{n-1}{(\epsilon})$ in D are each of the form $\frac{-1}{\epsilon}+O(\epsilon)$ for $\epsilon$ small. 

\noindent (b) Furthermore, let $g_\epsilon$ be the induced metric on $S^{n-1}{(\epsilon})$ and let $g_{0,\epsilon}$ be the standard Euclidean metric of curvature $\frac{1}{\epsilon^2}$. Then as $\epsilon\rightarrow 0, \frac{1}{\epsilon^2}{g_\epsilon}\rightarrow{\frac{1}{\epsilon^2}g_{0,\epsilon}}=g_{0,1}$ in the $C^{2}$-topology.

\noindent Below we use the following notation. A function $f(r)$ is $O(r)$ as $r\rightarrow 0$ if $\frac{f(r)}{r}\rightarrow constant$ as $r\rightarrow 0$.
\end{Lemma}

\noindent This lemma was originally proved in \cite{GL1}. In the appendix, we provide a complete proof, which includes details suppressed in the original, see Theorem \ref{GLlemma1app}. In order to deform the metric on $N$ we will construct a hypersurface in $N\times\mathbb{R}$. Let $r$ denote the radial distance from $S^{p}\times \{0\}$ on $N$ and $t$ the coordinate on $\mathbb{R}$. Let $\gamma$ be a $C^{2}$ curve in the $t-r$ plane which satisfies the following conditions, see Fig. \ref{gamma}. 

\noindent 1. For some $\bar{t}>0$, $\gamma$ lies entirely inside the rectangle $[0,\bar{r}]\times[0,{\bar{t}}]$, beginning at the point $(0,\bar{r})$ and ending at the point $(\bar{t},0)$. There are points $(0, r_1), (t_1', r_1'), (t_0,r_0)$ and $(t_\infty, r_\infty)$ on the interior of $\gamma$ with $0<r_\infty<r_0<\frac{r_1}{2}<r_1'<r_1<\bar{r}$ and $0<t_1'<t_0<t_\infty<\bar{t}$. We will assume that $\bar{t}-t_\infty$ is much larger than $r_\infty$.

\noindent 2. When $r\in[r_0,\bar{r}]$, $\gamma$ is the graph of a function $f_0$ with domain on the $r$-axis satisfying: $f_0(r)=0$ when $r\in[r_1,\bar{r}]$, $f_0(r)=t_1'-\tan{\theta_0}(r-r_1')$ for some $\theta_0\in(0,\frac{\pi}{2})$ when $r\in[r_0, r_1']$ and with $\dot{f_0}\leq 0$ and $\ddot{f_0}\geq 0$.

\noindent 3. When $r\in[0,r_\infty]$, $\gamma$ is the graph of a function $f_\infty$ defined over the interval $[t_\infty, \bar{t}]$ of the $t$-axis. The function $f_\infty$ is given by the formula $f_{\infty}(t)=f_{r_\infty}(\bar{t}-t)$ where $f_{r_\infty}$ is an $r_\infty$-torpedo function of the type described at the beginning of section \ref{torpedofunc}.

\noindent 4. Inside the rectangle $[t_0, t_\infty]\times[r_\infty, r_0]$, $\gamma$ is the graph of a $C^{2}$ function $f$ with $f(t_0)=r_0$, $f(t_\infty)=r_\infty$, $\dot{f}\leq 0$ and $\ddot{f}\geq 0$. 

\begin{figure}[htbp]
\begin{picture}(0,0)%
\includegraphics[height=60mm]{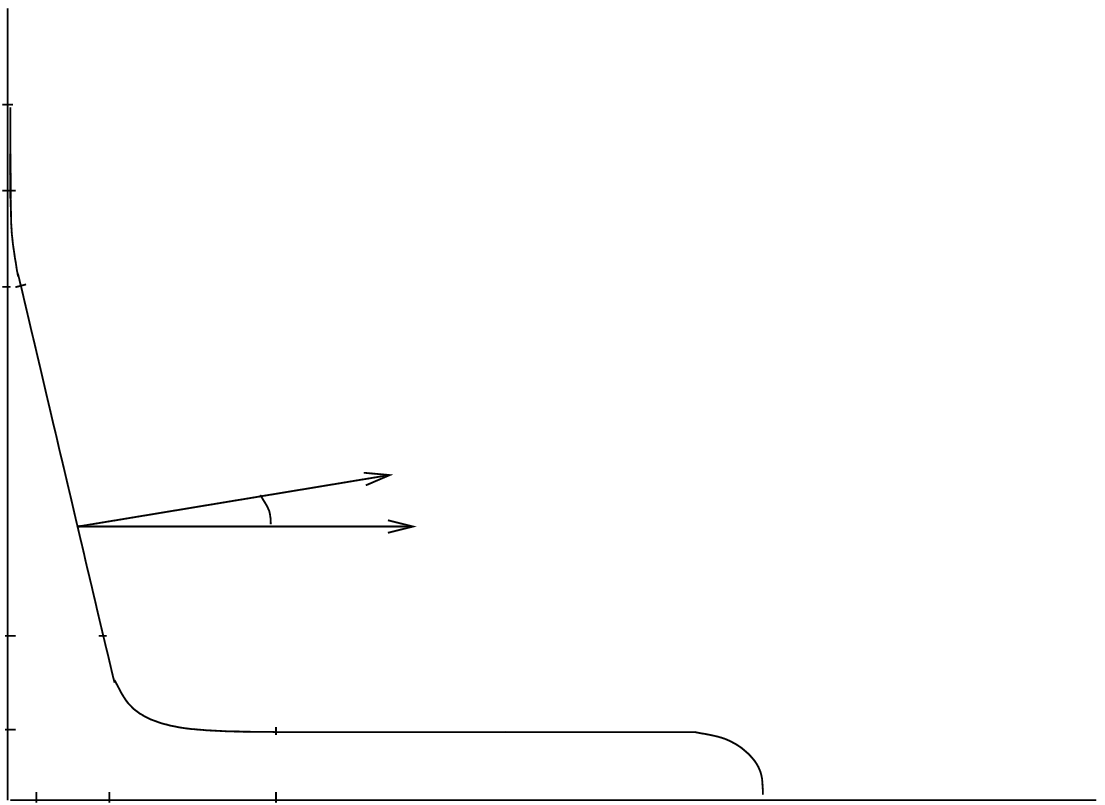}%
\end{picture}%
\setlength{\unitlength}{3947sp}%
\begingroup\makeatletter\ifx\SetFigFont\undefined%
\gdef\SetFigFont#1#2#3#4#5{%
  \reset@font\fontsize{#1}{#2pt}%
  \fontfamily{#3}\fontseries{#4}\fontshape{#5}%
  \selectfont}%
\fi\endgroup%
\begin{picture}(3889,2884)(1374,-4954)
\put(1180,-4690){\makebox(0,0)[lb]{\smash{{\SetFigFont{8}{8}{\rmdefault}{\mddefault}{\updefault}{\color[rgb]{0,0,0}$r_\infty$}%
}}}}
\put(1180,-4375){\makebox(0,0)[lb]{\smash{{\SetFigFont{8}{8}{\rmdefault}{\mddefault}{\updefault}{\color[rgb]{0,0,0}$r_0$}%
}}}}
\put(1180,-2825){\makebox(0,0)[lb]{\smash{{\SetFigFont{8}{8}{\rmdefault}{\mddefault}{\updefault}{\color[rgb]{0,0,0}$r_1$}%
}}}}
\put(1180,-3145){\makebox(0,0)[lb]{\smash{{\SetFigFont{8}{8}{\rmdefault}{\mddefault}{\updefault}{\color[rgb]{0,0,0}$r_1'$}%
}}}}
\put(1180,-2553){\makebox(0,0)[lb]{\smash{{\SetFigFont{8}{8}{\rmdefault}{\mddefault}{\updefault}{\color[rgb]{0,0,0}$\bar{r}$}%
}}}}
\put(1700,-5088){\makebox(0,0)[lb]{\smash{{\SetFigFont{8}{8}{\rmdefault}{\mddefault}{\updefault}{\color[rgb]{0,0,0}$t_0$}%
}}}}
\put(2294,-5088){\makebox(0,0)[lb]{\smash{{\SetFigFont{8}{8}{\rmdefault}{\mddefault}{\updefault}{\color[rgb]{0,0,0}$t_\infty$}%
}}}}
\put(2414,-3938){\makebox(0,0)[lb]{\smash{{\SetFigFont{8}{8}{\rmdefault}{\mddefault}{\updefault}{\color[rgb]{0,0,0}$\theta_0$}%
}}}}
\put(1800,-4375){\makebox(0,0)[lb]{\smash{{\SetFigFont{8}{8}{\rmdefault}{\mddefault}{\updefault}{\color[rgb]{0,0,0}$(t_0,r_0)$}%
}}}}
\put(4159,-5088){\makebox(0,0)[lb]{\smash{{\SetFigFont{8}{8}{\rmdefault}{\mddefault}{\updefault}{\color[rgb]{0,0,0}$\bar{t}$}%
}}}}
\put(2294,-4640){\makebox(0,0)[lb]{\smash{{\SetFigFont{8}{8}{\rmdefault}{\mddefault}{\updefault}{\color[rgb]{0,0,0}$(t_\infty, r_\infty)$}%
}}}}
\put(1430,-5088){\makebox(0,0)[lb]{\smash{{\SetFigFont{8}{8}{\rmdefault}{\mddefault}{\updefault}{\color[rgb]{0,0,0}$t_1'$}%
}}}}
\put(1500,-3145){\makebox(0,0)[lb]{\smash{{\SetFigFont{8}{8}{\rmdefault}{\mddefault}{\updefault}{\color[rgb]{0,0,0}$(t_1', r_1')$}%
}}}}
\end{picture}%
\caption{The curve $\gamma$}
\label{gamma}
\end{figure}

The curve $\gamma$ specifies a hypersurface in $N\times\mathbb{R}$ in the following way. Equip $N\times\mathbb{R}$ with the product metric $g+dt^{2}$. Define $M=M_{\gamma}$ to be the hypersurface, shown in Fig. \ref{hyperm} 
and defined 

\begin{equation*}
M_{\gamma}=\{(y,x,t)\in S^{p}\times D^{q+1}(\bar{r})\times\mathbb{R}:(r(x),t)\in{\gamma}\}.
\end{equation*}

\noindent We will denote by $g_\gamma$, the metric induced on the hypersurface $M$. The fact that $\gamma$ is a vertical line near the point $(0,\bar{r})$ means that $g_\gamma=g$, near $\p{N}$. Thus, $\gamma$ specifies a metric on $X$ which is the orginal metric $g$ outside of $N$ and then transitions smoothly to the metric $g_\gamma$. Later we will show that such a curve can be constructed so that $g_\gamma$ has positive scalar curvature. In the meantime, we will derive an expression for the scalar curvature of $g_\gamma$, by computing principal curvatures for $M$ with respect to the outward unit normal vector field and then utilising the Gauss curvature equation, see Lemmas \ref{principalM} and \ref{scalM}. Details of these computations can be found in the appendix, see Lemmas \ref{principalMapp} and \ref{scalarMapp} respectively.

\begin{figure}[htbp]
\begin{picture}(0,0)%
\includegraphics{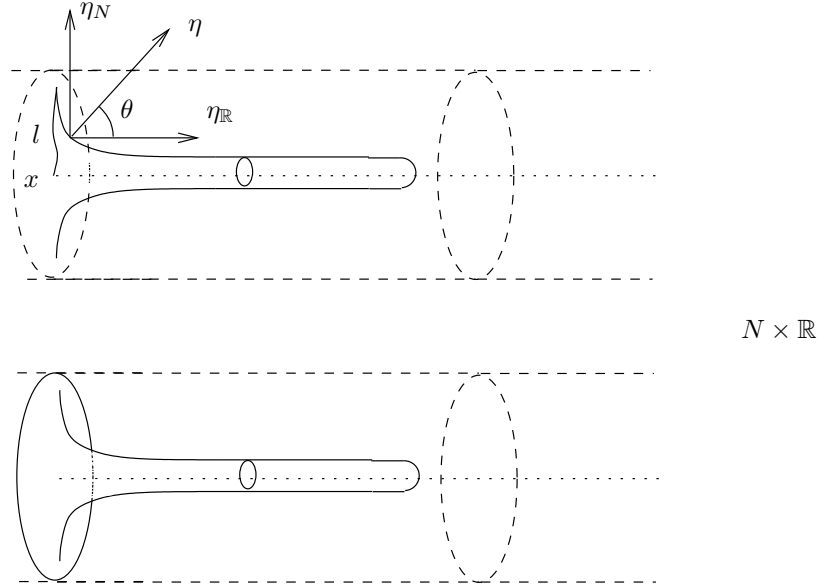}%
\end{picture}%
\setlength{\unitlength}{3947sp}%
\begingroup\makeatletter\ifx\SetFigFont\undefined%
\gdef\SetFigFont#1#2#3#4#5{%
  \reset@font\fontsize{#1}{#2pt}%
  \fontfamily{#3}\fontseries{#4}\fontshape{#5}%
  \selectfont}%
\fi\endgroup%
\begin{picture}(4666,3733)(1218,-3861)
\put(5869,-2318){\makebox(0,0)[lb]{\smash{{\SetFigFont{10}{8}{\rmdefault}{\mddefault}{\updefault}{\color[rgb]{0,0,0}$N\times\mathbb{R}$}
}}}}
\put(1716,-282){\makebox(0,0)[lb]{\smash{{\SetFigFont{10}{8}{\rmdefault}{\mddefault}{\updefault}{\color[rgb]{0,0,0}$\eta_N$}
}}}}
\put(2398,-386){\makebox(0,0)[lb]{\smash{{\SetFigFont{10}{8}{\rmdefault}{\mddefault}{\updefault}{\color[rgb]{0,0,0}$\eta$}
}}}}
\put(2512,-933){\makebox(0,0)[lb]{\smash{{\SetFigFont{10}{8}{\rmdefault}{\mddefault}{\updefault}{\color[rgb]{0,0,0}$\eta_{\mathbb{R}}$}
}}}}
\put(1972,-933){\makebox(0,0)[lb]{\smash{{\SetFigFont{10}{8}{\rmdefault}{\mddefault}{\updefault}{\color[rgb]{0,0,0}$\theta$}
}}}}
\put(1365,-1367){\makebox(0,0)[lb]{\smash{{\SetFigFont{10}{8}{\rmdefault}{\mddefault}{\updefault}{\color[rgb]{0,0,0}$x$}
}}}}
\put(1417,-1088){\makebox(0,0)[lb]{\smash{{\SetFigFont{10}{8}{\rmdefault}{\mddefault}{\updefault}{\color[rgb]{0,0,0}$l$}
}}}}
\end{picture}
\caption{The hypersurface $M$ in $N\times\mathbb{R}$, the sphere $S^{p}$ is represented schematically as a pair of points}
\label{hyperm}
\end{figure}

\begin{Lemma}\label{principalM}
The prinipal curvatures to $M$ with respect to the outward unit normal vector field have the form
\begin{equation}
\begin{array}{cl}
\lambda_j =
\begin{cases}
k & \text{if $j=1$}\\
(-\frac{1}{r}+O(r))\sin{\theta} & \text{if $2\leq j\leq q+1$}\\
O(1)\sin{\theta} & \text{if $q+2\leq j\leq n$}.
\end{cases}
\end{array}
\end{equation} 
Here $k$ is the curvature of $\gamma$, $\theta$ is the angle between the outward normal vector $\eta$ and the horizontal (or the outward normal to the curve $\gamma$ and the $t$-axis) and the corresponding principle directions $e_j$ are tangent to the curve $\gamma$ when $j=1$, the fibre sphere $S^{q}$ when $2\leq j\leq q+1$ and $S^{p}$ when $q+2\leq j\leq n$. 
\end{Lemma}

\begin{proof}
See appendix, lemma \ref{principalMapp}.
\end{proof}

\begin{Lemma}\label{scalM}
The scalar curvature of the metric induced on $M$ is given by
\begin{equation}\label{scalarMeqn}
\begin{split}
R^{M}&=R^{N}+\sin^{2}{\theta}\cdot O(1)-2k\cdot q\frac{\sin{\theta}}{r}\\
&\hspace{0.4cm} +2q(q-1)\frac{\sin^{2}{\theta}}{r^2}+k\cdot qO(r)\sin{\theta}.
\end{split}
\end{equation}
\end{Lemma}

\begin{proof}
See appendix, lemma \ref{scalarMapp}.
\end{proof}

\subsection{Part 2 of the proof: A continuous bending argument}

In this section we will prove the following lemma. 

\begin{Lemma}\label{Pushoutcurvehaspsc} 
{\it The curve $\gamma$ can be chosen so that the induced metric $g_\gamma$, on the hypersurface $M=M_{\gamma}$, has positive scalar curvature and is isotopic to the original metric $g$.}
\end{Lemma}

Before proving this lemma, it is worth simplifying some of our formulae. From formula (\ref{scalarMeqn}) we see that to keep $R^{M}>0$ we must choose $\gamma$ so that 

\begin{equation*}
\begin{array}{c}
k\left[2q\frac{\sin{\theta}}{r}+qO(r)\sin{\theta}\right] < R^{N}+\sin^{2}{\theta}\cdot O(1)+2q(q-1)\frac{\sin^{2}{\theta}}{r^2}.
\end{array}
\end{equation*}

\noindent This inequality can be simplified to  

\begin{equation*}
\begin{array}{c}
k[\frac{\sin{\theta}}{r}+O(r)\sin{\theta}]<R_{0}+\sin^{2}{\theta}\cdot O(1)+(q-1)\frac{\sin^{2}{\theta}}{r^2}
\end{array}
\end{equation*}

\noindent where 

\begin{equation*}
\begin{array}{c}
R_0=\frac{1}{2q}[\inf_{N}(R^{N})] 
\end{array}
\end{equation*}
\noindent and $\inf_{N}(R^{N})$ is the infimum of the function $R^{N}$ on the neighbourhood $N$. Simplifying further, we obtain

\begin{equation*}
\begin{array}{c}
k[1+O(r)r]<R_{0}\frac{r}{\sin{\theta}}+r{\sin{\theta}}\cdot O(1)+(q-1)\frac{\sin{\theta}}{r}.
\end{array}
\end{equation*}

\noindent Replace $O(r)$ with $C'r$ for some constant $C'>0$ and replace $O(1)$ with $-C$ where $C>0$, assuming the worst case scenario that $O(1)$ is negative. Now we have

\begin{equation}\label{cureqn}
\begin{array}{c}
k[1+C'r^{2}]<R_{0}\frac{r}{\sin{\theta}}+(q-1)\frac{\sin{\theta}}{r}-Cr{\sin{\theta}}.
\end{array}
\end{equation}

The proof of Lemma \ref{Pushoutcurvehaspsc} is quite complicated and so it is worth giving an overview. We denote by $\gamma^{0}$, the curve which in the $t-r$-plane runs vertically down the $r$-axis, beginning at $(0,\bar{r})$ and finishing at $(0,0)$. Now consider the curve $\gamma^{\theta_0}$, shown in Fig. \ref{fig:initialbend}. This curve begins as $\gamma^{0}$ before smoothly bending upwards over some angle small angle $\theta_0\in(0,\frac{\pi}{2})$ to proceed as a straight line segment before finally bending downwards to intersect the $t$-axis vertically. The corresponding hypersurface in $N\times\mathbb{R}$, constructed exactly as before, will be denoted by $M_{\gamma^{\theta_0}}$ and the induced metric by $g_{\gamma^{\theta_0}}$. The strict positivity of the scalar curvature of $g$ means that provided we choose $\theta_0$ to be sufficiently small, the scalar curvature of the metric $g_{\gamma^{\theta_0}}$ will be strictly positive. It will then be a relatively straightforward exercise to construct a homotopy of $\gamma^{\theta_0}$ back to $\gamma^{0}$ which induces an isotopy of the metrics $g_{\gamma^{\theta_0}}$ and $g$.

To obtain the curve $\gamma$, we must perform one final upward bending on $\gamma^{\theta_0}$. This will take place on the straight line piece below the first upward bend. This time we will bend the curve right around by an angle of $\frac{\pi}{2}-\theta_0$ to proceed as a horizontal line segment, before bending downwards to intersect the $t$-axis vertically, see Fig. \ref{gamma}. We must ensure throughout that inequality (\ref{cureqn}) is satisfied. In this regard, we point out that the downward bending, provided we maintain downward concavity, causes us no difficulty as here $k\leq 0$. The difficulty lies in performing an upward bending, where this inequality is reversed. 

Having constructed $\gamma$, our final task will be to demonstrate that it is possible to homotopy $\gamma$ back to $\gamma^{\theta_0}$ in such a way as to induce an isotopy between the metrics $g_\gamma$ and $g_{\gamma^{\theta_0}}$. This, combined with the previously constructed isotopy of $g_{\gamma}$ and $g$, will complete the proof.

\begin{figure}[htbp]
\begin{picture}(0,0)%
\includegraphics{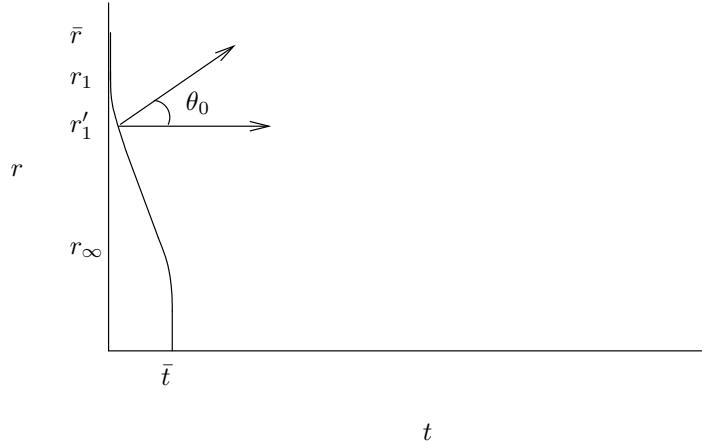}%
\end{picture}%
\setlength{\unitlength}{3947sp}%
\begingroup\makeatletter\ifx\SetFigFont\undefined%
\gdef\SetFigFont#1#2#3#4#5{%
  \reset@font\fontsize{#1}{#2pt}%
  \fontfamily{#3}\fontseries{#4}\fontshape{#5}%
  \selectfont}%
\fi\endgroup%
\begin{picture}(4365,2828)(1636,-3890)
\put(4239,-3824){\makebox(0,0)[lb]{\smash{{\SetFigFont{10}{8}{\rmdefault}{\mddefault}{\updefault}{\color[rgb]{0,0,0}$t$}%
}}}}
\put(1651,-2161){\makebox(0,0)[lb]{\smash{{\SetFigFont{10}{8}{\rmdefault}{\mddefault}{\updefault}{\color[rgb]{0,0,0}$r$}%
}}}}
\put(2751,-1749){\makebox(0,0)[lb]{\smash{{\SetFigFont{10}{8}{\rmdefault}{\mddefault}{\updefault}{\color[rgb]{0,0,0}$\theta_0$}%
}}}}
\put(2594,-3474){\makebox(0,0)[lb]{\smash{{\SetFigFont{10}{8}{\rmdefault}{\mddefault}{\updefault}{\color[rgb]{0,0,0}$\bar{t}$}%
}}}}
\put(2021,-2649){\makebox(0,0)[lb]{\smash{{\SetFigFont{10}{8}{\rmdefault}{\mddefault}{\updefault}{\color[rgb]{0,0,0}$r_\infty$}%
}}}}
\put(2021,-1326){\makebox(0,0)[lb]{\smash{{\SetFigFont{10}{8}{\rmdefault}{\mddefault}{\updefault}{\color[rgb]{0,0,0}$\bar{r}$}%
}}}}
\put(2021,-1586){\makebox(0,0)[lb]{\smash{{\SetFigFont{10}{8}{\rmdefault}{\mddefault}{\updefault}{\color[rgb]{0,0,0}$r_1$}%
}}}}
\put(2021,-1886){\makebox(0,0)[lb]{\smash{{\SetFigFont{10}{8}{\rmdefault}{\mddefault}{\updefault}{\color[rgb]{0,0,0}$r_1'$}%
}}}}

\end{picture}%

\caption{The curve $\gamma^{\theta_0}$ resulting from the initial bend}
\label{fig:initialbend}
\end{figure}

\begin{proof}\noindent{\bf The initial bending:} For some $\theta_0>0$, $\gamma^{\theta_0}$ will denote the curve depicted in Fig. \ref{fig:initialbend}, parameterised by the arc length parameter $s$. Beginning at $(0,\bar{r})$, the curve $\gamma^{\theta_0}$ runs downward along the vertical axis to the point $(0,r_1)$ for some fixed $0<r_1<\bar{r}$. It then bends upwards by an angle of $\theta_0$, proceeding as a straight line segment with slope $m_0=\frac{-1}{\tan{\theta_0}}$, before finally bending downwards and with downward concavity to intersect the $t$-axis vertically. The curvature of $\gamma^{\theta_0}$ at the point $\gamma^{\theta_0}(s)$ is denoted by $k(s)$ and $\theta=\theta(s)$ will denote the angle made by the normal vector to $\gamma^{\theta_0}$ and the $t$-axis, at the point $\gamma^{\theta_0}(s)$.    

The bending itself will be obtained by choosing a very small bump function for $k$, with support on an interval of length $\frac{r_1}{2}$, see Fig. \ref{fig:k(s)}. This will ensure that the entire upward bending takes place over some interval $[r_1',r_1]$ which is contained entirely in $[\frac{r_1}{2},r_1]$.  
The downward bending will then begin at $r=r_\infty$, for some $r_\infty\in(0,\frac{r_1}{2})$.

We will first show that the parameters $\theta_0\in(0,\frac{\pi}{2})$ and $r_1\in(0,\bar{r})$ can be chosen so that inequality (\ref{cureqn}) holds for all $\theta\in[0,\theta_0]$ and all $r\in(0,r_1]$. Begin by choosing some $\theta_0\in(0,\arcsin{\sqrt{\frac{R_0}{C}}})$. This guarantees that the right hand side of (\ref{cureqn}) remains positive for all $\theta\in[0,\theta_0]$. For now, the variable $\theta$ is assumed to lie in $[0,\theta_0]$. Provided $\theta$ is close to zero, the term $R_{0}\frac{r}{\sin{\theta}}$ is positively large and dominates. When $\theta=0$ the right hand side of (\ref{cureqn}) is positively infinite. Once $\theta$ becomes greater than zero, the term $(q-1)\frac{\sin{\theta}}{r}$ can be made positively large by choosing $r$ small, and so can be made to dominate. Recall here that $q\geq 2$ by the assumption that the original surgery sphere had codimension at least three. It is therefore possible to choose $r_1>0$ so that inequality (\ref{cureqn}) holds for all $\theta\in[0,\theta_0]$ and for all $r\in(0,r_1]$. Note also that without the assumption that the scalar curvature of the original metric $g$ is strictly positive, this argument fails.

\begin{figure}[htbp]
\begin{picture}(0,0)%
\includegraphics{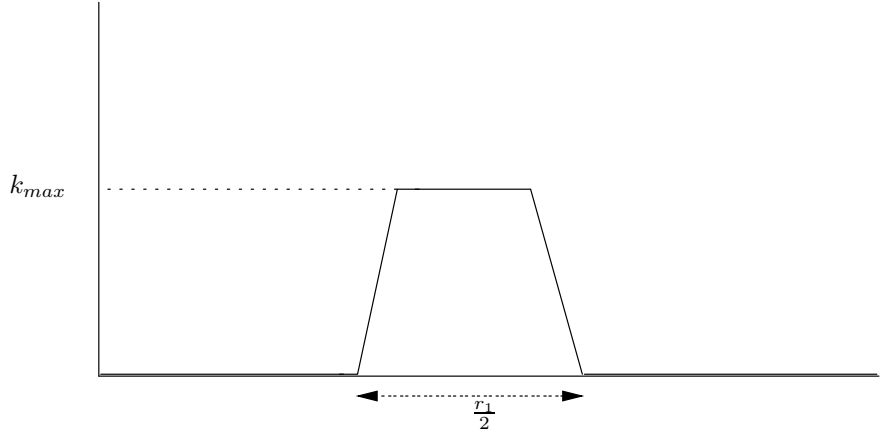}%
\end{picture}%
\setlength{\unitlength}{3947sp}%
\begingroup\makeatletter\ifx\SetFigFont\undefined%
\gdef\SetFigFont#1#2#3#4#5{%
  \reset@font\fontsize{#1}{#2pt}%
  \fontfamily{#3}\fontseries{#4}\fontshape{#5}%
  \selectfont}%
\fi\endgroup%
\begin{picture}(5490,2903)(811,-3977)
\put(3726,-3711){\makebox(0,0)[lb]{\smash{{\SetFigFont{10}{8}{\rmdefault}{\mddefault}{\updefault}{\color[rgb]{0,0,0}$\frac{r_1}{2}$}%
}}}}
\put(826,-2286){\makebox(0,0)[lb]{\smash{{\SetFigFont{10}{8}{\rmdefault}{\mddefault}{\updefault}{\color[rgb]{0,0,0}$k_{max}$}%
}}}}
\end{picture}%

\caption{The bump function $k$}
\label{fig:k(s)}
\end{figure}

We will now bend $\gamma^{0}$ to $\gamma^{\theta_0}$, smoothly increasing $\theta$ from $0$ to $\theta_0$. We do this by specifying a bump function $k$ which describes the curvature along $\gamma^{\theta_0}$, see Fig. \ref{fig:k(s)}. This gives

\begin{equation*}
\begin{array}{c}
\Delta\theta=\int kds\approx \frac{1}{2}r_1\cdot k_{max}.
\end{array}
\end{equation*}

\noindent This approximation can be made as close to equality as we wish. If necessary re-choose $\theta_0$ so that $\theta_0<\frac{1}{2}r_1\cdot k_{max}$. Note that $r_1$ has been chosen to make inequality (\ref{cureqn}) hold for all $\theta\in[0,\theta_0]$ and so rechoosing a smaller $\theta_0$ does not affect the choice of $r_1$. We need to show that $k_{max}>0$ can be found so that

\begin{equation*}
\begin{array}{c}
k_{max}[1+C'{r_1}^{2}]<R_{0}\frac{r_1}{\sin{\theta}}+(q-1)\frac{\sin{\theta}}{r_1}-C{r_1}{\sin{\theta}},
\end{array}
\end{equation*}

\noindent for all $\theta\in[0,\theta_0]$. But from the earlier argument, $r_1$ and $\theta_0$ have been chosen so that the right-hand side of this inequality is positive for all $\theta\in[0,\theta_0]$. So some such $k_{max}>0$ exists. This completes the initial upward bending.

The curve $\gamma^{\theta_0}$ then proceeds as a straight line before bending downwards, with downward concavity, to vertically intersect the $t$-axis. This downward concavity ensures that $k\leq 0$ and so inequality (\ref{cureqn}) is easily satisfied, completing the construction of $\gamma^{\theta_0}$.
\\

\noindent {\bf The initial isotopy:} Next we will show that $\gamma^{\theta_0}$ can be homotopied back to $\gamma^{0}$ in such a way as to induce an isotopy between the metrics $g_{\gamma^{\theta_0}}$ and $g$. Treating $\gamma^{\theta_0}$ as the graph of a smooth function $f_{0}$ over the interval $[0,\bar{r}]$, we can compute the curvature $k$, this time in terms of $r$, as

\begin{equation*}
k=\frac{\ddot{f_{0}}}{(1+\dot{f_{0}}^{2})^{\frac{3}{2}}}.
\end{equation*}

By replacing $f_{0}$ with $\lambda f_{0}$ where $\lambda\in[0,1]$ we obtain a homotopy from $\gamma^{\theta_{0}}$ back to $\gamma^{0}$. To ensure that the induced metric has positive scalar curvature at each stage in this homotopy it is enough to show that on the interval $[\frac{r_1}{2}, r_1]$, $k^{\lambda}\leq k$ for all $\lambda\in[0,1]$, where $k^{\lambda}$ is the curvature of $\lambda f_{0}$. Note that away from this interval downward concavity means that $k^{\lambda}\leq 0$ for all $\lambda$ and so inequality (\ref{cureqn}) is easily satisfied.

We wish to show that for all $\lambda\in[0,1]$,

\begin{equation*}
\frac{\lambda\ddot{f_{0}}}{(1+\lambda^{2}\dot{f_{0}}^{2})^{\frac{3}{2}}}\leq\frac{\ddot{f_{0}}}{(1+\dot{f_{0}}^{2})^{\frac{3}{2}}}.
\end{equation*}

\noindent A slight rearrangement of this inequality gives

\begin{equation*}
\frac{\ddot{f_{0}}}{((\lambda^{\frac{-2}{3}})(1+\lambda^{2}\dot{f_{0}}^{2}))^{\frac{3}{2}}}\leq\frac{\ddot{f_{0}}}{(1+\dot{f_{0}}^{2})^{\frac{3}{2}}},
\end{equation*}
\noindent and hence it is enough to show that 

\begin{equation*}
\frac{1+\lambda^{2}\dot{f_{0}}^{2}}{1+\dot{f_{0}}^{2}}\geq\lambda^{\frac{2}{3}} \hspace{1cm}\text{for all $\lambda\in[0,1]$}.
\end{equation*}

\noindent Replacing $\lambda^{\frac{2}{3}}$ with $\mu$ and $\dot{f_{0}}^{2}$ with $b$ we obtain the following inequality.

\begin{equation}\label{polly}
\begin{array}{c}
{\mu}^{3}b-\mu b-\mu+1\geq0.
\end{array}
\end{equation}

The left hand side of this inequality is zero when $\mu=1$ or when $\mu=\frac{-b^{2}\pm\sqrt{b^{2}+4b}}{2b^{2}}$. We may assume that $\theta_0$ has been chosen small enough so that $\tan^{2}(\theta_0)<\frac{1}{4}$. Thus $b=\dot{f_{0}}^{2}<\frac{1}{4}$. A simple computation then shows that the left hand side of (\ref{polly}) is non-zero when $\mu$ (and thus $\lambda$) is in $[0,1]$, and so the inequality holds.
\\

\noindent {\bf The final bending:} We will now construct the curve $\gamma$ so that the induced metric $g_\gamma$ has positive scalar curvature. From the description of $\gamma$ given in Part 1, we see that it is useful to regard $\gamma$ as consisting of three pieces. When $r>r_0$, $\gamma$ is just the curve $\gamma^{\theta_0}$ constructed above and when $r\in[0,r_\infty]$, $\gamma$ is the graph of the  concave downward function $f_\infty$. In both of these cases, inequality (\ref{cureqn}) is easily satisfied. The third piece is where the difficulty lies. In the rectangle $[t_0, t_\infty]\times[r_\infty, r_0]$ we must specify a curve which connects the previous two pieces to form a $C^{2}$ curve and satisfies inequality (\ref{cureqn}). This will be done by constructing an appropriate $C^{2}$ function $f:[t_0, t_\infty]\longrightarrow[r_\infty, r_0]$. Before discussing the construction of $f$ we observe that inequality (\ref{cureqn}) can be simplified even further. 
 
Choose $r_0\in(0,\frac{r_1}{2})$ so that $0<r_0<min\{\frac{1}{\sqrt{4C}},\frac{1}{\sqrt{2C'}}\}$. Now, when $r\in(0,r_0]$ and $\theta\geq\theta_0$, we have

\begin{equation*}
\begin{array}{cl}
(q-1)\frac{\sin{\theta}}{r}-Cr{\sin{\theta}}& \geq\sin{\theta[\frac{q-1}{r}-Cr]}\\
&\geq\frac{\sin{\theta}}{r}[1-Cr^{2}].
\end{array}
\end{equation*}

\noindent When $r<\frac{1}{\sqrt{4C}}$, $r^{2}<\frac{1}{4C}$.
So $Cr^{2}<\frac{1}{4}$ and $1-Cr^{2}>\frac{3}{4}$. Thus

\begin{equation*}
\begin{array}{c}
(q-1)\frac{\sin{\theta}}{r}-Cr{\sin{\theta}}\geq\frac{3}{4}\frac{\sin{\theta}}{r}.
\end{array}
\end{equation*}

\noindent Also $r<\sqrt{\frac{1}{2C'}}$.
So $r^{2}<\frac{1}{2C'}$ giving that $2C'r^{2}<1$. Thus, $1+C'r^{2}<\frac{3}{2}$.
Hence from inequality (\ref{cureqn}) we get

\begin{equation*}
k<\frac{2}{3}.\frac{3}{4}\frac{\sin{\theta}}{r}=\frac{\sin{\theta}}{2r}.
\end{equation*}

\noindent So if we begin the second bend when $r\in [0,r_0]$, it suffices to maintain 

\begin{equation}\label{keqn}
k<\frac{\sin{\theta}}{2r}. 
\end{equation}

It should be pointed out that the inequality (\ref{keqn}) {\bf only holds when $\theta > \theta_0$ and does not hold for only $\theta>0$, no matter how small $r$ is chosen.} The following argument demonstrates this. Assuming $\theta$ close to zero and using the fact that $k(s)=\frac{d\theta}{ds}$, we can assume (\ref{keqn}) is
 
\begin{equation*}
\frac{d\theta}{ds}<\frac{{\theta}}{2r}.
\end{equation*}

\noindent But this is 
\begin{equation*}
\frac{d\log(\theta)}{ds}<\frac{1}{2r},
\end{equation*}
\noindent the left hand side of which is unbounded as $\theta$ approaches $0$. It is for this reason that the initial bend and hence the {\bf strict positivity} of the scalar curvature of $g$ is so important. 
 
\begin{Remark}\label{Gajercomment}
From the above one can see that the inequality on page 190 of \cite{Gajer} breaks down when $\theta$ is near $0$. In this case the bending argument aims at maintaining non-negative mean curvature.  Since apriori the mean curvature is not strictly positive, an analogous initial bend to move $\theta$ away from 0 is not possible.   
\end{Remark}

We will now restrict our attention entirely to the rectangle $[t_0, t_\infty]\times[r_\infty, r_0]$. Here we regard $\gamma$ as the graph of a function $f$. Thus we obtain,
\begin{equation*}
\sin{\theta}=\frac{1}{\sqrt{1+\dot{f}^{2}}}
\end{equation*}
\noindent and 
\begin{equation*}
k=\frac{\ddot{f}}{(1+\dot{f}^{2})^{\frac{3}{2}}}.
\end{equation*}
\noindent Hence, (\ref{keqn}) gives rise to the following differential inequality
\begin{equation*}
\frac{\ddot{f}}{(1+\dot{f}^{2})^{\frac{3}{2}}}<\frac{1}{\sqrt{1+\dot{f}^{2}}}\frac{1}{2f}.
\end{equation*}

\noindent This simplifies to

\begin{equation}\label{diffkeqn}
{\ddot{f}}<\frac{{1+\dot{f}^{2}}}{2f}.
\end{equation}

\noindent Of course to ensure that $\gamma$ is a $C^{2}$ curve we must insist that as well as satisfying (\ref{diffkeqn}), $f$ must also satisfy conditions (\ref{cond1}), (\ref{cond2}) and (\ref{cond3}) below. 

\begin{eqnarray}
\label{cond1}
&f(t_0)=r_0,\qquad 
f(t_\infty)>0, \qquad\\
\label{cond2}
&\dot{f}(t_0)=m_0,\qquad 
\dot{f}(t_\infty)=0,\qquad\\
\label{cond3}
&\ddot{f}(t_0)=0,\qquad
\ddot{f}(t_\infty)=0,
\end{eqnarray}

\noindent where $m_0=\frac{-1}{\tan{\theta_0}}$. The fact that such a function can be constructed is the subject of the following lemma. Having constructed such a function, $r_\infty$ will then be set equal to $f(t_\infty)$ and the construction of $\gamma$ will be complete.

\begin{Lemma}\label{technical}
For some $t_\infty>t_0$, there is a $C^{2}$ function $f:[t_0,t_\infty]\longrightarrow[0,r_0]$ which satisfies inequality (\ref{diffkeqn}) as well as conditions  
(\ref{cond1}), (\ref{cond2}) and (\ref{cond3}).
\end{Lemma}

\begin{proof}
The following formula describes a family of functions, all of which satisfy inequality (\ref{diffkeqn}).

\begin{equation*}
\begin{array}{c}
f(t)=c+\frac{C_1}{4}(t-C_2)^{2}, \hspace{1.0cm}\text{where $C_1,C_2>0\hspace{2mm}\text{and}\hspace{2mm}c\in(0,\frac{1}{C_1})$.}
\end{array}
\end{equation*}

\noindent Such a function $f$ has first and second derivatives

\begin{eqnarray*}
&\dot{f}=\frac{C_1}{2}(t-C_2)\hspace{2mm}\text{and}\hspace{2mm} \ddot{f}=\frac{C_1}{2}.
\end{eqnarray*}

\noindent We will shortly see that $C_1$ and $C_2$ can be chosen so that on the interval $[t_0, C_2]$, $f(t_0)=r_0$, $\dot{f}(t_0)= m_0$, $\dot{f}(C_2)= 0$ and $f(C_2)=c>0$. The choice of $C_1$ needs to be very large which makes $\ddot{f}$ a large positive constant. Thus, some adjustment is required near the end points if such a function is to satisfy the requirements of the lemma. We will achieve this by restricting the function to some proper subinterval $[t_0', t_\infty']\subset [t_0, C_2]$ and pasting in appropriate transition functions on the intervals $[t_0, t_0']$ and $[t_\infty', t_\infty]$ (where $t_\infty$ is close to $C_2$).

More precisely, let $t_0'-t_0=\delta_0$, $t_\infty-t_\infty'=\delta_\infty$ and $C_2-t_\infty'=\frac{\delta_\infty}{2}$. We will now show that for appropriate choices of $C_1, C_2, \delta_0$ and $\delta_\infty$, the following function satisfies the conditions of the lemma. To aid the reader, we include the graph of the second derivative of this function, see Fig. \ref{secondderivoff}.  

\begin{equation}\label{pwf(t)}
\begin{array}{c}
f(t) =
\begin{cases}
r_0+m_0(t-t_0)+\frac{C_1}{12\delta_0}(t-t_0)^{3}, & \text{if $t\in[t_0, t_0']$}\\
c+\frac{C_1}{4}(t-C_2)^{2}, & \text{if $t\in[t_0',t_\infty']$}\\
c-\frac{C_1}{48}{\delta_\infty}^{2}-\frac{C_1}{12\delta_\infty}(t-t_\infty)^3, & \text{if $t\in[t_\infty', t_\infty]$}.
\end{cases}
\end{array}
\end{equation} 

\begin{figure}[htbp]
\begin{picture}(0,0)%
\includegraphics{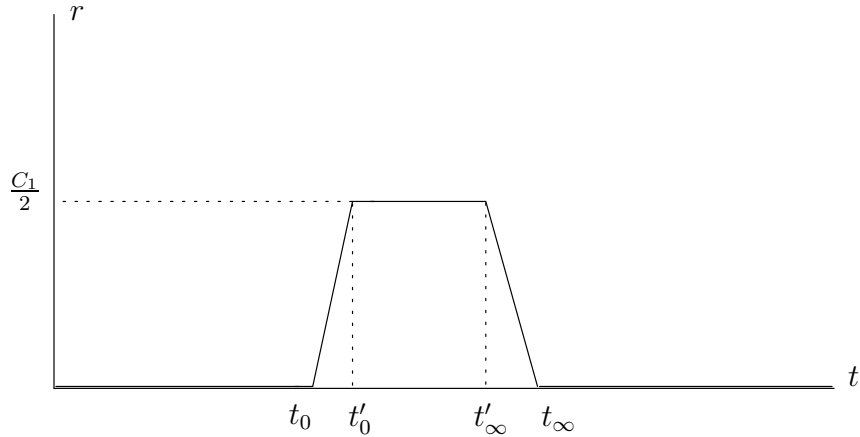}%
\end{picture}%
\setlength{\unitlength}{3947sp}%
\begingroup\makeatletter\ifx\SetFigFont\undefined%
\gdef\SetFigFont#1#2#3#4#5{%
  \reset@font\fontsize{#1}{#2pt}%
  \fontfamily{#3}\fontseries{#4}\fontshape{#5}%
  \selectfont}%
\fi\endgroup%
\begin{picture}(5305,2770)(1086,-3740)
\put(1101,-2249){\makebox(0,0)[lb]{\smash{{\SetFigFont{12}{14.4}{\rmdefault}{\mddefault}{\updefault}{\color[rgb]{0,0,0}$\frac{C_1}{2}$}%
}}}}
\put(2864,-3661){\makebox(0,0)[lb]{\smash{{\SetFigFont{12}{14.4}{\rmdefault}{\mddefault}{\updefault}{\color[rgb]{0,0,0}$t_0$}%
}}}}
\put(3239,-3674){\makebox(0,0)[lb]{\smash{{\SetFigFont{12}{14.4}{\rmdefault}{\mddefault}{\updefault}{\color[rgb]{0,0,0}$t_0'$}%
}}}}
\put(4026,-3674){\makebox(0,0)[lb]{\smash{{\SetFigFont{12}{14.4}{\rmdefault}{\mddefault}{\updefault}{\color[rgb]{0,0,0}$t_\infty'$}%
}}}}
\put(4451,-3674){\makebox(0,0)[lb]{\smash{{\SetFigFont{12}{14.4}{\rmdefault}{\mddefault}{\updefault}{\color[rgb]{0,0,0}$t_\infty$}%
}}}}
\put(6376,-3424){\makebox(0,0)[lb]{\smash{{\SetFigFont{12}{14.4}{\rmdefault}{\mddefault}{\updefault}{\color[rgb]{0,0,0}$t$}%
}}}}
\put(1489,-1124){\makebox(0,0)[lb]{\smash{{\SetFigFont{12}{14.4}{\rmdefault}{\mddefault}{\updefault}{\color[rgb]{0,0,0}$r$}%
}}}}
\end{picture}%
\caption{The second derivative of $f$}
\label{secondderivoff}
\end{figure}

\noindent A simple check shows that $f(t_0)=r_0$, $\dot{f}(t_0)=m_0$ and $\ddot{f}(t_0)=0$. Now we must show that $C_1$ can be chosen so that this function is $C^{2}$ at $t_0'$. We begin by solving, for $t_0'$, the equation

\begin{equation*}
\begin{array}{c}
c+\frac{C_{1}}{4}(t_0'-C_2)^{2}=r_0+m_0(t_0'-t_0)+\frac{C_1}{12\delta_0}(t_0'-t_0)^{3}.
\end{array}
\end{equation*}

\noindent This results in the following formula for $t_0'$,

\begin{equation*}
\begin{array}{c}
t_0'=C_2-\sqrt{\frac{4}{C_1}(r_0+m_0\delta_0+\frac{C_1}{12}\delta_0^{2}-c)}.
\end{array}
\end{equation*}

\noindent Equating the first derivatives of the first two components of (\ref{pwf(t)}) at $t_0'$ and replacing $t_0'$ with the expression above results in the following equation.
\begin{equation}\label{firstderiveqn}
\begin{array}{c}
C_1(r_0-c)-m_0^{2}=-\frac{C_1}{2}\delta_0 m_0-\frac{C_1^{2}}{48}\delta_0^{2}.
\end{array}
\end{equation}

The second derivatives of the first two components of (\ref{pwf(t)}) agree at $t_0'$ and so provided $C_1$ and $\delta_0$ are chosen to satisfy (\ref{firstderiveqn}), $f$ is $C^{2}$ at $t_0'$. It remains to show that $\delta_0$ can be chosen so that $f$ satisfies inequality (\ref{diffkeqn}) on $[t_0, t_0']$. 
The parameter $C_1$ varies continuously with respect to $\delta_0$. Denoting by $\bar{C_1}$, the solution to the equation $\bar{C_1}(r_0-c)-m_0^{2}=0$, it follows from equation (\ref{firstderiveqn}), that for small $\delta_0$, $C_1$ is given by a formula $C_1(\delta_0)=\bar{C_1}+\epsilon(\delta_0)$ for some continuous parameter $\epsilon$ with $\epsilon(0)=0$ and $\epsilon(\delta_0)>0$ when $\delta_0>0$. When $\delta_0=0$, we obtain the strict inequality 

\begin{equation*}
\bar{C_1}<\frac{1+m_0^{2}}{r_0}.
\end{equation*}

\noindent Thus, there exists some sufficiently small $\delta_0$, so that for all $s\in[0,1]$,

\begin{equation*}
C_1=\bar{C_1}+\epsilon(\delta_0)<\frac{1+(m_0+\frac{C_1}{4}\delta_0 s)^{2}}{r_0},
\end{equation*} 
 
\noindent while at the same time,

\begin{equation*}
\begin{array}{c}
-r_0<m_0\delta_0s+\frac{C_1}{12}\delta_0^{2}s^{3}\leq 0.
\end{array}
\end{equation*}

\noindent Hence, 
\begin{equation}\label{sineq}
C_1<\frac{1+(m_0+\frac{C_1}{4}\delta_0 s)^{2}}{r_0+m_0\delta_0 s+\frac{C_1}{12}\delta_0^{2}s^{3}}
\end{equation}

\noindent holds for all $s\in[0,1]$. Replacing $s$ with $\frac{t-t_0}{\delta_0}$ in (\ref{sineq}) yields inequality (\ref{diffkeqn}) for $t\in[t_0,t_0']$ and so $f$ satisfies (\ref{diffkeqn}) at least on $[t_0,t_\infty']$. 

Given $r_0$, the only choice we have made so far in the construction of $f$, is the choice of $\delta_0$. This choice determines uniquely, the choices of $C_1$ and $C_2$. Strictly speaking, we need to choose some $c$ in $(0,\frac{1}{C_1})$ but we can always regard this as given by the choice of $C_1$, by setting $c=\frac{1}{2C_1}$ say. There is one final choice to be made and that is the choice of $\delta_\infty$.
 
Some elementary calculations show that $f$ is $C^{2}$ at $t_\infty'$. The choice of $\delta_\infty$ is completely independent of any of the choices we have made so far and so can be made arbitrarily small. Thus, an almost identical argument to the one made when choosing $\delta_0$ shows that for a sufficiently small choice of $\delta_\infty$, inequality (\ref{diffkeqn}) is satisfied when $t\in[t_\infty', t_\infty]$.
Also, the independence of $\delta_0$ and $C_1$ means that $f(t_\infty)=c-\frac{C_1}{12}\delta_\infty^{2}$ can be kept strictly positive by ensuring $\delta_\infty$ is sufficiently small. The remaining conditions of the lemma are then trivial to verify.
\end{proof}

\noindent{\bf The final isotopy:} The final task in the proof of Lemma \ref{Pushoutcurvehaspsc} is the construction of a homotopy between $\gamma$ and $\gamma^{\theta_0}$ which induces an isotopy between the metrics $g_\gamma$ and $g_{\gamma^{\theta_0}}$. We will begin by making a very slight adjustment to $\gamma$. Recall that the function $f$ has as its second derivative: a bump function with support on the interval $[t_0, t_\infty]$, see Fig. \ref{secondderivoff}. By altering this bump function on the region $[t_\infty',t_\infty]$, we make adjustments to $f$. In particular, we will replace $f$ with the $C^{2}$ function which agrees with $f$ on $[t_0, t_\infty']$ but whose second derivative is the bump function shown in Fig. \ref{alterbump}, with support on $[t_0, t_\infty'']$, where $t_\infty''\in[C_2, t_\infty]$. We will denote this new function $f^{\infty}$.

\begin{figure}[htbp]

\begin{picture}(0,0)%
\includegraphics{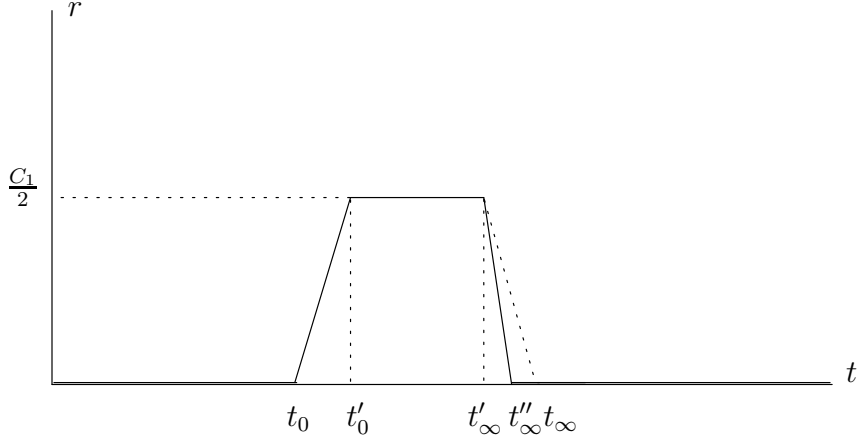}%
\end{picture}%
\setlength{\unitlength}{3947sp}%
\begingroup\makeatletter\ifx\SetFigFont\undefined%
\gdef\SetFigFont#1#2#3#4#5{%
  \reset@font\fontsize{#1}{#2pt}%
  \fontfamily{#3}\fontseries{#4}\fontshape{#5}%
  \selectfont}%
\fi\endgroup%
\begin{picture}(5305,2795)(1086,-3765)
\put(1101,-2249){\makebox(0,0)[lb]{\smash{{\SetFigFont{12}{14.4}{\rmdefault}{\mddefault}{\updefault}{\color[rgb]{0,0,0}$\frac{C_1}{2}$}%
}}}}
\put(2864,-3699){\makebox(0,0)[lb]{\smash{{\SetFigFont{12}{14.4}{\rmdefault}{\mddefault}{\updefault}{\color[rgb]{0,0,0}$t_0$}%
}}}}
\put(3239,-3699){\makebox(0,0)[lb]{\smash{{\SetFigFont{12}{14.4}{\rmdefault}{\mddefault}{\updefault}{\color[rgb]{0,0,0}$t_0'$}%
}}}}
\put(6376,-3424){\makebox(0,0)[lb]{\smash{{\SetFigFont{12}{14.4}{\rmdefault}{\mddefault}{\updefault}{\color[rgb]{0,0,0}$t$}%
}}}}
\put(1489,-1124){\makebox(0,0)[lb]{\smash{{\SetFigFont{12}{14.4}{\rmdefault}{\mddefault}{\updefault}{\color[rgb]{0,0,0}$r$}%
}}}}
\put(4256,-3699){\makebox(0,0)[lb]{\smash{{\SetFigFont{12}{14.4}{\rmdefault}{\mddefault}{\updefault}{\color[rgb]{0,0,0}$t_\infty''$}%
}}}}
\put(4476,-3699){\makebox(0,0)[lb]{\smash{{\SetFigFont{12}{14.4}{\rmdefault}{\mddefault}{\updefault}{\color[rgb]{0,0,0}$t_\infty$}%
}}}}
\put(4001,-3699){\makebox(0,0)[lb]{\smash{{\SetFigFont{12}{14.4}{\rmdefault}{\mddefault}{\updefault}{\color[rgb]{0,0,0}$t_\infty'$}%
}}}}
\end{picture}%

\caption{The second derivative of the function $f^{\infty}$}
\label{alterbump}
\end{figure}

When $t_\infty''=t_\infty$, no change has been made and $f^{\infty}=f$. When $t_\infty''<t_\infty$, the derivative of $f^{\infty}$ on the interval $[t_\infty'', t_\infty]$ is a negative constant, causing the formerly horizontal straight line piece of $\gamma$ to tilt downwards with negative slope. Thus, by continuously decreasing $t_\infty''$ from $t_\infty$ by some sufficiently small amount, we can homotopy $\gamma$ to a curve of the type shown in Fig. \ref{GLcurvebent}, where the second straight line piece now has small negative slope before bending downwards to intersect the $t$-axis vertically at $\bar{t}$. Note that the rectangle $[t_0, t_\infty]\times[r_\infty, r_0]$ is now replaced by the rectangle $[t_0, t_\infty'']\times[r_\infty'', r_0]$, where $f^{\infty}(t_\infty'')=r_\infty''$. It is easy to see how, on $[t_\infty'', \bar{t}]$, $\gamma$ can be homotopied through curves each of which is the graph of a $C^{2}$ function with non-positive second derivative, thus satisfying inequality (\ref{diffkeqn}). We do need to verify however, that on $[t_\infty', t_\infty'']$, this inequality is valid. Recall, this means showing that

\begin{equation*}
\ddot{f^{\infty}}<\frac{1+\dot{f^{\infty}}^{2}}{2f^{\infty}}.
\end{equation*}

When $t_\infty''=t_\infty$, $f^{\infty}=f$ and so this inequality is already strict on the interval $[t_\infty', t_\infty'']$. Now suppose $t_\infty''$ is slightly less than $t_\infty$. Then, on $[t_\infty', t_\infty'']$, $\ddot{f^{\infty}}\leq \ddot{f}$, while the $2$-jets of $f$ and $f^{\infty}$ agree at $t_\infty'$. This means that $\dot{f^{\infty}}\leq\dot{f}$ and $f^{\infty}\leq f$ on $[t_\infty', t_\infty'']$. But $\dot{f}<0$ on this interval and so $\dot{f^{\infty}}^{2}\geq \dot{f}^{2}$. Also, provided $t_\infty''$ is sufficiently close to $t_\infty$, we can keep $f^{\infty}>0$ and sufficiently large on this interval so that the curve $\gamma$ can continue as the graph of a decreasing non-negative concave downward function all the way to the point $\bar{t}$. Thus, the inequality in (\ref{diffkeqn}) actually grows as $t_\infty''$ decreases.

\begin{figure}[htbp]
\begin{picture}(0,0)%
\includegraphics[height=60mm]{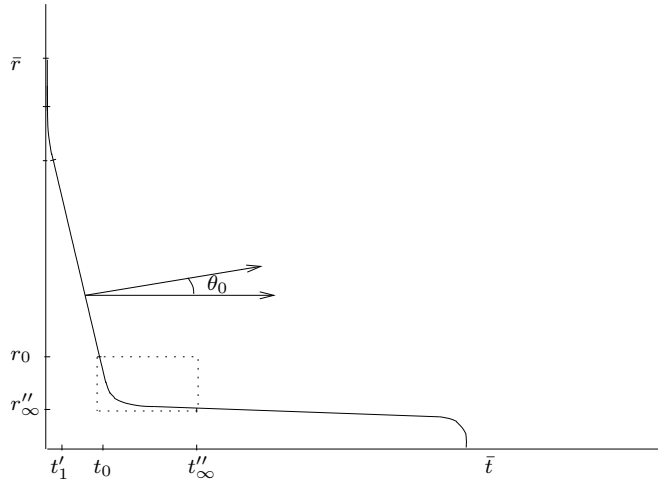}%
\end{picture}%
\setlength{\unitlength}{3947sp}%
\begingroup\makeatletter\ifx\SetFigFont\undefined%
\gdef\SetFigFont#1#2#3#4#5{%
  \reset@font\fontsize{#1}{#2pt}%
  \fontfamily{#3}\fontseries{#4}\fontshape{#5}%
  \selectfont}%
\fi\endgroup%
\begin{picture}(3889,2884)(1374,-4954)
\put(1180,-4690){\makebox(0,0)[lb]{\smash{{\SetFigFont{8}{8}{\rmdefault}{\mddefault}{\updefault}{\color[rgb]{0,0,0}$r_\infty''$}%
}}}}
\put(1180,-4375){\makebox(0,0)[lb]{\smash{{\SetFigFont{8}{8}{\rmdefault}{\mddefault}{\updefault}{\color[rgb]{0,0,0}$r_0$}%
}}}}
\put(1180,-2553){\makebox(0,0)[lb]{\smash{{\SetFigFont{8}{8}{\rmdefault}{\mddefault}{\updefault}{\color[rgb]{0,0,0}$\bar{r}$}%
}}}}
\put(1700,-5088){\makebox(0,0)[lb]{\smash{{\SetFigFont{8}{8}{\rmdefault}{\mddefault}{\updefault}{\color[rgb]{0,0,0}$t_0$}%
}}}}
\put(2294,-5088){\makebox(0,0)[lb]{\smash{{\SetFigFont{8}{8}{\rmdefault}{\mddefault}{\updefault}{\color[rgb]{0,0,0}$t_\infty''$}%
}}}}
\put(2414,-3938){\makebox(0,0)[lb]{\smash{{\SetFigFont{8}{8}{\rmdefault}{\mddefault}{\updefault}{\color[rgb]{0,0,0}$\theta_0$}%
}}}}
\put(4159,-5088){\makebox(0,0)[lb]{\smash{{\SetFigFont{8}{8}{\rmdefault}{\mddefault}{\updefault}{\color[rgb]{0,0,0}$\bar{t}$}%
}}}}
\put(1430,-5088){\makebox(0,0)[lb]{\smash{{\SetFigFont{8}{8}{\rmdefault}{\mddefault}{\updefault}{\color[rgb]{0,0,0}$t_1'$}%
}}}}
\end{picture}%
\caption{The effect of such an alteration on the curve $\gamma$}
\label{GLcurvebent}
\end{figure}

It remains to show that this slightly altered $\gamma$ can be homotopied back to $\gamma^{\theta_{0}}$ in such a way as to induce an isotopy of metrics. To ease the burden of notation we will refer to the function $f^{\infty}$ as simply $f$ and the rectangle $[t_0,t_\infty'']\times[r_\infty'', r_0]$ as simply $[t_0,t_\infty]\times[r_\infty, r_0]$. It is important to remember that $f$ differs from the function constructed in Lemma \ref{technical} in that $m_0\leq\dot{f}<0$ on $[t_0, t_\infty]$. We wish to continuously deform the graph of $f$ to obtain the straight line of slope $m_0$ intersecting the point $(t_0, r_0)$. We will denote this straight line segment by $l$, given by the formula $l(r)=r_0+m_0(t-t_0)$. We will now construct a homotopy by considering the functions which are inverse to $f$ and $l$, see Fig. \ref{lastisotopy}. Consider the linear homotopy $h^{-1}_s=(1-s)f^{-1}+sl^{-1}$, where $s\in[0,1]$. Let $h_s$ denote the corresponding homotopy from $f$ to $l$, where for each $s$, $h_s$ is inverse to $h^{-1}_s$. Note that the domain of $h_s$ is $[t_0,(1-s)t_\infty+sl^{-1}(r_\infty)]$. For each $r\in[r_\infty, r_0]$, $\dot{h^{-1}_s}(r)\leq\dot{f^{-1}}(r)$. This means that for any $s\in[0,1]$ and any $r\in[r_\infty, r_0]$, $\dot{h_s}(t_s)\geq\dot{f}(t)$, where $h_s(t_s)=f(t)=r$. As the second derivative of $h_s$ is bounded by $\ddot{f}$, this means that inequality (\ref{diffkeqn}) is satisfied throughout the homotopy.

\begin{figure}[htbp]
\begin{picture}(0,0)%
\includegraphics{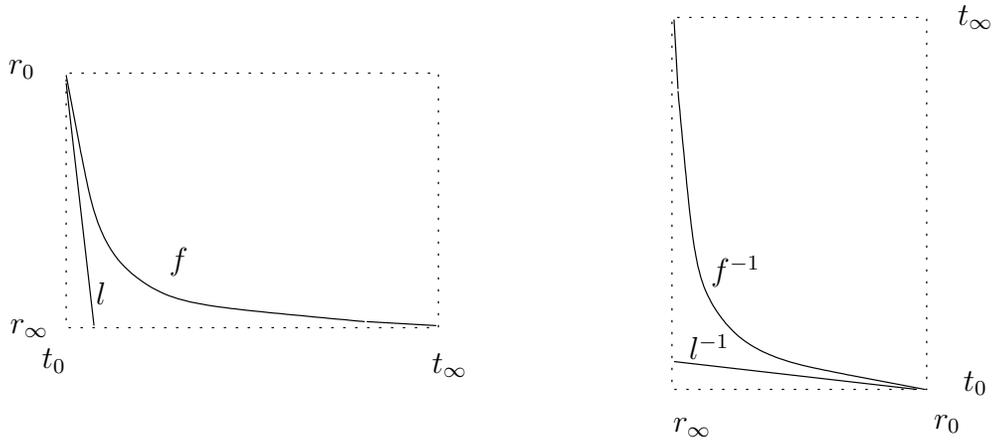}%
\end{picture}%
\setlength{\unitlength}{3947sp}%
\begingroup\makeatletter\ifx\SetFigFont\undefined%
\gdef\SetFigFont#1#2#3#4#5{%
  \reset@font\fontsize{#1}{#2pt}%
  \fontfamily{#3}\fontseries{#4}\fontshape{#5}%
  \selectfont}%
\fi\endgroup%
\begin{picture}(6030,2783)(861,-3065)
\put(1089,-2599){\makebox(0,0)[lb]{\smash{{\SetFigFont{12}{14.4}{\rmdefault}{\mddefault}{\updefault}{\color[rgb]{0,0,0}$t_0$}%
}}}}
\put(1426,-2199){\makebox(0,0)[lb]{\smash{{\SetFigFont{12}{14.4}{\rmdefault}{\mddefault}{\updefault}{\color[rgb]{0,0,0}$l$}%
}}}}

\put(3539,-2624){\makebox(0,0)[lb]{\smash{{\SetFigFont{12}{14.4}{\rmdefault}{\mddefault}{\updefault}{\color[rgb]{0,0,0}$t_\infty$}%
}}}}
\put(876,-749){\makebox(0,0)[lb]{\smash{{\SetFigFont{12}{14.4}{\rmdefault}{\mddefault}{\updefault}{\color[rgb]{0,0,0}$r_0$}%
}}}}
\put(889,-2374){\makebox(0,0)[lb]{\smash{{\SetFigFont{12}{14.4}{\rmdefault}{\mddefault}{\updefault}{\color[rgb]{0,0,0}$r_\infty$}%
}}}}
\put(1889,-1974){\makebox(0,0)[lb]{\smash{{\SetFigFont{12}{14.4}{\rmdefault}{\mddefault}{\updefault}{\color[rgb]{0,0,0}$f$}%
}}}}
\put(5051,-2999){\makebox(0,0)[lb]{\smash{{\SetFigFont{12}{14.4}{\rmdefault}{\mddefault}{\updefault}{\color[rgb]{0,0,0}$r_\infty$}%
}}}}
\put(6689,-2974){\makebox(0,0)[lb]{\smash{{\SetFigFont{12}{14.4}{\rmdefault}{\mddefault}{\updefault}{\color[rgb]{0,0,0}$r_0$}%
}}}}
\put(6876,-2736){\makebox(0,0)[lb]{\smash{{\SetFigFont{12}{14.4}{\rmdefault}{\mddefault}{\updefault}{\color[rgb]{0,0,0}$t_0$}%
}}}}
\put(6839,-436){\makebox(0,0)[lb]{\smash{{\SetFigFont{12}{14.4}{\rmdefault}{\mddefault}{\updefault}{\color[rgb]{0,0,0}$t_\infty$}%
}}}}
\put(5151,-2541){\makebox(0,0)[lb]{\smash{{\SetFigFont{12}{14.4}{\rmdefault}{\mddefault}{\updefault}{\color[rgb]{0,0,0}$l^{-1}$}%
}}}}
\put(5289,-2049){\makebox(0,0)[lb]{\smash{{\SetFigFont{12}{14.4}{\rmdefault}{\mddefault}{\updefault}{\color[rgb]{0,0,0}$f^{-1}$}%
}}}}
\end{picture}%
\caption{The graphs of the functions $f$ and $l$ and their inverses}
\label{lastisotopy}
\end{figure}

This homotopy extends easily to the part of $\gamma$ on the region where $t\geq (1-s)t_\infty+sl^{-1}(r_\infty)$, which can easily be homotopied through curves, each the graph of a concave downward decreasing non-negative function. The result is a homtopy between $\gamma$ and $\gamma^{\theta_0}$, through curves which satisfy inequality (\ref{cureqn}) at every stage. This, combined with the initial isotopy, induces an isotopy through metrics of positive scalar curvature between $g$ and $g_\gamma$, completing the proof of Lemma \ref{Pushoutcurvehaspsc}.
\end{proof}

\subsection{Part 3 of the proof: Isotopying to a standard product} 

Having constructed the psc-metric $g_\gamma$ and having demonstrated that $g_\gamma$ is isotopic to the original metric $g$, one final task remains. We must show that the metric $g_\gamma$ can be isotopied to a psc-metric which, near the embedded sphere $S^{p}$, is the product $g_p+g_{tor}^{q+1}(\delta)$. Composing this with the isotopy from Part 2 yields the desired isotopy from $g$ to $g_{std}$ and proves Theorem \ref{IsotopyTheorem}.  

We denote by $\pi:\N \rightarrow S^{p}$, the normal bundle to the embedded $S^{p}$ in $X$. The Levi-Civita connection on $X$, with respect to the metric $g_{\gamma}$, gives rise to a normal connection on the total space of this normal bundle . This gives rise to a horizontal distribution $\mathcal{H}$ on the total space of $\N$. Now equip the fibres of the bundle $\N$ with the metric $g_{tor}^{q+1}(\delta)$. Equip $S^{p}$ with the metric $g_{\gamma}|_{S^{p}}$, the induced metric on $S^{p}$. The projection $\pi:(\N,\tilde{g})\rightarrow (S^{p},\check{g})$ is now a Riemannian submersion with base metric $\check{g}=g_{\gamma}|_{S^{p}}$ and fibre metric $\hat{g}=g_{tor}^{q+1}(\delta)$. The metric $\tilde{g}$ denotes the unique submersion metric arising from $\check{g},\hat{g}$ and $\mathcal{H}$. 

Our focus will mostly be on the restriction of this Riemannian submersion to the disk bundle, $\pi:D\N(\epsilon)\rightarrow S^{p}$. We will retain $\tilde{g}$, $\hat{g}$ and $\check{g}$ to denote the relevant restrictions. Before saying anything more specific about this disk bundle, it is worth introducing some useful notation. For some $t_L\in(t_\infty, \bar{t}-r_\infty)$, we define the following submanifolds of $M$,   
\begin{equation*}
\begin{array}{c}
M[t_L, \bar{t}]=\{(y,x,t)\in S^{p}\times D^{q+1}(\bar{r})\times\mathbb{R}:(r(x),t)\in{\gamma}\hspace{2mm} \text{and $t\geq t_L$}\}.
\end{array}
\end{equation*}
\noindent and 
\begin{equation*}
\begin{array}{c}
M[t_\infty, t_L]=\{(y,x,t)\in S^{p}\times D^{q+1}(\bar{r})\times\mathbb{R}:(r(x),t)\in{\gamma}\hspace{2mm} \text{and $t_\infty \leq t \leq t_L$}\}.
\end{array}
\end{equation*}
\begin{figure}[htbp]
\begin{picture}(0,0)%
\includegraphics{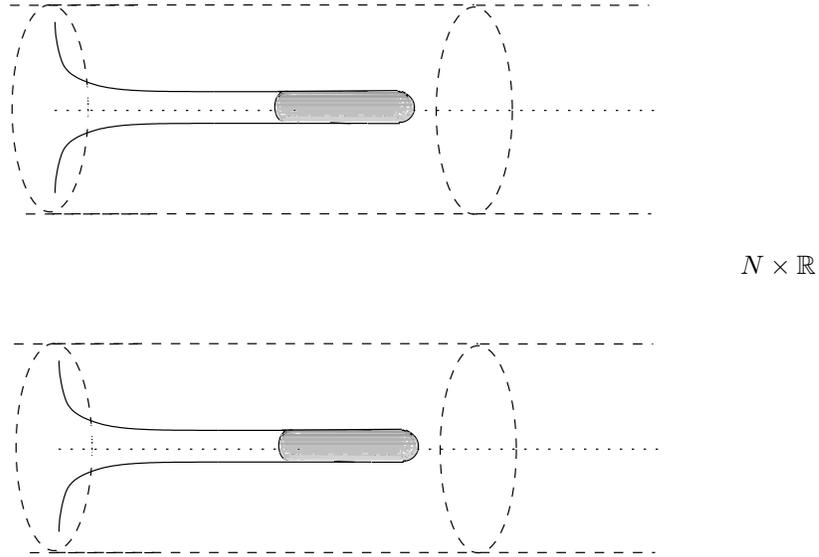}%
\end{picture}%
\setlength{\unitlength}{3947sp}%
\begingroup\makeatletter\ifx\SetFigFont\undefined%
\gdef\SetFigFont#1#2#3#4#5{%
  \reset@font\fontsize{#1}{#2pt}%
  \fontfamily{#3}\fontseries{#4}\fontshape{#5}%
  \selectfont}%
\fi\endgroup%
\begin{picture}(4670,3467)(1214,-4082)
\put(5869,-2318){\makebox(0,0)[lb]{\smash{{\SetFigFont{10}{8}{\rmdefault}{\mddefault}{\updefault}{\color[rgb]{0,0,0}$N\times\mathbb{R}$}%
}}}}
\end{picture}%
\caption{The shaded piece denotes the region $M[t_L,\bar{t}]$}
\label{fig:torpedopiece}
\end{figure}

\noindent Note that $M[t_L, \bar{t}]$ is, for appropriately small $\epsilon$, the disk bundle $D\N(\epsilon)$ and $M[t_\infty, t_L]$ is a cylindrical region (diffeomorphic to $S^{p}\times S^{q}\times [t_\infty, t_L]$) which connects this disk bundle with the rest of $M$. We will make our primary adjustments on the disk bundle $D\N(\epsilon)$, where we will construct an isotopy from the metric $g_\gamma$ to a metric which is a product. The cylindrical piece will be then be used as a transition region.

On $D\N(\epsilon)$ we can use the exponential map to compare the metrics $g_\gamma$ and $\tilde{g}$. Replacing the term $r_\infty$ with $\delta$, we observe the following convergence.

\begin{Lemma} \label{C2convergence}
There is $C^{2}$ convergence of the metrics $g_\gamma$ and $\tilde{g}$ as $\delta\rightarrow 0$.
\end{Lemma}
\begin{proof} Treating $g_\gamma$ as a submersion metric (or at least the metric obtained by pulling back $g_\gamma$ via the exponential map), it suffices to show convergence of the fibre metrics. In the case of $g_\gamma$, the metric on each fibre $D^{q+1}(\epsilon)_{y}$, where $y\in S^{p}$, is of the form

\begin{equation*}
\begin{array}{c}
\hat{g_\gamma}=dt^{2}+{g|_{S^{q}(f_{\delta}(\bar{t}-t))}}_{y}.
\end{array}
\end{equation*}

\noindent Here $t\in[t_L, \bar{t}]$ and recall that ${S^{q}(f_{\delta}(\bar{t}-t))}_{y}$ is the geodesic fibre sphere of radius $f_{\delta}(\bar{t}-t)$ at the point $y\in S^{p}$. In the same coordinates, the fibre metric for $\tilde{g}$ is 

\begin{equation*}
\begin{array}{c}
\hat{g}=dt^{2}+f_{\delta}(\bar{t}-t)^{2}ds_{q}^{2}.
\end{array}
\end{equation*}

\noindent We know from Lemma \ref{GLlemma1} that as $r\rightarrow 0$, \\
\begin{equation*}
\begin{array}{c}
{g|_{S^{q}(r)}}_{y}\longrightarrow r^{2}ds_{q}^{2}
\end{array}
\end{equation*}
\noindent in the $C^{2}$ topology. Now, $0<f_{\delta}(\bar{t}-t)\leq \delta$ and so as $\delta\rightarrow 0$, we get that

\begin{equation*}
\begin{array}{c}
{g|_{S^{q}(f_{\delta}(\bar{t}-t))}}_{y}\longrightarrow f_{\delta}(\bar{t}-t)^{2}ds_{q}^{2}.
\end{array}
\end{equation*}
\end{proof}

Hence we can isotopy $g_{\gamma}$, through submersion metrics, to one which pulls back on $S^{p}\times D^{q+1}(\epsilon)$ to the submersion metric $\tilde{g}$. In fact, we can do this with arbitrarily small curvature effects and so maintain positive scalar curvature. Furthermore, the fact that there is $C^{2}$ convergence of ${g_{\gamma}|_{S^{q}(f(\bar{t}-t))}}_{x}$ to $f(\bar{t}-t)^{2}ds_{q}^{2}$ means that we can ensure a smooth transition along the cylinder $M[t_\infty, t_L]$, although this may neccessitate making the cylindrical piece very long.

Now, by the formulae of O'Neill, we get that the scalar curvature of $\tilde{g}$ is

\begin{equation*}
\begin{array}{c}
\tilde{R}=\check{R}\circ\pi+\hat{R}-|A|^{2}-|T|^{2}-|\bar{n}|^{2}-2\check{\delta}(\bar{n}).
\end{array}
\end{equation*}

\noindent where $\tilde{R}, \check{R}$ and $\hat{R}$ are the scalar curvatures of $\tilde{g}, \check{g}$ and $\hat{g}$ respectively. For full definitions and formulae for $A,T,\bar{n}$ and $\check{\delta}$, see \cite{B}. Briefly, the terms  $T$ and $A$ are the first and second tensorial invariants of the submersion. Here $T$ is the obstruction to the bundle having totally geodesic fibres and so by construction $T=0$, while $A$ is the obstruction to the integrability of the distribution. The $\bar{n}$ term is also $0$ as $\bar{n}$ is the mean curvature vector and vanishes when $T$ vanishes. We are left with

\begin{equation}\label{O'Neill}
\begin{array}{c}
\tilde{R}=\check{R}\circ\pi+\hat{R}-|A|^{2}.
\end{array}
\end{equation}

We wish to deform $\tilde{g}$ through Riemannian submersions to one which has base metric $g_p$, preserving positive scalar curvature throughout the deformation. We can make $\hat{R}$ arbitrarily positively large by choosing small $\delta$. As the deformation takes place over a compact interval, curvature arising from the base metric is bounded. We must ensure, however, as we shrink the fibre metric, that $|A|^{2}$ is not affected in a significant way.

Letting $\tau=\bar{t}-t$, the metric on the fibre is

\begin{equation*}
\begin{array}{cl}
g_{tor}^{q+1}(\delta)&=d{\tau}^{2}+f_{\delta}(\tau)^{2}ds_{q}^{2}\\
&=\delta^{2}d(\frac{\tau}{\delta})^{2}+\delta^{2}f_{1}(\frac{\tau}{\delta})^{2}ds_{q}^{2}\\
&=\delta^{2}g_{tor}^{q+1}(1).
\end{array}
\end{equation*}

\noindent The canonical variation formula, \cite{B}, now gives that

\begin{equation*}
\begin{array}{c}
\tilde R=R_{\delta^{2}}=\frac{1}{\delta^{2}}\hat{R}+\check{R}\circ\pi-\delta^{2}|A|^{2}
\end{array}
\end{equation*}

\noindent Thus, far from the $|A|^{2}$ term becoming more significant as we shrink $\delta$, it actually diminishes.

Having isotopied $\tilde{g}$ through positive scalar curvature Riemannian submersions to obtain a submersion metric with base metric $\check{g}=g_p$ and fibre metric $\hat{g}=g_{tor}^{q+1}(\delta)$, we finish by isotopying through Riemannian submersions to the product metric $g_p+g_{tor}^{q+1}(\delta)$. This involves a linear homotopy of the distribution to one which is flat i.e. where $A$ vanishes. 
As $|A|^{2}$ is bounded throughout, we can again shrink $\delta$ if necessary to ensure positivity of the scalar curvature.

At this point we have constructed an isotopy between $\tilde{g}=g_{std}|_{D\N(\epsilon)}$ and $g_p+g_{tor}^{q+1}(\delta)$. In the original Gromov-Lawson construction, this isotopy is performed only on the sphere bundle $S\N(\epsilon)$ and so the resulting metric is the product $g_p+\delta^{2}ds_{q}^{2}$ (in this case $g_p=ds_{p}^{2}$). In fact, the restriction of the above isotopy to the boundary of the disk bundle $D\N(\epsilon)$ is precisely this Gromov-Lawson isotopy. Thus, as the metric on $D\N(\epsilon)$ is being isotopied from $g_\gamma$ to $g_p+g_{tor}^{q+1}(\delta)$, we can continuously transition along the cylinder $M[t_\infty, t_L]$ from this metric to the orginal metric $g_\gamma$. Again, this may require a considerable lengthenning of the cylindrical piece. This completes the proof of Theorem \ref{IsotopyTheorem}.

\end{proof}

\subsection{Applying Theorem \ref{IsotopyTheorem} over a compact family of psc-metrics}
Before proceeding with the proof of Theorem \ref{ImprovedsurgeryTheorem} we make an important observation. Theorem \ref{IsotopyTheorem} can be extended to work for a compact family of positive scalar curvature metrics on $X$ as well as a compact family of embedded surgery spheres. A compact family of psc-metrics on $X$ will be specified with a continuous map from some compact indexing space $B$ into the space $\Riem^{+}(X)$. In the case of a compact family of embedded surgery spheres, we need to introduce some notation. The set of smooth maps $C^{\infty}(W, Y)$ between the compact manifolds $W$ and $Y$ can be equipped with a standard $C^{\infty}$ topology, see chapter 2 of \cite{Hirsch}. Note that as $W$ is compact there is no difference between the so called ``weak" and ``strong" topologies on this space. Contained in $C^{\infty}(W, Y)$, as an open subspace, is the space $Emb(W, Y)$ of smooth embeddings. We can now specify a compact family of embedded surgery spheres on a compact manifold $X$, with a continuous map from some compact indexing space $C$ into $Emb(S^{p}, X)$.

\begin{Theorem}\label{GLcompact}
Let $X$ be a smooth compact manifold of dimension $n$ and $B$ and $C$ a pair of compact spaces. Let $\mathcal{B}=\{g_{b}\in\Riem^{+}(X):{b\in B}\}$ be a continuous family of psc-metrics on $X$ and $\mathcal{C}=\{{i_c}\in Emb(S^{p}, X):{c\in C}\}$, a continuous family of embeddings, with $p+q+1=n$ and $q\geq 2$. Finally, let $g_p$ be any metric on $S^{p}$. Then, for some $\delta>0$, there is a continuous map

\begin{equation*}
\begin{split}
\mathcal{B}\times \mathcal{C}&\longrightarrow \Riem^{+}(X)\\
(g_b,i_c)&\longmapsto g_{std}^{b,c}
\end{split}
\end{equation*}

\noindent satisfying

\noindent (i) Each metric $g_{std}^{b,c}$ has the form $g_p+g_{tor}^{q+1}(\delta)$ on a tubular neighbourhood of $i_c (S^{p})$ and is the original metric $g_b$ away from this neighbourhood. 

\noindent (ii) For each $c\in C$, the restriction of this map to $\mathcal{B}\times\{ i_c \}$ is homotopy equivalent to the inclusion $\mathcal{B}\hookrightarrow \Riem^{+}(X)$. 
\end{Theorem}

\begin{proof} For each pair $b,c$, the exponential map $\exp_b$ of the metric $g_b$ can be used to specify a tubular neighbourhood $N_{b,c}(\bar{r})$ of the embedded sphere $i_c(S^{p})$, exactly as in Part 1 of Theorem \ref{IsotopyTheorem}. Compactness gives that the infimum of injectivity radii over all metrics $g_b$ on $X$ is some positive constant and so a single choice $\bar{r}>0$ can be found, giving rise to a continuous family of such tubular neighbourhoods $\{N_{b,c}=N_{b,c}(\bar{r}):b,c\in B\times C\}$. Each metric $g_{b}$ may be adjusted in $N_{b,c}$ by specifying a hypersurface $M_{\gamma}^{b,c}\subset N_{b,c}\times\mathbb{R}$ constructed with respect to a curve $\gamma$, exactly as described in the proof of Theorem \ref{IsotopyTheorem}. Equipping each $N_{b,c}\times\mathbb{R}$ with the metric $g_b|_{N{b,c}}+dt^{2}$ induces a continuous family of metrics $g_{\gamma}^{b,c}$ on the respective hypersurfaces $M_{\gamma}^{b,c}$.

We will first show that a single curve $\gamma$ can be chosen so that the resulting metrics $g_\gamma^{b,c}$ have positive scalar curvature for all $b$ and $c$. The homotopy of $\gamma$ to the vertical line segment in Part 2 of the proof Theorem \ref{IsotopyTheorem} can be applied exactly as before, inducing an isotopy between $g_\gamma^{b,c}$ and $g_b$ which varies continuously with respect to $b$ and $c$. Finally, Part 3 of Theorem \ref{IsotopyTheorem} can be generalised to give rise to an isotopy between $g_\gamma^{b,c}$ and $g_{std}^{b,c}$, which again varies continuously with respect to $b$ and $c$.

Recall from the proof of Theorem \ref{IsotopyTheorem} that for any curve $\gamma$, the scalar curvature on the hypersurface $M=M_\gamma$ is given by:

\begin{equation*}
\begin{array}{cl}
R^{M}&=R^{N}+\sin^{2}{\theta}.O(1)-2k\cdot q\frac{\sin{\theta}}{r}\\
&\hspace{0.4cm} +2q(q-1)\frac{\sin^{2}{\theta}}{r^2}+k\cdot qO(r)\sin{\theta}.
\end{array}
\end{equation*}

\noindent The $O(1)$ term comes from the principal curvatures on the embedded surgery sphere $S^{p}$ and the Ricci curvature on $N$, both of which are bounded. Over a compact family of psc-metrics $g_b, b\in B$ and and a compact family of embeddings $i_c, c\in C$, these curvatures remain bounded and so the $O(1)$ term is unchanged. Here, the tubular neighbourhood $N$ is replaced with the continuous family of tubular neighbourhoods $N_{b,c}$ described above. Recall that we can specify all of these neighbourhoods with a single choice of radial distance $\bar{r}$. The $O(r)$ term comes from the principal curvatures on the geodesic spheres $S^{q-1}(r)$, which were computed in Lemma \ref{GLlemma1}. This computation works exactly the same for a compact family of metrics and so this $O(r)$ term is unchanged. The expression now becomes

\begin{equation*}
\begin{array}{cl}
R^{M_{b,c}}&=R^{N_{b,c}}+\sin^{2}{\theta}.O(1)-2k\cdot q\frac{\sin{\theta}}{r}\\
&\hspace{0.4cm} +2q(q-1)\frac{\sin^{2}{\theta}}{r^2}+k\cdot qO(r)\sin{\theta}.
\end{array}
\end{equation*}

\noindent Inequality (\ref{cureqn}) can be obtained as before as 

\begin{equation*}
\begin{array}{c}
k[1+C'r^{2}]<R_{0}\frac{r}{\sin{\theta}}+(q-1)\frac{\sin{\theta}}{r}-Cr{\sin{\theta}},
\end{array}
\end{equation*}

\noindent where in this case $R_0=\frac{1}{2q}[\inf(R^{N_{b,c}})]$, taken over all pairs $b,c$. The important thing is that $R_0$ is still positive.
The construction of a curve $\gamma$ which satisfies this inequality then proceeds exactly as in Part 2 of Theorem \ref{IsotopyTheorem}. The resulting curve $\gamma$ specifies a family of hypersurfaces $M_{\gamma}^{b,c}\subset N_{b,c}\times\mathbb{R}$. For each $(b,c)$, the induced metric on $M_{\gamma}^{b,c}$ has positive scalar curvature. The curve $\gamma$ can then be homotopied back to the vertical line, exactly as in Part 2 of Theorem \ref{IsotopyTheorem}, inducing a continuous deformation of the family $\{g_{\gamma}^{b,c}\}$ to the family $\{g_b\}$.

Part 3 of Theorem \ref{IsotopyTheorem}, can be applied almost exactly as before. The bundle $\mathcal{N}$ and distribution $\mathcal{H}$ are now replaced with continuous families $\mathcal{N}_{b,c}$ and $\mathcal{H}_{b,c}$, giving rise to a continuous family of Riemannian submersions $\pi_{b,c}:(\mathcal{N}_{b,c},\tilde{g}_{b,c})\rightarrow(i_{c}(S^{p}),\check{g}_{b,c})$ where the base metric $\check{g}_{b,c}=g_b|_{i_c(S^{p})}$, the fibre metric is $\hat{g}=g_{tor}^{q+1}(\delta)$ as before and $\tilde{g}_{b,c}$ is the respective submersion metric. By compactness, a single choice of $\epsilon$ gives rise to a family of disk bundles $D\N_{b,c}(\epsilon)$ all specifying appropriate submanifolds $M_{\gamma}^{b,c}[t_\infty, t_L]$ and $M_{\gamma}^{b,c}[t_L,\bar{t}]$ of $M_{\gamma}^{b,c}$ (see Part 3 of Theorem \ref{IsotopyTheorem} for details). Lemma \ref{C2convergence} easily generalises to show that as $\delta\rightarrow 0$ there is uniform $C^{2}$ convergence $g_\gamma^{b,c}\rightarrow \tilde{g}_{b,c}$. Thus, there is a continuously varying family of isotopies over $b$ and $c$, through psc-submersion metrics, deforming each $g_\gamma^{b,c}$ into $\tilde{g}_{b,c}$. 

Formula \ref{O'Neill} now generalises to give the following formula for the scalar curvature of $\tilde{g}_{b,c}$, varying continuously over $b$ and $c$.

\begin{equation}\label{O'Neillfamily}
\begin{array}{c}
\tilde{R_{b,c}}=\check{R_{b,c}}\circ\pi_{b,c}+\hat{R}-|A_{b,c}|^{2}.
\end{array}
\end{equation}

\noindent Here $\tilde{R_{b,c}}, \check{R_{b,c}}$ and $\hat{R}$ denote the scalar curvatures of $\tilde{g}_{b,c}, \check{g}_{b,c}$ and $\hat{g}$ respectively. The term $A_{b,c}$ satisfies all of the properties of $A$ in formula (\ref{O'Neill}), namely $|A_{b,c}|$ is bounded and in fact diminishes uniformly as $\delta$ decreases. Thus there is a sufficiently small $\delta>0$, so that the family $\{\tilde{g}_{b,c}\}$ can be isotopied through families of psc-submersion metrics to the desired family $\{g_{std}^{b,c}\}$, as in the proof of Theorem \ref{IsotopyTheorem}. 
\end{proof}

\begin{Remark}
Note that Theorem \ref{GLcompact} claims only the existence of such a map. To write down a well-defined function of this type means incorporating the various parameter choices made in the construction of Theorem \ref{IsotopyTheorem}. For our current purposes, in this paper, that is not necessary. 
\end{Remark}

\begin{Corollary}\label{GLisotopy}
Let $g$ and $h$ be isotopic psc-metrics on $X$. Let $g'$ and $h'$ be respectively the metrics obtained by application of the Surgery Theorem on a codimension$\geq 3$ surgery. Then $g'$ and $h'$ are isotopic.
\end{Corollary}

\subsection{The proof of Theorem \ref{ImprovedsurgeryTheorem}(The Improved Surgery Theorem)}
\begin{proof}  Recall that $g$ denotes a positive scalar curvature metric on the closed manifold $X^{n}$, $i:S^{p}\hookrightarrow X$ denotes an embedding of the sphere $S^{p}$ with trivial normal bundle and that $p+q+1=n$ with $q\geq 2$. Let $W$ denote the trace of a surgery on $X$ with respect to this embedded sphere. We wish to extend $g$ over $W$ to obtain a psc-metric which is product near the boundary.

Corollary \ref{GLconc} implies the existence of a psc-metric $\bar{g}$ on the cylinder $X\times I$ so that near $X\times\{0\}$, $\bar{g}=g+ds^{2}$ and near $X\times\{1\}$, $\bar{g}=g_{std}+ds^{2}$ where $g_{std}$ is the metric obtained in Theorem \ref{IsotopyTheorem}. Thus, by choosing $g_p=\epsilon^{2}ds_p^{2}$, near $S^{p}$ the metric $g_{std}$ has the form $\epsilon^{2}ds_{p}^{2}+g_{tor}^{q+1}(\delta)$ for some sufficiently small $\delta>0$. Using the exponential map for $g_{std}$ we can specify a tubular neighbourhood of $S^{p}$ , $N=S^{p}\times D^{q+1}(\bar{r})$, so that the restriction of $g_{std}$ on $N$ is precisely the metric $\epsilon^{2}ds_{p}^{2}+g_{tor}^{q+1}(\delta)$. As before, $N$ is equipped with the coordinates $(y,x)$ where $y\in S^{p}$, $x\in D^{q+1}(\bar{r})$ and $D^{q+1}(\bar{r})$ is the Euclidean disk of radius $\bar{r}$. The quantity $r$ will denote the Euclidean radial distance on $D^{q+1}(\bar{r})$. Moreover, we may assume that $\delta$ is arbitrarily small and that the tube part of $g_{tor}^{q+1}(\delta)$ is arbitrarily long, thus the quantity $\bar{r}-\delta$ can be made as large as we like.

We will now attach a handle $D^{p+1}\times D^{q+1}$ to the cylinder $X\times I$. Recall that in section \ref{prelim}, we equipped the plane $\mathbb{R}^{n+1}=\mathbb{R}^{p+1}\times\mathbb{R}^{q+1}$ with a metric $h=g_{tor}^{p+1}(\epsilon)+g_{tor}^{q+1}(\delta)$. By equipping $\mathbb{R}^{p+1}$ and $\mathbb{R}^{q+1}$ with standard spherical coordinates $(\rho, \phi)$ and $(r,\theta)$, we realised the metric $h$ as 
\begin{equation*}
\begin{array}{c}
h=d\rho^{2}+f_\epsilon(\rho)^{2}ds_{p}^{2}+dr^{2}+f_\delta(r)^{2}ds_{q}^{2},
\end{array}
\end{equation*}
where $f_\epsilon, f_\delta:(0,\infty)\rightarrow(0,\infty)$ are the torpedo curves defined in section \ref{prelim}. The restriction of $h$ to the disk product $D^{p+1}(\bar{\rho})\times D^{q+1}(\bar{r})$ is the desired handle metric, where $\bar{\rho}$ is as large as we like. We can then glue the boundary component $\p({D^{p+1}(\bar{\rho})})\times D^{q+1}(\bar{r})$ to $N$ with the isometry

\begin{equation*}
\begin{split}
S^{p}\times D^{q+1}(\bar{r})&\longrightarrow N\\
(y,x)&\longmapsto(i(y),L_y(x)),
\end{split}
\end{equation*}

\noindent where $L_y\in \rm{O}(q+1)$ for all $y\in S^{p}$. Different choices of map $y\mapsto L_y\in \rm{O}(q+1)$ give rise to different framings of the embedded surgery sphere $S^{p}$ in $X$. The resulting manifold (which is not yet smooth) is represented in Fig. \ref{Gajercorrectionmetric}. Recall that $\bar{\rho}$ and $\bar{r}$ are radial distances with respect to the Euclidean metric on $\mathbb{R}^{p+1}$ and $\mathbb{R}^{q+1}$ respectively. By choosing $\epsilon$ and $\delta$ sufficiently small and the corresponding tubes long enough, we can ensure that $\frac{\pi}{2}\epsilon<\bar{\rho}$ and $\frac{\pi}{2}\delta<\bar{r}$. 

\begin{figure}[htbp]
\begin{picture}(0,0)%
\includegraphics{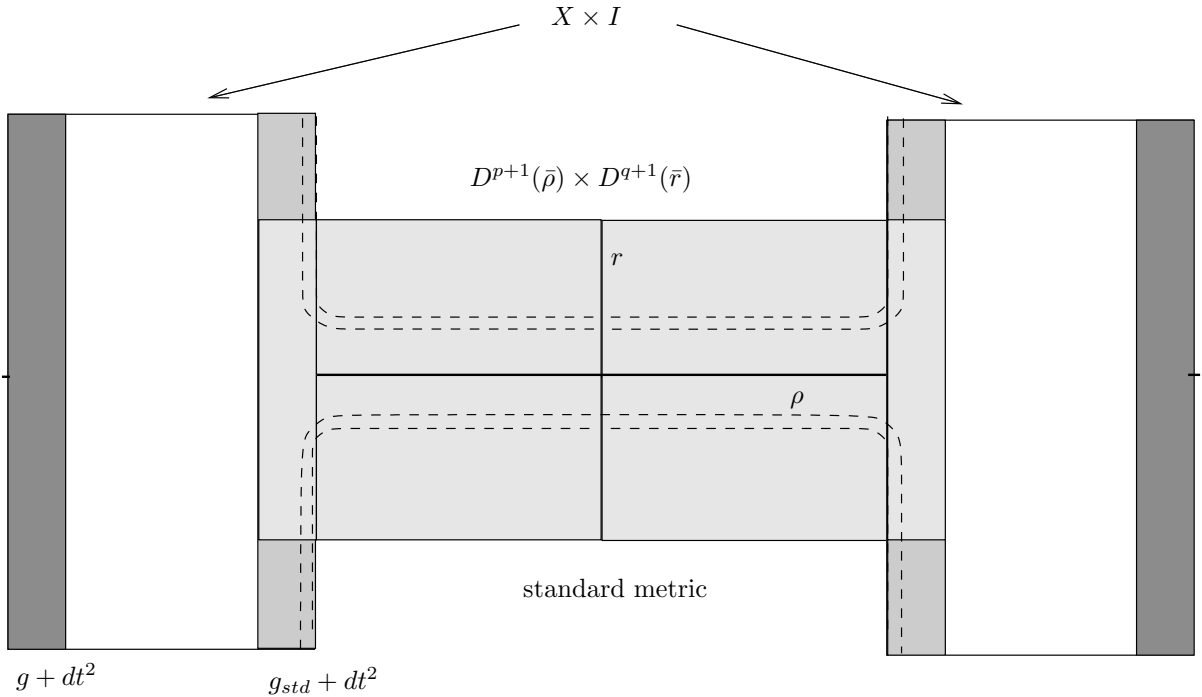}%
\end{picture}%
\setlength{\unitlength}{3947sp}%
\begingroup\makeatletter\ifx\SetFigFont\undefined%
\gdef\SetFigFont#1#2#3#4#5{%
  \reset@font\fontsize{#1}{#2pt}%
  \fontfamily{#3}\fontseries{#4}\fontshape{#5}%
  \selectfont}%
\fi\endgroup%
\begin{picture}(7569,4395)(1529,-4390)
\put(4489,-1161){\makebox(0,0)[lb]{\smash{{\SetFigFont{10}{8}{\rmdefault}{\mddefault}{\updefault}{\color[rgb]{0,0,0}$D^{p+1}(\bar{\rho})\times D^{q+1}(\bar{r})$}%
}}}}
\put(5001,-149){\makebox(0,0)[lb]{\smash{{\SetFigFont{10}{8}{\rmdefault}{\mddefault}{\updefault}{\color[rgb]{0,0,0}$X\times I$}%
}}}}
\put(1639,-4299){\makebox(0,0)[lb]{\smash{{\SetFigFont{10}{8}{\rmdefault}{\mddefault}{\updefault}{\color[rgb]{0,0,0}$g+dt^{2}$}%
}}}}
\put(3226,-4324){\makebox(0,0)[lb]{\smash{{\SetFigFont{10}{8}{\rmdefault}{\mddefault}{\updefault}{\color[rgb]{0,0,0}$g_{std}+dt^{2}$}%
}}}}
\put(4826,-3749){\makebox(0,0)[lb]{\smash{{\SetFigFont{10}{8}{\rmdefault}{\mddefault}{\updefault}{\color[rgb]{0,0,0}standard metric}%
}}}}
\put(6501,-2524){\makebox(0,0)[lb]{\smash{{\SetFigFont{10}{8}{\rmdefault}{\mddefault}{\updefault}{\color[rgb]{0,0,0}$\rho$}%
}}}}
\put(5376,-1661){\makebox(0,0)[lb]{\smash{{\SetFigFont{10}{8}{\rmdefault}{\mddefault}{\updefault}{\color[rgb]{0,0,0}$r$}%
}}}}
\end{picture}%

\caption{The metric $(X\times I,\bar{g})\cup (D^{p+1}(\bar{\rho})\times D^{q+1}(\bar{r}),h)$ and the smooth handle represented by the dashed curve}
\label{Gajercorrectionmetric}
\end{figure}

Two tasks remain. Firstly, we need to make this attaching smooth near the corners. This will be done in the obvious way by specifying a smooth hypersurface inside $D^{p+1}(\bar{\rho})\times D^{q+1}(\bar{r})$ which meets $N$ smoothly near its boundary, as shown by the dashed curve in Fig. \ref{Gajercorrectionmetric}. This is similar to the hypersurface $M$ constructed in the original Gromov-Lawson construction. Again we must ensure that the metric induced on this hypersurface has positive scalar curvature. This is considerably easier than in the case of $M$, given the ambient metric we are now working with. We will in fact show that the metric induced on this hypersurface is precisely the metric obtained by the Gromov-Lawson construction. The second task is to show that this metric can be adjusted to have a product structure near the boundary.

The spherical coordinates $(\rho, \phi, r, \theta)$ on the handle $D^{p+1}(\bar{\rho})\times D^{q+1}(\bar{r})$ can be extended to overlap with $X\times I$ on $N(\bar{r})\times [1-\epsilon_1, 1]$, where $\epsilon_1$ is chosen so that $\bar{g}|_{X\times[1-\epsilon_1, 1]}=g_{std}+dt^{2}$. We denote this region $D^{p+1}(\bar{\rho})\times D^{q+1}(\bar{r})$. Let $E$ be the embedding

\begin{equation}\label{map;Gajercorrection}
\begin{split}
E:[0,\epsilon_1]\times[0,\infty)&\longrightarrow\mathbb{R}\times\mathbb{R}\\
(s,t)&\longmapsto(a_1(s,t),a_2(s,t))
\end{split}
\end{equation}

\noindent shown in Fig. \ref{fig:embE}.

\begin{figure}[htbp]
\begin{picture}(0,0)%
\includegraphics{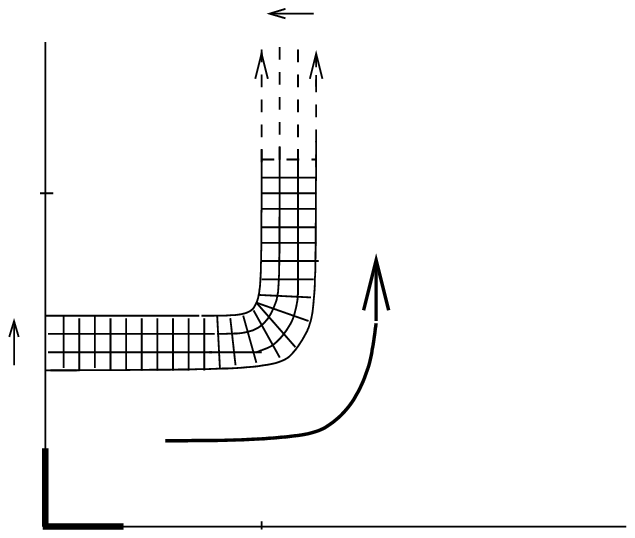}%
\end{picture}%
\setlength{\unitlength}{3947sp}%
\begingroup\makeatletter\ifx\SetFigFont\undefined%
\gdef\SetFigFont#1#2#3#4#5{%
  \reset@font\fontsize{#1}{#2pt}%
  \fontfamily{#3}\fontseries{#4}\fontshape{#5}%
  \selectfont}%
\fi\endgroup%
\begin{picture}(3277,3033)(1124,-3140)
\put(1739,-3049){\makebox(0,0)[lb]{\smash{{\SetFigFont{10}{8}{\rmdefault}{\mddefault}{\updefault}{\color[rgb]{0,0,0}$\epsilon\frac{\pi}{2}$}%
}}}}
\put(1189,-2724){\makebox(0,0)[lb]{\smash{{\SetFigFont{10}{8}{\rmdefault}{\mddefault}{\updefault}{\color[rgb]{0,0,0}$\delta\frac{\pi}{2}$}%
}}}}
\put(1139,-2024){\makebox(0,0)[lb]{\smash{{\SetFigFont{10}{8}{\rmdefault}{\mddefault}{\updefault}{\color[rgb]{0,0,0}$s$}%
}}}}
\put(2664,-261){\makebox(0,0)[lb]{\smash{{\SetFigFont{10}{8}{\rmdefault}{\mddefault}{\updefault}{\color[rgb]{0,0,0}$s$}%
}}}}
\put(4064,-2699){\makebox(0,0)[lb]{\smash{{\SetFigFont{10}{8}{\rmdefault}{\mddefault}{\updefault}{\color[rgb]{0,0,0}$\rho$}%
}}}}
\put(1714,-586){\makebox(0,0)[lb]{\smash{{\SetFigFont{10}{14.4}{\rmdefault}{\mddefault}{\updefault}{\color[rgb]{0,0,0}$r$}%
}}}}
\put(1626,-1711){\makebox(0,0)[lb]{\smash{{\SetFigFont{10}{14.4}{\rmdefault}{\mddefault}{\updefault}{\color[rgb]{0,0,0}$E(0,0)$}%
}}}}
\put(1639,-2286){\makebox(0,0)[lb]{\smash{{\SetFigFont{10}{14.4}{\rmdefault}{\mddefault}{\updefault}{\color[rgb]{0,0,0}$E(\epsilon_1,0)$}%
}}}}
\put(3151,-2361){\makebox(0,0)[lb]{\smash{{\SetFigFont{10}{14.4}{\rmdefault}{\mddefault}{\updefault}{\color[rgb]{0,0,0}$t$}%
}}}}
\put(2651,-2974){\makebox(0,0)[lb]{\smash{{\SetFigFont{10}{14.4}{\rmdefault}{\mddefault}{\updefault}{\color[rgb]{0,0,0}$\bar{\rho}$}%
}}}}
\put(1439,-1286){\makebox(0,0)[lb]{\smash{{\SetFigFont{10}{14.4}{\rmdefault}{\mddefault}{\updefault}{\color[rgb]{0,0,0}$\bar{r}$}%
}}}}
\end{picture}%

\caption{The embedding $E$}
\label{fig:embE}
\end{figure}

The map $E$ will satisfy the following conditions. 
 
(1) For each $s_0\in[0,\epsilon_1]$, $E(s_0,t)$ is the curve $(t,c_2(s_0))$ when $t\in[0,\frac{\epsilon\pi}{2}]$, and ends as the unit speed vertical line $(c_1(s_0),t)$. Here $c_1$ and $c_2$ are functions on $[0,\epsilon_1]$ defined as follows. For each $s$, $c_1(s)=\bar{\rho}+s$ and $c_2(s)=c_2(0)-s$, where $c_1(0)-\epsilon_1>\frac{\pi}{2}\epsilon$ and $c_2(s)>\frac{\pi}{2}\delta$.

(2) For each $t_0\in[0,\infty)$, the path $E(s,t_0)$ runs orthogonally to the levels $E(s_0,t)$ for each $s_0\in[0,\epsilon_1]$. That is, for each $(s_0,t_0)$, $\frac{\partial E}{\partial t}(s_0,t_0).\frac{\partial E}{\partial t}(s_0,t_0)=0$. 

\noindent Provided $\epsilon_1$ is chosen sufficiently small, the map
\begin{equation*}
\begin{split}
Z:[0,\epsilon_1]\times(0,\infty)\times{S^{p}}\times S^{q}&\rightarrow{D}^{p+1}(\bar{\rho}+\epsilon_1)\times{D}^{q+1}(\bar{r})\\
(s,t,\phi,\theta)&\mapsto(a_1(s,t),\phi,a_2(s,t),\theta)
\end{split}
\end{equation*}

\noindent parameterises a region in $D^{p+1}(\bar{\rho}+\epsilon_1)\times D^{q+1}(\bar{r})$. Consider the hypersurface parameterised by the $Z(s=0,\phi,t,\theta)$. The metric induced on the region bounded by this hypersurface extends $\bar{g}$ as a psc-metric over the trace of the surgery. Now we need to show that this metric can be deformed to one which is a product near the boundary while maintaining positive scalar curvature.

We begin by computing the metric near the boundary with respect to the parameterisation $Z$. Letting 

\begin{equation*}
\begin{array}{c}
Y_s={\frac{\partial a_1}{\partial s}}^{2}+ {\frac{\partial a_2}{\partial s}}^{2}
\qquad \text{and}\qquad Y_t={{\frac{\partial a_1}{\partial t}}^{2}+{\frac{\partial a_2}{\partial t}}^{2}},
\end{array}
\end{equation*}

\begin{equation*}
\begin{array}{cl}
Z^{*}(d\rho^{2}+f_{\epsilon}(\rho)^{2}ds_{p}^{2}+dr^{2}+f_{\delta}(r)^{2}ds_{q}^{2})&=d{a_1}^{2}+f_{\epsilon}(a_1)^{2}ds_{p}^{2}+d{a_2}^{2}+f_{\delta}(a_2)^{2}ds_{q}^{2}\\
&=Y_s(s,t)d{s}^{2}+Y_{t}(s,t)dt^{2}+f_{\epsilon}(a_1)^{2}ds_{p}^{2} +f_{\delta}(a_2)^{2}ds_{q}^{2}.
\end{array}
\end{equation*}

\noindent Now on the straight pieces of our neighbourhood, it is clear that $Y_s=1$ and $Y_t=1$. Thus on the straight region running parallel to the horizontal axis, the metric is

\begin{equation*}
\begin{array}{cl}
d{s}^{2}+dt^{2}+f_{\epsilon}(a_1)^{2}ds_{p}^{2}+f_{\delta}(a_2)^{2}ds_{q}^{2}&=d{s}^{2}+dt^{2}+f_{\epsilon}(t)^{2}ds_{p}^{2} +f_{\delta}(c_2(s))^{2}ds_{q}^{2}\\
&=d{s}^{2}+dt^{2}+f_{\epsilon}(t)^{2}ds_{p}^{2}+\delta^{2}ds_{q}^{2},\hspace{0.4cm}\text{since $c_2>\frac{\pi}{2}\delta$.}
\end{array}
\end{equation*}

\noindent On the straight region running parallel to the vertical axis, the metric is

\begin{equation*}
\begin{array}{clc}
d{s}^{2}+dt^{2}+f_{\epsilon}(a_1)^{2}ds_{p}^{2}+f_{\delta}(a_2)^{2}ds_{q}^{2}&=d{s}^{2}+dt^{2}+f_{\epsilon}(c_1(s))^{2}ds_{p}^{2} +f_{\delta}(t)^{2}ds_{q}^{2}\\
&=d{s}^{2}+dt^{2}+\epsilon^{2}ds_{p}^{2}+{\delta}^{2}ds_{q}^{2},\\
&=d{s}^{2}+dt^{2}+f_{\epsilon}(t)^{2}ds_{p}^{2}+{\delta}^{2}ds_{q}^{2},\\
\end{array}
\end{equation*}

\noindent The second equality holds because $c_1>\frac{\pi}{2}\epsilon$ and $t>\frac{\pi}{2}\delta$. The last equality follows from the fact that $t>c_1>\frac{\pi}{2}\epsilon$ and $t>\frac{\pi}{2}\delta$. As we do not have unit speed curves in $s$ and $t$, the best we can say about the remaining ``bending" region is that the metric is of the form

\begin{equation*}
\begin{array}{c}
Y_sd{s}^{2}+Y_{t}dt^{2}+\epsilon^{2}ds_{p}^2+{\delta}^2 ds_{q}^{2}.
\end{array}
\end{equation*}

\begin{figure}[htbp]
\begin{picture}(0,0)%
\includegraphics{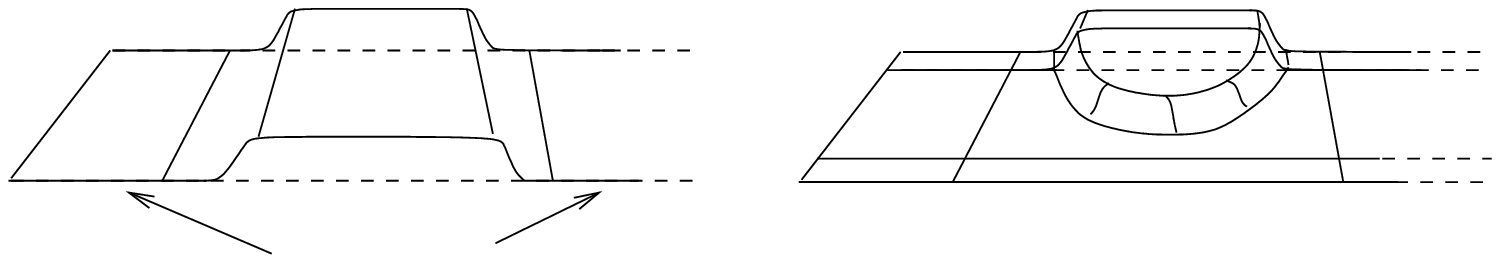}%
\end{picture}%
\setlength{\unitlength}{3947sp}%
\begingroup\makeatletter\ifx\SetFigFont\undefined%
\gdef\SetFigFont#1#2#3#4#5{%
  \reset@font\fontsize{#1}{#2pt}%
  \fontfamily{#3}\fontseries{#4}\fontshape{#5}%
  \selectfont}%
\fi\endgroup%
\begin{picture}(8539,1620)(749,-2215)
\put(2564,-2149){\makebox(0,0)[lb]{\smash{{\SetFigFont{10}{8}{\rmdefault}{\mddefault}{\updefault}{\color[rgb]{0,0,0}$Y_s=Y_t=1$}%
}}}}
\put(1289,-1036){\makebox(0,0)[lb]{\smash{{\SetFigFont{10}{8}{\rmdefault}{\mddefault}{\updefault}{\color[rgb]{0,0,0}$s=\epsilon_1$}%
}}}}
\put(764,-1624){\makebox(0,0)[lb]{\smash{{\SetFigFont{10}{8}{\rmdefault}{\mddefault}{\updefault}{\color[rgb]{0,0,0}$s=0$}%
}}}}
\put(4664,-1100){\makebox(0,0)[lb]{\smash{{\SetFigFont{10}{8}{\rmdefault}{\mddefault}{\updefault}{\color[rgb]{0,0,0}no change}%
}}}}
\put(6451,-1874){\makebox(0,0)[lb]{\smash{{\SetFigFont{10}{8}{\rmdefault}{\mddefault}{\updefault}{\color[rgb]{0,0,0}$Y_s'=Y_t'=1$}%
}}}}
\put(2589,-749){\makebox(0,0)[lb]{\smash{{\SetFigFont{10}{8}{\rmdefault}{\mddefault}{\updefault}{\color[rgb]{0,0,0}$Y_s,Y_t\neq 1$}%
}}}}
\end{picture}%

\caption{Adjusting $Y_s$ and $Y_t$}
\label{fig:XsYt}
\end{figure}

The graphs of $Y_s$ and $Y_t$ are surfaces, shown schematically in Fig. \ref{fig:XsYt}. Outside of a compact region, $Y_s=1$ and $Y_t=1$. We can replace $Y_s$ and $Y_t$ with smooth functions $Y_s'$ and $Y_t'$, so that on $[\epsilon_2,\epsilon_1]\times(0,\infty))$, $Y_s=Y_s'$ and $Y_t=Y_t'$ and so that on $[0,\epsilon_3]\times(0,\infty)$, $Y_s'=Y_t'=1$ for some $\epsilon_1>\epsilon_2>\epsilon_3>0$. Moreover, this can be done so that $Y_s-Y_s'$ and $Y_t-Y_t'$ have support in a compact region. Any curvature resulting from these changes is bounded and completely independent of the metric on the sphere factors. Thus we can always choose $\delta$ sufficiently small to guarantee the positive scalar curvature of the resulting metric

\begin{equation*}
\begin{array}{c}
Y_{s}'d{s}^{2}+Y_{t}'dt^{2}+f_{\epsilon}(t)^{2}ds_{p}^{2}+{\delta}^{2}ds_{q}^{2}
\end{array}
\end{equation*}

\noindent which, when $s\in[0,\epsilon_3]$, is the metric
\begin{equation*}
\begin{array}{c}
d{s}^{2}+dt^{2}+f_{\epsilon}(t)^{2}ds_{p}^{2}+{\delta}^{2}ds_{n}^{2}.
\end{array}
\end{equation*}

\noindent This is of course the metric $ds^{2}+g_{tor}^{p+1}(\epsilon)+\delta^{2}ds_{q}^{2}$,  completing the proof of Theorem \ref{ImprovedsurgeryTheorem}.

\end{proof}

\section{Constructing Gromov-Lawson cobordisms}\label{GLcobordsection}
In section \ref{surgerysection} we showed that a psc-metric $g$ on $X$ can be extended over the trace of a codimension$\geq 3$ surgery to a psc-metric with product structure near the boundary. Our goal in section \ref{GLcobordsection} is to generalise this result in the form of Theorem \ref{GLcob}. Here $\{W^{n+1};X_0,X_1\}$ denotes a smooth compact cobordism and $g_0$ is a psc-metric on $X_0$. If $W$ can be decomposed into a union of elementary cobordisms, each the trace of a codimension$\geq 3$ surgery, then we should be able to extend $g_0$ to a psc-metric on $W$, which is product near the boundary, by repeated application of Theorem \ref{ImprovedsurgeryTheorem}. Two questions now arise. Assuming $W$ admits such a decomposition, how do we realise it? Secondly, how many such decompositions can we realise? In order to answer these questions it is worth briefly reviewing some basic Morse Theory. For a more thorough account of this see \cite{Smale} and \cite{Hatcher}.

\subsection{Morse Theory and admissible Morse functions}
Let $\F=\F(W)$ denote the space of smooth functions $f:W\rightarrow I$ on the cobordism $\{W;X_0,X_1\}$ with the property that $f^{-1}(0)=X_0$ and $f^{-1}(1)=X_1$, and having no critical points near $\p {W}$. The space $\F$ is a subspace of the space of smooth functions on $W$ with its standard $C^{\infty}$ topology, see chapter 2 of \cite{Hirsch} for the full definition. A function $f\in \F$ is a Morse function if whenever $w$ is a critical point of $f$, $det(D^{2}f(w))\neq 0$. Here $D^{2}f(w)$ denotes the Hessian of $f$ at $w$. The Morse index of the critical point $w$ is the number of negative eigenvalues of $D^{2}f(w)$. The well known Morse Lemma, Lemma 2.2 of \cite{Smale}, then says that there is a coordinate chart $\{x=(x_1,x_2,\cdots,x_{n+1})\}$ near $w$, with $w$ identified with $(0,\cdots,0)$, so that in these coordinates,
\begin{equation} \label{Morsequadratic}
 f(x)=c-x_1^{2}-\cdots-x_{p+1}^{2}+x_{p+2}^{2}+\cdots+x_{n+1}^{2},
\end{equation} 
\noindent where $c=f(w)$. Here $p+1$ is the Morse index of $w$ and this coordinate chart is known as a {\it Morse coordinate chart}.

\begin{figure}[htbp]
\begin{picture}(0,0)%
\includegraphics{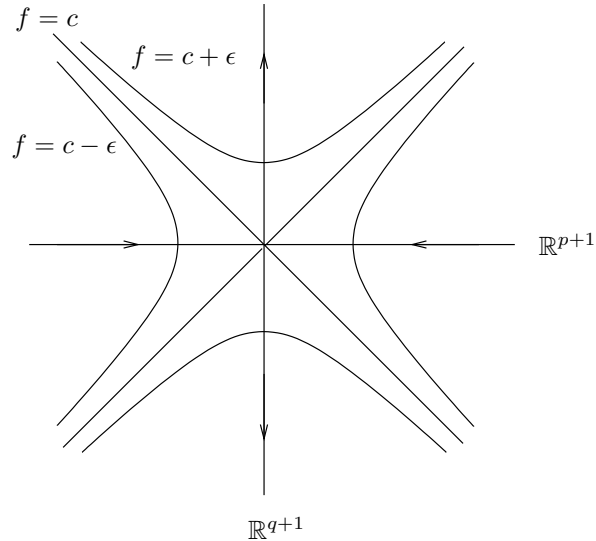}%
\end{picture}%
\setlength{\unitlength}{3947sp}%
\begingroup\makeatletter\ifx\SetFigFont\undefined%
\gdef\SetFigFont#1#2#3#4#5{%
  \reset@font\fontsize{#1}{#2pt}%
  \fontfamily{#3}\fontseries{#4}\fontshape{#5}%
  \selectfont}%
\fi\endgroup%
\begin{picture}(3317,3445)(2874,-6927)
\put(4351,-6861){\makebox(0,0)[lb]{\smash{{\SetFigFont{10}{8}{\rmdefault}{\mddefault}{\updefault}{\color[rgb]{0,0,0}$\mathbb{R}^{q+1}$}%
}}}}
\put(6176,-5074){\makebox(0,0)[lb]{\smash{{\SetFigFont{10}{8}{\rmdefault}{\mddefault}{\updefault}{\color[rgb]{0,0,0}$\mathbb{R}^{p+1}$}%
}}}}
\put(2864,-4436){\makebox(0,0)[lb]{\smash{{\SetFigFont{10}{8}{\rmdefault}{\mddefault}{\updefault}{\color[rgb]{0,0,0}$f=c-\epsilon$}%
}}}}
\put(3614,-3874){\makebox(0,0)[lb]{\smash{{\SetFigFont{10}{8}{\rmdefault}{\mddefault}{\updefault}{\color[rgb]{0,0,0}$f=c+\epsilon$}%
}}}}
\put(2889,-3636){\makebox(0,0)[lb]{\smash{{\SetFigFont{10}{8}{\rmdefault}{\mddefault}{\updefault}{\color[rgb]{0,0,0}$f=c$}%
}}}}
\end{picture}%
\caption{Morse coordinates around a critical point}
\label{morsecoords}
\end{figure}

Inside of this coordinate chart it is clear that level sets below the critical level are diffeomorphic to $S^{p}\times D^{q+1}$ and that level sets above the critical level are diffeomorphic to $D^{p+1}\times S^{q}$ where $p+q+1=n$, see Fig. \ref{morsecoords}. In the case where $f$ has exactly one critical point $w$ of index $p+1$, the cobordism $W$ is diffeomorphic to the trace of a $p$-surgery on $X_0$. If $W$ admits a Morse function $f$ with no critical points then by theorem 3.4 of \cite{Smale}, $W$ is diffeomorphic to the cylinder $X_0\times I$ (and consequently $X_0$ is diffeomorphic to $X_1$).

The critical points of a Morse function are isolated and as $W$ is compact, $f$ will have only finitely many. Denote the critical points of $f$ as $w_0, w_1, \cdots, w_k$ where each $w_i$ has Morse index $p_i+1$. We will assume that $0<f(w_0)=c_0\leq f(w_1)=c_1\leq\cdots f(w_k)=c_k<1$. 

\begin{Definition}
\rm{The Morse function $f$ is {\em well-indexed} if critical points on the same level set have the same index and for all $i$, $p_i\leq p_{i+1}$.} \end{Definition}

In the case when the above inequalities are all strict, $f$ decomposes $W$ into a union of elementary cobordisms $C_0 \cup C_1 \cup \cdots \cup C_k$. Here each $C_i=f^{-1}([c_{i-1}+\tau, c_{i}+\tau])$ when $0<i<k$, and $C_0=f^{-1}([0,c_0+\tau])$ and $C_{k}=f^{-1}([c_{k-1}+\tau,1])$, for some appropriately small $\tau>0$. Each $C_i$ is the trace of a $p_i$-surgery. When these inequalities are not strict, in other words $f$ has level sets with more than one critical point, then $W$ is decomposed into a union of cobordisms $C_0'\cup C_1'\cup\cdots \cup C_l'$ where $l<k$. A cobordism $C_i'$ which is not elementary, is the trace of several distinct surgeries. It is of course possible, with a minor adjustment of $f$, to make the above inequalities strict.

By equipping $W$ with a Riemannian metric $m$, we can define ${grad}_{m}f$ the gradient vector field for $f$. This metric is called a {\it background metric} for $f$ and has no relation to any of the other metrics mentioned here. In particular, no assumptions are made about its curvature. More generally, we define {\it gradient-like} vector fields on $W$ with respect to $f$ and $m$, as follows.  

\begin{Definition}
\rm{
A {\em gradient-like} vector field with respect to $f$ and $m$ is a vector field $V$ on $W$ satisfying the following properties.

(1) $df_{x}(V_x)>0$ for all $x$ in $W$.

(2) Each critical point $w$ of $f$, lies in a neighbouhood $U$ so that for all $x\in U$, $V_x={grad}_{m}f(x)$. 
}
\end{Definition}  

\noindent We point out that the space of background metrics for a particular Morse function $f:W\rightarrow I$ is a convex space. So too, is the space of gradient-like vector fields associated with any particular pair $(f,m)$, see chapter 2, section 2 of \cite{Hatcher}. We can now define an admissible Morse function on $W$.

\begin{Definition}
\rm{
An {\em admissible Morse function} $f$ on a compact cobordism $\{W;X_0,X_1\}$ is a triple $f=(f,m,V)$ where $f:W\rightarrow I$ is a Morse function, $m$ is a background metric for $f$, $V$ is a gradient like vector field with respect to $f$ and $m$, and finally, any critical point of $f$ has Morse index less than or equal to $n-2$
}
\end{Definition}

\begin{Remark}
We emphasise the fact that an admissible Morse function is actually a triple consisting of a Morse function, a Riemannian metric and a gradient-like vector field. However, to ease the burden of notation, an admissible Morse function $(f, m, V)$ will be denoted as simply $f$. 
\end{Remark}

\begin{figure}[htbp]
\begin{picture}(0,0)%
\includegraphics{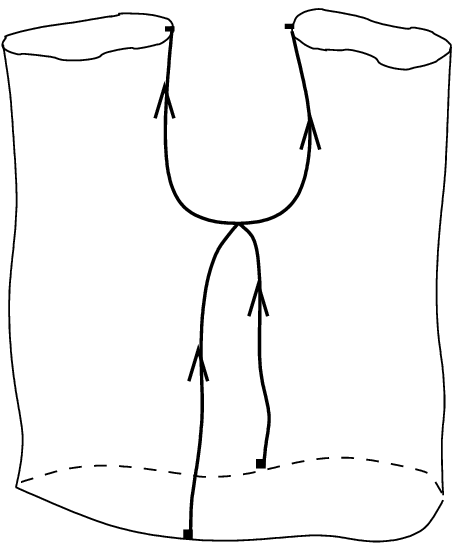}%
\end{picture}%
\setlength{\unitlength}{3947sp}%
\begingroup\makeatletter\ifx\SetFigFont\undefined%
\gdef\SetFigFont#1#2#3#4#5{%
  \reset@font\fontsize{#1}{#2pt}%
  \fontfamily{#3}\fontseries{#4}\fontshape{#5}%
  \selectfont}%
\fi\endgroup%
\begin{picture}(2177,2814)(1974,-3771)
\put(3076,-2136){\makebox(0,0)[lb]{\smash{{\SetFigFont{10}{8}{\rmdefault}{\mddefault}{\updefault}{\color[rgb]{0,0,0}$w$}%
}}}}
\put(2164,-1836){\makebox(0,0)[lb]{\smash{{\SetFigFont{10}{8}{\rmdefault}{\mddefault}{\updefault}{\color[rgb]{0,0,0}$K_{-}^{q+1}(w)$}%
}}}}
\put(3314,-2924){\makebox(0,0)[lb]{\smash{{\SetFigFont{10}{8}{\rmdefault}{\mddefault}{\updefault}{\color[rgb]{0,0,0}$K_{-}^{p+1}(w)$}%
}}}}
\put(3039,-3611){\makebox(0,0)[lb]{\smash{{\SetFigFont{10}{8}{\rmdefault}{\mddefault}{\updefault}{\color[rgb]{0,0,0}$S_{-}^{p}(w)$}%
}}}}
\put(2851,-1111){\makebox(0,0)[lb]{\smash{{\SetFigFont{10}{8}{\rmdefault}{\mddefault}{\updefault}{\color[rgb]{0,0,0}$S_{+}^{q}(w)$}%
}}}}
\end{picture}%

\caption{Trajectory spheres for a critical point $w$ on an elementary cobordism}
\label{trajflow}
\end{figure}

Associated to each critical point $w$ of index $p+1$ is a pair of trajectory spheres $S_{-}^{p}(w)$ and $S_{+}^{q}(w)$, respectively converging towards and emerging from $w$, see Fig. \ref{trajflow}. As usual $p+q+1=n$. Let us assume for simplicity that $f$ has exactly one critical point $w$ and that $w$ has Morse index $p+1$. Then associated to $w$ is an embedded sphere $S^{p}=S_{-}^{p}(w)$ in $X_0$ which follows a trajectory towards $w$. The trajectory itself consists of the union of segments of integral curves of the gradient vector field beginning at the embedded $S^{p}\subset X_0$ and ending at $w$. It is topologically a $(p+1)$-dimensional disk $D^{p+1}$. We denote it $K_{-}^{p+1}(w)$. Similarly, there is an embedded sphere $S^{q}=S_{+}^{q}\subset X_1$ which bounds a trajectory $K_{+}^{q}(w)$ (homeomorphic to a disk $D^{q}$) emerging from $w$. Both spheres are embedded with trivial normal bundle and the elementary cobordism $W$ is in fact diffeomorphic to the trace of a surgery on $X_0$ with respect to $S^{p}$.

We are now in a position to prove Theorem \ref{GLcob}. This is the construction, given a positive scalar curvature metric $g_0$ on $X_0$ and an admissible Morse function $f$ on $W$, of a psc-metric $\bar{g}=\bar{g}(g_0,f)$ on $W$ which extends $g_0$ and is a product near the boundary. As pointed out in the introduction, the metric $\bar{g}$ is known as a {\it Gromov-Lawson cobordism with respect to $g_0$ and $f$}. The resulting metric induced on $X_1$, $g_1=\bar{g}|_{X_1}$, is said to be {\it Gromov-Lawson cobordant} to $g_0$.
\\

\noindent {\bf Theorem \ref{GLcob}.} 
{\em Let $\{W^{n+1};X_0,X_1\}$ be a smooth compact cobordism. Suppose $g_0$ is a metric of positive scalar curvature on $X_0$ and $f:W\rightarrow I$ is an admissible Morse function. Then there is a psc-metric $\bar{g}=\bar{g}(g_0,f)$ on $W$ which extends $g_0$ and has a product structure near the boundary.}

\begin{proof}
Let $f$ be an admissible Morse function on $W$. Let $m$ be the background metric on $W$, as described above. Around each critical point $w_i$ of $f$ we choose mutually disjoint Morse coordinate balls $B(w_i)=B_{m}(w_i,\bar{\epsilon})$ where $\bar{\epsilon}>0$ is some sufficiently small constant. For the moment, we may assume that $f$ has only one critical point $w$ of Morse index $p+1$ where as usual $p+q+1=n$ and $q\geq 2$. Let $c=f(w)\in(0,1)$. Associated to $w$ are the trajectory spheres $S^{p}=S_{-}^{p}(w)$ and $S_{+}^{q}(w)$, defined earlier in this section. Let $N=S^{p}\times D^{q+1}(\bar{r})\subset X_0$ denote the tubular neighbourhood defined in the proof of Theorem \ref{IsotopyTheorem}, constructed using the exponential map for the metric $g_0$. The {\em normalised} gradient-like flow of $f$ (obtained by replacing $V$ with $\frac{V}{m(V,V)}$ away from critical points and smoothing with an appropriate bump function) gives rise to a diffeomorphism from $f^{-1}([0,\epsilon_0])$ to $f^{-1}([0,c-\tau])$ where $0<\epsilon_0<c-\tau<c$. In particular, normalisation means that it maps $f^{-1}([0,\epsilon_0])$ diffeomorphically onto $f^{-1}([c-\tau-\epsilon_0, c-\tau])$. For sufficiently small $\bar{r}$, $\epsilon_0$ and $\tau$, the level set $f^{-1}(c-\tau)$ may be chosen to intersect with $B(w)$ so as to contain the image of $N\times [0,\epsilon_0]$ under this diffeomorphism, see Fig. \ref{actionofflow}. 

\begin{figure}[htbp]

\begin{picture}(0,0)%
\includegraphics{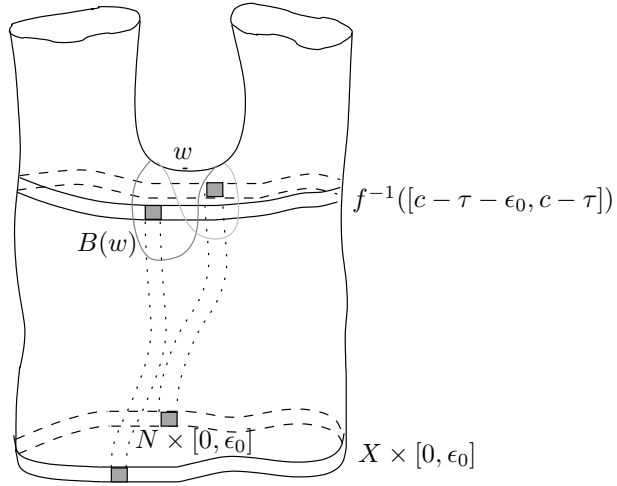}%
\end{picture}%
\setlength{\unitlength}{3947sp}%
\begingroup\makeatletter\ifx\SetFigFont\undefined%
\gdef\SetFigFont#1#2#3#4#5{%
  \reset@font\fontsize{#1}{#2pt}%
  \fontfamily{#3}\fontseries{#4}\fontshape{#5}%
  \selectfont}%
\fi\endgroup%
\begin{picture}(2217,3032)(1974,-4223)
\put(2409,-2761){\makebox(0,0)[lb]{\smash{{\SetFigFont{10}{8}{\rmdefault}{\mddefault}{\updefault}{\color[rgb]{0,0,0}$B(w)$}%
}}}}
\put(2776,-4001){\makebox(0,0)[lb]{\smash{{\SetFigFont{10}{8}{\rmdefault}{\mddefault}{\updefault}{\color[rgb]{0,0,0}$N\times{[0,\epsilon_0]}$}%
}}}}
\put(4176,-4086){\makebox(0,0)[lb]{\smash{{\SetFigFont{10}{8}{\rmdefault}{\mddefault}{\updefault}{\color[rgb]{0,0,0}$X\times{[0,\epsilon_0]}$}%
}}}}
\put(4139,-2486){\makebox(0,0)[lb]{\smash{{\SetFigFont{10}{8}{\rmdefault}{\mddefault}{\updefault}{\color[rgb]{0,0,0}$f^{-1}([c-\tau-\epsilon_0,c-\tau])$}%
}}}}
\put(3014,-2186){\makebox(0,0)[lb]{\smash{{\SetFigFont{10}{8}{\rmdefault}{\mddefault}{\updefault}{\color[rgb]{0,0,0}$w$}%
}}}}
\end{picture}%

\caption{The action of the gradient-like flow on $N\times [0,\epsilon_0]$}
\label{actionofflow}
\end{figure}

\begin{figure}[htbp]

\begin{picture}(0,0)%
\includegraphics{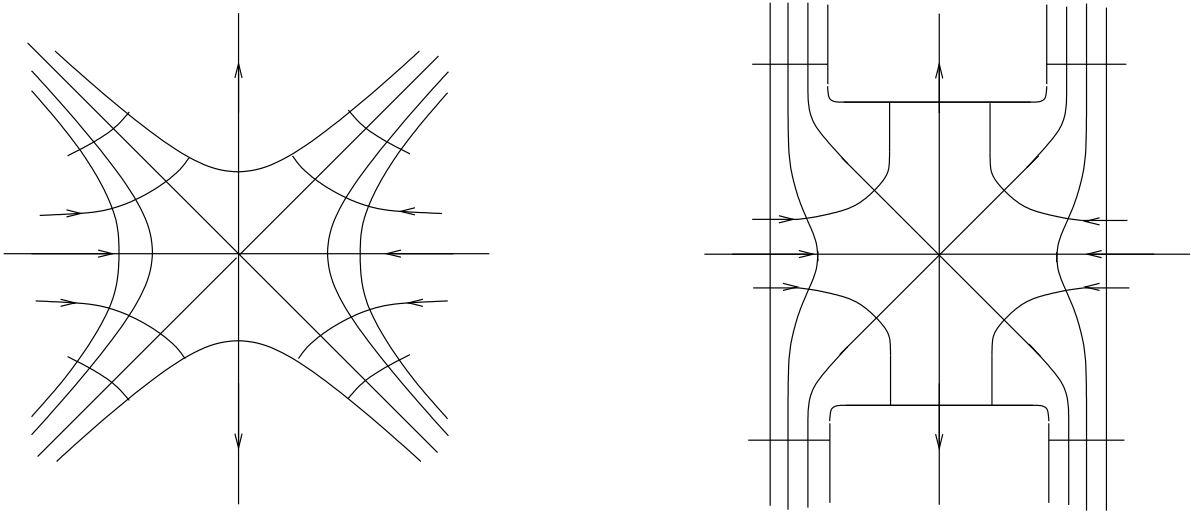}%
\end{picture}%
\setlength{\unitlength}{3947sp}%
\begingroup\makeatletter\ifx\SetFigFont\undefined%
\gdef\SetFigFont#1#2#3#4#5{%
  \reset@font\fontsize{#1}{#2pt}%
  \fontfamily{#3}\fontseries{#4}\fontshape{#5}%
  \selectfont}%
\fi\endgroup%
\begin{picture}(7475,3217)(1364,-4095)
\end{picture}%
\caption{A diffeomorphism on the handle.}
\label{coordchange}
\end{figure}
In fact, we can use the normalised gradient-like flow to construct a diffeomorphism between $X_0\times [0,c-\tau]$ and $f^{-1}([0,c-\tau])$ which for each $s\in[0,c-\tau]$, maps $X\times\{s\}$ diffeomorphically onto $f^{-1}(s)$. Corollary \ref{GLconc} then allows us to extend the metric $g_0$ from $X_0$ as a psc-metric over $f^{-1}([0,c-\tau])$ which is product near the boundary. Moreover this extension can be constructed so that the resulting metric $\bar{g_0}$, is the product $g_0+dt^{2}$ outside of $B(w)$ and on $f^{-1}([c-\tau-\epsilon_0, c-\tau])$ is the metric $(g_0)_{std}+dt^{2}$ where $(g_0)_{std}$ is the metric constructed in Theorem \ref{IsotopyTheorem} with respect to $N$. Recall that on $X_0$, the metric $(g_0)_{std}$ is the original metric $g_0$ outside of $N$ but that near $S^{p}$, $(g_0)_{std}=\epsilon^{2}ds_{p}^{2}+g_{tor}^{q+1}(\delta)$. Choose $r_0\in(0,\bar{r})$, so that on the neighbourhood $N(r_0)=S^{p}\times D^{q+1}(r_0)\subset N$, $(g_0)_{std}=\epsilon^{2}ds_{p}^{2}+g_{tor}^{q+1}(\delta)$. Observe that the trajectory of $X_0\setminus N(r_0)$ does not pass any critical points. Thus it is possible to extend $\bar{g_0}$ as $(g_0)_{std}+dt^2$ along this trajectory up to the level set $f^{-1}(c+\tau)$. To extend this metric over the rest of $f^{-1}[0,c+\tau]$ we use a diffeomorphism of the type described in Fig. \ref{coordchange} to adjust coordinates near $w$. Thus, away from the origin, the level sets and flow lines of $f$ are the vertical and horizontal lines of the the standard Cartesian plane and the extension along the trajectory of $X_0\setminus N(r_0)$ is assumed to take place on this region, see Fig. \ref{extension}. Over the rest of $f^{-1}(c+\tau)$, the metric $\bar{g_0}$ can be extended extended as the metric constructed in Theorem \ref{ImprovedsurgeryTheorem}.

\begin{figure}[htbp]
\begin{picture}(0,0)%
\includegraphics{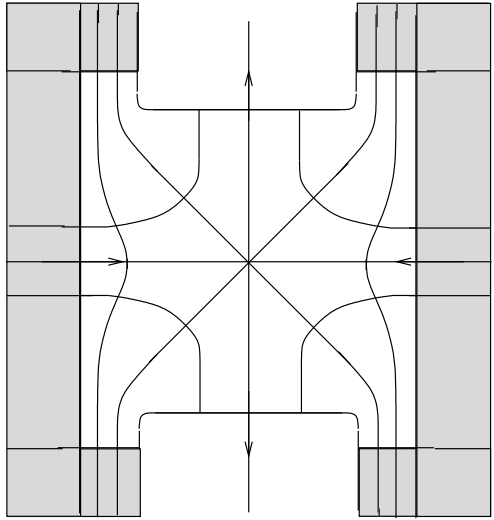}%
\end{picture}%
\setlength{\unitlength}{3947sp}%
\begingroup\makeatletter\ifx\SetFigFont\undefined%
\gdef\SetFigFont#1#2#3#4#5{%
  \reset@font\fontsize{#1}{#2pt}%
  \fontfamily{#3}\fontseries{#4}\fontshape{#5}%
  \selectfont}%
\fi\endgroup%
\begin{picture}(3074,3261)(2202,-4567)
\end{picture}%
\caption{Extending $\bar{g_0}$ along the trajectory of $X_0\setminus N(r_0)$ to the level set $f^{-1}(c+\tau)$.}
\label{extension}
\end{figure}

At this stage we have constructed a psc-metric on $f^{-1}(c+\tau)$, which extends the original metric $g_0$ on $X_0$ and is product near the boundary. As $f^{-1}([c+\tau,1])$ is diffeomorphic to the cylinder $X_1\times[c+\tau,1]$, this metric can then be extended as a product metric over the rest of $W$. This construction is easily generalised to the case where $f$ has more than one critical point on the level set $f^{-1}(c)$. In the case where $f$ has more than one critical level, and thus decomposes $W$ into cobordisms $C_0'\cup C_1'\cup\cdots \cup C_l'$ as described above, repeated application of this construction over each cobordism results in the desired metric $\bar{g}(g_0,f)$.
\end{proof}    

\subsection{A reverse Gromov-Lawson cobordism}

Given a Morse triple $f=(f,m,V)$ on a smooth compact cobordism $\{W;X_0,X_1\}$, with $f^{-1}(0)=X_0$ and $f^{-1}(1)=X_1$, we denote by $1-f$, the Morse triple $(1-f, m, -V)$ which has the gradient-like flow of $f$, but running in the opposite direction. In particular,  $(1-f)^{-1}(0)=X_1$ and $(1-f)^{-1}(1)=X_0$ and so it is easier to think of this as simply ``turning the cobordism $W$ upside down". Although $1-f$ has the same critical points as $f$, there is a change in the indices. Each critical point of $f$ with index $p+1$ is a critical point of $1-f$ with index $q+1$, where $p+q+1=n$ and $\dim W=n+1$. Just as $f$ describes a sequence of surgeries which turns $X_0$ into $X_1$, $1-f$ describes a sequence of complementary surgeries which reverses this process and turns $X_1$ back into $X_0$. 
 
Given an admissible Morse function $f$ on a cobordism $\{W; X_0, X_1\}$, Theorem \ref{GLcob} allows us to construct, from a psc-metric $g_0$ on $X_0$, a new psc-metric $g_1$ on $X_1$. Suppose now that $1-f$ is also an admissible Morse function. The following theorem describes what happens if we reapply the construction of Theorem \ref{GLcob} on the metric $g_1$ with respect to the function $1-f$.  
\\

\noindent {\bf Theorem \ref{Reversemorse}.}
{\em Let $\{W^{n+1};X_0,X_1\}$ be a smooth compact cobordism, $g_0$ a psc-metric on $X_0$ and $f:W\rightarrow I$, an admissible Morse function. Suppose that $1-f$ is also an admissible Morse function. Let $g_1=\bar{g}(g_0,f)|_{X_1}$ denote the restriction of the Gromov-Lawson cobordism $\bar{g}(g_0,f)$ to $X_1$. Let $\bar{g}(g_1,1-f)$ be a Gromov-Lawson cobordism with respect to $g_1$ and $1-f$ and let $g_0'=\bar{g}(g_1, 1-f)|_{X_0}$ denote the restriction of this metric to $X_0$. Then $g_0$ and $g_0'$ are canonically isotopic metrics in $\Riem^{+}(X_0)$.}

\begin{proof}
It is enough to consider the case where $f$ has a single critical point $w$ of index $p+1$. The metric $g_1$ is the restriction of the metric $\bar{g}(g_0, f)$, constructed in Theorem \ref{GLcob}, to $X_1$. In constructing the metric $g_0'$ we apply the Gromov-Lawson construction to this metric with respect to surgery on an embedded sphere $S^{q}$. The admissible Morse function $1-f$ determines a neighbourhood $S^{q}\times D^{p+1}$ on which this surgery takes place. Recall that, by construction, the metric $g_1$ is already the standard metric $\delta^{2}ds_{q}^{2}+g_{tor}^{p+1}(\epsilon)$ near this embedded sphere. Thus, $g_0'$ is precisely the metric obtained by applying the Gromov-Lawson construction on this standard piece. Removing a tubular neighbourhood of $S^{q}$ in this standard region results in a metric on $X_1\setminus{S^{q}\times D^{p+1}}$, which is the standard product $\delta^{2}ds_{q}^{2}+\epsilon^{2}ds_{p}^{2}$. The construction is completed by attaching the product $D^{q+1}\times S^{p}$ with the standard metric $g_{tor}^{q+1}(\delta)+\epsilon^{2}ds_{p}^{2}$. In Fig. \ref{reversemorse} we represent this, using a dashed curve, as a hypersurface of the standard region. The resulting metric is isotopic to the metric $(g_0)_{std}$, the metric obtained from $g_0$ in Theorem \ref{IsotopyTheorem}, by a continuous rescaling of the tube length of the torpedo factor. In turn $(g_0)_{std}$ can then be isotopied back to $g_0$ by Theorem \ref{IsotopyTheorem}.

\begin{figure}[htbp]
\begin{picture}(0,0)%
\includegraphics{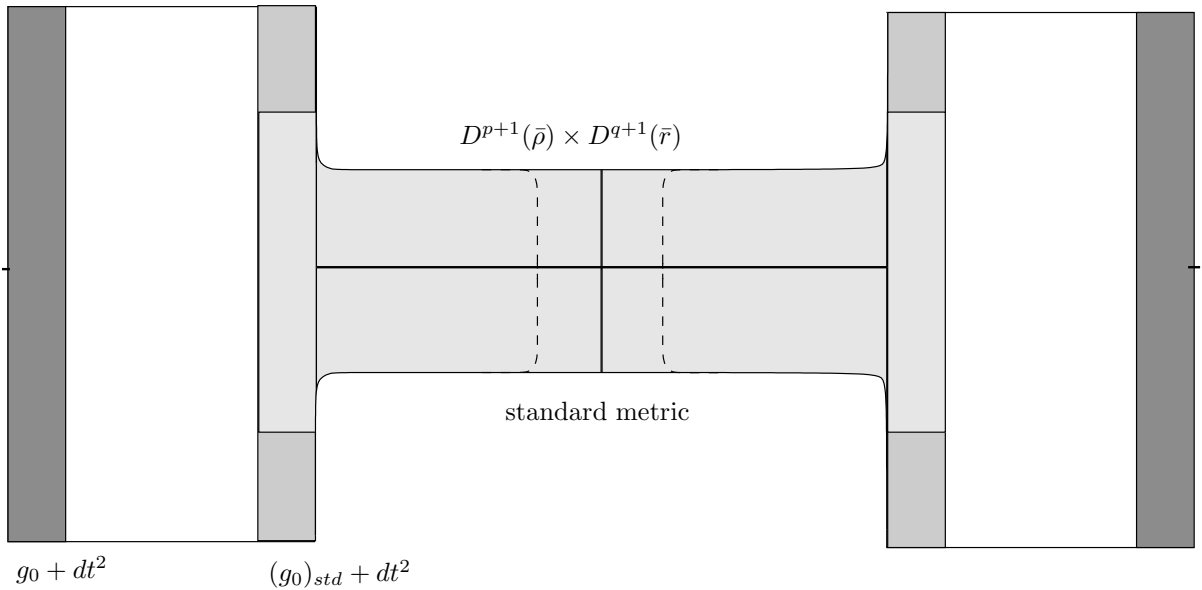}%
\end{picture}%
\setlength{\unitlength}{3947sp}%
\begingroup\makeatletter\ifx\SetFigFont\undefined%
\gdef\SetFigFont#1#2#3#4#5{%
  \reset@font\fontsize{#1}{#2pt}%
  \fontfamily{#3}\fontseries{#4}\fontshape{#5}%
  \selectfont}%
\fi\endgroup%
\begin{picture}(7569,3697)(1529,-4390)
\put(1639,-4299){\makebox(0,0)[lb]{\smash{{\SetFigFont{10}{8}{\rmdefault}{\mddefault}{\updefault}{\color[rgb]{0,0,0}$g_0+dt^{2}$}%
}}}}
\put(3226,-4324){\makebox(0,0)[lb]{\smash{{\SetFigFont{10}{8}{\rmdefault}{\mddefault}{\updefault}{\color[rgb]{0,0,0}$(g_0)_{std}+dt^{2}$}%
}}}}
\put(4426,-1574){\makebox(0,0)[lb]{\smash{{\SetFigFont{10}{8}{\rmdefault}{\mddefault}{\updefault}{\color[rgb]{0,0,0}$D^{p+1}(\bar{\rho})\times D^{q+1}(\bar{r})$}%
}}}}
\put(4714,-3311){\makebox(0,0)[lb]{\smash{{\SetFigFont{10}{8}{\rmdefault}{\mddefault}{\updefault}{\color[rgb]{0,0,0}standard metric}%
}}}}
\end{picture}%

\caption{The metric induced by $\bar{g}(g_1, -f)$ on a level set below the critical level}
\label{reversemorse}
\end{figure} 

\end{proof}

\subsection{Continuous families of Morse functions}

The construction of Theorem \ref{GLcob} easily generalises to the case of a compact family of admissible Morse functions. Before doing this we should briefly discuss the space $\M= \M(W)$ of Morse functions in $\F=\F(W)$. It is well known that $\M$ is an open dense subspace of $\F$, see theorem 2.7 of \cite{Smale}. Let $\tilde{\F}$ denote the space of triples $(f,m,V)$ so that $f\in \F$, $m$ is a backgound metric for $f$ and $V$ is a gradient-like vector field with respect to $f$ and $m$. The space $\tilde{\F}$ is then homotopy equivalent to the space ${\F}$. In fact, by equipping $W$ with a fixed background metric $\bar{m}$, the inclusion map
\begin{equation}
f\longmapsto(f,\bar{m},{grad}_{\bar{m}}f)
\end{equation}
forms part of a deformation retract of $\tilde{\F}$ down to $\F$, see chapter 2, section 2 of \cite{Hatcher} for details.
 
Denote by $\tilde{\M}=\tilde{\M}(W)$, the subspace of $\tilde{F}$, of triples $(f,m,V)$ where $f$ is a Morse function. Elements of $\tilde{\M}$ will be known as {\it Morse triples}. The above deformation retract then restricts to a deformation retract of $\tilde{\M}$ to $\M$. The subspace of $\tilde{\M}$ consisting of admissible Morse functions will be denoted $\tilde{\M}^{adm}=\tilde{\M}^{adm}(W)$. To economise in notation we will shorten $(f,m,V)$ to simply $f$. Let $f_0,f_1\in \tilde{\M}$.

\begin{Definition}
\rm{
The Morse triples $f_0$ and $f_1$ are {\em isotopic} if they lie in the same path component of $\tilde{\M}$. A path $f_t,t\in I$ connecting $f_0$ and $f_1$ is called an {\em isotopy} of Morse triples.  
}
\end{Definition}

\begin{Remark}
This dual use of the word isotopy is unfortunate, however, it should be clear from context which meaning we wish to employ.
\end{Remark}

\noindent In order for two Morse triples $f_0$ and $f_1$ to lie in the same path component of $\tilde{\M}$, it is necessary that both have the same number of index $p$ critical points for each $p\in\{0,1,\cdots n+1\}$. Thus if $f_0$ and $f_1$ are both admissible Morse functions, an isotopy of Morse triples connecting $f_0$ to $f_1$ is contained entirely in $\tilde{\M}^{adm}$. We now prove Theorem \ref{GLcobordismcompact}.
\\

\noindent {\bf Theorem \ref{GLcobordismcompact}.}
{\em Let $\{W, X_0, X_1\}$ be a smooth compact cobordism and let $B$ and $C$ be a pair of compact spaces. Let $\mathcal{B}=\{g_b\in \Riem^{+}(X_0):b\in B\}$ be a continuous family of psc-metrics on $X_0$ and let $\mathcal{C}=\{f_c\in\tilde{\M}^{adm}(W):c\in C\}$ be a continuous family of admissible Morse functions on $W$. Then there is a continuous map 

\begin{equation*}
\begin{split}
\mathcal{B}\times \mathcal{C}&\longrightarrow \Riem^{+}(W)\\
(g_b,f_c)&\longmapsto \bar{g}_{b,c}=\bar{g}(g_b, f_c)
\end{split}
\end{equation*}

\noindent so that for each pair $(b,c)$, the metric $\bar{g}_{b,c}$ is the metric constructed in Theorem \ref{GLcob}.}

\begin{proof} For each $c\in C$, $f_c$ will have the same number of critical points of the same index. The family $\{f_c\}$ can be thought of as a continuous rearrangement of the critical points. In turn this means a continuous rearrangement of embedded surgery spheres. The proof then follows directly Lemma \ref{GLcompact}. 

\end{proof}

\begin{Corollary}\label{GLmetricisotopy}
Let $f_t,t\in I$ be an isotopy in the space admissible Morse functions, $\tilde{\M}^{adm}(W)$. Then there is a continuous family of psc-metrics $\bar{g}_t$ on $W$ so that for each $t$, $\bar{g}_t=\bar{g}(g_0,f_t)$ is a Gromov-Lawson cobordism of the type constructed in Theorem \ref{GLcob}. In particular, $\bar{g}_t|_{X_1}, t\in I$ is an isotopy of psc-metrics on $X_1$. 
\end{Corollary}

\begin{Definition}
\rm{
A Morse triple $(f,m,V)$ is {\em well indexed} if the Morse function $f$ is well indexed.
}
\end{Definition}

\begin{Theorem}\cite{Smale}\label{rearrangement}
Let $f\in\tilde{\M}$. Then there is a well-indexed Morse triple $\bar{f}$ which lies in the same path component of $\tilde{\M}$.
\end{Theorem}

\noindent This is basically theorem 4.8 of \cite{Smale}, which proves this fact for Morse functions. We only add that it holds for Morse triples also. 

\begin{proof}
It is sufficient to consider the case where $f$ has exactly two critical points $w$ and $w'$ with $f(w)=c<\frac{1}{2}<f(w')=c'$. The proof of the more general case is exactly the same. Now suppose that $w$ has index $p+1$, $w'$ has index $p'+1$ and $p\geq p'$. Denote by $K_{w}$, the union of trajectories $K_{-}^{p+1}(w)$ and $K_{+}^{q+1}(w)$ associated with $w$. As always, $p+q+1=n$. Similarly $K_{w'}$ will denote the union of trajectories $K_{-}^{p'+1}(w')$ and $D_{+}^{q'+1}(w')$ associated with $w'$ where $p'+q'+1=n$.

\begin{figure}[htbp]

\begin{picture}(0,0)%
\includegraphics{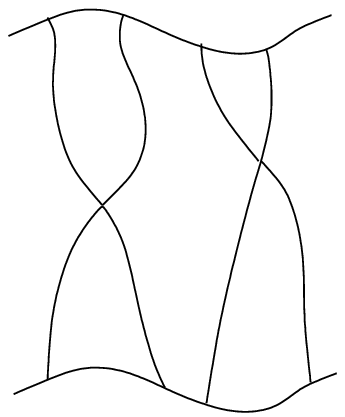}%
\end{picture}%
\setlength{\unitlength}{3947sp}%
\begingroup\makeatletter\ifx\SetFigFont\undefined%
\gdef\SetFigFont#1#2#3#4#5{%
  \reset@font\fontsize{#1}{#2pt}%
  \fontfamily{#3}\fontseries{#4}\fontshape{#5}%
  \selectfont}%
\fi\endgroup%
\begin{picture}(2065,1967)(1636,-2902)
\put(1664,-2836){\makebox(0,0)[lb]{\smash{{\SetFigFont{10}{8}{\rmdefault}{\mddefault}{\updefault}{\color[rgb]{0,0,0}$X_0$}%
}}}}
\put(1651,-1099){\makebox(0,0)[lb]{\smash{{\SetFigFont{10}{8}{\rmdefault}{\mddefault}{\updefault}{\color[rgb]{0,0,0}$X_1$}%
}}}}
\put(3389,-1749){\makebox(0,0)[lb]{\smash{{\SetFigFont{10}{8}{\rmdefault}{\mddefault}{\updefault}{\color[rgb]{0,0,0}$w'$}%
}}}}
\put(2351,-1961){\makebox(0,0)[lb]{\smash{{\SetFigFont{10}{8}{\rmdefault}{\mddefault}{\updefault}{\color[rgb]{0,0,0}$w$}%
}}}}
\end{picture}%

\caption{Non-intersecting trajectories $K_{w}$ and $K_{w'}$}
\label{fig:nonintersectingtrajectories}
\end{figure}

We begin with simpler case when $K_{w}$ and $K_{w'}$ do not intersect. For any $0<a'<a<1$, Theorem 4.1 of \cite{Smale} provides a construction for a well-indexed function $\bar{f}$ with critical points $w$ and $w'$ but with $f(w')=a'$ and $f(w)=a$. The construction can be applied continuously and so replacing $0<a'<a<1$ with a pair of continuous functions $0<a_t'<a_t<1$, with $a_0'=c$, $a_0=c'$, $a_1'=a'$, $a_1=a$ and $t\in I$ results in the desired isotopy.

\begin{figure}[htbp]

\begin{picture}(0,0)%
\includegraphics{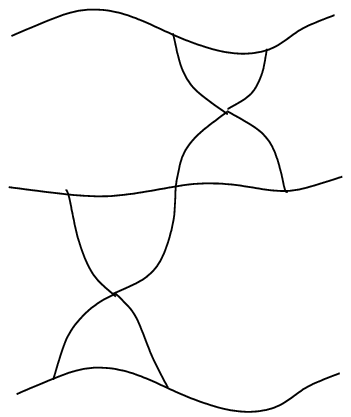}%
\end{picture}%
\setlength{\unitlength}{3947sp}%
\begingroup\makeatletter\ifx\SetFigFont\undefined%
\gdef\SetFigFont#1#2#3#4#5{%
  \reset@font\fontsize{#1}{#2pt}%
  \fontfamily{#3}\fontseries{#4}\fontshape{#5}%
  \selectfont}%
\fi\endgroup%
\begin{picture}(2077,1967)(1636,-2902)
\put(1564,-1836){\makebox(0,0)[lb]{\smash{{\SetFigFont{10}{8}{\rmdefault}{\mddefault}{\updefault}{\color[rgb]{0,0,0}$f^{-1}(\frac{1}{2})$}%
}}}}
\put(1874,-2836){\makebox(0,0)[lb]{\smash{{\SetFigFont{10}{8}{\rmdefault}{\mddefault}{\updefault}{\color[rgb]{0,0,0}$X_0$}%
}}}}
\put(1874,-1099){\makebox(0,0)[lb]{\smash{{\SetFigFont{10}{8}{\rmdefault}{\mddefault}{\updefault}{\color[rgb]{0,0,0}$X_1$}%
}}}}
\put(2864,-1524){\makebox(0,0)[lb]{\smash{{\SetFigFont{10}{8}{\rmdefault}{\mddefault}{\updefault}{\color[rgb]{0,0,0}$w'$}%
}}}}
\put(2326,-2386){\makebox(0,0)[lb]{\smash{{\SetFigFont{10}{8}{\rmdefault}{\mddefault}{\updefault}{\color[rgb]{0,0,0}$w$}%
}}}}
\put(2501,-1764){\makebox(0,0)[lb]{\smash{{\SetFigFont{10}{8}{\rmdefault}{\mddefault}{\updefault}{\color[rgb]{0,0,0}$S^{q}_{+}$}%
}}}}
\put(3114,-2034){\makebox(0,0)[lb]{\smash{{\SetFigFont{12}{14.4}{\rmdefault}{\mddefault}{\updefault}{\color[rgb]{0,0,0}$S^{p'}_{-}$}%
}}}}
\end{picture}%
\caption{Intersecting trajectories}
\label{trajectoryintersectory}
\end{figure}

In general, the trajectory spheres of two distinct Morse critical points may well intersect, see Fig. \ref{trajectoryintersectory}. 
However, provided certain dimension restrictions are satisfied, it is possible to continuously move one trajectory sphere out of the way of the other trajectory sphere. This is theorem 4.4 of \cite{Smale}. We will not reprove it here, except to say that the main technical tool required in the proof is lemma 4.6 of \cite{Smale}, which we state below. 
\begin{Lemma}\cite{Smale}
Suppose $M$ and $N$ are two submanifolds of dimension $m$ and $n$ in a manifold $V$ of dimension $v$. If $M$ has a product neighbourhood in $V$, and $m+n<v$, then there exists a diffeomorphism $h$ of $V$ onto iteslf smoothly isotopic to the identity, such that $h(M)$ is disjoint from $N$.
\end{Lemma}

The following observation then makes theorem 4.4 of \cite{Smale} possible. 
Let $S_{+}^{q}$ and $S_{-}^{p'}$ denote the respective intersections of $f^{-1}(\frac{1}{2})$ with $K_{+}^{q+1}(w)$ and $K_{-}^{p'+1}(w')$. Adding up dimensions, we see that 

\begin{equation*}
\begin{split}
q+p'&=n-p-1+p'\\
&\leq n-p'+1+p'\\
&\leq n-1. 
\end{split}
\end{equation*}

\noindent We can now isotopy $f$ to have disjoint $K_{w}$ and $K_{w'}$ and then proceed as before.

\end{proof}

\begin{Corollary}\label{rearrangementmetric}
Any Gromov-Lawson cobordism $\bar{g}(g_0,f)$ can be isotopied to a Gromov-Lawson cobordism $\bar{g}(g_0,\bar{f})$ which is obtained from a well-indexed admissible Morse function $\bar{f}$.
\end{Corollary}

\begin{proof}
This follows immediately from Theorem \ref{rearrangement} above and Corollary \ref{GLmetricisotopy}. 
\end{proof}

\section{Constructing Gromov-Lawson concordances}\label{GLconcordsection}

Replacing $X_0$ with $X$ and the metric $g_0$ with $g$, we now turn our attention to the case when the cobordism $\{W;X_0,X_1\}$ is the cylinder $X\times I$. By equipping $X\times I$ with an admissible Morse function $f$, we can use Theorem \ref{GLcob} to extend the psc-metric $g$ over $X\times I$ as a Gromov-Lawson cobordism $\bar{g}=\bar{g}(g,f)$. The resulting metric is known as a {\it Gromov-Lawson concordance} or more specifically, a {\it Gromov-Lawson concordance of $g$ with respect to $f$} and the metrics $g_0=\bar{g}|_{X\times\{0\}}$ and $g_1=\bar{g}|_{X\times\{1\}}$ are said to be {\it Gromov-Lawson concordant}.
 
\subsection{Applying the Gromov-Lawson technique over a pair of cancelling surgeries}

In this section, we will construct a Gromov-Lawson concordance on the cylinder $S^{n}\times I$. It is possible to decompose this cylinder into the union of two elementary cobordisms, one the trace of a $p$-surgery, the other the trace of a $(p+1)$-surgery. The second surgery therefore undoes the topological effects of the first surgery. Later in the section, we will show how such a decomposition of the cylinder can be realised by a Morse function with two ``cancelling" critical points. Assuming that $n-p\geq 4$, the standard round metric $ds_{n}^{2}$ can be extended over the union of these cobordisms by the technique of Theorem \ref{GLcob}, resulting in a Gromov-Lawson concordance. To understand this concordance we need to analyse the geometric effects of applying the Gromov-Lawson construction over the two cancelling surgeries.

\begin{Example}\label{sphereeg}
\rm{
Let $S^{n}$ represent the standard smooth $n-$sphere equipped with the round metric $g=ds_{n}^{2}$. We will perform two surgeries, the first a $p-$surgery on $S^{n}$ and the second, a $(p+1)-$surgery on the resulting manifold. The second surgery will have the effect of undoing the topological change made by the first surgery and restoring the original topology of $S^{n}$. Later we will see that the union of the resulting traces will in fact form a cylinder $S^{n}\times I$. 

In section \ref{prelim} we saw that $S^{n}$ can be decomposed as a union of sphere-disk products. Assuming that $p+q+1=n$ we obtain,
\begin{equation*}
\begin{split}
S^{n} &= \p D^{n+1},\\
&=\p (D^{p+1}\times D^{q+1}),\\
&=S^{p}\times D^{q+1}\cup_{S^{p}\times S^{q}} D^{p+1}\times S^{q}.
\end{split}
\end{equation*}
\noindent Here we are are assuming that $q\geq 3$. Let $S^{p}\times {\oD}\hookrightarrow S^{n}$ be the embedding obtained by the inclusion
\begin{equation*}
S^{p}\times {\oD}\hookrightarrow S^{p}\times D^{q+1}\cup_{S^{p}\times S^{q}} D^{p+1}\times S^{q}.
\end{equation*}
\noindent We will now perform a surgery on this embedded $p-$sphere. This involves first removing the embedded $S^{p}\times {\oD}$ to obtain $S^{n}\setminus(S^{p}\times {\oD})=D_{-}^{p+1}\times S^{q}$, and then attaching $(D_{+}^{p+1}\times S^{q})$ along the common boundary $S^{p}\times S^{q}$. The attaching map here is given by restriction of the orginal embedding to the boundary. The resulting manifold is of course the sphere product $S^{p+1}\times S^{q}$ where the disks $D_{-}^{p+1}$ and ${D_{+}}^{p+1}$ are hemi-spheres of the $S^{p+1}$ factor.

By performing a surgery on an embedded $p-$sphere in $S^{n}$ we have obtained a manifold which is diffeomorphic to $S^{p+1}\times S^{q}$. By applying the Gromov-Lawson construction to the metric $g$ we obtain a positive scalar curvature metric $g'$ on $S^{p+1}\times S^{q}$, see Fig. \ref{GLsphere1}. This metric is the original round metric on an $S^{n}\setminus(S^{p}\times D^{q+1})$ piece and is $g_{tor}^{p+1}(\epsilon)\times \delta^{2}ds_{q}^{2}$ on a $D^{p+1}\times S^{q}$ piece for some small $\delta>0$. There is also a piece diffeomorphic to $S^{p}\times S^{q}\times I$ where the metric smoothly transitions between the two forms, the construction of which took up much section \ref{surgerysection}.

\begin{figure}[htbp]

\begin{picture}(0,0)%
\includegraphics{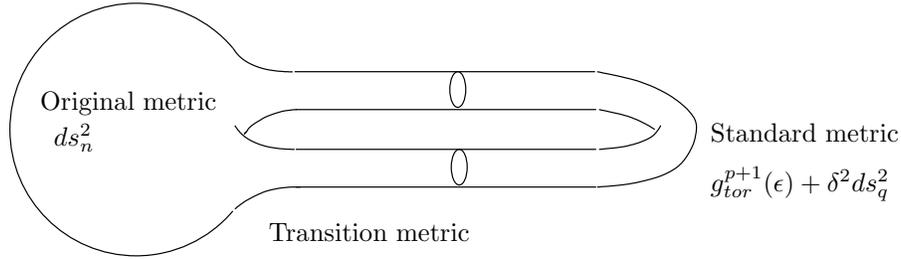}%
\end{picture}%
\setlength{\unitlength}{3947sp}%
\begingroup\makeatletter\ifx\SetFigFont\undefined%
\gdef\SetFigFont#1#2#3#4#5{%
  \reset@font\fontsize{#1}{#2pt}%
  \fontfamily{#3}\fontseries{#4}\fontshape{#5}%
  \selectfont}%
\fi\endgroup%
\begin{picture}(5706,1606)(217,-2927)
\put(1851,-2824){\makebox(0,0)[lb]{\smash{{\SetFigFont{10}{8}{\rmdefault}{\mddefault}{\updefault}{\color[rgb]{0,0,0}Transition metric}%
}}}}
\put(4626,-2199){\makebox(0,0)[lb]{\smash{{\SetFigFont{10}{8}{\rmdefault}{\mddefault}{\updefault}{\color[rgb]{0,0,0}Standard metric}%
}}}}
\put(4626,-2511){\makebox(0,0)[lb]{\smash{{\SetFigFont{10}{8}{\rmdefault}{\mddefault}{\updefault}{\color[rgb]{0,0,0}$g_{tor}^{p+1}(\epsilon)+{\delta}^{2}ds_{q}^{2}$}%
}}}}
\put(414,-1999){\makebox(0,0)[lb]{\smash{{\SetFigFont{10}{8}{\rmdefault}{\mddefault}{\updefault}{\color[rgb]{0,0,0}Original metric}%
}}}}
\put(501,-2224){\makebox(0,0)[lb]{\smash{{\SetFigFont{10}{8}{\rmdefault}{\mddefault}{\updefault}{\color[rgb]{0,0,0}$ds_{n}^{2}$}%
}}}}
\end{picture}%
\caption{The geometric effect of the first surgery}
\label{GLsphere1}
\end{figure}

We will now perform a second surgery, this time on the manifold $(S^{p+1}\times S^{q}, g')$. We wish to obtain a manifold which is diffeomorphic to the orginal $S^{n}$, that is, we wish to undo the $p-$surgery we have just performed. Consider the following decomposition of $S^{p+1}\times S^{q}$.

\begin{equation*}
\begin{split}
S^{p+1}\times S^{q} &= S^{p+1}\times (D_{-}^{q}\cup D_{+}^{q})\\
&=(S^{p+1}\times D_{-}^{q})\cup (S^{p+1}\times D_{+}^{q}).
\end{split}
\end{equation*}

\noindent Again, the inclusion map gives us an embedding of $S^{p+1}\times {{D}_{-}^{q}}$. Removing $S^{p+1}\times {\oDqm}$ and attaching $D^{p+2}\times S^{q-1}$ along the boundary gives

\begin{equation*}
\begin{split}
 D^{p+2}\times S^{q-1}\cup S^{p+1}\times D_{+}^{q}
 &=\p (D^{p+2}\times D^{q})\\
 &=\p (D^{n+1})\\
 &=S^{n}.
\end{split}
\end{equation*}

\begin{figure}[htbp]

\begin{picture}(0,0)%
\includegraphics{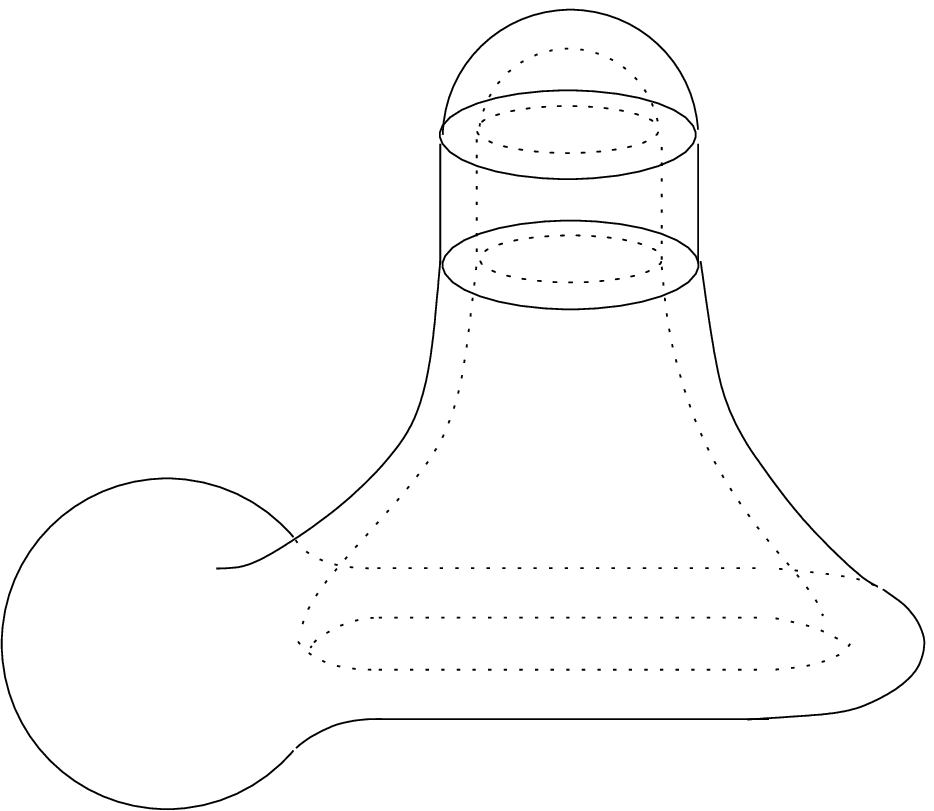}%
\end{picture}%
\setlength{\unitlength}{3947sp}%
\begingroup\makeatletter\ifx\SetFigFont\undefined%
\gdef\SetFigFont#1#2#3#4#5{%
  \reset@font\fontsize{#1}{#2pt}%
  \fontfamily{#3}\fontseries{#4}\fontshape{#5}%
  \selectfont}%
\fi\endgroup%
\begin{picture}(5354,3856)(3287,-6901)
\put(3401,-6074){\makebox(0,0)[lb]{\smash{{\SetFigFont{10}{8}{\rmdefault}{\mddefault}{\updefault}{\color[rgb]{0,0,0}Original metric}%
}}}}
\put(3438,-6349){\makebox(0,0)[lb]{\smash{{\SetFigFont{10}{8}{\rmdefault}{\mddefault}{\updefault}{\color[rgb]{0,0,0}$ds_{n}^{2}$}%
}}}}
\put(5226,-5649){\makebox(0,0)[lb]{\smash{{\SetFigFont{10}{8}{\rmdefault}{\mddefault}{\updefault}{\color[rgb]{0,0,0}Transition metric}%
}}}}
\put(6601,-6699){\makebox(0,0)[lb]{\smash{{\SetFigFont{10}{8}{\rmdefault}{\mddefault}{\updefault}{\color[rgb]{0,0,0}Formerly standard metric}%
}}}}
\put(4014,-3649){\makebox(0,0)[lb]{\smash{{\SetFigFont{10}{8}{\rmdefault}{\mddefault}{\updefault}{\color[rgb]{0,0,0}Standard metric}%
}}}}
\put(4014,-3961){\makebox(0,0)[lb]{\smash{{\SetFigFont{10}{8}{\rmdefault}{\mddefault}{\updefault}{\color[rgb]{0,0,0}$g_{tor}^{p+2}(\epsilon')+{\delta'}^{2}ds_{q-1}^{2}$}%
}}}}
\end{picture}%

\caption{The geometric effect of the second surgery: a different metric on $S^{n}$}
\label{fig:GLsphere}
\end{figure}

Applying the Gromov-Lawson construction to the metric $g'$ with respect to this second surgery produces a metric which looks very different to the orginal round metric on $S^{n}$, see Fig. \ref{fig:GLsphere}. Roughly speaking, $g''$ can be thought of as consisting of four pieces: the original piece where $g''=g$ and which is diffeomorphic to a disk $D^{n}$, the new standard piece where $g''=g_{tor}^{p+2}(\epsilon)+{\delta'}^{2}ds_{q-1}^{2}$ and which is diffeomorphic to $D^{p+2}\times S^{q-1}$, the old standard piece where $g''= g_{tor}^{p+1}(\epsilon')+\delta^{2}ds_{q}^{2}$, this time only on a region diffeomorphic to $D^{p+1}\times D^{q}$ and finally, a transition metric which connects these pieces. Later on we will need to describe this metric in more detail. 
}
\end{Example}

\subsection {Cancelling Morse critical points: The weak and strong cancellation theorems}
We will now show that the cylinder $S^{n}\times I$ can be decomposed into a union of two elementary cobordisms which are the traces of the surgeries described above. This decomposition is obtained from a Morse function $f:S^{n}\times I\rightarrow I$ which satisfies certain properties. The following theorem, known as the {\it weak cancellation theorem} is proved in chapter 5 of \cite{Smale}. It is also discussed, in much greater generality, in chapter 5, section 1 of \cite{Hatcher}.

\begin{Theorem}\cite{Smale}\label{cancelthm}
Let $\{W^{n+1};X_0,X_1\}$ be a smooth compact cobordism and $f:W\rightarrow I$ be a Morse triple on $W$. Letting $p+q+1=n$, suppose that $f$ satisfies the following conditions.
 
\noindent (a) The function $f$ has exactly $2$ critical points $w$ and $z$ and $0<f(w)<c<f(z)<1$. 

\noindent(b) The points $w$ and $z$ have Morse index $p+1$ and $p+2$ respectively. 

\noindent(c) On $f^{-1}(c)$, the trajectory spheres $S_{+}^{p}(w)$ and $S_{-}^{q}(z)$, respectively emerging from the critical point $w$ and converging toward the critical point $z$, intersect transversely as a single point. 

\noindent Then the critical points $w$ and $z$ cancel and $W$ is diffeomorphic to $X_0\times I$.
\end{Theorem}

The proof of \ref{cancelthm} in \cite{Smale} is attributed to Morse. The fact that $S_{+}^{p}(w)$ and $S_{-}^{q}(z)$ intersect transversely as a point means that there is a single {\it trajectory arc} connecting $w$ and $z$. It is possible to alter the vector field $V$ on an arbitrarily small neighbourhood of this arc to obtain a nowhere zero gradient-like vector field $V'$ which agrees with $V$ outside of this neighbourhood. This in turn gives rise to a Morse function $f'$ with gradient-like vector field $V'$, which agrees with $f$ outside this neighbourhood and has {\bf no} critical points, see Fig. \ref{trajectoryarc}.

\begin{figure}[htbp]
\begin{picture}(0,0)%
\includegraphics{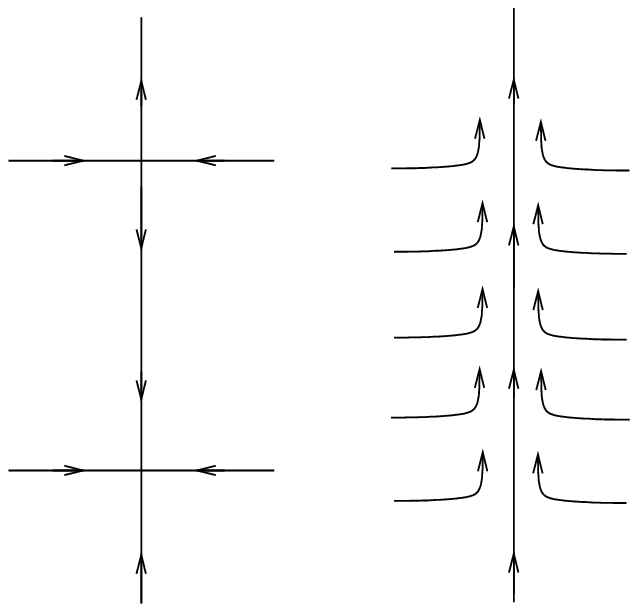}%
\end{picture}%
\setlength{\unitlength}{3947sp}%
\begingroup\makeatletter\ifx\SetFigFont\undefined%
\gdef\SetFigFont#1#2#3#4#5{%
  \reset@font\fontsize{#1}{#2pt}%
  \fontfamily{#3}\fontseries{#4}\fontshape{#5}%
  \selectfont}%
\fi\endgroup%
\begin{picture}(3471,3117)(1411,-5749)
\put(1426,-3661){\makebox(0,0)[lb]{\smash{{\SetFigFont{10}{8}{\rmdefault}{\mddefault}{\updefault}{\color[rgb]{0,0,0}$\mathbb{R}^{p+1}$}%
}}}}
\put(2489,-2786){\makebox(0,0)[lb]{\smash{{\SetFigFont{10}{8}{\rmdefault}{\mddefault}{\updefault}{\color[rgb]{0,0,0}$\mathbb{R}^{q+1}$}%
}}}}
\end{picture}%

\caption{Altering the gradient-like vector field along the trajectory arc}
\label{trajectoryarc}
\end{figure}

\noindent The desired decomposition of $S^{n}\times I$ can now be realised by a Morse function $f:S^{n}\times I\rightarrow I$ which satisfies (a), (b) and (c) above as well as the condition that $n-p\geq 4$. Application of Theorem \ref{GLcob} with respect to an admissible Morse function $f$ which satisfies (a), (b) and (c) will result in a Gromov-Lawson concordance on $S^{n}\times I$ between $g=ds_{n}^{2}$ and the metric $g''$ described above. Equivalently, one can think of this as obtained by two applications of Theorem \ref{ImprovedsurgeryTheorem}, one for each of the elementary cobordisms specified by $f$.  

\subsection{A strengthening of Theorem \ref{cancelthm}}
There is a strengthening of theorem \ref{cancelthm} in the case where $W,X_0$ and $X_1$ are simply connected. Before stating it, we should recall what is meant by the intersection number of two manifolds. Let $M$ and $M'$ be two smooth submanifolds of dimensions $r$ and $s$ in a smooth manifold $N$ of dimension $r+s$ and suppose that $M$ and $M'$ intersect transversely as the set of points $\{n_1,n_2,\cdots,n_l\}$ in $N$. Suppose also that $M$ is oriented and that the normal bundle $\N(M')$ of $M'$ in $N$ is oriented. At $n_i$, choose a positively oriented $r-$frame $v_1,\cdots ,v_r$ of linearly independent vectors which span the tangent space $T_{n_i}M$. Since the intersection at $n_i$ is transverse, this frame is a basis for the normal fibre $\N_{n_i}(M')$.

\begin{Definition}
{\rm The {\em intersection number of $M$ and $M'$ at $n_i$} is defined to be $+1$ or $-1$ according as the vectors $v_1,\cdots,v_r$ represent a positively or negatively oriented basis for the fibre $\N_{n_i}(M')$. The {\em intersection number $M'.M$ of $M$ and $M'$} is the sum of intersection numbers over all points $n_i$.} 
\end{Definition}

\begin{Remark}
In the expression $M'.M$, we adopt the convention that the manifold with oriented normal bundle is written first. 
\end{Remark}

We now state the {\it strong cancellation theorem}. This is theorem 6.4 of \cite{Smale}.

\begin{Theorem}\cite{Smale}\label{cancelthm2}
Let $\{W^{n+1};X_0,X_1\}$ be a smooth compact cobordism where $W,X_0$ and $X_1$ are simply connected manifolds. Let $f:W\rightarrow I$ be a Morse triple on $W$. Letting $p+q+1=n$, suppose that $f$ satisfies the following conditions.
 
\noindent (a') The function $f$ has exactly $2$ critical points $w$ and $z$ and $0<f(w)<c<f(z)<1$. 

\noindent(b') The points $w$ and $z$ have Morse index $p+1$ and $p+2$ respectively and $1\leq p\leq n-4$. 

\noindent(c') On $f^{-1}(c)$, the trajectory spheres $S_{+}^{p}(w)$ and $S_{-}^{q}(z)$ have intersection number $S_{+}^{p}(w).S_{-}^{q}(z)=1$ or $-1$.

\noindent Then the critical points $w$ and $z$ cancel and $W$ is diffeomorphic to $X_0\times I$. In fact, $f$ can be altered near $f^{-1}(c)$ so that the trajectory spheres intersect transversely at a single point and the conclusions of theorem \ref{cancelthm} then apply.
\end{Theorem}
   
Simple connectivity plays an important role in the proof. It of course guarantees the orientability conditions we need but more importantly it is used to simplify the intersection of the trajectory spheres. Roughly speaking, if $n_1$ and $n_2$ are two intersection points with opposite intersection, there are arcs connecting these points in each of the trajectory spheres, whose union forms a loop contractible in $f^{-1}(c)$ which misses all other intersection points. An isotopy can be constructed (which involves contracting this loop) to move the trajectory spheres to a position where the intersection set contains no new elements but now excludes $n_1$ and $n_2$. 

\begin{Remark}
The hypothesis that critical points of $f$ have index at least 2 is necessary, as the presence of index $1$ critical points would spoil the assumption of simple connectivity.
\end{Remark}

\subsection{Standardising the embedding of the second surgery sphere}\label{stanemb}
In Example \ref{sphereeg}, the second surgery sphere $S^{p+1}$ was regarded as the union of two hemi-spheres $D_{-}^{p+1}$ and $D_{+}^{p+1}$, the latter hemi-sphere coming from the handle attachment. It was assumed in the construction of the metric $g''$, that the disk $D_{+}^{p+1}$ was embedded so that the metric induced by $g'$ was precisely the $g_{tor}^{p+1}(\epsilon)$ factor of the handle metric. Now let $f$ be an admissible Morse function on $X\times I$ which satisfies conditions (a), (b) and (c) above. This specifies two trajectory spheres $S_{-}^{p}$ and $S_{-}^{p+1}$ corresponding to the critical points $w$ and $z$ respectively. On the level set $f^{-1}(c)$, the spheres $S_{+}^{q}$ and $S_{-}^{p+1}$ intersect transversely at a single point $\alpha$. Now suppose we extend a psc-metric $g$ on $X$ over $f^{-1}[0,c]$ in the manner of Theorem \ref{GLcob}, denoting by $g'$ the induced metric on $f^{-1}(c)$. In general, the metric induced by $g'$ on $S_{-}^{p+1}$ near $\alpha$ will not be $g_{tor}^{p+1}(\epsilon)$. We will now show that such a metric can be obtained with a minor adjustment of the Morse function $f$.

Let $\mathbb{R}^{n+1}=\mathbb{R}^{p+1}\times \mathbb{R}^{q+1}$ denote the Morse coordinate neighbourhood near $w$. Here $\mathbb{R}^{p+1}$ and $\mathbb{R}^{q+1}$ denote the respective inward and outward trajectories at $w$. Let $\mathbb{R}$ denote the $1$-dimensional subspace of $\mathbb{R}^{q+1}$ spanned by the vector based at zero and ending at $\alpha$. Finally, let $D^{p+1}$ denote the intersection of $f^{-1}(c)$ with the plane $\mathbb{R}\times\mathbb{R}^{p+1}$. The metric induced by $g'$ on $D^{p+1}$ is precisely the $g_{tor}^{p+1}(\epsilon)$ factor of the handle metric.   

\begin{figure}[htbp]

\begin{picture}(0,0)%
\includegraphics{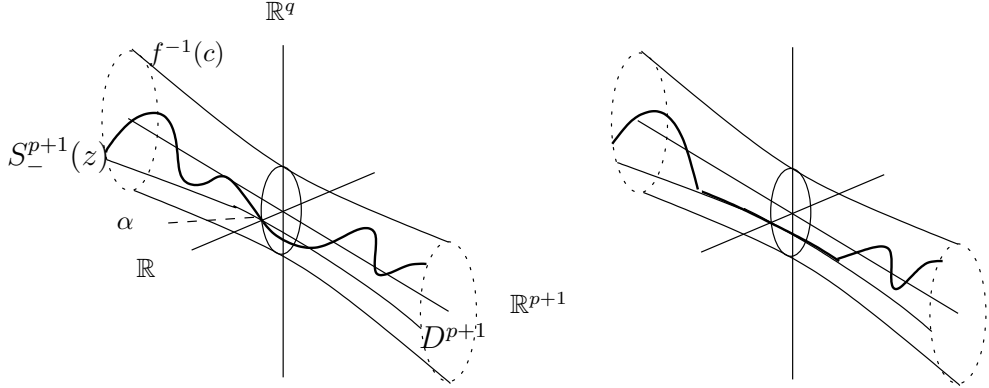}%
\end{picture}%
\setlength{\unitlength}{3947sp}%
\begingroup\makeatletter\ifx\SetFigFont\undefined%
\gdef\SetFigFont#1#2#3#4#5{%
  \reset@font\fontsize{#1}{#2pt}%
  \fontfamily{#3}\fontseries{#4}\fontshape{#5}%
  \selectfont}%
\fi\endgroup%
\begin{picture}(5820,2460)(1561,-3605)
\put(2239,-1556){\makebox(0,0)[lb]{\smash{{\SetFigFont{10}{8}{\rmdefault}{\mddefault}{\updefault}{\color[rgb]{0,0,0}$f^{-1}(c)$}%
}}}}
\put(2039,-2611){\makebox(0,0)[lb]{\smash{{\SetFigFont{10}{8}{\rmdefault}{\mddefault}{\updefault}{\color[rgb]{0,0,0}$\alpha$}%
}}}}
\put(4509,-3141){\makebox(0,0)[lb]{\smash{{\SetFigFont{10}{8}{\rmdefault}{\mddefault}{\updefault}{\color[rgb]{0,0,0}$\mathbb{R}^{p+1}$}%
}}}}
\put(2164,-2911){\makebox(0,0)[lb]{\smash{{\SetFigFont{10}{8}{\rmdefault}{\mddefault}{\updefault}{\color[rgb]{0,0,0}$\mathbb{R}$}%
}}}}
\put(2976,-1299){\makebox(0,0)[lb]{\smash{{\SetFigFont{10}{8}{\rmdefault}{\mddefault}{\updefault}{\color[rgb]{0,0,0}$\mathbb{R}^{q}$}%
}}}}
\put(3939,-3361){\makebox(0,0)[lb]{\smash{{\SetFigFont{12}{14.4}{\rmdefault}{\mddefault}{\updefault}{\color[rgb]{0,0,0}$D^{p+1}$}%
}}}}
\put(1351,-2211){\makebox(0,0)[lb]{\smash{{\SetFigFont{12}{14.4}{\rmdefault}{\mddefault}{\updefault}{\color[rgb]{0,0,0}$S_{-}^{p+1}(z)$}%
}}}}
\end{picture}%

\caption{Isotopying $S_{-}^{p+1}$ near $\alpha$ to coincide with $D^{p+1}$.}
\label{adjustsphere}
\end{figure}

\begin{Lemma}\label{neatembedding}
It is possible to isotopy the trajectory sphere $S_{-}^{p+1}(z)$, so that on $f^{-1}(c)$ it agrees, near $\alpha$, with the disk $D^{p+1}$. 
\end{Lemma}

\begin{proof}
Choose a coordinate chart $\mathbb{R}^{n}$ in $f^{-1}(c)$ around $\alpha$, where $\alpha$ is identified with $0$ and $\mathbb{R}^{n}$ decomposes as $\mathbb{R}^{p+1}\times \mathbb{R}^{q}$. The intersection of $S_{-}^{p+1}(z)$ with this chart is a $(p+1)$-dimensional disk in $\mathbb{R}^{n}$ which intersects with $\mathbb{R}^{q}$ transversely at the orgin. Thus, near the orgin, $S_{-}^{p+1}(z)$ is the graph of a function over $\mathbb{R}^{p+1}$ and so we can isotopy it to an embedding which is the plane $\mathbb{R}^{p+1}$ on some neighbourhood of $0$, and the original $S_{-}^{p+1}(z)$ away from this neighbourhood.
\end{proof}

\noindent Thus, the Morse function $f$ can be isotopied to a Morse function where the standard part of the metric $g'$ induces the $g_{tor}^{p+1}(\epsilon)$ factor of the handle metric on $S_{-}^{p+1}(\alpha)$, see Fig. \ref{geomissue}.

\begin{figure}[htbp]

\begin{picture}(0,0)%
\includegraphics{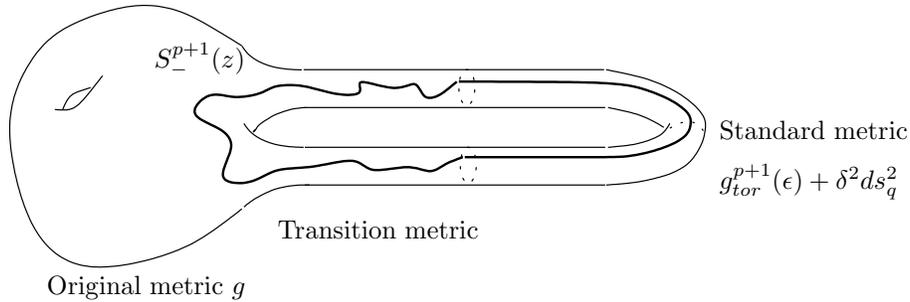}%
\end{picture}%
\setlength{\unitlength}{3947sp}%
\begingroup\makeatletter\ifx\SetFigFont\undefined%
\gdef\SetFigFont#1#2#3#4#5{%
  \reset@font\fontsize{#1}{#2pt}%
  \fontfamily{#3}\fontseries{#4}\fontshape{#5}%
  \selectfont}%
\fi\endgroup%
\begin{picture}(5767,1895)(156,-3240)
\put(1851,-2824){\makebox(0,0)[lb]{\smash{{\SetFigFont{10}{8}{\rmdefault}{\mddefault}{\updefault}{\color[rgb]{0,0,0}Transition metric}%
}}}}
\put(4626,-2511){\makebox(0,0)[lb]{\smash{{\SetFigFont{10}{8}{\rmdefault}{\mddefault}{\updefault}{\color[rgb]{0,0,0}$g_{tor}^{p+1}(\epsilon)+\delta^{2}ds_{q}^{2}$}%
}}}}
\put(4626,-2199){\makebox(0,0)[lb]{\smash{{\SetFigFont{10}{8}{\rmdefault}{\mddefault}{\updefault}{\color[rgb]{0,0,0}Standard metric}%
}}}}
\put(401,-3174){\makebox(0,0)[lb]{\smash{{\SetFigFont{10}{8}{\rmdefault}{\mddefault}{\updefault}{\color[rgb]{0,0,0}Original metric $g$}%
}}}}
\put(1076,-1749){\makebox(0,0)[lb]{\smash{{\SetFigFont{10}{8}{\rmdefault}{\mddefault}{\updefault}{\color[rgb]{0,0,0}$S_{-}^{p+1}(z)$}%
}}}}
\end{picture}%

\caption{The embedded sphere $S_{-}^{p+1}(z)$ after adjustment}
\label{geomissue}
\end{figure}

\section{Gromov-Lawson concordance implies isotopy in the case of two cancelling surgeries}\label{doublesurgsection}

We continue to employ the notation of the previous section in stating the following theorem.

\begin{Theorem}\label{concisodoublesurgery}
Let $f:X\times I\rightarrow I$ be an admissible Morse function on $X$ which satisfies conditions (a),(b) and (c) of Theorem \ref{cancelthm} above. Let $g$ be a metric of positive scalar curvature on $X$ and $\bar{g}=\bar{g}(g,f)$, a Gromov-Lawson concordance with respect to $f$ and $g$ on $X\times I$. Then the metric $g''=\bar{g}|_{X\times\{1\}}$ on $X$ is isotopic to the original metric $g$. 
\end{Theorem}

\noindent We will postpone the proof of Theorem \ref{concisodoublesurgery} for now. Later we will show that this theorem contains the main geometric argument needed to prove that any metrics which are Gromov-Lawson concordant are actually isotopic. Before doing this, we need to introduce some more terminology.

\subsection{Connected sums of psc-metrics}
 Suppose $(X,g_X)$ and $(Y, g_Y)$ are Riemannian manifolds of positive scalar curvature with dim X=dim Y$\geq 3$. A {\it psc-metric connected sum} of $g_X$ and $g_Y$ is a positive scalar curvature metric on the connected sum $X \# Y$, obtained using the Gromov-Lawson technique for connected sums on $g_X$ and $g_Y$. Recall that on $X$, the Gromov-Lawson technique involves modifying the metric on a disk $D=D^{n}$ around some point $w\in X$, by pushing out geodesic spheres around $w$ to form a tube. It is possible to construct this tube so that the metric on it has positive scalar curvature and so that it ends as a Riemannian cylinder $S^{n-1}\times I$. Furthermore the metric induced on the $S^{n-1}$ factor can be chosen to be arbitrarily close to the standard round metric and so we can isotopy this metric to the round one. By Lemma \ref{isotopyimpliesconc}, we obtain a metric on $X\setminus D^{n}$ which has positive scalar curvature and which near the boundary is the standard product $\delta^{2}ds_{n-1}^{2}+dt^{2}$ for some (possibly very small) $\delta$. Repeating this procedure on $Y$ allows us to form a psc-metric connected sum of $(X,g_X)$ and $(Y, g_{Y})$ which we denote

\begin{equation*}
(X,g_X)\#(Y, g_{Y}).
\end{equation*}

\subsection{An analysis of the metric $g''$, obtained from the second surgery}

Recall that $f$ specifies a pair of cancelling surgeries. The first surgery is on an embedded sphere $S^{p}$ and we denote the resulting surgery manifold by $X'$. Applying the Gromov-Lawson construction results in a metric $g'$ on $X'$, which is the orginal metric $g$ away from $S^{p}$ and transitions on a region diffeomorphic to $S^{p}\times S^{q}\times I$ to a metric which is the standard product $g_{tor}^{p+1}(\epsilon)\times \delta^{2}ds_{q}^{2}$ on the handle $D^{p+1}\times S^{q}$. The second surgery sphere, embedded in $X'$, is denoted $S^{p+1}$. In section \ref{stanemb}, we showed that it is reasonable to assume that on the standard region, the restriction of $g'$ to the sphere $S^{p+1}$ is precisely the $g_{tor}^{p+1}(\epsilon)$ factor of the standard metric $g_{tor}^{p+1}(\epsilon)+\delta^{2}ds_q^{2}$, see Fig. \ref{geomissue}. Applying the Gromov-Lawson construction to $g'$ with respect to this second surgery, results in a metric $g''$ on $X$ which is concordant to $g$. 
\\

\begin{figure}[htbp]

\begin{picture}(0,0)%
\includegraphics{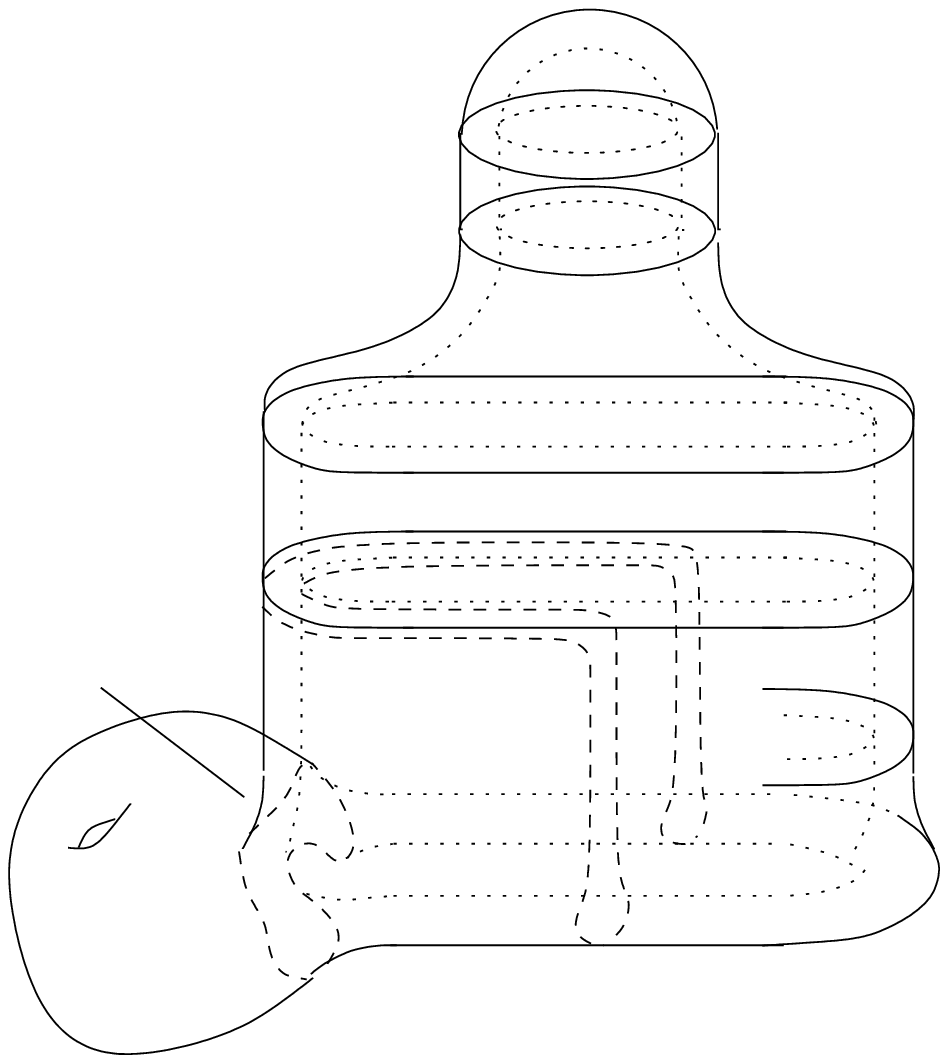}%
\end{picture}%
\setlength{\unitlength}{3947sp}%
\begingroup\makeatletter\ifx\SetFigFont\undefined%
\gdef\SetFigFont#1#2#3#4#5{%
  \reset@font\fontsize{#1}{#2pt}%
  \fontfamily{#3}\fontseries{#4}\fontshape{#5}%
  \selectfont}%
\fi\endgroup%
\begin{picture}(6861,5036)(361,-6870)
\put(1114,-6111){\makebox(0,0)[lb]{\smash{{\SetFigFont{10}{8}{\rmdefault}{\mddefault}{\updefault}{\color[rgb]{0,0,0}Original metric}%
}}}}
\put(4639,-6211){\makebox(0,0)[lb]{\smash{{\SetFigFont{10}{8}{\rmdefault}{\mddefault}{\updefault}{\color[rgb]{0,0,0}Old standard metric}%
}}}}
\put(5514,-4974){\makebox(0,0)[lb]{\smash{{\SetFigFont{10}{8}{\rmdefault}{\mddefault}{\updefault}{\color[rgb]{0,0,0}$ds^{2}+g_{tor}^{p+1}(\epsilon)+\delta^{2}ds_{q-1}^{2}$}%
}}}}
\put(5514,-4261){\makebox(0,0)[lb]{\smash{{\SetFigFont{10}{8}{\rmdefault}{\mddefault}{\updefault}{\color[rgb]{0,0,0}$ds^{2}+g_{Dtor}^{p+1}(\epsilon)+\delta^{2}ds_{q-1}^{2}$}%
}}}}
\put(1851,-2274){\makebox(0,0)[lb]{\smash{{\SetFigFont{10}{8}{\rmdefault}{\mddefault}{\updefault}{\color[rgb]{0,0,0}Standard metric}%
}}}}
\put(2526,-6211){\makebox(0,0)[lb]{\smash{{\SetFigFont{10}{8}{\rmdefault}{\mddefault}{\updefault}{\color[rgb]{0,0,0}Old transition metric}%
}}}}
\put(5476,-5724){\makebox(0,0)[lb]{\smash{{\SetFigFont{10}{8}{\rmdefault}{\mddefault}{\updefault}{\color[rgb]{0,0,0}Easy transition metric}%
}}}}
\put(1639,-2649){\makebox(0,0)[lb]{\smash{{\SetFigFont{10}{8}{\rmdefault}{\mddefault}{\updefault}{\color[rgb]{0,0,0}$g_{tor}^{p+2}(\epsilon)+\delta^{2}ds_{q-1}^{2}$}%
}}}}
\put(376,-4800){\makebox(0,0)[lb]{\smash{{\SetFigFont{10}{8}{\rmdefault}{\mddefault}{\updefault}{\color[rgb]{0,0,0}New transition metric}%
}}}}
\put(1489,-6999){\makebox(0,0)[lb]{\smash{{\SetFigFont{10}{8}{\rmdefault}{\mddefault}{\updefault}{\color[rgb]{0,0,0}$X\setminus D$}%
}}}}
\put(3889,-6899){\makebox(0,0)[lb]{\smash{{\SetFigFont{10}{8}{\rmdefault}{\mddefault}{\updefault}{\color[rgb]{0,0,0}$D$}%
}}}}

\end{picture}%

\caption{The metric $g''$}
\label{adjg''}
\end{figure}

In Fig. \ref{adjg''}, we describe a metric which is obtained from $g''$ by a only a very minor adjustment. We will discuss the actual adjustment a little later but, as it can be obtained through an isotopy, to ease the burden of notation we will still denote the metric by $g''$. This metric agrees with $g$ outside of an $n$-dimensional disk $D$, see Fig. \ref{adjg''}. The restriction of $g''$ to this disk can be thought of as consisting of several regions. Near the boundary of the disk, and represented schematically by two dashed curves, is a cylindrical region which is diffeomorphic to $S^{n-1}\times I$. This cylindrical region will be known as the {\em connecting cylinder}, see Fig. \ref{transgg}. We will identify the sphere which bounds $X\setminus D$ with $S^{n-1}\times \{1\}$. This sphere is contained in a region where $g''=g$ and so we know very little about the induced metric on this sphere. 

\begin{figure}[htbp]

\begin{picture}(0,0)%
\includegraphics{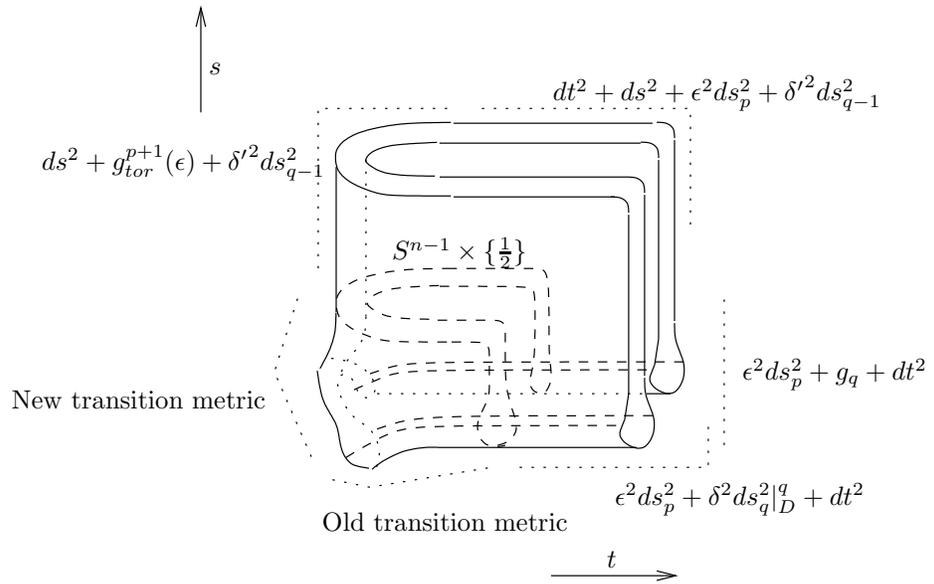}%
\end{picture}%
\setlength{\unitlength}{3947sp}%
\begingroup\makeatletter\ifx\SetFigFont\undefined%
\gdef\SetFigFont#1#2#3#4#5{%
  \reset@font\fontsize{#1}{#2pt}%
  \fontfamily{#3}\fontseries{#4}\fontshape{#5}%
  \selectfont}%
\fi\endgroup%
\begin{picture}(4617,3641)(349,-3490)
\put(4151,-3004){\makebox(0,0)[lb]{\smash{{\SetFigFont{10}{8}{\rmdefault}{\mddefault}{\updefault}{\color[rgb]{0,0,0}$\epsilon^{2}ds_{p}^{2}+\delta^{2}ds_{q}^{2}|_D^{q}+dt^{2}$}%
}}}}
\put(4951,-2217){\makebox(0,0)[lb]{\smash{{\SetFigFont{10}{8}{\rmdefault}{\mddefault}{\updefault}{\color[rgb]{0,0,0}$\epsilon^{2}ds_{p}^{2}+g_q+dt^{2}$}%
}}}}
\put(2313,-3167){\makebox(0,0)[lb]{\smash{{\SetFigFont{10}{8}{\rmdefault}{\mddefault}{\updefault}{\color[rgb]{0,0,0}Old transition metric}%
}}}}
\put(551,-898){\makebox(0,0)[lb]{\smash{{\SetFigFont{10}{8}{\rmdefault}{\mddefault}{\updefault}{\color[rgb]{0,0,0}$ds^{2}+g_{tor}^{p+1}(\epsilon)+{\delta'}^{2}ds_{q-1}^{2}$}%
}}}}
\put(2751,-1461){\makebox(0,0)[lb]{\smash{{\SetFigFont{10}{8}{\rmdefault}{\mddefault}{\updefault}{\color[rgb]{0,0,0}$S^{n-1}\times\{\frac{1}{2}\}$}%
}}}}
\put(1601,-298){\makebox(0,0)[lb]{\smash{{\SetFigFont{10}{8}{\rmdefault}{\mddefault}{\updefault}{\color[rgb]{0,0,0}$s$}%
}}}}
\put(4101,-3398){\makebox(0,0)[lb]{\smash{{\SetFigFont{10}{8}{\rmdefault}{\mddefault}{\updefault}{\color[rgb]{0,0,0}$t$}%
}}}}
\put(3764,-461){\makebox(0,0)[lb]{\smash{{\SetFigFont{10}{8}{\rmdefault}{\mddefault}{\updefault}{\color[rgb]{0,0,0}$dt^{2}+ds^{2}+\epsilon^{2}ds_{p}^{2}+{\delta'}^{2}ds_{q-1}^{2}$}%
}}}}
\put(364,-2399){\makebox(0,0)[lb]{\smash{{\SetFigFont{10}{14.4}{\rmdefault}{\mddefault}{\updefault}{\color[rgb]{0,0,0}New transition metric}%
}}}}
\end{picture}%

\caption{The connecting cylinder $S^{n-1}\times I$}
\label{transgg}
\end{figure} 

The region $S^{n-1}\times[\frac{1}{2}, 1]$ is where most of the transitioning happens from the old metric $g$ to the standard form. This transition metric consists in part of the {\em old transition metric} from the first surgery and the {\em new transition metric} from the second surgery. The old transition metric is on a region which is diffeomorphic to $S^{p}\times D^{q}\times [\frac{1}{2}, 1]$ (schematically this is the region below the horizontal dashed lines near the bottom of Fig. \ref{transgg}) while the new transition metric is on a region which is diffeomorphic to $D^{p+1}\times S^{q-1}\times [\frac{1}{2}, 1]$. On the second cylindrical piece $S^{n-1}\times [0, \frac{1}{2}]$, the metric $g''$ is much closer to being standard.  

Turning our attention away from the connecting cylinder for a moment, it is clear that the metric $g''$ agrees with the standard part of the metric $g'$ on a region which is diffeomorphic to $D^{p+1}\times D^{q}$, see Fig. \ref{adjg''}. Here $g''$ is the metric $g_{tor}^{p+1}(\epsilon)+\delta^{2}ds_{q}^{2}|_{D^{q}}$ and we call this piece the {\em old standard metric}. The old standard metric transitions through an {\em easy transition metric} on a region diffeomorphic to $I\times D^{p+1}\times S^{q}$ to take the form $ds^{2}+g_{tor}^{p+1}(\epsilon')+{\delta'}^{2}ds_{q-1}^{2}$. This particular transition is known as the {\em easy transition metric} as it is far simpler than the previous transitions. 

Returning now to the second cylindrical piece of the connecting cylinder, we see that there is a neighbourhood of $S^{n-1}\times [0,\frac{1}{2}]$, containing both the old standard and easy transition regions where the metric $g''$ takes the form of a product $\epsilon^{2}ds_{p}^{2}+dt^{2}+g_q$, where the metric $g_q$ is a metric on the disk $D^{q}$, see Fig. \ref{firstalteration}. Shortly we will write represent $g_q$ more explicitly. 

\begin{figure}[htbp]
\hspace*{-40mm}
\begin{picture}(0,0)%
\includegraphics{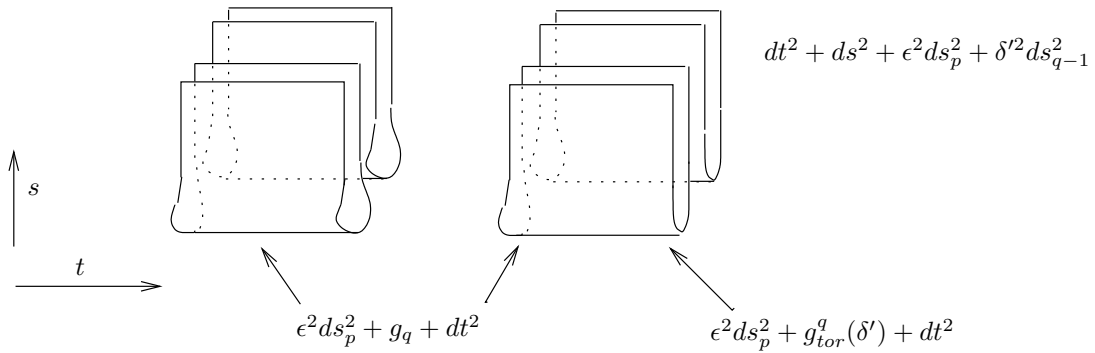}%
\end{picture}%
\setlength{\unitlength}{3947sp}%
\begingroup\makeatletter\ifx\SetFigFont\undefined%
\gdef\SetFigFont#1#2#3#4#5{%
  \reset@font\fontsize{#1}{#2pt}%
  \fontfamily{#3}\fontseries{#4}\fontshape{#5}%
  \selectfont}%
\fi\endgroup%
\begin{picture}(4769,2171)(772,-4102)
\put(5526,-2274){\makebox(0,0)[lb]{\smash{{\SetFigFont{10}{8}{\rmdefault}{\mddefault}{\updefault}{\color[rgb]{0,0,0}$dt^{2}+ds^{2}+\epsilon^{2}ds_{p}^{2}+\delta'^{2}ds_{q-1}^{2}$}%
}}}}
\put(2586,-4011){\makebox(0,0)[lb]{\smash{{\SetFigFont{10}{8}{\rmdefault}{\mddefault}{\updefault}{\color[rgb]{0,0,0}$\epsilon^{2}ds_{p}^{2}+g_q+dt^{2}$}%
}}}}
\put(5189,-4036){\makebox(0,0)[lb]{\smash{{\SetFigFont{10}{8}{\rmdefault}{\mddefault}{\updefault}{\color[rgb]{0,0,0}$\epsilon^{2}ds_{p}^{2}+g_{tor}^{q}(\delta')+dt^{2}$}%
}}}}
\put(901,-3124){\makebox(0,0)[lb]{\smash{{\SetFigFont{10}{8}{\rmdefault}{\mddefault}{\updefault}{\color[rgb]{0,0,0}$s$}%
}}}}
\put(1201,-3624){\makebox(0,0)[lb]{\smash{{\SetFigFont{10}{8}{\rmdefault}{\mddefault}{\updefault}{\color[rgb]{0,0,0}$t$}%
}}}}
\end{picture}%

\caption{A neighbourhood of $S^{n-1}\times [0,\frac{1}{2}]$ on which
  $g''$ has a product structure (left) and the metric resulting from
  an isotopy on this neighbourhood (right)}
\label{firstalteration}
\end{figure}

Returning once more to the metric $g''$ on $D$, we observe that outside of $S^{n-1}\times I$ and away from the old standard and easy transition regions, the metric is almost completely standard. The only difference between this metric and the metric $g''$ constructed in Example \ref{sphereeg} is the fact that the metric on the second surgery sphere $S^{p+1}$ is first isotopied to the double torpedo metric $g_{Dtor}^{p+1}(\epsilon)$ before finally transitioning to the round metric $\epsilon^{2}ds_{p+1}^{2}$. This gives a concordance between the metric $g_{Dtor}^{p+1}+\delta'^{2}ds_{q-1}^{2}$ and $\epsilon^{2}ds_{p+1}^{2}+\delta'^{2}ds_{q-1}^{2}$ which is capped off on the remaining $D^{p+2}\times S^{q-1}$ by the {\em new standard metric} $g_{tor}^{p+2}(\epsilon)+\delta'^{2}ds_{q-1}^{2}$. This completes our initial analysis of $g''$.

\subsection{The proof of Theorem \ref{concisodoublesurgery}}
\begin{proof}
We will perform a sequence of adjustments on each of the metrics $g$ and $g''$. Beginning with the metric $g''$, we will construct $g_1''$ and $g_2''$ each of which is isotopic to the previous one. Similarly, we will construct isotopic metrics $g_1$, $g_2$ and $g_3$. The metrics $g_3$ and $g_2''$ will then be demonstrably isotopic.
\\
 
\noindent {\bf The initial adjustment of $g''$.} We will begin by making some minor adjustments to the metric $g''$ to obtain the metric $g_1''$.  Recall again that on the part of the connecting cylinder identified with $S^{n-1}\times [0,\frac{1}{2}]$, the metric $g''$ is somewhat standard. We observed that on a particular region of $S^{n-1}\times [0,\frac{1}{2}]$, $g''$ takes the form $\epsilon^{2}ds_{p}^{2}+g_q+dt^{2}$. Here $g_q$ can be written more explicitly as

\begin{equation*}
g_q=dr^{2}+F(r)^{2}ds_{q-1}^{2},
\end{equation*}

\noindent where $r$ is the radial distance cordinate and $F$ is a function of the type shown in Fig. \ref{adjfunc}.

\begin{figure}[htbp]
\begin{picture}(0,0)%
\includegraphics{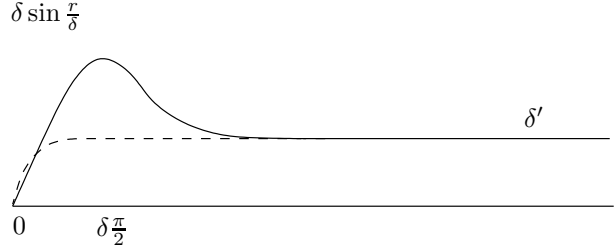}%
\end{picture}%
\setlength{\unitlength}{3947sp}%
\begingroup\makeatletter\ifx\SetFigFont\undefined%
\gdef\SetFigFont#1#2#3#4#5{%
  \reset@font\fontsize{#1}{#2pt}%
  \fontfamily{#3}\fontseries{#4}\fontshape{#5}%
  \selectfont}%
\fi\endgroup%
\begin{picture}(3815,1657)(2661,-3602)
\put(2689,-3536){\makebox(0,0)[lb]{\smash{{\SetFigFont{10}{8}{\rmdefault}{\mddefault}{\updefault}{\color[rgb]{0,0,0}$0$}%
}}}}
\put(2676,-2199){\makebox(0,0)[lb]{\smash{{\SetFigFont{10}{8}{\rmdefault}{\mddefault}{\updefault}{\color[rgb]{0,0,0}$\delta \sin{\frac{r}{\delta}}$}%
}}}}
\put(5901,-2861){\makebox(0,0)[lb]{\smash{{\SetFigFont{10}{8}{\rmdefault}{\mddefault}{\updefault}{\color[rgb]{0,0,0}$\delta'$}%
}}}}
\put(3214,-3549){\makebox(0,0)[lb]{\smash{{\SetFigFont{10}{8}{\rmdefault}{\mddefault}{\updefault}{\color[rgb]{0,0,0}$\delta\frac{\pi}{2}$}%
}}}}
\end{picture}

\caption{The function $F$, with $f_{\delta'}$ shown as the dashed curve}
\label{adjfunc}
\end{figure}

\noindent A linear homotopy of $F$ to the torpedo function $f_{\delta'}$ induces an isotopy from the metric $g_q$ to the metric $g_{tor}^{q}(\delta')$. With an appropriate rescaling, it is possible to isotopy the metric $\epsilon^{2}ds_{p}^{2}+g_q+dt^{2}$ to one which is unchanged near $S^{n-1}\times\{\frac{1}{2}\}$ but near $S^{n-1}\times\{0\}$, is the standard product $\epsilon^{2}ds_{p}^{2}+g_{tor}^{q}(\delta')+dt^{2}$. This isotopy then easily extends to an isotopy of $g''$ resulting in a metric which, on the old standard and easy transition regions, is now $g_{tor}^{p+1}(\epsilon)+g_{tor}^{q}(\delta')$ away from $S^{n-1}\times \{\frac{1}{2}\}$, see Fig. \ref{generalgg}.
\\

\begin{figure}
\hspace*{-40mm}
\begin{picture}(0,0)%
\includegraphics{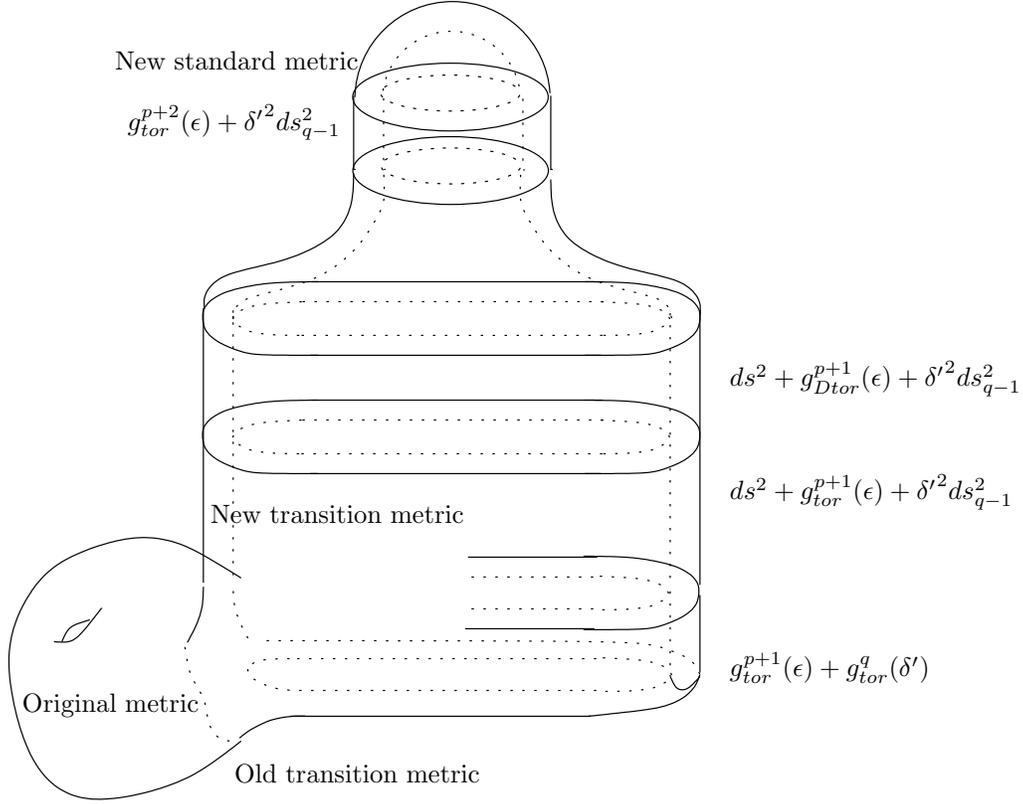}%
\end{picture}%
\setlength{\unitlength}{3947sp}%
\begingroup\makeatletter\ifx\SetFigFont\undefined%
\gdef\SetFigFont#1#2#3#4#5{%
  \reset@font\fontsize{#1}{#2pt}%
  \fontfamily{#3}\fontseries{#4}\fontshape{#5}%
  \selectfont}%
\fi\endgroup%
\begin{picture}(4555,5036)(974,-6870)
\put(1070,-6311){\makebox(0,0)[lb]{\smash{{\SetFigFont{10}{8}{\rmdefault}{\mddefault}{\updefault}{\color[rgb]{0,0,0}Original metric}%
}}}}
\put(2401,-6749){\makebox(0,0)[lb]{\smash{{\SetFigFont{10}{8}{\rmdefault}{\mddefault}{\updefault}{\color[rgb]{0,0,0}Old transition metric}%
}}}}
\put(5514,-4974){\makebox(0,0)[lb]{\smash{{\SetFigFont{10}{8}{\rmdefault}{\mddefault}{\updefault}{\color[rgb]{0,0,0}$ds^{2}+g_{tor}^{p+1}(\epsilon)+{\delta'}^{2}ds_{q-1}^{2}$}%
}}}}
\put(5514,-4261){\makebox(0,0)[lb]{\smash{{\SetFigFont{10}{8}{\rmdefault}{\mddefault}{\updefault}{\color[rgb]{0,0,0}$ds^{2}+g_{Dtor}^{p+1}(\epsilon)+{\delta'}^{2}ds_{q-1}^{2}$}%
}}}}
\put(2251,-5124){\makebox(0,0)[lb]{\smash{{\SetFigFont{10}{8}{\rmdefault}{\mddefault}{\updefault}{\color[rgb]{0,0,0}New transition metric}%
}}}}
\put(1731,-2649){\makebox(0,0)[lb]{\smash{{\SetFigFont{10}{8}{\rmdefault}{\mddefault}{\updefault}{\color[rgb]{0,0,0}$g_{tor}^{p+2}(\epsilon)+{\delta'}^{2}ds_{q-1}^{2}$}%
}}}}
\put(1651,-2274){\makebox(0,0)[lb]{\smash{{\SetFigFont{10}{8}{\rmdefault}{\mddefault}{\updefault}{\color[rgb]{0,0,0}New standard metric}%
}}}}

\put(5514,-6086){\makebox(0,0)[lb]{\smash{{\SetFigFont{10}{8}{\rmdefault}{\mddefault}{\updefault}{\color[rgb]{0,0,0}$g_{tor}^{p+1}(\epsilon)+g_{tor}^{q}(\delta')$}%
}}}}
\end{picture}%

\caption{The metric $g_1''$ resulting from the initial adjustment}
\label{generalgg}
\end{figure}

\begin{figure}[htbp]
\begin{picture}(0,0)%
\includegraphics{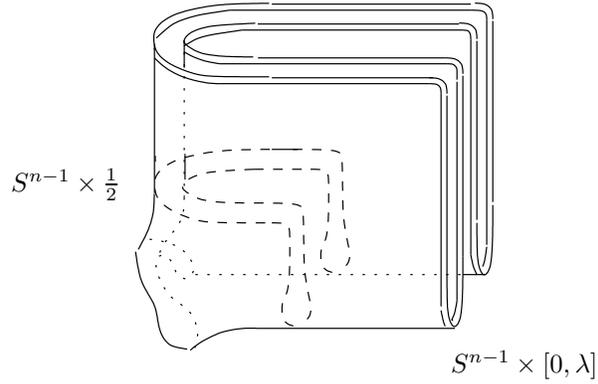}%
\end{picture}%
\setlength{\unitlength}{3947sp}%
\begingroup\makeatletter\ifx\SetFigFont\undefined%
\gdef\SetFigFont#1#2#3#4#5{%
  \reset@font\fontsize{#1}{#2pt}%
  \fontfamily{#3}\fontseries{#4}\fontshape{#5}%
  \selectfont}%
\fi\endgroup%
\begin{picture}(3552,2397)(1874,-3440)
\put(5151,-3374){\makebox(0,0)[lb]{\smash{{\SetFigFont{10}{8}{\rmdefault}{\mddefault}{\updefault}{\color[rgb]{0,0,0}$S^{n-1}\times[0,\lambda]$}%
}}}}
\put(2389,-2236){\makebox(0,0)[lb]{\smash{{\SetFigFont{10}{8}{\rmdefault}{\mddefault}{\updefault}{\color[rgb]{0,0,0}$S^{n-1}\times\frac{1}{2}$}%
}}}}
\end{picture}%

\caption{The collar neighbourhood $S^{n-1}\times [0,\lambda]$}
\label{alteredg''} 
\end{figure} 

\noindent {\bf The embedding $\bar{J}$.} For sufficiently small $\lambda$, the cylindrical portion $S^{n-1}\times[0,\lambda]$ of the connecting cylinder $S^{n-1}\times I$ is contained entirely in a region where $g''=g_{tor}^{p+1}(\epsilon)+g_{tor}^{q}(\delta')$. Recall that in section \ref{embedmixedtorp}, we equipped the plane $\mathbb{R}^{n}=\mathbb{R}^{p+1}\times\mathbb{R}^{q}$ with this metric, then denoted by $h=g_{tor}^{p+1}(\epsilon)+g_{tor}^{q}(\delta')$. In standard spherical coordinates $(\rho, \phi), (r, \theta)$ for $\mathbb{R}^{p+1}$ and $\mathbb{R}^{q}$ respectively, we can represent this metric with the explicit formula

\begin{equation}
h=d\rho^{2}+f_{\epsilon}(\rho)^{2}ds_{p}^{2}+dr^{2}+f_{\delta'}(r)^{2}ds_{q}^{2},
\end{equation}

\noindent where $f_{\epsilon}, f_{\delta'}$ are the standard $\epsilon$ and $\delta$ torpedo functions defined on $(0,\infty)$. The restriction of $g''$ to the region $S^{n-1}\times [0,\lambda]$ is now isometric to an annular region of $(\mathbb{R}^{n},h)$ shown in Fig. \ref{imageJ}. For a more geometrically accurate schematic of $(\mathbb{R}^{n},h)$, see Fig. \ref{h} in section \ref{prelim}.

\begin{figure}[htbp]
\begin{picture}(0,0)%
\includegraphics{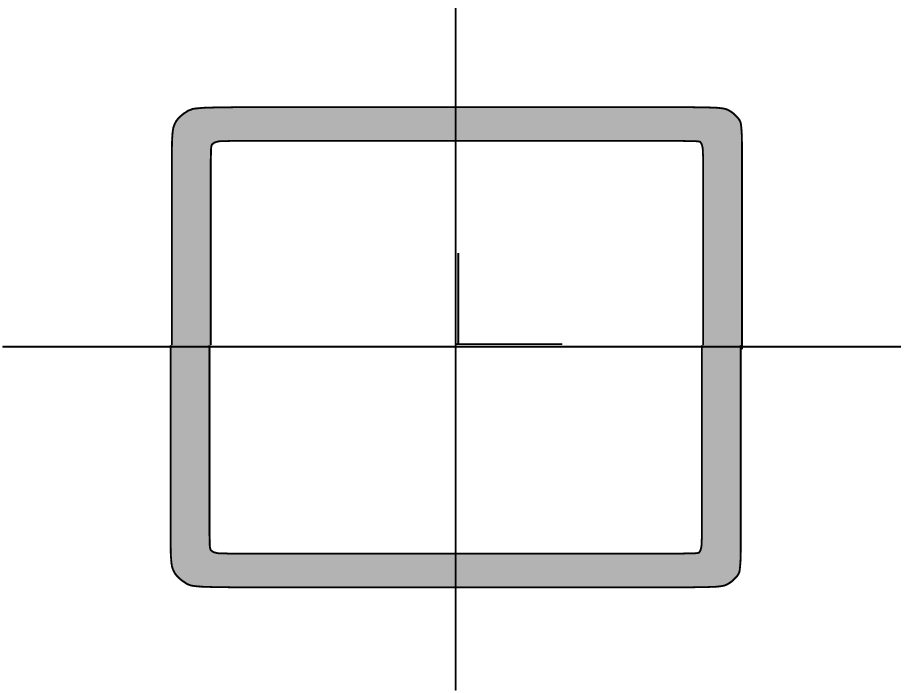}%
\end{picture}%
\setlength{\unitlength}{3947sp}%
\begingroup\makeatletter\ifx\SetFigFont\undefined%
\gdef\SetFigFont#1#2#3#4#5{%
  \reset@font\fontsize{#1}{#2pt}%
  \fontfamily{#3}\fontseries{#4}\fontshape{#5}%
  \selectfont}%
\fi\endgroup%
\begin{picture}(4336,3299)(2240,-6861)
\put(4514,-4911){\makebox(0,0)[lb]{\smash{{\SetFigFont{10}{8}{\rmdefault}{\mddefault}{\updefault}{\color[rgb]{0,0,0}$\frac{\pi}{2}\delta'$}%
}}}}
\put(4764,-5461){\makebox(0,0)[lb]{\smash{{\SetFigFont{10}{8}{\rmdefault}{\mddefault}{\updefault}{\color[rgb]{0,0,0}$\frac{\pi}{2}\epsilon$}%
}}}}
\put(6464,-5474){\makebox(0,0)[lb]{\smash{{\SetFigFont{10}{8}{\rmdefault}{\mddefault}{\updefault}{\color[rgb]{0,0,0}$\mathbb{R}^{p+1}$}%
}}}}
\put(4539,-3761){\makebox(0,0)[lb]{\smash{{\SetFigFont{10}{8}{\rmdefault}{\mddefault}{\updefault}{\color[rgb]{0,0,0}$\mathbb{R}^{q}$}%
}}}}
\end{picture}%

\caption{The image of $\bar{J}$} 
\label{imageJ}
\end{figure} 

There is an isometric embedding $\bar{J}$ of the cylindrical portion $S^{n-1}\times[0,\lambda]$ of the connecting cylinder $S^{n-1}\times I$ into $(\mathbb{R}^{n}, h)$. Let $\bar{a}$ denote an embedding

\begin{equation*}
\begin{split}
\bar{a}:[0,\lambda]\times[0,b]&\rightarrow\mathbb{R}\times\mathbb{R} \\
(t_1, t_2)&\mapsto(a_1(t_1,t_2),\phi,a_2(t_1,t_2),\theta)
\end{split}
\end{equation*}

\noindent which satisfies the following properties. 

\noindent (1) For each $t_1\in[0,\lambda]$, the restriction of $\bar{a}$ to $\{t_1\}\times [0,b]$ is a smooth curve in the first quadrant of $\mathbb{R}^{2}$ which begins at $(c_1+t_1,0)$, follows a vertical trajectory, bends by an angle of $\frac{\pi}{2}$ towards the vertical axis in the form of a circular arc and continues as a horizontal line to end at $(0,c_1+t_1)$. We will assume that $c_1>\max\{\frac{\pi}{2}\epsilon,\frac{\pi}{2}\delta'\}$ and that the bending takes place above the horizontal line through line $(0,\delta\frac{\pi}{2})$, see Fig. \ref{fig:xandy}.

\noindent (2) At each point $(t_1, t_2)$, the products $\frac{\p{a_1}}{\p{t_1}}.\frac{\p{a_1}}{\p{t_2}}$ and $\frac{\p{a_1}}{\p{t_1}}.\frac{\p{a_1}}{\p{t_2}}$ are both zero.

\noindent For some such $\bar{a}$, there is a map $\bar{J}$ defined

\begin{equation*}
\begin{split}
\bar{J}:[0,\lambda]\times[0,b]\times{S^{p}}\times S^{q-1}&\rightarrow\mathbb{R}^{p+1}\times\mathbb{R}^{q}\\
(t_1,t_2,\phi,\theta)&\mapsto(a_1(t_1,t_2),\phi,a_2(t_1,t_2),\theta)
\end{split}
\end{equation*}

\noindent which isometrically embeds the cylindrical piece $(S^{n-1}\times [0,\lambda], g''|_{S^{n-1}\times[0,\lambda]})$ into $(\mathbb{R}^{n}, h)$, see Fig. \ref{imageJ}. Furthermore, assumption (2) above means that the metric $g''|_{S^{n-1}\times[0,\lambda]}$ can be foliated as $dt_1^{2}+g_{t_1}''$ where $g_{t_1}''$ is the metric induced on the restriction of $\bar{J}$ to $\{t_1\}\times[0,\lambda]\times S^{p}\times S^{q}$. For each $t_1\in [0,\lambda]$, the metric $g_{t_1}''$ is a mixed torpedo metric $g_{Mtor}^{p,q-1}$. These metrics are of course not isometric, but differ only in that the tube lengths of the various torpedo parts vary continuously. 
\\

\noindent {\bf Isotopying the metric on $S^{n-1}\times [0,\lambda]$ to the ``connected sum" metric $g_2''$.} Given two copies of the plane $\mathbb{R}^{n}$, each equipped with the metric $h$, we can apply the Gromov-Lawson technique to construct a connected sum $(\mathbb{R}^{n}, h)\#(\mathbb{R}^{n}, h)$. This technique determines a psc-metric by removing a disk around each origin and gluing the resulting manifolds together with an appropriate psc-metric on the cylinder $S^{n-1}\times I$. In this section, we will isotopy the metric $g''|_{S^{n-1}\times [0,\lambda]}$ to obtain precisely this cylinder metric, see Fig. \ref{connsum}. Importantly, this isotopy will fix the metric near the ends of the cylinder and so will extend easily to an isotopy of $g_1''$ on all of $X$ to result in the metric $g_2''$.

\begin{figure}[htbp]  
\begin{picture}(0,0)%
\includegraphics{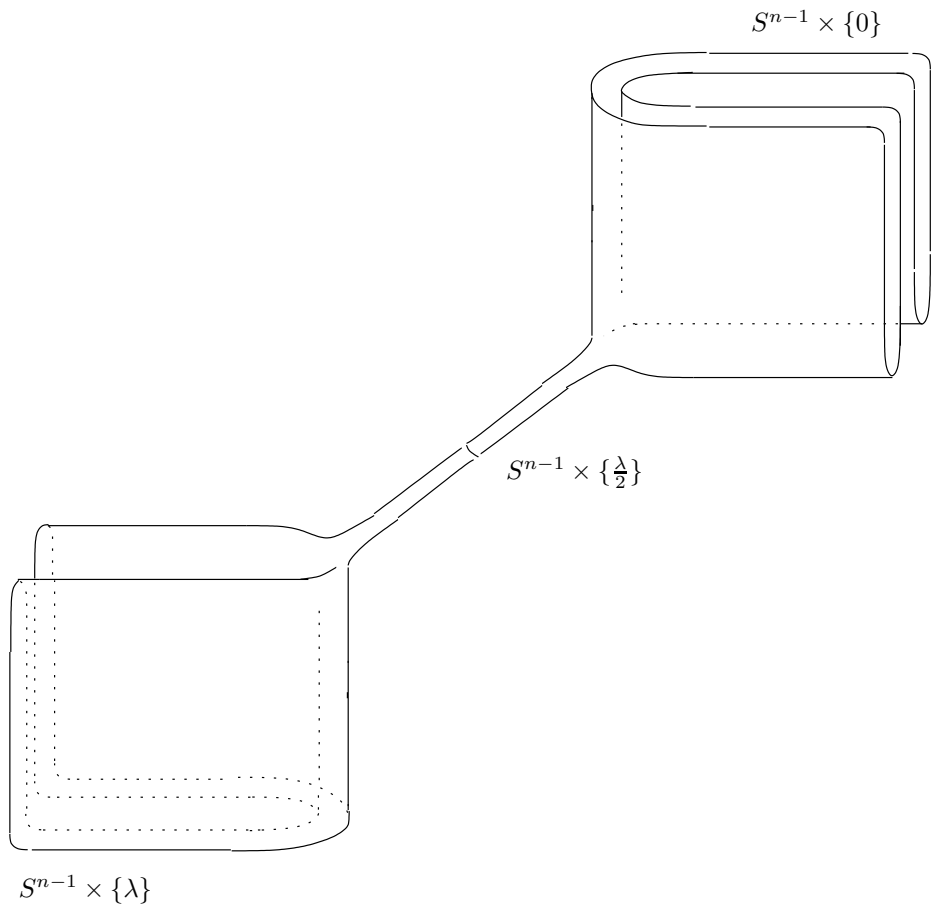}%
\end{picture}%
\setlength{\unitlength}{3947sp}%
\begingroup\makeatletter\ifx\SetFigFont\undefined%
\gdef\SetFigFont#1#2#3#4#5{%
  \reset@font\fontsize{#1}{#2pt}%
  \fontfamily{#3}\fontseries{#4}\fontshape{#5}%
  \selectfont}%
\fi\endgroup%
\begin{picture}(5805,5658)(2846,-6115)
\put(5976,-3424){\makebox(0,0)[lb]{\smash{{\SetFigFont{10}{8}{\rmdefault}{\mddefault}{\updefault}{\color[rgb]{0,0,0}$S^{n-1}\times\{\frac{\lambda}{2}\}$}%
}}}}
\put(2914,-6049){\makebox(0,0)[lb]{\smash{{\SetFigFont{10}{8}{\rmdefault}{\mddefault}{\updefault}{\color[rgb]{0,0,0}$S^{n-1}\times\{\lambda\}$}%
}}}}
\put(7514,-611){\makebox(0,0)[lb]{\smash{{\SetFigFont{10}{8}{\rmdefault}{\mddefault}{\updefault}{\color[rgb]{0,0,0}$S^{n-1}\times\{0\}$}%
}}}}
\end{picture}%

\caption{The metric obtained by isotopying $g''|_{{S^{n-1}\times[0,\lambda]}}$ to the cylinder metric of Gromov-Lawson ``connected sum" construction}
\label{connsum}
\end{figure}

Let $\bar{a}^{t_1}$ denote the curve which is the image of the map $\bar{a}$ restricted to $\{t_1\}\times[0,b]$ and $\bar{J}^{t_1}$ will denote the embedded sphere in $\mathbb{R}^{n}$ which is the image of the map $\bar{J}$ on $\{t_1\}\times[0,b]\times S^{p}\times S^{q-1}$. We define the map $K^{\tau}$ as

\begin{equation*}
\begin{array}{cl}
{K^{\tau}}:[0,\frac{\pi}{2}\tau]\times{S^{p}}\times S^{q-1}&\longrightarrow\mathbb{R}^{p+1}\times\mathbb{R}^{q}\\
(t,\phi,\theta)&\longmapsto(\tau\cos(\frac{t}{\tau}),\phi,\tau\sin(\frac{t}{\tau}),\theta)
\end{array}
\end{equation*}

\noindent For each $\tau>0$, the image of $K^{\tau}$ in $(\mathbb{R}^{n}, h)$ is a geodesic sphere of radius $\tau$. Now consider the region, shown in Fig. \ref{foliatingh}, bounded by the embedded spheres $\bar{J}^{\frac{\lambda}{2}}$ and $K^{\tau}$ where $\tau$ is assumed to be very small. Let $c^{\tau}$ denote the circular arc given by $c^{\tau}(t)=(\tau\cos(\frac{t}{\tau}),\tau\sin(\frac{t}{\tau})$, for $t\in[0,\frac{\pi}{2}\tau]$. It is easy to construct a smooth homotopy between $\bar{a}^{\frac{\lambda}{2}}$ and $c^{\tau}$ through curves $(x_{\nu}, y_{\nu}), \nu\in I$ where $c^{\tau}=(x_{0}, y_0)$ and $\bar{a}^{\frac{\lambda}{2}}=(x_1, y_1)$. For example, this can be done by smoothly shrinking the straight edge pieces of $\bar{a}^{\frac{\lambda}{2}}$ to obtain a piece which is within arbitrarily small smoothing adjustments of being a circular arc, the radius of which can then be smoothly shrunk as required, see Fig. \ref{homotopyfoliate}. 

\begin{figure}[htbp]
\begin{picture}(0,0)%
\includegraphics{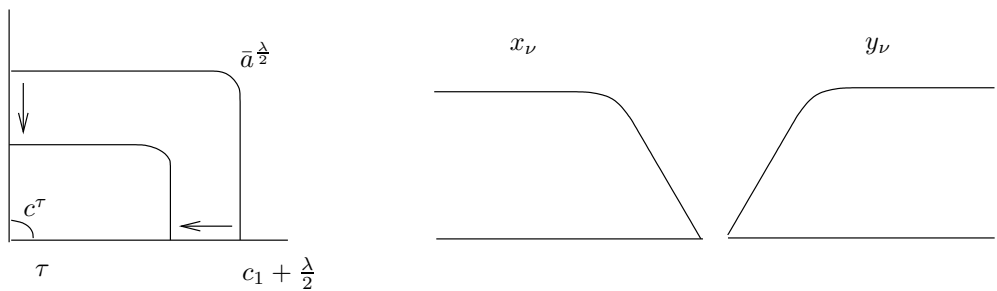}%
\end{picture}%
\setlength{\unitlength}{3947sp}%
\begingroup\makeatletter\ifx\SetFigFont\undefined%
\gdef\SetFigFont#1#2#3#4#5{%
  \reset@font\fontsize{#1}{#2pt}%
  \fontfamily{#3}\fontseries{#4}\fontshape{#5}%
  \selectfont}%
\fi\endgroup%
\begin{picture}(6214,1778)(563,-3054)
\put(739,-2963){\makebox(0,0)[lb]{\smash{{\SetFigFont{10}{8}{\rmdefault}{\mddefault}{\updefault}{\color[rgb]{0,0,0}$\tau$}%
}}}}
\put(2039,-2988){\makebox(0,0)[lb]{\smash{{\SetFigFont{10}{8}{\rmdefault}{\mddefault}{\updefault}{\color[rgb]{0,0,0}$c_1+\frac{\lambda}{2}$}%
}}}}
\put(664,-2588){\makebox(0,0)[lb]{\smash{{\SetFigFont{10}{8}{\rmdefault}{\mddefault}{\updefault}{\color[rgb]{0,0,0}$c^{\tau}$}%
}}}}
\put(2026,-1650){\makebox(0,0)[lb]{\smash{{\SetFigFont{10}{8}{\rmdefault}{\mddefault}{\updefault}{\color[rgb]{0,0,0}$\bar{a}^{\frac{\lambda}{2}}$}%
}}}}
\put(3720,-1542){\makebox(0,0)[lb]{\smash{{\SetFigFont{10}{8}{\rmdefault}{\mddefault}{\updefault}{\color[rgb]{0,0,0}$x_\nu$}%
}}}}
\put(5952,-1536){\makebox(0,0)[lb]{\smash{{\SetFigFont{10}{8}{\rmdefault}{\mddefault}{\updefault}{\color[rgb]{0,0,0}$y_\nu$}%
}}}}
\end{picture}%

\caption{Homotopying the curve $\bar{a}^{\frac{\lambda}{2}}$ to $c^{\tau}$}
\label{homotopyfoliate}
\end{figure}

By smoothly varying the length of domain intervals of $x_\nu$ and $y_\nu$ with respect to $\nu$ we can ensure that the curve $(x_\nu, y_\nu)$ is unit speed for all $\nu$. The above homotopy gives rise to a foliation of the region contained between $\bar{J}^{\frac{\lambda}{2}}$ and $K^{\tau}$, see Fig. \ref{foliatingh}, and a corresponding foliation of the metric $h$ on this region. Letting $l\in[0,1]$ denote the coordinate running orthogonal to the curves given by the above homotopy, we can write the metric $h=dl^{2}+h_l$. Moreover, the metric $h_l$ can be computed explicitly as

\begin{equation*}
h_l=dt^{2}+f_\epsilon(x(t))^{2}ds_{p}^{2}+f_{\delta'}(y(t))^{2}ds_{q}^{2}
\end{equation*}

\noindent where $x=x_{\nu}$ and $y=y_\nu$ for some $\nu$. An elementary calculation shows that $-1\leq\dot{x_\nu}\leq 0$, $0\leq \dot{y_\nu}\leq 1$, $\ddot{x_\nu}\leq 0$ and $\ddot{y_\nu}\leq 0$. A further elementary calculation now shows that the functions $f_\epsilon(x(t))$ and  $f_{\delta'}(y(t))$ belong to the spaces $\U$ and $\V$ defined in section \ref{prelim}. Thus, by Lemma \ref{lem:markslemm1.5}, the metric $h_l$ has positive scalar curvature and so the decomposition of $h$ into $dl^{2}+h_l$ induces an isotopy between the metric $h_1=g_{\frac{\lambda}{2}}''$ and the metric $h_0$ induced by $h$ on the geodesic sphere of radius $\tau$.

\begin{figure} 
\begin{picture}(0,0)%
\includegraphics{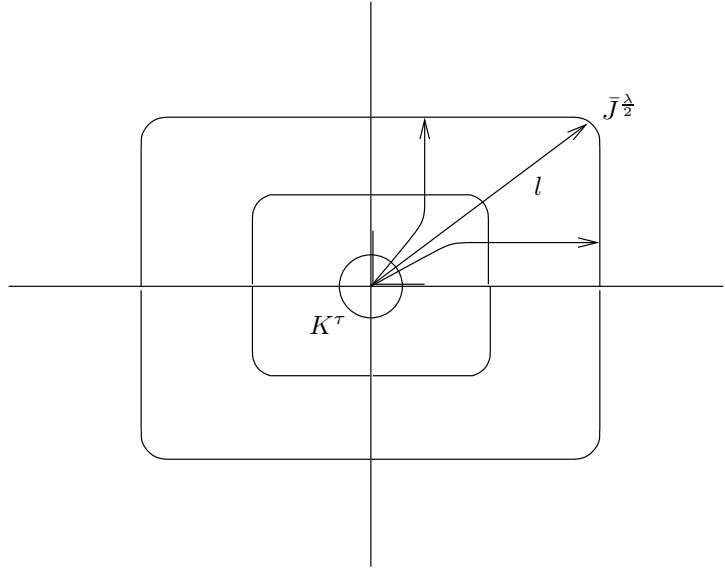}%
\end{picture}%
\setlength{\unitlength}{3947sp}%
\begingroup\makeatletter\ifx\SetFigFont\undefined%
\gdef\SetFigFont#1#2#3#4#5{%
  \reset@font\fontsize{#1}{#2pt}%
  \fontfamily{#3}\fontseries{#4}\fontshape{#5}%
  \selectfont}%
\fi\endgroup%
\begin{picture}(4512,3574)(2864,-5198)
\put(6601,-2374){\makebox(0,0)[lb]{\smash{{\SetFigFont{10}{8}{\rmdefault}{\mddefault}{\updefault}{\color[rgb]{0,0,0}$\bar{J}^{\frac{\lambda}{2}}$}%
}}}}
\put(4764,-3724){\makebox(0,0)[lb]{\smash{{\SetFigFont{10}{8}{\rmdefault}{\mddefault}{\updefault}{\color[rgb]{0,0,0}$K^{\tau}$}%
}}}}
\put(6176,-2849){\makebox(0,0)[lb]{\smash{{\SetFigFont{10}{8}{\rmdefault}{\mddefault}{\updefault}{\color[rgb]{0,0,0}$l$}%
}}}}
\end{picture}%

\caption{The region bounded by $\bar{J}^{\frac{\lambda}{2}}$ and $K^{\tau}$}
\label{foliatingh}
\end{figure}

Recall that the restriction of the metric $g''$ to $S^{n-1}\times [0, \frac{\lambda}{2}]$, isometrically embeds into $(\mathbb{R}^{n}, h)$ as the region between the curves $\bar{J}^{0}$ and $\bar{J}^{\frac{\lambda}{2}}$. Using the foliation $h=dl^{2}+h_l$, this metric can now be continuously extended as the metric $h$ over the rest of the region between $\bar{J}^{0}$ and $K^{\tau}$, see Fig. \ref{extendh}. As the curve $K^{\tau}$ is a geodesic sphere with respect to $h$, this metric can then be continuously extended as the metric obtained by the Gromov-Lawson construction, to finish as a round cylinder metric. The metric $g''|_{S^{n-1}\times [0,\frac{\lambda}{2}]}$ has now been isotopied to one half of the metric depicted in Fig. \ref{connsum} without making any adjustment near $S^{n-1}\times \{0\}$.

An analogous construction can be performed on $g''|_{S^{n-1}\times [\frac{\lambda}{2}, \lambda]}$, this time making no alteration to the metric near ${S^{n-1}\times \{\lambda\}}$. Both constructions can be combined to form the desired isotopy by making a minor modifcation to ensure that at each stage, the metric near $S^{n-1}\times\{\frac{\lambda}{2}\}$ is a psc-Riemannian cylinder. Such a modification is possible because of the fact that the above foliation decomposes $h$ into an isotopy of psc-metrics.   
\\

\begin{figure}[htbp]
\begin{picture}(0,0)%
\includegraphics{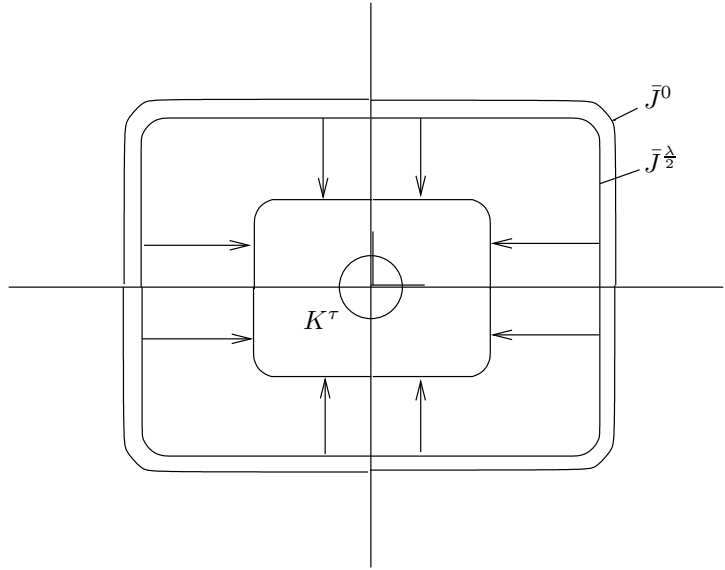}%
\end{picture}%
\setlength{\unitlength}{3947sp}%
\begingroup\makeatletter\ifx\SetFigFont\undefined%
\gdef\SetFigFont#1#2#3#4#5{%
  \reset@font\fontsize{#1}{#2pt}%
  \fontfamily{#3}\fontseries{#4}\fontshape{#5}%
  \selectfont}%
\fi\endgroup%
\begin{picture}(4512,3574)(2864,-5198)
\put(4726,-3686){\makebox(0,0)[lb]{\smash{{\SetFigFont{10}{8}{\rmdefault}{\mddefault}{\updefault}{\color[rgb]{0,0,0}$K^{\tau}$}%
}}}}
\put(6864,-2699){\makebox(0,0)[lb]{\smash{{\SetFigFont{10}{8}{\rmdefault}{\mddefault}{\updefault}{\color[rgb]{0,0,0}$\bar{J}^{\frac{\lambda}{2}}$}%
}}}}
\put(6864,-2299){\makebox(0,0)[lb]{\smash{{\SetFigFont{10}{8}{\rmdefault}{\mddefault}{\updefault}{\color[rgb]{0,0,0}$\bar{J}^{0}$}%
}}}}
\end{picture}%

\caption{Isotopying the metric $g''|_{S^{n-1}\times[0,\frac{\lambda}{2}]}$ to the metric $h$ on the region bounded by $\bar{J}^{0}$ and $K^{\tau}$}
\label{extendh}
\end{figure}

\noindent {\bf Isotopying the metric $g$.}
In this step we will perform three successive adjustments on the metric $g$, resulting in succesive positive scalar curvature metrics $g_1, g_2$ and $g_3$. Each adjustment will result in a metric which is isotopic to the previous one and thus to $g$. 

In adjusting the metric $g$, we wish to mimic as closely as possible, the Gromov-Lawson technique applied in the construction of $g''$. The main difficulty of course is that we are prevented from making any topological change to the manifold $X$. Thus, the first adjustment is one we have seen before. The metric $g_1$ is precisely the metric $g_{std}$ constructed in Theorem \ref{IsotopyTheorem}, this being the closest we can get to the original Gromov-Lawson construction without changing the topology of $X$, see Fig. \ref{g_1}. The metric $g_1$ is the original metric $g$ outside of a tubular neighbourhood of the embedded $S^{p}$. It then transitions to a standard form so that near $S^{p}$ it is $\epsilon^{2}ds_{p}^{2}+g_{tor}^{q+1}(\delta)$ for some suitably small $\delta>0$. We will refer to this region as the {\it standard region} throughout this proof, see Fig. \ref{g_1}. From Theorem \ref{IsotopyTheorem}, we know that $g_1$ is isotopic to the original $g$. We make two important observations. 

(i) All of the data regarding the effects of the Gromov-Lawson construction on $(X,g)$, is contained in the metric $g_1$. 

(ii) The embedded disk $D_{-}^{p+1}$ agrees entirely with the non-standard part of the embedded sphere $S^{p+1}$. 

\begin{figure}[htbp]
\begin{picture}(0,0)%
\includegraphics{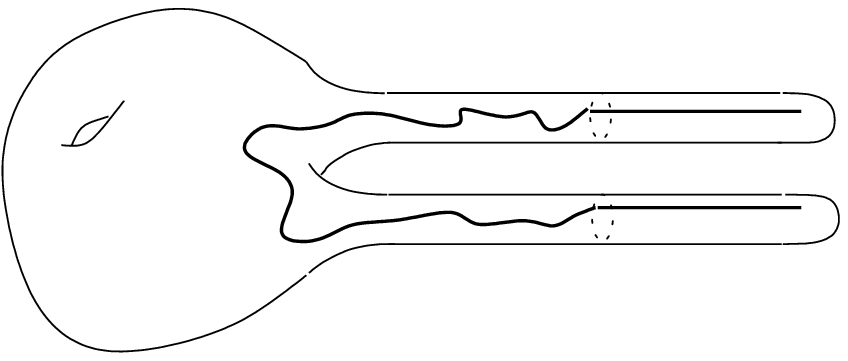}%
\end{picture}%
\setlength{\unitlength}{3947sp}%
\begingroup\makeatletter\ifx\SetFigFont\undefined%
\gdef\SetFigFont#1#2#3#4#5{%
  \reset@font\fontsize{#1}{#2pt}%
  \fontfamily{#3}\fontseries{#4}\fontshape{#5}%
  \selectfont}%
\fi\endgroup%
\begin{picture}(4705,1895)(156,-3240)
\put(1851,-2824){\makebox(0,0)[lb]{\smash{{\SetFigFont{10}{8}{\rmdefault}{\mddefault}{\updefault}{\color[rgb]{0,0,0}Transition metric}%
}}}}
\put(401,-3174){\makebox(0,0)[lb]{\smash{{\SetFigFont{10}{8}{\rmdefault}{\mddefault}{\updefault}{\color[rgb]{0,0,0}Original metric $g$}%
}}}}
\put(1164,-1924){\makebox(0,0)[lb]{\smash{{\SetFigFont{10}{8}{\rmdefault}{\mddefault}{\updefault}{\color[rgb]{0,0,0}$D_{-}^{p+1}$}%
}}}}
\put(3564,-2849){\makebox(0,0)[lb]{\smash{{\SetFigFont{10}{8}{\rmdefault}{\mddefault}{\updefault}{\color[rgb]{0,0,0}Standard metric}%
}}}}
\put(3589,-3111){\makebox(0,0)[lb]{\smash{{\SetFigFont{10}{8}{\rmdefault}{\mddefault}{\updefault}{\color[rgb]{0,0,0}$\epsilon^{2}ds_{p}^{2}+g_{tor}^{q}(\delta)$}%
}}}}
\end{picture}%

\caption{The metric $g_1$ on $X$, made standard near the embedded $S^{p}$}
\label{g_1}
\end{figure}

The aim of the next adjustment is to mimic as closely as possible the metric effects of the second surgery. The boundary of $D_{-}^{p+1}$ lies at the end of the standard region of $(X,g_1)$. Application of Theorem \ref{IsotopyTheorem} allows us adjust the metric near $D_{-}^{p+1}$ exactly as in the construction of $g''$. Near the boundary of $D_{-}^{p+1}$, the induced metric is standard and so we can transition (possibly very slowly) back to the metric $g_1$, see Fig. \ref{g_2}. The connecting cylinder $S^{n-1}\times I$ can be specified exactly as before and it is immediately obvious that the metric $g_2$ agrees with $g''$ on this region. The metric $g_3$ is now obtained by making precisely the adjustments made to the metric $g''$ in the region of $S^{n-1}\times [0,\frac{1}{2}]$.  
\\

\begin{figure}[htbp]
\begin{picture}(0,0)%
\includegraphics{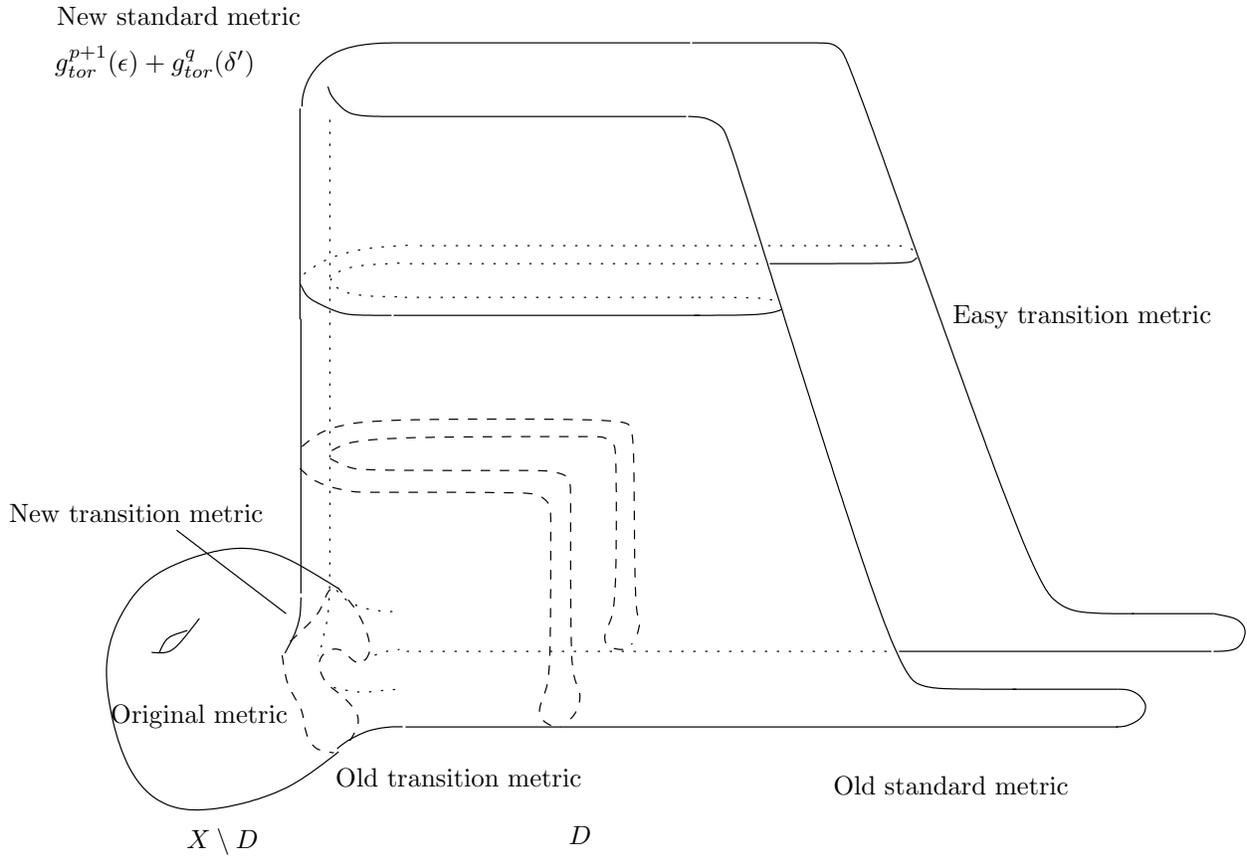}%
\end{picture}%
\setlength{\unitlength}{3947sp}%
\begingroup\makeatletter\ifx\SetFigFont\undefined%
\gdef\SetFigFont#1#2#3#4#5{%
  \reset@font\fontsize{#1}{#2pt}%
  \fontfamily{#3}\fontseries{#4}\fontshape{#5}%
  \selectfont}%
\fi\endgroup%
\begin{picture}(7788,5429)(361,-7165)
\put(1489,-7099){\makebox(0,0)[lb]{\smash{{\SetFigFont{10}{8}{\rmdefault}{\mddefault}{\updefault}{\color[rgb]{0,0,0}$X\setminus D$}%
}}}}
\put(3889,-7074){\makebox(0,0)[lb]{\smash{{\SetFigFont{10}{8}{\rmdefault}{\mddefault}{\updefault}{\color[rgb]{0,0,0}$D$}%
}}}}
\put(376,-5049){\makebox(0,0)[lb]{\smash{{\SetFigFont{10}{8}{\rmdefault}{\mddefault}{\updefault}{\color[rgb]{0,0,0}New transition metric}%
}}}}
\put(1014,-6311){\makebox(0,0)[lb]{\smash{{\SetFigFont{10}{8}{\rmdefault}{\mddefault}{\updefault}{\color[rgb]{0,0,0}Original metric}%
}}}}
\put(2426,-6711){\makebox(0,0)[lb]{\smash{{\SetFigFont{10}{8}{\rmdefault}{\mddefault}{\updefault}{\color[rgb]{0,0,0}Old transition metric}%
}}}}
\put(5551,-6761){\makebox(0,0)[lb]{\smash{{\SetFigFont{10}{8}{\rmdefault}{\mddefault}{\updefault}{\color[rgb]{0,0,0}Old standard metric}%
}}}}
\put(676,-1924){\makebox(0,0)[lb]{\smash{{\SetFigFont{10}{8}{\rmdefault}{\mddefault}{\updefault}{\color[rgb]{0,0,0}New standard metric}%
}}}}
\put(664,-2211){\makebox(0,0)[lb]{\smash{{\SetFigFont{10}{8}{\rmdefault}{\mddefault}{\updefault}{\color[rgb]{0,0,0}$g_{tor}^{p+1}(\epsilon)+g_{tor}^{q}(\delta')$}%
}}}}
\put(6301,-3799){\makebox(0,0)[lb]{\smash{{\SetFigFont{10}{8}{\rmdefault}{\mddefault}{\updefault}{\color[rgb]{0,0,0}Easy transition metric}%
}}}}
\end{picture}%

\caption{Adjusting the metric $g_1$ on a neighbourhood of the embedded disk $D_{-}^{p+1}$: Notice how no change is made near the boundary of this disk.}
\label{g_2}
\end{figure}

\noindent {\bf Comparing the metrics $g_2''$ and $g_3$.} At this stage we have constructed two metrics $g_2''$ and $g_3$ on $X$ which agree on $(X\setminus D)\cup (S^{n-1}\times[\frac{\lambda}{2}, \lambda])$. Near $S^{n-1}\times\{\frac{\lambda}{2}\}$, both metrics have the form of a standard round cylinder. The remaining region of $X$ is an $n$-dimensional disk which we denote $D'$. Here the metrics $g_2''$ and $g_3$ are quite different. Henceforth $g_2''$ and $g_3$ will denote the restriction of these metrics to the disk $D'$. As $g_2''$ and $g_3$ agree near the boundary of $D'$, to complete the proof it is enough to show that there is an isotopy from $g_2''$ to $g_3$ which fixes the metric near the boundary.

Both $g_2''$ and $g_3$ are obtained from metrics on the sphere $S^{n}$ by removing a point and pushing out a tube in the manner of the Gromov-Lawson connected sum construction. In both cases the point itself is the origin of a region which is isometrically identified with a neighbourhood of the origin in $(\mathbb{R}^{n}, h)$. We will denote by $\bar{g}_2''$ and $\bar{g}_3$, the respective sphere metrics which give rise to $g_2''$ and $g_3$ in this way, see Fig. \ref{ggggsphere} and Fig. \ref{g3sphere}.

\begin{figure}[htbp]
\begin{picture}(0,0)%
\includegraphics{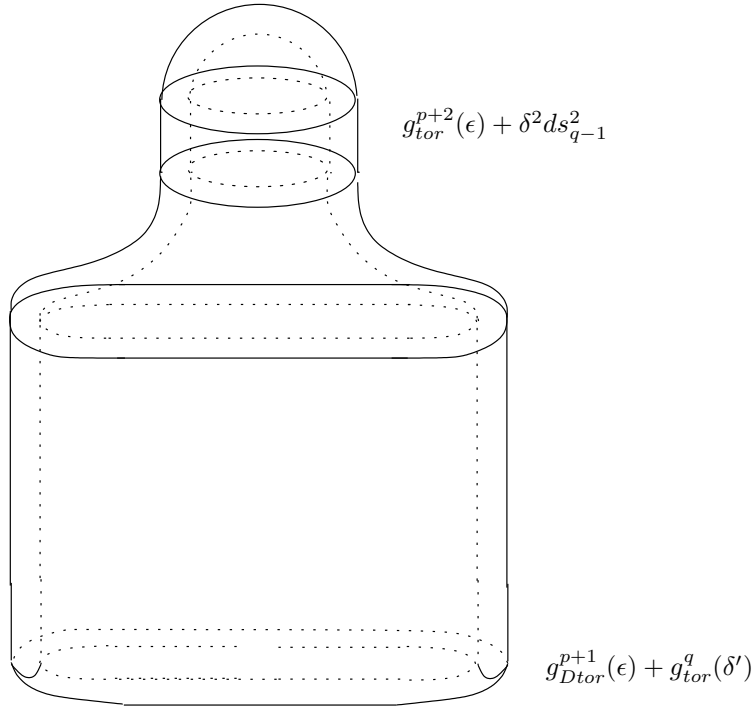}%
\end{picture}%
\setlength{\unitlength}{3947sp}%
\begingroup\makeatletter\ifx\SetFigFont\undefined%
\gdef\SetFigFont#1#2#3#4#5{%
  \reset@font\fontsize{#1}{#2pt}%
  \fontfamily{#3}\fontseries{#4}\fontshape{#5}%
  \selectfont}%
\fi\endgroup%
\begin{picture}(3394,4426)(1160,-4448)
\put(4539,-4249){\makebox(0,0)[lb]{\smash{{\SetFigFont{10}{8}{\rmdefault}{\mddefault}{\updefault}{\color[rgb]{0,0,0}$g_{Dtor}^{p+1}(\epsilon)+g_{tor}^{q}(\delta')$}%
}}}}
\put(3639,-837){\makebox(0,0)[lb]{\smash{{\SetFigFont{10}{8}{\rmdefault}{\mddefault}{\updefault}{\color[rgb]{0,0,0}$g_{tor}^{p+2}(\epsilon)+\delta^{2}ds_{q-1}^{2}$}%
}}}}
\end{picture}%

\caption{The metric $\bar{g}_2''$}
\label{ggggsphere}
\end{figure}

\begin{figure}[htbp]
\begin{picture}(0,0)%
\includegraphics{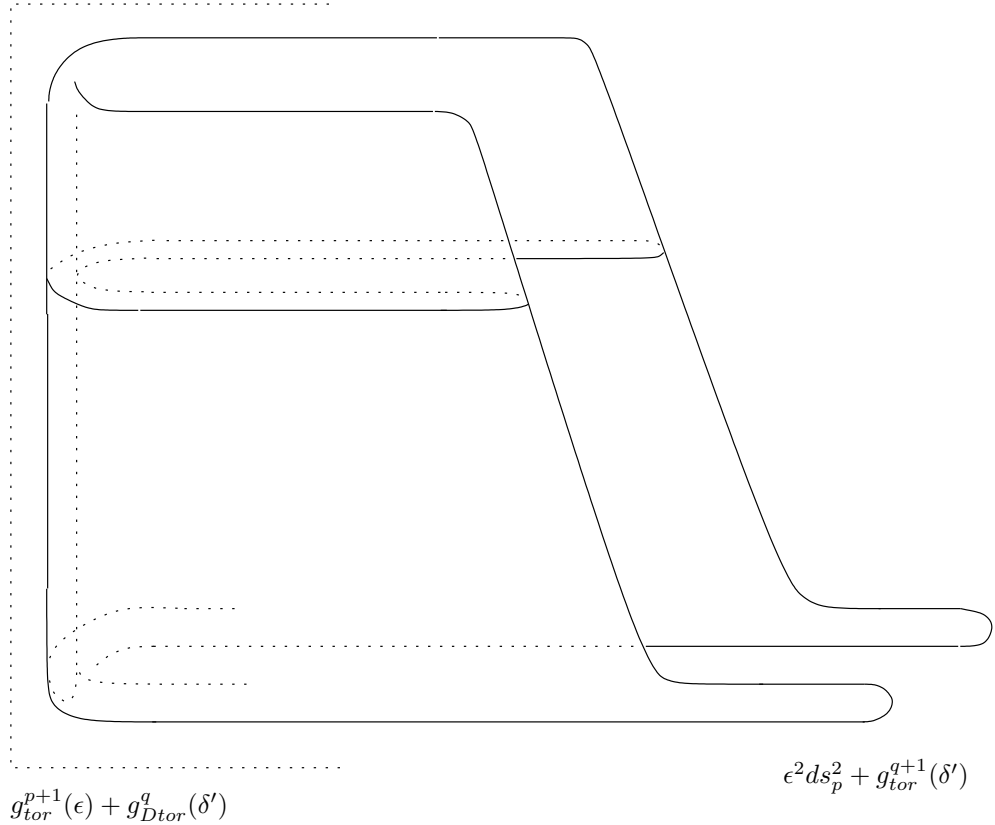}%
\end{picture}%
\setlength{\unitlength}{3947sp}%
\begingroup\makeatletter\ifx\SetFigFont\undefined%
\gdef\SetFigFont#1#2#3#4#5{%
  \reset@font\fontsize{#1}{#2pt}%
  \fontfamily{#3}\fontseries{#4}\fontshape{#5}%
  \selectfont}%
\fi\endgroup%
\begin{picture}(6188,5165)(430,-5520)
\put(5295,-5267){\makebox(0,0)[lb]{\smash{{\SetFigFont{10}{8}{\rmdefault}{\mddefault}{\updefault}{\color[rgb]{0,0,0}$\epsilon^{2}ds_p^{2}+g_{tor}^{q+1}(\delta')$}%
}}}}
\put(445,-5454){\makebox(0,0)[lb]{\smash{{\SetFigFont{10}{8}{\rmdefault}{\mddefault}{\updefault}{\color[rgb]{0,0,0}$g_{tor}^{p+1}(\epsilon)+g_{Dtor}^{q}(\delta')$}%
}}}}
\end{picture}%
\caption{The metric $\bar{g}_3$}
\label{g3sphere}
\end{figure}

The metrics $\bar{g}_2''$ and $\bar{g}_3$ isotopy easily to the respective mixed torpedo metrics $g_{tor}^{p+1,q-1}$ and $g_{tor}^{p, q}$ on $S^{n}$, see Fig. \ref{fig:mixed torpedoagain}, and are thus isotopic by the results of section \ref{prelim}, in particular Lemma \ref{toriso}. The proof of Theorem \ref{concisodoublesurgery} then follows from Theorem \ref{GLcompact}, where we showed that the Gromov-Lawson construction goes through for a compact family of psc-metrics.

\begin{figure}[htbp]
\vspace{1cm}

\includegraphics[height=30mm]{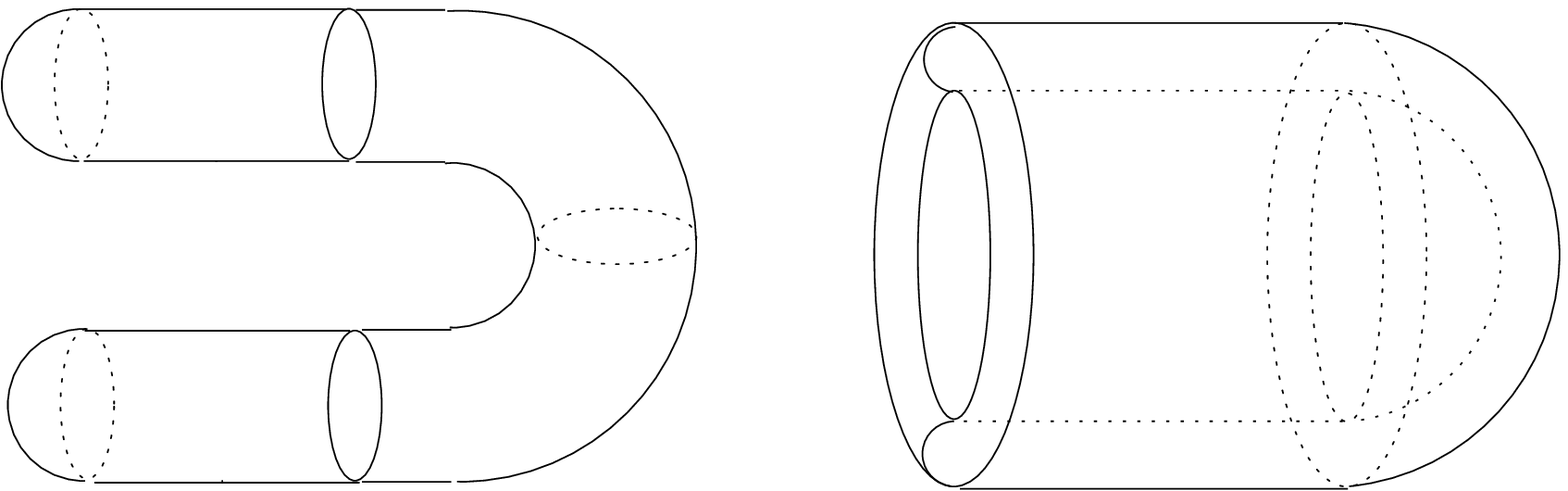}
\caption{The mixed torpedo metrics $g_{Mtor}^{p,q}$ and $g_{Mtor}^{p+1,q-1}$}
\label{fig:mixed torpedoagain}
\end{figure}

\end{proof}

\section{Gromov-Lawson concordance implies isotopy in the general case}\label{glconcisosection}

Theorem \ref{concisodoublesurgery} is the main tool needed in the proof of Theorem \ref{conciso}. The rest of the proof follows from Morse-Smale theory and all of the results needed to complete it are to be found in \cite{Smale}. Before we proceed with the proof of Theorem \ref{conciso}, it is worth discussing some of these results.

\subsection{A weaker version of Theorem \ref{conciso}}
Throughout, $\{W^{n+1},X_0,X_1\}$ is a smooth compact cobordism where $X_0$ and $X_1$ are closed manifolds of dimension $n$. Later on we will also need to assume that $X_0, X_1$ and $W$ are simply connected and that $n\geq 5$, although that is not necessary yet. Let $f$ denote a Morse triple on $W$, as defined in section \ref{GLcobordsection}. Recall this means that $f:W\rightarrow I$ is a Morse function which comes with extra data, a Riemannian metric $m$ on $W$ and a gradient-like vector field $V$ with respect to $f$ and $m$. Now by Theorem \ref{rearrangement}, $f$ can be isotopied to a Morse triple which is well-indexed. We will retain the name $f$ for this well-indexed Morse triple. As discussed in section \ref{GLcobordsection}, $f$ decomposes $W$ into a union of cobordisms $C_{0}\cup C_{1}\cup \cdots \cup{C_{n+1}}$ where each $C_{k}$ contains at most one critical level (contained in its interior) and all critical points of $f$ on this level have index $k$. For each $0\leq k\leq n+1$, we denote by $W_k$, the union $C_{0}\cup C_{1}\cup \cdots \cup{C_{k}}$. By setting $W_{-1}=X_0$, we obtain the following sequence of inclusions

\begin{equation*}
X_0=W_{-1}\subset W_{0}\subset W_{1}\subset \cdots \subset W_{n+1}=W,
\end{equation*} 

\noindent describing this decomposition. 

Suppose that $f$ has $l$ critical points of index $k$. Then for some $a<c<b$, the cobordism $C_k=f^{-1}[a,b]$, where $c$ is the only critical value between $a$ and $b$. The level set $f^{-1}(c)$ has $l$ critical points $w_1, \cdots, w_l$, each of index $k$. Associated to these critical points are trajectory disks $K_{-}^{k}(w_1), \cdots K_{-}^{k}(w_l)$ where each $K_{-}^{k}(w_i)$ has its boundary sphere $S_{-}^{k-1}(w_i)$ embedded in $f^{-1}(a)$. These trajectory disks determine a basis, by theorem 3.15 of \cite{Smale}, for the relative integral homology group $H_{k}(W_k,W_{k-1})$ which is isomorphic to $\mathbb{Z}\oplus \mathbb{Z}\oplus \cdots \oplus \mathbb{Z}$ ($l$ summands). 

We can now construct a chain complex $\C_{*}=\{\C_{k},\p\}$, where $\C_{k}=H_{k}(W_k,W_{k-1})$ and $\p:\C_{k}\rightarrow \C_{k-1}$ is the boundary homomorphism of the long exact sequence of the triple $W_{k-2}\subset W_{k-1}\subset W_{k}$. The fact that $\p^{2}=0$ is proved in theorem 7.4 of \cite{Smale}. Furthermore, theorem 7.4 gives that $H_{k}(\C_{*})=H_{k}(W,X_0)$.

\begin{Theorem}\label{concisoeasy} Let $X$ be a closed simply connected manifold with $dimX=n\geq5$ and let $g_0$ be a positive scalar curvature metric on $X$. Let $f$ be an admissible Morse function on $X\times I$ with no critical points of index $0$ or $1$.  Then the metrics $g_0$ and $g_1=\bar{g}|_{X\times\{1\}}$ are isotopic, where $\bar{g}=\bar{g}(g_0,f)$ is a Gromov-Lawson concordance on $X\times I$.
\end{Theorem}

\begin{proof}
By Corollary \ref{rearrangementmetric}, we may assume that $f$ is well indexed. Using the notation above, $f$ gives rise to a decomposition $X\times I=C_2\cup C_3\cup \cdots \cup C_{n-2}$ which in turn gives rise to a chain complex
\begin{equation*}
\C_{n-2}\rightarrow \C_{n-3}\rightarrow\cdots\rightarrow\C_{k+1}\rightarrow\C_{k}\rightarrow\cdots\rightarrow\C_2
\end{equation*}
 
\noindent where each $\C_{k}$ is a free abelian group. (Recall that all critical points of an admissible Morse function have index which is less than or equal to $n-2$.) Since $H_{*}(X\times I, X)=0$, it follows that the above sequence is exact. Thus, for each $\C_{k+1}$ we may choose elements $z_1^{k+1},\cdots, z_{l_{k+1}}^{k+1}\in \C_{k+1}$ and $b_1^{k+1},\cdots,b_{l_{k}}^{k+1}\in \C_{k+1}$ so that $\p(b_i^{k+1})=z_{i}^{k}$ for $i=1,\cdots,l_k$. Then $z_1^{k+1},\cdots, z_{l_{k+1}}^{k+1}, b_1^{k+1},\cdots,b_{l_{k}}^{k+1}$ is a basis for $\C_{k+1}$.

We will now restrict our attention to the cobordism $C_k\cup C_{k+1}$. Now let $w_1^{k+1},w_2^{k+1},\cdots, w_{l_{k+1}+l_{k}}^{k+1}$ denote the critical points of $f$ inside of $C_{k+1}$ and $w_1^{k},w_2^{k},\cdots, w_{l_{k}+l_{k-1}}^{k}$ denote the critical points of $f$ inside of $C_{k}$. As $2\leq k<k+1\leq n-2$, it follows from theorem 7.6 of \cite{Smale}, that $f$ can be perturbed so that the trajectory disks $K_{-}(w_1^{k+1}), \cdots K_{-}(w_{l_{k+1}+l_k}^{k+1})$ and $K_{-}(w_1^{k}), \cdots K_{-}(w_{l_{k}+l_{k-1}}^{k})$ represent the chosen bases for $\C_{k+1}$ and $\C_{k}$ respectively. 

Denote by $w_1^{k},w_2^{k},\cdots, w_{l_{k}}^{k}$, those critical points on $C_k$ which correspond to the elements $z_1^k,z_2^{k}\cdots, z_{l_{k}}^k$ of $\C_k$, i.e. the kernel of $\p:\C_k\rightarrow \C_{k-1}$. Denote by $w_1^{k+1},w_2^{k+2},\cdots,w_{l_k}^{k+1}$, those critical points in $C_{k+1}$ which correspond to the elements $b_1^{k+1},\cdots,b_{l_{k}}^{k+1}\in \C_{k+1}$. A slight perturbation of $f$, replaces $C_k\cup C_{k+1}$ with the decomposition $C_{k}'\cup C_{k}''\cup C_{k+1}''\cup C_{k+1}'$, see Fig. \ref{cancellingcobordisms}. Here $C_{k}'\cup C_{k}''$ is diffeomorphic to $C_k$, however, the critical points $w_1^{k},w_2^{k},\cdots, w_{l_{k}}^{k}$ have been moved to a level set above their orginal level, resulting in a pair of cobordisms each with one critical level. Similarly, we can move the critical points $w_1^{k+1},w_2^{k+2},\cdots,w_{l_k}^{k+1}$ down to a level set below their original level to replace $C_{k+1}$ with $C_{k+1}''\cup C_{k+1}'$. 
 
\begin{figure}[htbp]
\begin{picture}(0,0)%
\includegraphics{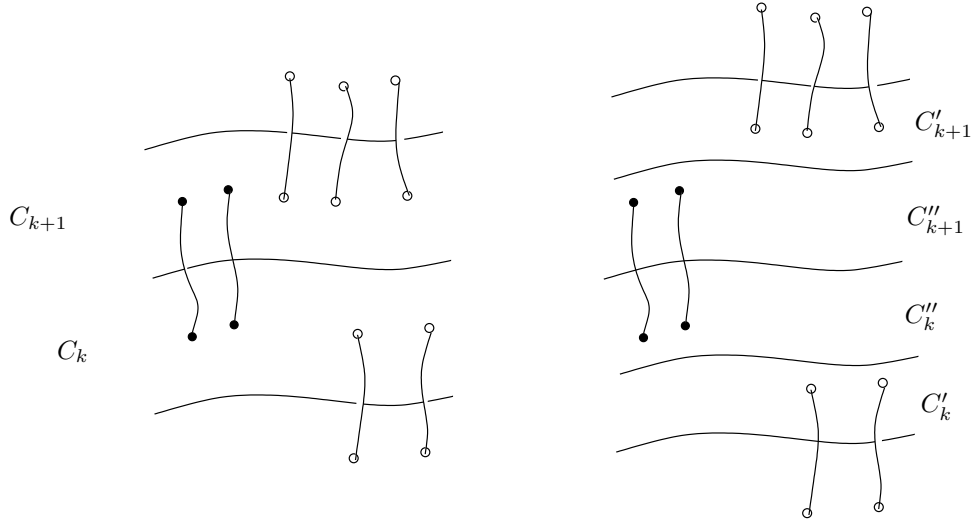}%
\end{picture}%
\setlength{\unitlength}{3947sp}%
\begingroup\makeatletter\ifx\SetFigFont\undefined%
\gdef\SetFigFont#1#2#3#4#5{%
  \reset@font\fontsize{#1}{#2pt}%
  \fontfamily{#3}\fontseries{#4}\fontshape{#5}%
  \selectfont}%
\fi\endgroup%
\begin{picture}(5755,3248)(1161,-4605)
\put(1176,-2761){\makebox(0,0)[lb]{\smash{{\SetFigFont{10}{8}{\rmdefault}{\mddefault}{\updefault}{\color[rgb]{0,0,0}$C_{k+1}$}%
}}}}
\put(1476,-3599){\makebox(0,0)[lb]{\smash{{\SetFigFont{10}{8}{\rmdefault}{\mddefault}{\updefault}{\color[rgb]{0,0,0}$C_k$}%
}}}}
\put(6864,-2161){\makebox(0,0)[lb]{\smash{{\SetFigFont{10}{8}{\rmdefault}{\mddefault}{\updefault}{\color[rgb]{0,0,0}$C_{k+1}'$}%
}}}}
\put(6814,-2761){\makebox(0,0)[lb]{\smash{{\SetFigFont{10}{8}{\rmdefault}{\mddefault}{\updefault}{\color[rgb]{0,0,0}$C_{k+1}''$}%
}}}}
\put(6801,-3374){\makebox(0,0)[lb]{\smash{{\SetFigFont{10}{8}{\rmdefault}{\mddefault}{\updefault}{\color[rgb]{0,0,0}$C_k''$}%
}}}}
\put(6901,-3936){\makebox(0,0)[lb]{\smash{{\SetFigFont{10}{8}{\rmdefault}{\mddefault}{\updefault}{\color[rgb]{0,0,0}$C_k'$}%
}}}}
\end{picture}%
\caption{Replacing $C_k\cup C_{k+1}$ with $C_{k}'\cup C_{k}''\cup C_{k+1}''\cup C_{k+1}'$}
\label{cancellingcobordisms}
\end{figure}

We now consider the the cobordism $C_{k}''\cup C_{k+1}''$. For some $a<c_k<c<c_{k+1}<b$, $C_{k}''\cup C_{k+1}''=f^{-1}[a,b]$, where $f^{-1}(c_k)$ contains all of the critical points of index $k$ and $f^{-1}(c_{k+1})$ contains all of the critical points of index $k+1$. Each critical point $w_i^{k}$ of index $k$ is associated with a critical point $w_i^{k+1}$ of index $k+1$. Using Van Kampen's theorem, we can show that $f^{-1}[a,b],f^{-1}(a)$ and $f^{-1}(b)$ are all simply connected, see remark 1 on page 70 of \cite{Smale}. 

Since $\p(b_i^{k+1})=z_{i}^{k}$, each pair of trajectory spheres has intersection $1$ or $-1$. The strong cancellation theorem, Theorem \ref{cancelthm2}, now gives that $f$ can be perturbed so that each pair of trajectory spheres intersects transversely on $f^{-1}(c)$ at a single point and that $f^{-1}[a,b]$ is diffeomorphic to $f^{-1}(a)\times[a,b]$.

Now consider the restriction of the metric $\bar{g}=\bar{g}(g_0,f)$ to $f^{-1}[a,b]$. Let $g_a$ and $g_b$ denote the induced metrics on  $f^{-1}(a)$ and $f^{-1}(b)$ respectively. The trajectories connecting the critical points of the first critical level with trajectory spheres in $f^{-1}(a)$ are mutually disjoint, as are those connecting critical points on the second critical level with the trajectory spheres on $f^{-1}(b)$. In turn, pairs of cancelling critical points can be connected by mutually disjoint arcs where each arc is the union of intersection points of the corresponding trajectory spheres. The metric $g_b$ is therefore obtained from $g_a$ by finitely many independent applications of the construction in Theorem \ref{concisodoublesurgery} and so $g_a$ and $g_b$ are isotopic. By repeating this argument as often as necessary we show that $g_0$ is isotopic to $g_1$. 
\end{proof}

\subsection{The proof of the main theorem}
\indent We can now complete the proof of Theorem \ref{conciso}. To do this, we must extend Theorem \ref{concisoeasy} to deal with the case of index $0$ and index $1$ critical points.
\\

\noindent {\bf Theorem \ref{conciso}.} {\em Let $X$ be a closed simply connected manifold with $dimX=n\geq5$ and let $g_0$ be a positive scalar curvature metric on $X$. Suppose $\bar{g}=\bar{g}(g_0,f)$ is a Gromov-Lawson concordance with respect to $g_0$ and an admissible Morse function $f:X\times I\rightarrow I$. Then the metrics $g_0$ and $g_1=\bar{g}|_{X\times\{1\}}$ are isotopic.}

\begin{proof} We will assume that $f$ is a well-indexed admissible Morse function on $W$. Using the notation of the previous theorem, $f$ decomposes $W$ into a union of cobordisms $C_{0}\cup C_{1}\cup \cdots \cup{C_{n-2}}$. In the case where $f$ has index $0$ critical points, $f$ can be perturbed so that for some $\epsilon>0$, $f^{-1}[0,\epsilon]$ contains all index $0$ critical points along with an equal number of index $1$ critical points. These critical points are arranged so that all index $0$ critical points are on the level $f^{-1}(c_0)$ and all index $1$ critical points are on the level set $f^{-1}(c_1 )$ where $0<c_0<c<c_1<\epsilon$. In theorem 8.1 of \cite{Smale}, it is proved that these critical points can be arranged into pairs of index $0$ and index $1$ critical points where each pair is connected by mutually disjoint arcs and each pair satisfies the conditions of theorem \ref{cancelthm}. Thus, Theorem \ref{concisodoublesurgery} gives that the metric $g_0$ is isotopic to the metric $g_\epsilon=\bar{g}(g_0,f)|_{f^{-1}(\epsilon)}$. 

If $f$ has no other critical points of index $1$, then Theorem \ref{concisoeasy} gives that $g_\epsilon$ is isotopic to $g_{1}$, completing the proof. We thus turn our attention to the case where $f$ has excess index $1$ critical points which do not cancel with critical points of index $0$. Each of these critical points is asscociated with a critical point of index $2$ and the intersection number of the corresponding trajectory spheres is $1$ or $-1$. Unfortunately, theorem \ref{cancelthm2} does not apply here as the presence of index $1$ critical points means the upper boundary component of $W_{1}$ is not simply connected. In turn, this prevents us from applying Theorem \ref{concisodoublesurgery}.

There is however, another way to deal with these excess index $1$ critical points which we will now summarise. It is possible to add in auxiliary pairs of index $2$ and index $3$ critical points. This can be done so that the newly added pairs have trajectory spheres which intersect transversely at a point and so satisfy the conditions of theorem \ref{cancelthm}. Furthermore, for each excess index $1$ critical point, such a pair of auxiliary critical points may be added so that the newly added index $2$ critical point has an incoming trajectory sphere which intersects transversely at a single point with the outgoing trajectory sphere of the index $1$ critical point. This allows us to use theorem \ref{cancelthm} and hence Theorem \ref{concisodoublesurgery} with respect to these index $1$, index $2$ pairs. The old index $2$ critical points now all have index $3$ critical points with which to cancel and so we can apply Theorem \ref{concisoeasy} to complete the proof. In effect, the excess index $1$ critical points are replaced by an equal number of index $3$ critical points. The details of this construction are to be found in the proof of theorem 8.1 of \cite{Smale} and so we will provide only a rough outline. The key result which makes this possible is a theorem by Whitney, which we state below.

\begin{Theorem}\cite{Whitney}\label{Whitney}
If two smooth embeddings of a smooth manifold $M$ of dimension $m$ into a smooth manifold $N$ of dimension $n$ are homotopic, then they are smoothly isotopic provided $n\geq 2m+2$. 
\end{Theorem}

Choose $\delta>0$ so that the metric $\bar{g}$ is a product metric on $X\times [1-\delta,1]$. Thus, $f$ has no critical points here either. On any open neighbourhood $U$ contained inside $f^{-1}[1-\delta,1]$, it is possible to replace the function $f$ with a new function $f_1$, so that outside $U$, $f_1=f$, but inside $U$, $f$ has a pair of critical points, $y$ and $z$ with respective indices $2$ and $3$ and so that on the cylinder $f^{-1}[1-\delta, 1]$, $f_1$ satisfies the conditions of Theorem \ref{cancelthm}. For a detailed proof of this fact, see lemma 8.2 of \cite{Smale}.  

\begin{Remark}
The Morse functions $f$ and $f_1$ are certainly not isotopic, as they have different numbers of critical points. However, this is not a problem as the following comment makes clear.
\end{Remark}

Recall that the metric $\bar{g}|_{X\times \{1-\delta\}}=g_1$ and that our goal is to show that $g_1$ is isotopic to $g_0$. By Theorem \ref{concisodoublesurgery}, the metric $\bar{g}|_{X\times \{1-\delta\}}$ is isotopic to $\bar{g}(g_0, f_1)|_{X\times\{1\}}$ and so it is enough to show that $g_0$ is isotopic to $\bar{g}(g_0, f_1)|_{X\times\{1\}}$ for some such $f_1$.  

For simplicity, we will assume that $f$ has no index $0$ critical points. We will assume that all of the critical points of index $1$ are on the level $f=c_1$. Choose points $a<c_1<b$ so that $f^{-1}[a,b]$ contains no other critical levels except $f^{-1}(c_1)$. Let $w$ be an index $1$ critical point of $f$. Emerging from $w$ is an outward trajectory whose intersection with the level set $f^{-1}(b)$ is an $n-1$-dimensional sphere $S_{+}^{n-1}(z)$. The following lemma is lemma 8.3 of \cite{Smale}. 

\begin{Lemma}\label{idealcircle}
There exists an embedded $1$-sphere $S=S^{1}$ in $f^{-1}(b)$ which intersects transversely with $S_{+}^{n-1}(z)$ at a single point and meets no other outward trajectory sphere. 
\end{Lemma}

Now replace $f$ with the function $f_1$ above. By Theorem \ref{rearrangement}, the function $f_1$ can be isotopied through admissible Morse functions to a well-indexed one $\bar{f_1}$. Consequently, the metric $\bar{g}(g_0, f_1)$ can be isotopied to a Gromov-Lawson concordance $\bar{g}(g_0,\bar{f_1})$. The critical points $y$ and $z$ have now been moved so that $y$ is on the same level as all of the other index $2$ critical points. There is a trajectory sphere $S_{-}^{1}(y)$, which is converging to $y$, embedded in $f^{-1}(b)$. Theorem \ref{Whitney} implies that $\bar{f_1}$ can be isotopied so as to move $S_{-}^{1}(y)$ onto the embedded sphere $S$ of Lemma \ref{idealcircle}. The resulting well-indexed admissible Morse function has the property that the outward trajectory spheres of index $1$ critical points intersect the inward trajectory spheres of their corresponding index $2$ critical points transversely at a point. 

We can make an arbitrarily small adjustment to $\bar{f_1}$ so that the index $2$ critical points which correspond to the kernel of the map $\p:\C_3\rightarrow \C_2$ are on a level set just above the level containing the remaining index $2$ critical points. Let $f^{-1}(c)$ denote a level set between these critical levels. Then $f^{-1}[0,c]$ is diffeomorphic to $X\times [0,c]$ and, by Theorem \ref{concisodoublesurgery}, the metric $g_0$ is isotopic to the metric $\bar{g}(g_0,\bar{f_1})|_{f^{-1}(c)}$. Furthermore, the cobordism $f^{-1}[c,1]$ is diffeomorphic to $X\times [c,1]$ and the restriction of $\bar{f_1}$ satisfies all of the conditions of Theorem \ref{concisoeasy}. This means that $\bar{g}(g_0,\bar{f_1})|_{f^{-1}(c)}$ is isotopic to $\bar{g}(g_0,\bar{f_1})|_{f^{-1}(1)}$, completing the proof.

\end{proof}

\section{Appendix: Curvature calculations from the Surgery Theorem}\label{append}

Below we provide detailed proofs of some Lemmas used in section \ref{surgerysection}, in the proof of Theorem \ref{IsotopyTheorem}. Lemma \ref{GLlemma1app}, in particular is exactly Lemma 1 from \cite{GL1}. We provide a detailed proof below. Lemmas \ref{principalMapp} and \ref{scalarMapp} below are curvature calculations. The resulting formulae arise in Gromov and Lawson's original proof of the Surgery Theorem, see \cite{GL1} or \cite{RS}.  

\begin{Lemma} \label{GLlemma1app} {\rm{(Lemma 1, \cite{GL1})}}

\noindent (a) The principal curvatures of the hypersurfaces $S^{n-1}{(\epsilon})$ in D are each of the form $\frac{-1}{\epsilon}+O(\epsilon)$ for $\epsilon$ small. 

\noindent (b) Furthermore, let $g_\epsilon$ be the induced metric on $S^{n-1}{(\epsilon})$ and let $g_{0,\epsilon}$ be the standard Euclidean metric of curvature $\frac{1}{\epsilon^2}$. Then as $\epsilon\rightarrow 0, \frac{1}{\epsilon^2}{g_\epsilon}\rightarrow{\frac{1}{\epsilon^2}g_{0,\epsilon}}=g_{0,1}$ in the $C^{2}$-topology.

\noindent Below we use the following notation. A function $f(r)$ is $O(r)$ as $r\rightarrow 0$ if $\frac{f(r)}{r}\rightarrow constant$ as $r\rightarrow 0$.\\
\end{Lemma}

\begin{proof}
We begin with the proof of (a). On $D$, in coordinates $x_1,...,x_n$, the metric $g$ has the form

\begin{equation}\label{3.1}
\begin{array}{c}
g_{ij}(x)=\delta_{ij}+ \sum{a_{ij}^{kl}x_{k}x_{l}+O(|x|^3)}=\delta_{ij}+O(|x|^3). 
\end{array}
\end{equation}

\noindent This follows from the Taylor series expansion of $g_{ij}(x)$ around $0$ and the fact that in a normal coordinate neighbourhood of $p=0$, $g_{ij}(0)=\delta_{ij}$ and $\Gamma_{ij}^{k}(0)=0$. 

Next we will show that the Christoffel symbols of the corresponding Levi-Civita connection have the form

\begin{equation*}
\begin{array}{c}
\Gamma_{ij}^{k}=\sum_{l}{\gamma_{ijk,l}}+O(|x|^2)=O(|x|).
\end{array}
\end{equation*}

\noindent Recall that the Christoffel symbols are given by the formula

\begin{equation*}
\begin{array}{c}
\Gamma_{ij}^{k}=\frac{1}{2}\sum_{l}{g^{kl}(g_{il,j}+g_{jl,i}-g_{ij,l})}.
\end{array}
\end{equation*}

\noindent Differentiating (\ref{3.1}), gives

\begin{equation*}
\begin{array}{c}
\frac{\partial}{\partial x_i}g_{jl}=0+\sum{a_{jl}^{st}(x_s \delta_{ti}+x_t \delta_{si})}+\hdots
\end{array}
\end{equation*}

\noindent Hence

\begin{equation*}
\begin{array}{c}
g_{il,j}+g_{jl,i}-g_{ij,l}=O(|x|).
\end{array}
\end{equation*}

\noindent We must now deal with the  $g^{kl}$ terms. Let $(g_{kl})=(I)+(Y)$ where the $I$ is the identity matrix and $Y$ is the matrix $(a_{kl}^{ij}x_{i}x_{j}+O(|x|^3))$. Recall the following elementary fact.\\
\begin{equation*}
\begin{array}{c}
(1+a)^{-1}=1-a+a^2-a^3+\hdots
\end{array}
\end{equation*}

\noindent Thus we can write

\begin{equation*}
\begin{array}{c}
(g^{kl})=(I)-(Y)+(Y)^2-(Y)^3+\hdots
\end{array}
\end{equation*}

\noindent Each component of this matrix has the form\\
\begin{equation*}
\begin{array}{c}
g^{kl}=\delta_{kl}+O(|x|^2).
\end{array}
\end{equation*}

\noindent Finally we obtain
\begin{equation*}
\begin{array}{c}
\Gamma_{ij}^{k}=\frac{1}{2}\sum_{l}{(\delta_{ij}+O(|x|^2))(O(|x|))}=O(|x|).
\end{array}
\end{equation*}

We will now compute the scalar second fundamental form on tangent vectors to the geodesic sphere $S^{n-1}(\epsilon)$. Consider the smooth curve $\alpha$ on $S^{n-1}(\epsilon)$ given by

\begin{equation*}
\begin{array}{c}
\alpha(s)=(\epsilon\cos{\frac{s}{\epsilon}},\epsilon\sin{\frac{s}{\epsilon}},0,\cdots,0). 
\end{array}
\end{equation*}

\noindent Velocity vectors of this curve are tangent vectors to $S^{n-1}(\epsilon)$ and have the form

\begin{equation*}
\begin{array}{c}
\dot{\alpha}(s)=(-\sin{\frac{s}{\epsilon}},\cos{\frac{s}{\epsilon}},0,\cdots,0). 
\end{array}
\end{equation*}

\noindent  Letting $\xi$ denote the exterior unit normal vector field to $S^{n-1}(\epsilon)$, we have

\begin{equation*}
\begin{array}{c}
\xi(\alpha(0))=\frac{\alpha(0)}{|\alpha(0)|}=(1,0,\cdots,0)=e_1.
\end{array}
\end{equation*}

\noindent and 

\begin{equation*}
\begin{array}{c}
\dot{\alpha}(0)=(0,1,0,\cdots,0)=e_2.
\end{array}
\end{equation*}

\noindent We will now proceed to compute the scalar second fundamental form at $\alpha(0)$. We denote by 

\begin{equation*}
\begin{array}{c}
A:T_{\alpha(0)} S^{n-1}(\epsilon)\times T_{\alpha(0)} S^{n-1}(\epsilon)\rightarrow \mathbb{R}, 
\end{array}
\end{equation*}

\noindent the scalar second fundamental form and by 

\begin{equation*}
\begin{array}{c}
S: T_{\alpha(0)} S^{n-1}(\epsilon)\rightarrow T_{\alpha(0)} S^{n-1}(\epsilon),
\end{array}
\end{equation*}

\noindent the shape operator, for the hypersurface $S^{n-1}(\epsilon)\subset D$. Recall that,

\begin{equation*}
\begin{array}{c}
S(X_{\alpha(0)})=-\nabla_{X}\xi,
\end{array}
\end{equation*}

\begin{equation*}
\begin{array}{c}
A(X_{\alpha(0)},Y_{\alpha(0)})=g(S(X_{\alpha(0)}), Y)
\end{array}
\end{equation*}

\noindent where $X,Y$ are tangent vector fields on $S^{n-1}(\epsilon)$, and that $A$ only depends on $X$ and $Y$ at $p$. We now compute

\begin{equation*}
\begin{array}{cl}
A(\dot{\alpha}(0),\dot{\alpha}(0))&=g(S(\dot{\alpha}(0)),\dot{\alpha})\\
&=g(-\nabla_{\dot{\alpha}}\xi,\dot{\alpha})_{\alpha(0)}\\
&=-\dot{\alpha}[g(\xi,\dot{\alpha})](\alpha(0))+g(\nabla_{\dot{\alpha}}{\dot{\alpha}},\xi)_{\alpha(0)}\\
&=0+g(\nabla_{\dot{\alpha}}{\dot{\alpha}},e_1)\\
&=\ddot{\alpha}^{(1)}(0)+\sum_{i,j}{\alpha_{ij}^{1}\dot{\alpha}^{(i)}\dot{\alpha}^{(j)}}(0).
\end{array}
\end{equation*}

\noindent The components of the velocity vector are

\begin{equation*}
\begin{array}{c}
\dot{\alpha}^{(1)}=-\sin{\frac{s}{\epsilon}}, \qquad
\dot{\alpha}^{(2)}=\cos{\frac{s}{\epsilon}}, \qquad
\dot{\alpha}^{(j)}=0, \hspace{2mm}j\geq 3,
\end{array}
\end{equation*}

\noindent while
\begin{equation*}
\begin{array}{c}
\ddot{\alpha}^{(1)}=-\frac{1}{\epsilon}\cos{\frac{s}{\epsilon}}.
\end{array}
\end{equation*}

\noindent Thus,
\begin{equation*}
\begin{array}{c}
(\ddot{\alpha}^{(1)}+\sum_{i,j}{\alpha_{ij}^{1}\dot{\alpha}^{(i)}\dot{\alpha}^{(j)}})(s)=-\frac{1}{\epsilon}\cos{\frac{s}{\epsilon}}+\alpha_{11}^{1}\sin^{2}{\frac{s}{\epsilon}}-2\alpha_{12}^{1}\sin{\frac{s}{\epsilon}}\cos{\frac{s}{\epsilon}}+\alpha_{22}^{1}\cos^{2}{\frac{s}{\epsilon}}.
\end{array}
\end{equation*}

\noindent Hence,
\begin{equation*}
\begin{array}{c}
(\ddot{\alpha}^{(1)}+\sum_{i,j}{\alpha_{ij}^{1}\dot{\alpha}^{(i)}\dot{\alpha}^{(j)}})(0)=-\frac{1}{\epsilon}+\alpha_{22}^{1}(\epsilon,0,\cdots,0)=-\frac{1}{\epsilon}+O(\epsilon).
\end{array}
\end{equation*}

\noindent We now have that $A(\dot{\alpha}(0),\dot{\alpha}(0))=-\frac{1}{\epsilon}+O(\epsilon)$. Finally we need to normalise the vector $\dot{\alpha}(0)$.

\noindent We can write
\begin{equation*}
\begin{array}{c}
A(\dot{\alpha}(0),\dot{\alpha}(0))=|\dot{\alpha}(0)|^{2}A(v,v)
\end{array}
\end{equation*}

\noindent where $v$ is the unit length vector $\frac{\dot{\alpha}(0)}{|\dot{\alpha}(0)|}$.

\begin{equation*}
\begin{array}{cl}
|\dot{\alpha}(0)|^{2}&=g(\dot{\alpha}(0),\dot{\alpha}(0))\\
&=g(e_2,e_2)_{(\epsilon,0,\cdots,0)}\\
&=g_{22}(\epsilon,0,\cdots,0)\\
&=\delta_{22}+(\sum_{k,l}{a_{22}^{kl}{x_k}{x_l}}+O(|x|^3))(\epsilon,0,\cdots,0)\\
&=1+a_{22}^{11}\epsilon^{2}+O(|x|^3)\\
&=1+O(\epsilon^2).
\end{array}
\end{equation*}

\noindent We now have that
\begin{equation*}
\begin{array}{c}
A(\dot{\alpha}(0),\dot{\alpha}(0))=(1+O(\epsilon^2))A(v,v).
\end{array}
\end{equation*}

\noindent That is 
\begin{equation*}
\begin{array}{c}
-\frac{1}{\epsilon}+O(\epsilon)=(1+O(\epsilon^2))A(v,v).
\end{array}
\end{equation*}

\noindent This means that
\begin{equation*}
\begin{array}{c}
A(v,v)=-\frac{1}{\epsilon}+O(\epsilon)+O(\epsilon^2)=-\frac{1}{\epsilon}+O(\epsilon).
\end{array}
\end{equation*}

\noindent By an orthogonal change of coordinates (another choice of orthonormal basis $\{e_1,...e_n\}$), this computation is valid for any unit vector. In particular, it holds if $v$ is a principal direction. Hence the principal curvatures have the desired form. This proves part (a). 

The second part of the lemma is more straightforward. We can compare the induced metrics $g_{\epsilon}$ on $S^{n-1}(\epsilon)$ for decreasing values of $\epsilon$ by pulling back onto $S^{n-1}(1)$ via the map

\centerline{$f_\epsilon:S^{n-1}(1) \rightarrow  S^{n-1}{(\epsilon)}$}
\centerline{$x\longmapsto \epsilon x$}

\noindent Then at a point $x$ where $|x|=1$, we have

\begin{equation*}
\begin{array}{cl}
\frac{1}{\epsilon^2}f_{\epsilon}^{*}(g_\epsilon)(x)&=\sum_{i,j}{g_{ij}(\epsilon x)}dx_i dx_j\\
&=\sum_{i,j}(\delta_{ij}+\epsilon^{2}{\sum_{i,j}{a_{ij}^{kl}x_k x_l}})dx_i dx_j+\epsilon^{3}(\text{higher order terms}).
\end{array}
\end{equation*}

\noindent In the $C^2$-topology (that is, in the zeroth, first and second order terms of the Taylor series expansion), $\frac{1}{\epsilon^2}f_{\epsilon}^{*}(g_\epsilon)$ converges to the standard Euclidean metric in some neighbourhood of $S^{n-1}(1)$ as $\epsilon\rightarrow 0$. As $f_\epsilon$ is a diffeomorphism, the metric $\frac{1}{\epsilon^2}(g_\epsilon)$ is isometric to $\frac{1}{\epsilon^2}f_{\epsilon}^{*}(g_\epsilon)$ and  converges (in $C^{2}$) to the standard metric in some neighbourhood of $S^{n-1}(\epsilon)$. This proves part (b) and completes the proof of Lemma \ref{GLlemma1}.
\end{proof}

Recall, in the proof of Theorem \ref{IsotopyTheorem}, we deform a psc-metric $g$ on a smooth manifold $X$ inside a tubular neighbourhood $N=S^{p}\times D^{q+1}$ of an embedded sphere $S^{p}$. Here $q\geq 2$. We do this by specifying a hypersurface $M$ inside $N\times \mathbb{R}$, shown in Fig. \ref{hypermappendix} and inducing a metric from the ambient metric $g+dt^{2}$. The hypersurface $M$ is defined as

\begin{equation*}
M_{\gamma}=\{(y,x,t)\in S^{p}\times D^{q+1}(\bar{r})\times\mathbb{R}:(r(x),t)\in{\gamma}\}.
\end{equation*}

where $\gamma$ is the curve shown in Fig. \ref{gammaappendix} and $r$ denotes radial distance from $S^{p}$ on $N$. The induced metric is denoted $g_\gamma$. The fact that $\gamma$ is a vertical line near the point $(0,\bar{r})$ means that $g_\gamma=g$, near $\p{N}$. Thus $\gamma$ specifies a metric on $X$ which is the orginal metric $g$ outside of $N$ and then transitions smoothly to the metric $g_\gamma$. For a more detailed description, see section \ref{surgerysection}. In the following lemmas we compute the scalar curvature of $g_\gamma$.

\begin{figure}[htbp]
\begin{picture}(0,0)%
\includegraphics[height=60mm]{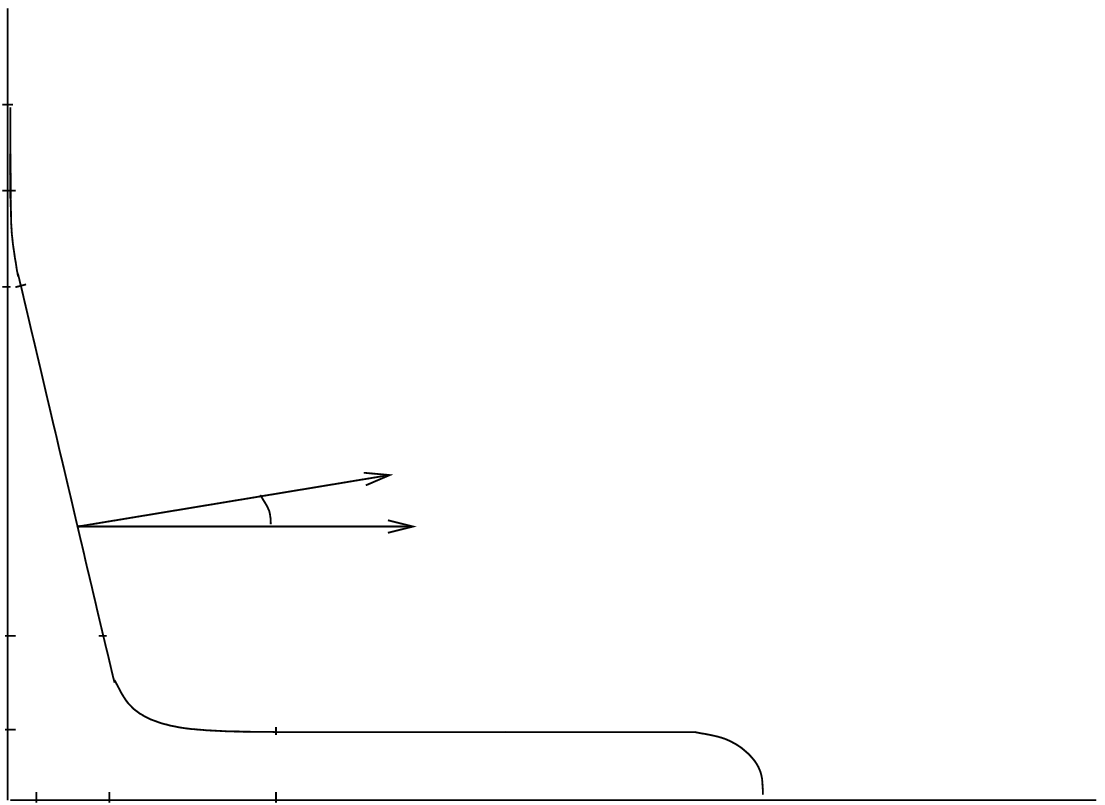}%
\end{picture}%
\setlength{\unitlength}{3947sp}%
\begingroup\makeatletter\ifx\SetFigFont\undefined%
\gdef\SetFigFont#1#2#3#4#5{%
  \reset@font\fontsize{#1}{#2pt}%
  \fontfamily{#3}\fontseries{#4}\fontshape{#5}%
  \selectfont}%
\fi\endgroup%
\begin{picture}(3889,2884)(1374,-4954)
\put(1180,-4690){\makebox(0,0)[lb]{\smash{{\SetFigFont{8}{8}{\rmdefault}{\mddefault}{\updefault}{\color[rgb]{0,0,0}$r_\infty$}%
}}}}
\put(1180,-4375){\makebox(0,0)[lb]{\smash{{\SetFigFont{8}{8}{\rmdefault}{\mddefault}{\updefault}{\color[rgb]{0,0,0}$r_0$}%
}}}}
\put(1180,-2825){\makebox(0,0)[lb]{\smash{{\SetFigFont{8}{8}{\rmdefault}{\mddefault}{\updefault}{\color[rgb]{0,0,0}$r_1$}%
}}}}
\put(1180,-3145){\makebox(0,0)[lb]{\smash{{\SetFigFont{8}{8}{\rmdefault}{\mddefault}{\updefault}{\color[rgb]{0,0,0}$r_1'$}%
}}}}
\put(1180,-2553){\makebox(0,0)[lb]{\smash{{\SetFigFont{8}{8}{\rmdefault}{\mddefault}{\updefault}{\color[rgb]{0,0,0}$\bar{r}$}%
}}}}
\put(1700,-5088){\makebox(0,0)[lb]{\smash{{\SetFigFont{8}{8}{\rmdefault}{\mddefault}{\updefault}{\color[rgb]{0,0,0}$t_0$}%
}}}}
\put(2294,-5088){\makebox(0,0)[lb]{\smash{{\SetFigFont{8}{8}{\rmdefault}{\mddefault}{\updefault}{\color[rgb]{0,0,0}$t_\infty$}%
}}}}
\put(2414,-3938){\makebox(0,0)[lb]{\smash{{\SetFigFont{8}{8}{\rmdefault}{\mddefault}{\updefault}{\color[rgb]{0,0,0}$\theta_0$}%
}}}}
\put(1800,-4375){\makebox(0,0)[lb]{\smash{{\SetFigFont{8}{8}{\rmdefault}{\mddefault}{\updefault}{\color[rgb]{0,0,0}$(t_0,r_0)$}%
}}}}
\put(4159,-5088){\makebox(0,0)[lb]{\smash{{\SetFigFont{8}{8}{\rmdefault}{\mddefault}{\updefault}{\color[rgb]{0,0,0}$\bar{t}$}%
}}}}
\put(2294,-4640){\makebox(0,0)[lb]{\smash{{\SetFigFont{8}{8}{\rmdefault}{\mddefault}{\updefault}{\color[rgb]{0,0,0}$(t_\infty, r_\infty)$}%
}}}}
\put(1430,-5088){\makebox(0,0)[lb]{\smash{{\SetFigFont{8}{8}{\rmdefault}{\mddefault}{\updefault}{\color[rgb]{0,0,0}$t_1'$}%
}}}}
\put(1500,-3145){\makebox(0,0)[lb]{\smash{{\SetFigFont{8}{8}{\rmdefault}{\mddefault}{\updefault}{\color[rgb]{0,0,0}$(t_1', r_1')$}%
}}}}
\end{picture}%
\caption{The curve $\gamma$}
\label{gammaappendix}
\end{figure}

\begin{figure}[htbp]
\begin{picture}(0,0)%
\includegraphics{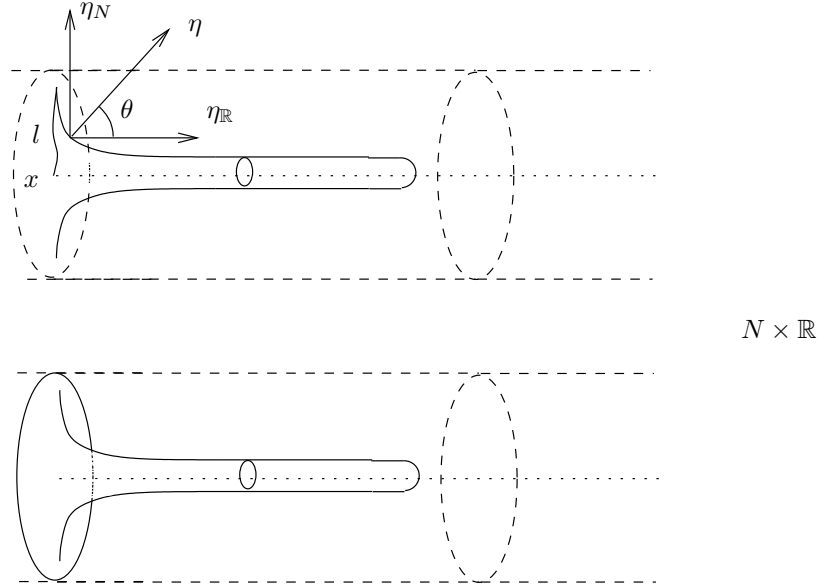}%
\end{picture}%
\setlength{\unitlength}{3947sp}%
\begingroup\makeatletter\ifx\SetFigFont\undefined%
\gdef\SetFigFont#1#2#3#4#5{%
  \reset@font\fontsize{#1}{#2pt}%
  \fontfamily{#3}\fontseries{#4}\fontshape{#5}%
  \selectfont}%
\fi\endgroup%
\begin{picture}(4666,3733)(1218,-3861)
\put(5869,-2318){\makebox(0,0)[lb]{\smash{{\SetFigFont{10}{8}{\rmdefault}{\mddefault}{\updefault}{\color[rgb]{0,0,0}$N\times\mathbb{R}$}
}}}}
\put(1716,-282){\makebox(0,0)[lb]{\smash{{\SetFigFont{10}{8}{\rmdefault}{\mddefault}{\updefault}{\color[rgb]{0,0,0}$\eta_N$}
}}}}
\put(2398,-386){\makebox(0,0)[lb]{\smash{{\SetFigFont{10}{8}{\rmdefault}{\mddefault}{\updefault}{\color[rgb]{0,0,0}$\eta$}
}}}}
\put(2512,-933){\makebox(0,0)[lb]{\smash{{\SetFigFont{10}{8}{\rmdefault}{\mddefault}{\updefault}{\color[rgb]{0,0,0}$\eta_{\mathbb{R}}$}
}}}}
\put(1972,-933){\makebox(0,0)[lb]{\smash{{\SetFigFont{10}{8}{\rmdefault}{\mddefault}{\updefault}{\color[rgb]{0,0,0}$\theta$}
}}}}
\put(1365,-1367){\makebox(0,0)[lb]{\smash{{\SetFigFont{10}{8}{\rmdefault}{\mddefault}{\updefault}{\color[rgb]{0,0,0}$x$}
}}}}
\put(1417,-1088){\makebox(0,0)[lb]{\smash{{\SetFigFont{10}{8}{\rmdefault}{\mddefault}{\updefault}{\color[rgb]{0,0,0}$l$}
}}}}
\end{picture}
\caption{The hypersurface $M$ in $N\times\mathbb{R}$, the sphere $S^{p}$ is represented schematically as a pair of points}
\label{hypermappendix}
\end{figure}

\begin{Lemma}\label{principalMapp}
The prinipal curvatures to $M$ with respect to the outward unit normal vector field have the form
\begin{equation}
\begin{array}{cl}
\lambda_j =
\begin{cases}
k & \text{if $j=1$}\\
(-\frac{1}{r}+O(r))\sin{\theta} & \text{if $2\leq j\leq q+1$}\\
O(1)\sin{\theta} & \text{if $q+2\leq j\leq n$}.
\end{cases}
\end{array}
\end{equation} 
Here $k$ is the curvature of $\gamma$, $\theta$ is the angle between the outward normal vector $\eta$ and the horizontal (or the outward normal to the curve $\gamma$ and the $t$-axis) and the corresponding principle directions $e_j$ are tangent to the curve $\gamma$ when $j=1$, the fibre sphere $S^{q}$ when $2\leq j\leq q+1$ and $S^{p}$ when $q+2\leq j\leq n$. 
\end{Lemma}

\begin{proof}
Let $w=(y,x,t)\in S^{p}\times D^{q+1}\times\mathbb{R}$ be a point on $M$. Let $l$ be the geodesic ray emanating from $y\times\{0\}$ in $N$ through the point $(y,x)$. The surface $l\times\mathbb{R}$ in $N\times R$ can be thought of as an embedding of $[0,\bar{r})\times \mathbb{R}$, given by the map $(r,t)\mapsto (l_r,t)$ where $l_r$ is the point on $l$ of length $r$ from $y\times\{0\}$. We will denote by $\gamma_{l}$, the curve $M\cap{l\times{\mathbb{R}}}$. This can be parameterised by composing the parameterisation of $\gamma$ with the above embedding. We will denote by $\dot{\gamma_l}_w$, the velocity vector of this curve at $w$. Finally we denote by $\eta$, the outward pointing unit normal vector field to $M$.

We now make a couple of observations. 

\noindent (a) The surface $l\times\mathbb{R}$ is a totally geodesic surface in $N\times\mathbb{R}$. This can be seen from the fact that any geodesic in $l\times\mathbb{R}$ projects onto geodesics in $l$ and $\mathbb{R}$. But $D\times\mathbb{R}$ is a Riemannian product and so such a curve is therefore a geodesic in $D\times\mathbb{R}$. 

\noindent (b) The vector $\eta$ is tangential to $l\times\mathbb{R}$. This can be seen by decomposing $\eta$ into orthogonal components
\begin{equation*}
\begin{array}{c}
\eta=\eta_N+\eta_\mathbb{R}.
\end{array}
\end{equation*}
\noindent Here $\eta_N$ is tangent to $N$ and $\eta_\mathbb{R}$ is tangent to $\mathbb{R}$. Now $\eta_N$ is orthogonal to the geodesic sphere $S^{q}(r)_y$, centered at $y\times\{0\}$ with radius $r=|x|$. By Gauss's Lemma, we know that $l$ runs orthogonally through $S^{q}(r)_y$ and so $\eta_N$ is tangent to $l$. Hence $\eta$ is tangent to $l\times {\mathbb{R}}$.

We will now show that $\gamma_l$ is a principal curve in $M$. Let $S^M$ denote the shape operator for $M$ in $N\times \mathbb{R}$ and $S^{\gamma_l}$, the shape operator for $\gamma_l$ in $l\times \mathbb{R}$. Both shape operators are defined with respect to $\eta$.

\begin{equation*}
\begin{array}{cl}
S^M(\dot{\gamma_l})&=-\nabla^{N\times{\mathbb{R}}}_{\dot{\gamma_l}}\eta\\
&=(-\nabla^{N\times{\mathbb{R}}}_{\dot{\gamma_l}}\eta)^{T}+(-\nabla^{N\times{\mathbb{R}}}_{\dot{\gamma_l}}\eta)^{\perp}\\
&=-\nabla^{l\times{\mathbb{R}}}_{\dot{\gamma_l}}\eta+0\\
&=S^{\gamma_l}(\dot{\gamma_l}).
\end{array}
\end{equation*} 

The third equality is a direct consequence of the fact that $l\times\mathbb{R}$ is a totally geodesic surface in $N\times\mathbb{R}$. Now as $T{\gamma_l}$, the tangent bundle of the curve $\gamma_l$, is a one-dimensional bundle, $\dot{\gamma_l}_w$ must be an eigenvector of $S^M$. Hence $\gamma_l$ is a principal curve. The corresponding principal curvature is of course the curvature of $\gamma$, which we denote by $k$.

At $w$, we denote the principal direction $\dot{\gamma_l}_w$ by $e_1$. The other principal directions we denote by $e_2,\cdots, e_n$, where $e_2,\cdots,e_{q+1}$ are tangent to the $S^{q}(r)$ factor and $e_{q+2},\cdots,e_{n}$ are tangent to $S^{p}$. Recall that the set $\{e_1,\cdots,e_n\}$ forms an orthonormal basis for $T_w M$. The corresponding principal curvatures will be denoted $\lambda_1=k,\lambda_2,\cdots,\lambda_n$. Our next task is to compute these principal curvatures.

Let $A$ denote the second fundamental form for $M$ in $N\times{\mathbb{R}}$ with respect to the outward normal vector $\eta$. Let $A^N$ denote the second fundamental form for $S^{p}\times S^{q}(r)$ in $N$, again with respect to $\eta$, and $\lambda_j^{N}$ the corresponding principal curvatures. When $2\leq j\leq n$,

\begin{equation*}
\begin{array}{cl}
\lambda_j&=A(e_j,e_j)\\
&=-g(\nabla_{e_j}^{N\times\mathbb{R}}\eta,e_j)\\
&=-g(\nabla_{e_j}^{N\times\mathbb{R}}({\cos{\theta}{\p_t}}+{\sin{\theta}\p_r}),e_j)\\
&=-g(\nabla_{e_j}^{N\times\mathbb{R}}{\cos{\theta}{\p_t}},e_j)-g(\nabla_{e_j}^{N\times\mathbb{R}}{\sin{\theta}\p_r},e_j).
\end{array}
\end{equation*}

\noindent where $\p_t$ and $\p_r$ are the coordinate vector fields for the $t$ and $r$ coordinates respectively. Now, 

\begin{equation*}
\begin{array}{cl}
\nabla_{e_j}^{N\times\mathbb{R}}{\cos{\theta}{\p_t}}&=\cos{\theta}\nabla_{e_j}^{N\times\mathbb{R}}{{\p_t}}+\partial_j(\cos{\theta})\cdot\p_t\\
&=\cos{\theta}\cdot 0+0\cdot\p_t\\
&=0.
\end{array}
\end{equation*}

\noindent However,

\begin{equation*}
\begin{array}{cl}
\nabla_{e_j}^{N\times\mathbb{R}}{\sin{\theta}\p_r}&=\sin{\theta}\nabla_{e_j}^{N\times\mathbb{R}}\p_r+\partial_j(\sin{\theta})\cdot\p_r\\
&=\sin{\theta}\nabla_{e_j}^{N\times\mathbb{R}}\p_r+0\cdot\p_r\\
&=\sin{\theta}\nabla_{e_j}^{N\times\mathbb{R}}\p_r.
\end{array}
\end{equation*}

\noindent Hence,

\begin{equation*}
\begin{array}{cl}
\lambda_j&=-\sin{\theta}.g(\nabla_{e_j}^{N\times\mathbb{R}}\p_r,e_j)\\
&=\sin{\theta}\cdot A^N(e_j,e_j)\\
&=\sin{\theta}\cdot\lambda_j^{N}.
\end{array}
\end{equation*}

\noindent We know from Lemma \ref{GLlemma1} that when $2\leq j\leq q+1$, $\lambda_j^{N}=-\frac{1}{r}+O(r)$. When $q+2\leq j\leq n$, $\lambda_j^{N}=O(1)$ as here the curvature is bounded. Hence the principal curvatures to $M$ are

\begin{equation*}
\begin{array}{c}
\lambda_j =
\begin{cases}
k & \text{if $j=1$}\\
(-\frac{1}{r}+O(r))\sin{\theta} & \text{if $2\leq j\leq q+1$}\\
O(1)\sin{\theta} & \text{if $q+2\leq j\leq n$}.
\end{cases}
\end{array}
\end{equation*} 

\end{proof}

\begin{Lemma}\label{scalarMapp}
The scalar curvature of the metric induced on $M$ is given by
\begin{equation}\label{scalarMeqnapp}
\begin{split}
R^{M}&=R^{N}+\sin^{2}{\theta}\cdot O(1)-2k\cdot q\frac{\sin{\theta}}{r}\\
&\hspace{0.4cm} +2q(q-1)\frac{\sin^{2}{\theta}}{r^2}+k\cdot qO(r)\sin{\theta}.
\end{split}
\end{equation}
\end{Lemma}

\begin{proof}
\noindent The Gauss Curvature Equation gives that

\begin{equation*}
\begin{array}{c}
\frac{1}{2}R^{M}=\sum_{i<j}(K_{ij}^{N\times\mathbb{R}}+\lambda_i\lambda_j)
\end{array}
\end{equation*}
where $K^{N\times\mathbb{R}}$ denotes sectional curvature on $N\times\mathbb{R}$. Before we continue we should examine $K^{N\times\mathbb{R}}$. When $2\leq i,j\leq n$, 

\begin{equation*}
\begin{array}{c}
K_{ij}^{N\times\mathbb{R}}=K_{ij}^{N}.
\end{array}
\end{equation*}

\noindent When $2\leq j\leq n$,

\begin{equation}\label{earlier}
\begin{array}{cl}
K_{ij}^{N\times\mathbb{R}}&= {Rm}^{N\times\mathbb{R}}(e_1,e_j,e_j,e_1)\\
&=Rm^{N\times\mathbb{R}}(-\cos{\theta}\p_r+\sin{\theta}\p_t,e_j,e_j,-\cos{\theta}\p_r+\sin{\theta}\p_t)\\
&=\cos^{2}{\theta}\cdot Rm^{N\times\mathbb{R}}(\p_r,e_j,e_j,\p_r) +\sin^{2}{\theta}\cdot Rm^{N\times\mathbb{R}}(\p_t,e_j,e_j,\p_t)\\
&=\cos^{2}{\theta}\cdot Rm^{N\times\mathbb{R}}(\p_r,e_j,e_j,\p_r)+\sin^{2}{\theta}\cdot 0\\
&=\cos^{2}{\theta}\cdot Rm^{N}(\p_r,e_j,e_j,\p_r)\\
&=\cos^{2}{\theta}\cdot K_{\p_rj}^{N}.
\end{array}
\end{equation}

\noindent Now,
 
\begin{equation*}
\begin{array}{cl}
\frac{1}{2}R^{M}&=\sum_{i<j}(K_{ij}^{N\times\mathbb{R}}+\lambda_i\lambda_j)\\
&=\sum_{j\leq 2}K_{1j}^{N\times\mathbb{R}}+\sum_{1\neq i<j}K_{ij}^{N\times\mathbb{R}}+\sum_{i<j}\lambda_i\lambda_j. 
\end{array}
\end{equation*}

\noindent From (\ref{earlier}),

\begin{equation*}
\begin{array}{c}
\sum_{j\leq 2}K_{1j}^{N\times\mathbb{R}}=(1-\sin^{2}{\theta})\sum_{j\geq 2}K_{\p_rj}^{N}.
\end{array}
\end{equation*}

\noindent Hence,

\begin{equation*}
\begin{array}{c}
\sum_{j\leq 2}K_{1j}^{N\times\mathbb{R}}+\sum_{1\neq i<j}K_{ij}^{N\times\mathbb{R}}=\frac{1}{2}R^{N}-\sin^{2}{\theta}\cdot Ric^N(\p_r,\p_r).
\end{array}
\end{equation*}

\noindent Next we deal with $\sum_{i<j}\lambda_i\lambda_j$.

\begin{equation*}
\begin{array}{cl}
\sum_{i<j}\lambda_i\lambda_j&= k\sum_{j\geq 2}\lambda_j+\sum_{2\leq i<j\leq q+1}\lambda_i\lambda_j+\sum_{2\leq i\leq q+1,q+2\leq j\leq n}\lambda_i\lambda_j+\sum_{q+2\leq i<j\leq n}\lambda_i\lambda_j\\\\
&=k\cdot{q}(-\frac{1}{r}+O(r))\sin{\theta}+kO(1)\sin{\theta}\\
&\hspace{0.4cm} +q(q-1)(-\frac{1}{r}+O(r))^{2}\sin^{2}{\theta}\\
& \hspace{0.4cm}+q(-\frac{1}{r}+O(r))O(1)\sin^{2}{\theta}\\
&\hspace{0.4cm} +O(1)\sin^{2}{\theta}.
\end{array}
\end{equation*}

\noindent Thus,

\begin{equation*}
\begin{array}{cl}
\frac{1}{2}R^{M}&=\frac{1}{2}R^{N}-\sin^{2}{\theta}\cdot Ric^{N}(\p_r,\p_r)\\
&\hspace{0.4cm}-q\frac{k\sin{\theta}}{r}+k\cdot qO(r)\sin{\theta}+k\cdot O(1)\sin{\theta}\\
&\hspace{0.4cm} +q(q-1)(\frac{1}{r^{2}}+O(1))\sin^{2}{\theta}\\
& \hspace{0.4cm}-\frac{q}{r}\sin^{2}{\theta}\cdot O(1)+q\sin^{2}{\theta}\cdot O(r)\\
&\hspace{0.4cm} +\sin^{2}{\theta}\cdot O(1)\\\\
&=\frac{1}{2}R^{N}-\sin^{2}{\theta}\cdot [Ric^{N}(\p_r,\p_r)+q(q-1)O(1)+O(1)+qO(r)]\\
&\hspace{0.4cm} -k\cdot q\frac{\sin{\theta}}{r}+q(q-1)\frac{\sin^{2}{\theta}}{r^2}+k\cdot qO(r)\sin{\theta}\\
&\hspace{0.4cm} +k\cdot O(1)\sin{\theta}-\frac{q}{r}\sin^{2}{\theta}\cdot O(1)\\\\
&=\frac{1}{2}R^{N}+\sin^{2}{\theta}\cdot O(1)-[k\cdot q\frac{\sin{\theta}}{r}-k\cdot O(1)\sin{\theta}]\\
&\hspace{0.4cm} +[q(q-1)\frac{\sin^{2}{\theta}}{r^2}-\frac{q}{r}\sin^{2}{\theta}\cdot O(1)]\\
&\hspace{0.4cm} +k\cdot qO(r)\sin{\theta}.
\end{array}
\end{equation*}

\noindent When $r$ is small, this reduces to (\ref{scalarMeqnapp}), i.e.

\begin{equation*}
\begin{array}{cl}
R^{M}&=R^{N}+\sin^{2}{\theta}\cdot O(1)-2k\cdot q\frac{\sin{\theta}}{r}\\
&\hspace{0.4cm} +2q(q-1)\frac{\sin^{2}{\theta}}{r^2}+k\cdot qO(r)\sin{\theta}.
\end{array}
\end{equation*}
\end{proof}

\vspace{1cm}

\noindent {\em E-mail address:} mwalsh3@uoregon.edu

\end{document}